\documentclass[a4paper]{amsart}
\usepackage{silence}
\usepackage{amsmath,amsthm,amssymb,latexsym,epic,bbm,comment,amscd}
\usepackage{thmtools,mathtools}
\usepackage{graphicx,color}
\usepackage{stmaryrd}
\usepackage[shortlabels]{enumitem}
\SetSymbolFont{stmry}{bold}{U}{stmry}{m}{n}
\usepackage[all,2cell]{xy}
\xyoption{2cell}

\allowdisplaybreaks[4]

\usepackage[all]{xy}
\usepackage[active]{srcltx}

\usepackage[table]{xcolor}
\usepackage{array}
\newcolumntype{C}{>{$}c<{$}}

\newcommand{\mone}[1][1]{{\,\text{-}#1}}

\font\sc=rsfs10
\newcommand{\cC}{\sc\mbox{C}\hspace{1.0pt}}
\newcommand{\cCH}{\sc\mbox{C}_{\mathcal{H}}\hspace{1.0pt}}

\newcommand{\cR}{\sc\mbox{R}\hspace{1.0pt}}

\newcommand{\cS}{\sc\mbox{S}\hspace{1.0pt}}
\newcommand{\cSo}{\sc\mbox{S}^{\,(0)}\hspace{1.0pt}}
\newcommand{\cSH}{\sc\mbox{S}_{\mathcal{H}}\hspace{1.0pt}}
\newcommand{\cSHo}{\sc\mbox{S}^{\,(0)}_{\mathcal{H}}\hspace{1.0pt}}
\newcommand{\cSJ}{\sc\mbox{S}_{\mathcal{J}}\hspace{1.0pt}}

\newcommand{\cD}{\sc\mbox{D}\hspace{1.0pt}}

\newcommand{\cA}{\sc\mbox{A}\hspace{1.0pt}}
\newcommand{\cAH}{\sc\mbox{A}_{\mathcal{H}}\hspace{1.0pt}}
\newcommand{\cAJ}{\sc\mbox{A}_{\mathcal{J}}\hspace{1.0pt}}
\newcommand{\cAHop}{\sc\mbox{A}^{\;\mathrm{op}}_{\mathcal{H}}\hspace{1.0pt}}

\newcommand{\cAHp}{\sc\mbox{A}^{\;\prime}_{\mathcal{H}}\hspace{1.0pt}}
\newcommand{\cAHpzero}{(\sc\mbox{A}^{\;\prime}_{\mathcal{H}}\hspace{1.0pt})^{(0)}}
\newcommand{\cAJp}{\sc\mbox{A}^{\;\prime}_{\mathcal{J}}\hspace{1.0pt}}

\newcommand{\cAHpop}{\sc\mbox{A}^{\;\prime,\mathrm{op}}_{\mathcal{H}}\hspace{1.0pt}}
\newcommand{\cAJpop}{\sc\mbox{A}^{\;\prime,\mathrm{op}}_{\mathcal{J}}\hspace{1.0pt}}

\newcommand{\cBH}{\sc\mbox{B}_{\mathcal{H}}\hspace{1.0pt}}
\newcommand{\cBHo}{\sc\mbox{B}^{(0)}_{\mathcal{H}}\hspace{1.0pt}}
\newcommand{\cBHop}{\sc\mbox{B}^{\,\mathrm{op}}_{\mathcal{H}}\hspace{1.0pt}}
\newcommand{\cBJ}{\sc\mbox{B}_{\mathcal{J}}\hspace{1.0pt}}
\newcommand{\cBJop}{\sc\mbox{B}^{\,\mathrm{op}}_{\mathcal{J}}\hspace{1.0pt}}
\newcommand{\cBJo}{\sc\mbox{B}^{(0)}_{\mathcal{J}}\hspace{1.0pt}}
\newcommand{\cE}{\sc\mbox{E}\hspace{1.0pt}}

\newcommand{\cO}{\sc\mbox{O}\hspace{1.0pt}}

\newcommand{\cV}{\sc\mbox{V}\hspace{1.0pt}}
\font\scc=rsfs7
\newcommand{\ccA}{\scc\mbox{A}\hspace{1.0pt}}
\newcommand{\ccAH}{\scc\mbox{A}_{\mathcal{H}}\hspace{1.0pt}}
\newcommand{\ccAHp}{\scc\mbox{A}^{\;\prime}_{\mathcal{H}}\hspace{1.0pt}}
\newcommand{\ccAJp}{\scc\mbox{A}^{\;\prime}_{\mathcal{J}}\hspace{1.0pt}}

\newcommand{\ccBH}{\scc\mbox{B}_{\mathcal{H}}\hspace{1.0pt}}
\newcommand{\ccBHo}{\scc\mbox{B}^{\,(0)}_{\mathcal{H}}\hspace{1.0pt}}

\newcommand{\ccC}{\scc\mbox{C}\hspace{1.0pt}}

\newcommand{\ccE}{\scc\mbox{E}\hspace{1.0pt}}

\newcommand{\ccS}{\scc\mbox{S}\hspace{1.0pt}}
\newcommand{\ccSH}{\scc\mbox{S}_{\mathcal{H}}\hspace{1.0pt}}

\newcommand{\ccR}{\scc\mbox{R}\hspace{1.0pt}}

\newcommand{\cccSH}{\scc\mbox{S}_{\mathcal{H}}\hspace{1.0pt}}

\newcommand{\cEnd}{\cE\mathrm{nd}}
\newcommand{\ccEnd}{\ccE\mathrm{nd}}

\newcommand{\CH}{\mathbf{C}_{\mathcal{H}}}
\newcommand{\CHo}{\mathbf{C}^{\,(0)}_{\mathcal{H}}}
\newcommand{\CL}{\mathbf{C}_{\mathcal{L}}}

\newcommand{\Cd}{\mathrm{C}_{d}}
\newcommand{\sd}{\square_{\Cd}}

\usepackage{tikz,tikz-cd}
\tikzset{anchorbase/.style={baseline={([yshift=-0.5ex]current bounding box.center)}},
smallnodes/.style={font=\scriptsize,text height=0.75ex,text depth=0.15ex},
cstrand/.style={line width=1.5,color=black,solid},
dstrand/.style={line width=1.5,color=blue,densely dotted},
xstrand/.style={line width=1.5,color=red,densely dotted},
}
\usetikzlibrary{decorations}
\usetikzlibrary{decorations.markings}
\usetikzlibrary{decorations.pathreplacing}
\usetikzlibrary{decorations.pathmorphing}
\tikzstyle directed=[postaction={decorate,decoration={markings,
mark=at position #1 with {\arrow[draw=black, line width=0.3mm]{>}}}}]
\tikzstyle rdirected=[postaction={decorate,decoration={markings,
mark=at position #1 with {\arrow[draw=black, line width=0.3mm]{<}}}}]
\tikzstyle marked=[postaction={decorate,decoration={markings,
mark=at position #1 with {\draw[ultra thin,black,fill=black] (0,0) circle (.08cm);}}}]
\tikzstyle box=[minimum height=0.4cm,draw,rounded corners, draw,rectangle,solid]

\newtheorem{theoremm}{Theorem}[section]

\declaretheorem[style=plain,name=Theorem,numberlike=theoremm]{theorem}
\declaretheorem[style=plain,name=Lemma,numberlike=theoremm]{lemma}
\declaretheorem[style=plain,name=Proposition,numberlike=theoremm]{proposition}
\declaretheorem[style=plain,name=Corollary,numberlike=theoremm]{corollary}

\declaretheorem[style=remark,name=Example,numberlike=theorem]{example}
\declaretheorem[style=remark,name=Remark,numberlike=theorem]{remark}

\numberwithin{equation}{section}

\usepackage{hyperref}
\hypersetup{
pdftoolbar=true,
pdfmenubar=true,
pdffitwindow=false,
pdfstartview={FitH},
pdftitle={Simple transitive \texorpdfstring{$2$}{2}-representations of Soergel bimodules for finite Coxeter types},
pdfauthor={Marco Mackaay, Volodymyr Mazorchuk, Vanessa Miemietz, Daniel Tubbenhauer and Xiaoting Zhang},
pdfsubject={},
pdfcreator={Marco Mackaay, Volodymyr Mazorchuk, Vanessa Miemietz, Daniel Tubbenhauer and Xiaoting Zhang},
pdfproducer={Marco Mackaay, Volodymyr Mazorchuk, Vanessa Miemietz, Daniel Tubbenhauer and Xiaoting Zhang},
pdfkeywords={},
pdfnewwindow=true,
colorlinks=true,
linkcolor=black,
citecolor=black,
filecolor=black,
urlcolor=black,
linkbordercolor=black,
citebordercolor=black,
urlbordercolor=black,
linktocpage=true
}

\newcommand{\nnfootnote}[1]{%
\begin{NoHyper}
\renewcommand\thefootnote{}\footnote{#1}%
\addtocounter{footnote}{-1}%
\end{NoHyper}
}

\begin{document}
\title[Simple transitive $2$-representations of Soergel bimodules]
{Simple transitive $2$-representations of Soergel bimodules for finite Coxeter types}

\author[M. Mackaay, V. Mazorchuk, V. Miemietz, D. Tubbenhauer and X. Zhang]
{Marco Mackaay, Volodymyr Mazorchuk, Vanessa Miemietz,\\ Daniel Tubbenhauer and Xiaoting Zhang}

\begin{abstract}
In this paper we show that Soergel bimodules for
finite Coxeter types have only finitely many equivalence classes of simple transitive $2$-representations and we complete their classification in all types but $H_{3}$ and $H_{4}$.
\end{abstract}

\nnfootnote{\textit{Mathematics Subject Classification 2020.} Primary: 18N10, 18D25; Secondary: 16D20, 18A50, 18N25, 20C08.}
\nnfootnote{\textit{Keywords.} $2$-categories, $2$-representations, Soergel bimodules, Lusztig's asymptotic bicategories.}

\maketitle
\setcounter{tocdepth}{1}
\tableofcontents

\setlength{\parskip}{5pt}


\section{Introduction}\label{section:introduction}


The series of papers \cite{MM1}, \cite{MM2}, \cite{MM3}, \cite{MM4}, \cite{MM5}, \cite{MM6}
laid the foundations for a systematic theory of finitary
$2$-representations of finitary $2$-categories, which is a categorical
analog of the theory of finite dimensional representations of
finite dimensional algebras. This theory is a part of
modern $2$-representation theory which originated in seminal papers
like \cite{Os2}, \cite{ENO}, \cite{ChRo} and \cite{KhLa}.
The main difference to the ``alternative''
theory of module categories over tensor categories,
see \cite{EGNO} and references therein,
is not that $2$-categories can have more than one object: every
finitary $2$-category gives rise to a multifinitary monoidal category with the same
finitary $2$-representation theory and vice versa, see \cite[Remark 4.3.7]{EGNO} and
\cite[Subsection 2.4]{MMMTZ2}. The main difference is that finitary $2$-representations and
$2$-categories are not assumed to be abelian, but additive and idempotent complete. This is
important for the main character of this paper, the one-object $\mathbb{C}$-linear
$2$-category $\cS=\cS(W,S)$ of Soergel bimodules for a Coxeter system $(W,S)$ of finite type
over the coinvariant algebra of $W$, which is finitary
but not abelian. (Our choice of ground field in this paper is crucial, as $\cS$ behaves very
differently over a field of positive characteristic.) One can, of course, consider an
abelianization $\underline{\cC}$ of a finitary $2$-category $\cC$, but such abelianizations
do not admit adjunctions in general, which is a serious obstruction. Recall that a finitary
$2$-category is (quasi) fiat if it has a duality structure satisfying conditions akin to those
for (rigid or) pivotal monoidal categories (this is the categorical analog of an algebra having
an involution). Tensor categories are rigid by assumption and many of the
structural results on their module categories do not hold without that assumption, see \cite{EGNO}, but $\underline{\cC}$ is not equivalent to a multitensor
category for a general quasi fiat $2$-category $\cC$.
For this reason, there are many differences between
the representation theories of quasi fiat $2$-categories and multitensor categories, e.g.
cell structures have played a key role in the first one but no role at all in the second. Even the
formulation of analogous results in the two theories, and the respective proofs,
can have important differences, e.g. the role of coalgebras in quasi fiat $2$-categories
is analogous to that of algebras in multitensor categories, but technically more involved,
compare e.g. \cite[Section 4]{MMMT} and \cite[Subsections 7.8 to 7.10]{EGNO}.
However, the two representation theories coincide
when $\cC$ is fusion or pivotal fusion, i.e. locally semisimple (meaning that morphism
categories are semisimple) and quasi fiat or fiat, which is very helpful for
the purpose of this paper, as we will explain below.

A key result in finitary $2$-representation theory is the (weak) Jordan--H\"{o}lder theorem \cite[Theorem 8]{MM5}
for finitary $2$-representations, in which the role of the simples is played by the so-called
simple transitive $2$-representations. Classifying
the simple transitive $2$-representations of a given finitary $2$-category $\cC$
is therefore a fundamental problem in $2$-representation theory, which we will refer to as the
\emph{Classification Problem}. When $\cC$ is fusion, there are only finitely many equivalence
classes of simple transitive $2$-representations by
(a consequence of) Ocneanu rigidity, see e.g. \cite[Proposition 3.4.6 and
Corollary 9.1.6]{EGNO}. For certain fusion $2$-categories
those equivalence classes have even been classified explicitly,
see Section \ref{appendix}. However, when $\cC$ is not locally semisimple,
that finiteness result need not be true, see e.g. \cite[Subsection 4.3]{EO},
and in most cases the Classification Problem is unsolved.

In this paper we address the Classification Problem for the fiat $2$-category
$\cS$ of Soergel bimodules. This $2$-category, which is naturally graded and categorifies the
corresponding Hecke algebra $\mathsf{H}=\mathsf{H}(W,S)$,
was introduced by Soergel \cite{So}, \cite{So2} to give an alternative proof of the famous
Kazhdan--Lusztig conjectures of positive integrality in the case of Weyl groups and
provides a framework in which to generalize and approach these conjectures for arbitrary
Coxeter groups, see e.g. \cite{So3}. A general and completely algebraic proof using Soergel
bimodules was eventually found by Elias and Williamson \cite{EW}. Since their introduction,
Soergel bimodules and their (graded) $2$-representations have played a fundamental role in
representation theory, both in characteristic zero, see e.g.
\cite{KiMa}, \cite{KMR} (and the above mentioned papers), and,
more recently, in positive characteristic, see e.g. \cite{LW}, \cite{RW}. It is therefore
interesting to study its (graded) $2$-representation theory,
initially in characteristic zero. However, only very partial results on the classification
of its simple transitive $2$-representations were known. To be precise, the Classification Problem
was completely solved for finite Coxeter type A in \cite{MM5} both in the graded and the ungraded
setting, for finite dihedral type in \cite{KMMZ} and \cite{MT} in the graded setting
and almost completely in the ungraded setting, and for a small number of other special
cases in other Coxeter types, see e.g. \cite{KMMZ} and \cite{MMMZ}, both
in the graded and the ungraded setting. In all these
cases, the classification turned out to be finite. However, in general it was not even known if
the number of equivalence classes of simple transitive $2$-representations of $\cS$
is finite, left aside any classification. In this paper, we show that there are indeed only
finitely many equivalence classes of simple transitive
$2$-representations of $\cS$ for any finite Coxeter type, and classify them completely
for all finite Coxeter types except $H_{3}$ and $H_{4}$.
All our results are proved in the graded setting, but they hold in the ungraded setting
as well, see Subsection \ref{subsection:ungraded}.

Let us explain our results in some more detail. First of all, the Classification Problem for
$\cS$ can be reduced to that of certain subquotients. To do that, one uses the cell structure
of $\cS$, which is the categorical analog of the Kazhdan--Lusztig cell structure of $\mathsf{H}$.
By \cite[Subsection 3.2]{CM}, for every (graded or ungraded) simple transitive
$2$-representation $\mathbf{M}$ of $\cS$ there is a unique two-sided cell,
called \emph{apex}, that is not annihilated by $\mathbf{M}$ and is maximal in the two-sided
order with respect to that property. The Classification Problem for $\cS$ can therefore be
addressed apex by apex, and for a fixed two-sided cell $\mathcal{J}$ we denote the $2$-category of (graded)
simple transitive $2$-representations of $\cS$ with apex $\mathcal{J}$ by $\cS\text{-}
\mathrm{(g)stmod}_{\mathcal{J}}$. Next, we can reduce the Classification Problem even
further by \emph{(strong) $\mathcal{H}$-reduction}, see \cite[Theorem 15]{MMMZ} or \cite[Theorem 4.32]{MMMTZ2}.
For any diagonal $\mathcal{H}$-cell $\mathcal{H}\subseteq \mathcal{J}$, which, by definition,
the intersection of a left cell in $\mathcal{J}$ and the dual right cell,
there is a subquotient $\cSH$ of $\cS$ whose only left, right and two-sided cells are the
trivial cell and $\mathcal{H}$. By \cite[Theorem 4.32]{MMMTZ2}, there is a biequivalence
\begin{gather*}
\cS\text{-}\mathrm{(g)stmod}_{\mathcal{J}}\simeq
\cSH\text{-}\mathrm{(g)stmod}_{\mathcal{H}}.
\end{gather*}
This is a major reduction, because $\cSH$ is much smaller than $\cS$ and one can pick any diagonal
$\mathcal{H}$-cell in $\mathcal{J}$ (which is helpful in practice, because not all diagonal
$\mathcal{H}$-cells of $\mathcal{J}$ necessarily have the same structure, see Section \ref{appendix}).

By the results in this paper, the benefits turn out to be even bigger.
Based on Elias and Williamson's results in \cite{EW}, Lusztig defined
in \cite[Subsection 18.15]{Lu2} a locally semisimple bicategory
$\cAJ$, for any two-sided cell $\mathcal{J}$
of any Coxeter group. This bicategory categorifies the summand of
the \emph{asymptotic Hecke algebra} corresponding to the cell $\mathcal{J}$.
The bicategory $\cAJ$ is pivotal fusion by \cite{EW2} and contains a one-object pivotal fusion full subbicategory
$\cAH$ for any diagonal $\mathcal{H}$-cell $\mathcal{H}\subset\mathcal{J}$, which we call the
\emph{asymptotic bicategory} associated to $\mathcal{H}$. Being fusion, the bicategory
$\cAH$ has only one cell, corresponding to $\mathcal{H}$, which is left, right,
two-sided and diagonal simultaneously, so all its simple transitive $2$-representations have apex $\mathcal{H}$.
The main insight of this paper is that these asymptotic bicategories completely determine
the (graded) simple transitive $2$-representations of $\cS$. To be precise, our main result
(Theorem \ref{theorem:main}) is the existence of a biequivalence of graded $2$-categories
\begin{gather*}
\cSH\text{-}\mathrm{gstmod}_{\mathcal{H}}\simeq
\cAH\text{-}\mathrm{stmod}^{\prime}
\end{gather*}
for every diagonal $\mathcal{H}$-cell $\mathcal{H}$ of $\cS$ for
any finite Coxeter type. Here $\cAH\text{-}\mathrm{stmod}^{\prime}$ is
the graded $2$-category with translation obtained from $\cAH\text{-}\mathrm{stmod}$
by a well-known and straightforward construction, which we will recall in Subsection
\ref{subsection:grading}. As a matter of fact, $\cAH\text{-}\mathrm{stmod}^{\prime}$ and
$\cAH\text{-}\mathrm{stmod}$ are biequivalent as ungraded $2$-categories, but not as
graded $2$-categories. In the ungraded case, therefore, the above biequivalence
becomes
\begin{gather*}
\cSH\text{-}\mathrm{stmod}_{\mathcal{H}}\simeq
\cAH\text{-}\mathrm{stmod},
\end{gather*}
see Subsection \ref{subsection:ungraded}.
By \cite[Corollary 9.1.6 and Proposition 3.4.6]{EGNO} and strong $\mathcal{H}$-reduction, this implies that the number of equivalence classes of (graded) simple transitive $2$-representations of
$\cS$ is finite for any finite Coxeter type, c.f. Corollary \ref{cor:gradedfiniteness} and
the end of Subsection \ref{subsection:ungraded}.
Moreover, for all finite Coxeter types but $H_{3}$ and $H_{4}$,
all two-sided cells contain a diagonal $\mathcal{H}$-cell $\mathcal{H}$ for which both
$\cAH$ and $\cAH\text{-}\mathrm{stmod}$ are known explicitly, as summarized in Section \ref{appendix}.
For some two-sided cells in Coxeter types $H_{3}$ and $H_{4}$ this is unfortunately not true, see Section \ref{appendix}. Therefore,
we get a complete solution of the Classification Problem for $\cS$ for all finite Coxeter types
but $H_{3}$ and $H_{4}$, and even for those two Coxeter types we get a complete solution
for more than half the number of apexes, see Section \ref{appendix}. Finally,
Theorem \ref{theorem:main} also implies (see Theorem \ref{theorem:main-local-ss})
that $\cS\text{-}\mathrm{(g)stmod}_{\mathcal{J}}$ is locally (graded) semisimple, meaning
that its morphism categories are all (graded) semisimple, for any two-sided cell
$\mathcal{J}$ and any finite Coxeter type.

Let us briefly sketch the key ingredients of the proof of Theorem \ref{theorem:main}. Let
$\mathbf{C}_{\mathcal{H}}$ be the cell $2$-representation of $\cSH$ with apex
$\mathcal{H}$, which categorifies the Kazhdan--Lusztig cell module of the Hecke algebra
associated to $\mathcal{H}$. The main ingredient in the proof of Theorem \ref{theorem:main}
is the existence of a graded biequivalence
\begin{gather*}
\cEnd_{\ccSH}(\mathbf{C}_{\mathcal{H}})\simeq\cAHpop,
\end{gather*}
c.f. Corollary \ref{cor:ssendo}.
(For the experts,
we remark that the proof of this proposition uses that the Duflo involution
$\Cd\in\mathcal{H}$ is a separable Frobenius algebra in $\cSH$,
as we show in Subsection \ref{s7.5} and
Proposition \ref{proposition:dot diagram}.) In the ungraded setting,
the above biequivalence becomes
\begin{gather*}
\cEnd_{\ccSH}(\mathbf{C}_{\mathcal{H}})\simeq\cAHop
\end{gather*}

Another important ingredient
in the proof of our main result is the \emph{double centralizer theorem}, see \cite[Theorem 5.2]{MMMTZ2}, which implies that there is a biequivalence
\begin{gather*}
\cSH\text{-}\mathrm{(g)stmod}_{\mathcal{H}}\simeq
\big(\cEnd_{\ccSH}(\mathbf{C}_{\mathcal{H}})\big)^{\mathrm{op}}\text{-}\mathrm{(g)stmod},
\end{gather*}
as we show in Proposition \ref{prop:main-theorem-prop2}. (Note that the bicategory
$\cBH$ in that proposition, which will be defined in \eqref{def:Cdbicomodules},
is graded biequivalent to
$\big(\cEnd_{\ccSH}(\mathbf{C}_{\mathcal{H}})\big)^{\mathrm{op}}$ by
Corollary \ref{cor:End-Bicomod}.)

Along the way, we also show several results that are interesting in their own right, e.g.
the aforementioned fact that $\Cd\in\mathcal{H}$ is a separable Frobenius algebra in
$\cSH$. In order to achieve this, we prove in Proposition \ref{prop:Frobenius}
that the underlying algebra $\mathsf{B}$ of $\mathbf{C}_{\mathcal{H}}$
is a finite dimensional positively graded weakly symmetric Frobenius algebra of graded length $2\mathbf{a}$, where $\mathbf{a}$ is the
value of Lusztig's $\mathbf{a}$-function on $\mathcal{H}$. In certain cases $\mathsf{B}$
is known to be symmetric, see Remark \ref{remark:symmetric}. However, we do not know if $\mathsf{B}$ is symmetric in general.
\medskip

\textbf{Acknowledgments.}
All computer assisted calculations in this paper were done using SageMath.

We thank Ben Elias, Meinolf Geck, Matt Hogancamp, Hankyung Ko, Raphael Rouquier and
Geordie Williamson for stimulating discussions and helpful exchanges of emails.
D.T. would also like to thank
Ben Elias for mellifluous influence
which is hardly superfluous.
We thank the referee for a very careful reading of
the paper and many helpful comments.

Significant parts of this research were done when all the authors met at
the University of Z{\"u}rich in September 2018, at Uppsala
University in April 2019 and at the University of East Anglia in July 2019. Hospitality and support of these universities are gratefully acknowledged.

M.M. was supported in part by Funda\c{c}{\~a}o para a Ci\^{e}ncia e a
Tecnologia (Portugal), projects UID/MAT/04459/2013 (Center for Mathematical Analysis,
Geometry and Dynamical Systems - CAMGSD) and PTDC/MAT-PUR/31089/2017 (Higher Structures and Applications).
Vo.Ma. is partially supported by
the Swedish Research Council and G{\"o}ran Gustafsson Stiftelse.
Va.Mi. is partially supported by EPSRC grant EP/S017216/1, which also funded the workshop at UEA in July 2019.
D.T. is supported by SageMath.
X.Z. is supported by G{\"o}ran Gustafsson Stiftelse and
National Natural Science Foundation of China (Grant No.12101422).
\vspace{1cm}


\section{Recollections}\label{section:Recollections}



\subsection{Categorical conventions}\label{s2.0}


Categories $\mathcal{C}$ and $2$-categories, $2$-semicategories and bicategories $\cC$ in this paper are assumed to be essentially small.
We also view a monoidal category as a $2$-category with one (possibly unspecified)
object; a perspective which we will use throughout, e.g. for Soergel bimodules.
We will also use the following notation:
\begin{enumerate}[\textbullet]

\item objects in categories (which are not morphism categories in $2$-categories)
are denoted by letters such as $X\in\mathcal{C}$, and morphisms by $f\in\mathcal{C}$;

\item objects in $2$-categories are denoted by $\mathtt{i}\in\cC$, $1$-morphisms by
$\mathrm{F}\in\cC$ and $2$-morphisms by Greek letters such as $\alpha\in\cC$;

\item for any $\mathtt{i},\mathtt{j}\in\cC$, we denote by $\cC(\mathtt{i},\mathtt{j})$
the corresponding morphism category;

\item identity $1$-morphisms are denoted by $\mathbbm{1}_{\mathtt{i}}$ and identity
$2$-morphisms by $\mathrm{id}_{\mathrm{F}}$, where the subscripts are sometimes omitted;

\item we write $\mathrm{F}\mathrm{G}=\mathrm{F}\circ\mathrm{G}$ for composition of $1$-morphisms, and
$\circ_{\mathrm{v}}$ and $\circ_{\mathrm{h}}$ denote vertical and horizontal compositions
of $2$-morphisms, respectively.

\end{enumerate}

We will also use bicategories, silently adapting definitions and results to the weaker setting if
necessary, using \cite{MMMTZ2}. We will stress when we do not work with genuine $2$-categories,
$2$-functors etc. The reader is referred to e.g. \cite{ML}, \cite{Le} or \cite{Be} for these and related notions.


\subsection{Finitary and fiat $2$-categories, and their $2$-representations}\label{s2.1}


Let $\Bbbk$ be an algebraically closed field.

A category $\mathcal{C}$ is called \emph{finitary} (over $\Bbbk$)
if it is equivalent to the category of finitely generated, injective (or projective) modules over
some associative, finite dimensional $\Bbbk$-algebra.
These categories assemble into a $2$-category $\mathfrak{A}^{f}_{\Bbbk}$ having
additive, $\Bbbk$-linear functors and natural transformations as $1$- and $2$-morphisms,
respectively. Similarly, a $2$-category $\cC$ is \emph{finitary} (over $\Bbbk$) if it has finitely many objects,
all identity $1$-morphisms $\mathbbm{1}_{\mathtt{i}}$ are indecomposable
and each morphism category $\cC(\mathtt{i},\mathtt{j})$ is finitary over $\Bbbk$
with all compositions being (bi)additive and
$\Bbbk$-(bi)linear.
We further say that $\cC$ is \emph{fiat} if it has a weak antiinvolution
${}^{\star}$, reversing the direction of both $1$- and $2$-morphisms,
and adjunction $2$-morphisms associated to ${}^{\star}$. If ${}^{\star}$ is just a weak
antiequivalence of finite order, then $\cC$ is called \emph{quasi fiat}.
Finally, a finitary $2$-category $\cC$ is called \emph{locally semisimple} if its morphism
categories are all semisimple. Let us note that our use of the term ``locally semisimple''
is similar to such standard categorical terminology as  ``locally small'', with ``locally'' referring to $1$-morphism categories.

\begin{remark}
For completeness, dropping the assumption of the identity $1$-morphisms being indecomposable gives
what we call multifinitary, multifiat or quasi multifiat $2$-categories. Most of the theory
goes through for these as well, see \cite{MMMTZ2}, but we will not need that level of
generality in this paper.
\end{remark}

\begin{example}
For a finite group $G$, the (strictified) one-object $2$-category $\cR\mathrm{ep}(G,\Bbbk)$ of
finite dimensional
representations of $G$ over $\Bbbk$ is fiat if and only if the algebra $\Bbbk[G]$ has finite representation type.
This is true, for example, if $\mathrm{char}(\Bbbk)\nmid\#G$, in which case
$\cR\mathrm{ep}(G,\Bbbk)$ is locally semisimple.

Let $\Bbbk=\mathbb{C}$.
Another example of a fiat $2$-category is $\cS=\cS_{\mathbb{C}}(W,\mathtt{S})$, the
one-object $2$-category of Soergel bimodules over the coinvariant algebra of a
finite Coxeter group, see Section \ref{s5}.
\end{example}

\begin{example}\label{ex:projbimod}
For any finite dimensional algebra $\mathsf{B}$ we have an associated one-object $2$-category $\cC_{\mathsf{B}}$, called the \emph{$2$-category of projective functors}, whose $1$-morphisms are direct sums of functors with summands isomorphic to the identity functor or to tensoring with projective $\mathsf{B}$-$\mathsf{B}$-bimodules. (Despite the name of $\cC_{\mathsf{B}}$, the identity functor is not a
projective functor, but is needed to make $\cC_{\mathsf{B}}$ a genuine $2$-category
instead of a $2$-semicategory.) Assume that $\mathsf{B}$ is basic and connected and that
$1=e_{1}+\dots+e_{n}$ is a splitting of the identity into
orthogonal primitive idempotents, then $\{\mathsf{B}e_{i}\otimes_{\Bbbk} e_{j}\mathsf{B}\mid i,j=1,\dots,n\}$ is a complete and irredundant set of indecomposable projective $\mathsf{B}$-$\mathsf{B}$-bimodules. By e.g. \cite[Lemma 45 and its proof]{MM1}, for every $i,j=1,\dots,n$,
\begin{gather*}
(\mathsf{B}e_{i}\otimes_{\Bbbk}e_{j}\mathsf{B})^{\star}\cong (e_{j}\mathsf{B})^{\star}\otimes_{\Bbbk}e_{i}\mathsf{B}.
\end{gather*}
The left injective $\mathsf{B}$-modules $(e_{j}\mathsf{B})^{\star}$, for $j=1,\dots,n$,
are projective if and only if $\mathsf{B}$ is Frobenius (since $\mathsf{B}$ is assumed to be basic,
it is self-injective if and only if it is Frobenius). Therefore, $\cC_{\mathsf{B}}$ is quasi fiat if and only if
$\mathsf{B}$ is Frobenius.

Assume that $\cC_{\mathsf{B}}$ is quasi fiat. Then $(e_{j}\mathsf{B})^{\star}\cong\mathsf{B}e_{\sigma(j)}$, for $j=1,\dots,n$, where $\sigma$ is the Nakayama
permutation. Applying ${}^{\star}$ twice yields
\begin{gather*}
(\mathsf{B}e_{i}\otimes_{\Bbbk}e_{j}\mathsf{B})^{\star\star}\cong \mathsf{B}e_{\sigma(i)}\otimes_{\Bbbk}e_{\sigma(j)}\mathsf{B},
\end{gather*}
for all $i,j=1,\dots,n$. We thus see that $\cC_{\mathsf{B}}$ is fiat if and only if $\mathsf{B}$
weakly symmetric. For more details, see \cite[Subsection 7.3]{MM1}.
\end{example}

A locally semisimple quasi fiat $2$-category is called a \emph{fusion} $2$-category
and a locally semisimple fiat $2$-category can be equipped with the structure of a \emph{pivotal fusion} $2$-category.
Note that one-object (pivotal) fusion $2$-categories, a.k.a. (pivotal) fusion categories, form an important class of
tensor categories which has been intensively studied, see e.g. \cite{EGNO}.

\begin{example}
The $2$-category $\cR\mathrm{ep}(G,\Bbbk)$ is pivotal fusion unless $\mathrm{char}(\Bbbk)|\#G$.
By contrast, the $2$-category $\cS$ is not (pivotal) fusion, since it is not locally semisimple.
\end{example}

\begin{remark}
Throughout this paper, we will consistently use the above $2$-categorical terminology.
Let us, for convenience, list the correspondence between our
terminology and the one used in \cite{EGNO}  in the abelian setting for one-object $2$-categories/monoidal categories:
\begin{enumerate}[\textbullet]

\item quasi fiat $\leftrightarrow$ finitary and rigid;

\item fiat $\leftrightarrow$ finitary and pivotal;

\item locally semisimple $\leftrightarrow$ semisimple;

\item (pivotal) fusion $2$-category $\leftrightarrow$ (pivotal) fusion category.

\end{enumerate}
Note that the notion of fiat only requires the existence of
a natural isomorphism between the identity and ${}^{\star\star}$, whereas the notion of pivotal requires
a choice of such a natural isomorphism. In practice, this subtle difference is not important for us  because the abstract results hold for fiat $2$-categories, whereas in the examples the pivotal structure is known and fixed.
\end{remark}

For a finitary $2$-category $\cC$, a \emph{finitary $2$-representation}
$\mathbf{M}$ is an additive, $\Bbbk$-linear $2$-functor from $\cC$ to $\mathfrak{A}^{f}_{\Bbbk}$.
Finitary $2$-representations of $\cC$ form a $2$-category, as we will explain in
Subsection \ref{s:repcats}; in particular, there is an appropriate notion of equivalence.
The \emph{underlying category} of $\mathbf{M}$ is defined as
\begin{gather*}
\mathcal{M}:=\coprod_{\mathtt{i}\in\ccC}\mathbf{M}(\mathtt{i}).
\end{gather*}
The \emph{rank} of $\mathbf{M}$ is by definition the number of
isomorphism classes of indecomposable objects in $\mathcal{M}$.
Moreover, we will often use the action notation $\mathrm{F}\,X:=\mathbf{M}(\mathrm{F})(X)$
for $2$-representations.

\begin{example}\label{example:yoneda}
If $\cC$ is finitary, then the so-called \emph{principal} or \emph{Yoneda} $2$-representation
$\mathbf{P}_{\mathtt{i}}:=\cC(\mathtt{i},{}_{-})$ is finitary, for all $\mathtt{i}\in\cC$.
\end{example}

A finitary $2$-representation $\mathbf{M}$ is called \emph{transitive} if,
for any $\mathtt{i}\in\cC$ and any non-zero object $X\in\mathbf{M}(\mathtt{i})$, the
additive closure (in the sense of being closed under direct sums, direct
summands and isomorphisms)
\begin{gather*}
\mathrm{add}\big(
\{\mathrm{F}\,X\mid\mathtt{j}\in\cC,\mathrm{F}\in\cC(\mathtt{i},\mathtt{j})\}\big)
\end{gather*}
coincides with $\mathcal{M}$.
Recall that an \emph{ideal} of $\mathbf{M}$ is by definition
a $\cC$-stable ideal of $\mathcal{M}$.
A \emph{transitive} $2$-representation $\mathbf{M}$ is said to be \emph{simple transitive}
provided that it has no non-trivial ideals. Moreover, every transitive
$2$-representation has a unique simple transitive quotient.

The importance of simple transitive $2$-representations is explained, in particular,
by the existence of a weak version of the Jordan--H{\"o}lder theorem.
Namely, for any finitary $2$-representation $\mathbf{M}$
of $\cC$, there is a finite filtration
\begin{gather*}
0=\mathbf{M}_{0}\subset\mathbf{M}_{1}\subset\dots\subset\mathbf{M}_m=\mathbf{M}
\end{gather*}
where every $2$-representation $\mathbf{M}_k$ generates an ideal
$\mathbf{I}_k$ in $\mathbf{M}_{k+1}$ such that
$\mathbf{M}_{k+1}/\mathbf{I}_{k}$ is transitive, and thus, has a unique associated
simple transitive quotient $\mathbf{L}_{k+1}$. Up to equivalence and
ordering, the set $\{\mathbf{L}_k\mid 1\leq k\leq m\}$ is an invariant of $\mathbf{M}$.

The above motivates the \emph{Classification Problem}, i.e. the classification of
simple transitive $2$-representations for a fixed finitary $2$-category.

\begin{example}
For the $2$-category $\cR\mathrm{ep}(G,\mathbb{C})$ the Classification Problem has a
well-known solution in terms of subgroups of $G$ and their group cohomology, see Section \ref{appendix} for details.

By contrast, solving the Classification Problem for $\cS$ is the main objective of this paper.
\end{example}

See \cite[Subsections 2.2, 2.3 and 2.4]{MM1}, \cite[Subsection 2.5]{MM6},
\cite[Subsection 2.3]{MM3} and \cite[Subsection 3.5]{MM5} for details.


\subsection{Cells and cell $2$-representations}\label{s2.2}


For every finitary $2$-category $\cC$, one has the notion of \emph{cells}:
For any pair of indecomposable $1$-morphisms $\mathrm{F}$ and $\mathrm{G}$,
we define $\mathrm{F}\geq_L\mathrm{G}$ if $\mathrm{F}$ is isomorphic to a direct summand
of $\mathrm{H}\mathrm{G}$, for some $1$-morphism $\mathrm{H}$. This produces the
\emph{left} preorder $\geq_L$, for which the equivalence classes are called \emph{left cells}.
Similarly one obtains \emph{right cells} and \emph{two-sided cells}.
By \cite[Subsection 3.2]{CM}, for any transitive $2$-representation $\mathbf{M}$ there is a unique
two-sided cell $\mathcal{J}$, an invariant of $\mathbf{M}$ called the \emph{apex},
which is not annihilated by $\mathbf{M}$
and is maximal, in the two-sided order, with respect to this property.

\begin{example}\label{rem:fusioncell}
Any one-object (pivotal) fusion $2$-category $\cC$, a.k.a.
(pivotal) fusion category, has only one cell, which is left, right and two-sided.
This follows from the fact that
$\mathrm{X}\mathrm{X}^{\star}$ and $\mathrm{X}^{\star}\mathrm{X}$ both contain the
identity $1$-morphism as a direct summand, for all indecomposable $\mathrm{X}\in\cC$.
\end{example}

\begin{example}
When $\mathrm{char}(\Bbbk)\nmid\#G$, the $2$-category
$\cR\mathrm{ep}(G,\Bbbk)$ is pivotal fusion and has, therefore, only one cell,
which is left, right and two-sided. When $\mathrm{char}(\Bbbk)\mid\#G$,
the $2$-category $\cR\mathrm{ep}(G,\Bbbk)$ has more than one cell, for example, the projective modules form
a two-sided cell.

As we will recall in Subsection \ref{s5.1},
the cells of $\cS$ are given by the Kazhdan--Lusztig cells in the sense of \cite{KL}.
\end{example}

Each left cell $\mathcal{L}$ of $\cC$ can be used to define
a \emph{cell $2$-representation} $\mathbf{C}_{\mathcal{L}}$
as follows. First we note that all $1$-morphisms
in $\mathcal{L}$ have the same domain, say $\mathtt{i}$.
Define a $2$-subrepresentation $\mathbf{M}^{\geq\mathcal{L}}$ of $\mathbf{P}_{\mathtt{i}}$
using the induced action of $\cC$ on
\begin{gather*}
\mathrm{add}\big(
\{\mathrm{F}\mid\mathrm{F}\geq_{L}\mathcal{L}\}
\big).
\end{gather*}
The $2$-representation $\mathbf{M}^{\geq\mathcal{L}}$ has a unique maximal ideal
$\mathbf{I}$
and we define
\begin{gather*}
\mathbf{C}_{\mathcal{L}}:=\mathbf{M}^{\geq\mathcal{L}}/\mathbf{I},
\end{gather*}
which is a simple transitive $2$-representation by construction. If $\cC$ is quasi fiat, then
the apex of $\mathbf{C}_{\mathcal{L}}$ is the two-sided cell containing $\mathcal{L}$.

\begin{example}
The (unique) cell $2$-representation of $\cR\mathrm{ep}(G,\mathbb{C})$ coincides with
its unique principal $2$-representation.

The cell $2$-representations of $\cS$ categorify the Kazhdan--Lusztig cell representations, see Subsection \ref{s5.1}.
\end{example}

Finally, recall that a $2$-category is called \emph{$\mathcal{J}$-simple}, where
$\mathcal{J}$ is a two-sided cell, if any non-zero $2$-ideal contains the identity $2$-morphisms
of all $1$-morphisms in $\mathcal{J}$. If $\cC$ is fiat, then $\cC$ has
an associated $\mathcal{J}$-simple subquotient $2$-category $\cC_{\mathcal{J}}$, whose only two-sided cells are $\mathcal{J}$
and the two-sided cells containing $\mathbbm{1}_{\mathtt{i}}$, such that $\mathtt{i}$
is the source of some
$1$-morphism in $\mathcal{J}$. Note that these cells may all coincide, e.g. if $\cC$ is fusion.

For further details we refer to \cite[Subsection 4.5]{MM1}, \cite[Subsection 3.2]{CM}, \cite[Section 3]{MM5},
\cite[Subsection 4.2]{MMMZ} and \cite[Subsections 2.5 and 2.6]{MMMTZ2}.


\subsection{Coalgebra and algebra $1$-morphisms}\label{s2.23}


Finitary $2$-categories can be injectively or projectively abelianized.
The \emph{injective abelianization} is denoted $\underline{\cC}$
and the \emph{projective abelianization} is denoted $\overline{\cC}$. Moreover,
$\cC$ embeds into $\underline{\cC}$ or
into $\overline{\cC}$, and the isomorphism closure of the image of this
embedding is the $2$-full $2$-subcategory of injective
or projective $1$-morphisms, respectively. In particular, in these abelianizations,
each indecomposable $1$-morphism $\mathrm{F}\in\cC$ has an associated simple socle or head, respectively.

These abelianizations are rather technical and not all properties of $\cC$ carry over
to the abelianizations. In particular, the abelianizations of fiat $2$-categories are usually not even finitary
and the involution ${}^{\star}$ only gives rise to an antiequivalence between
$\underline{\cC}$ and $\overline{\cC}$.

The same abelianizations, mutatis mutandis, exist for finitary $2$-representations, where we use the same notation.

See \cite[Section 3]{MMMT} for details.

\begin{example}
When $\mathrm{char}(\Bbbk)\nmid\#G$, we have
$\underline{\cR\mathrm{ep}(G,\Bbbk)}\simeq\overline{\cR\mathrm{ep}(G,\Bbbk)}\simeq\cR\mathrm{ep}(G,\Bbbk)$,
because $\cR\mathrm{ep}(G,\Bbbk)$ is locally semisimple. In contrast, $\cS$ is not abelian
and neither $\underline{\cS}$ nor $\overline{\cS}$ are equivalent to it.
\end{example}

Abelianizations play a key role in the construction and study of $2$-representations:
Recall that a \emph{coalgebra}
$\mathrm{C}:=(\mathrm{C},\delta_{\mathrm{C}},\epsilon_{\mathrm{C}})$ in $\underline{\cC}$
is a $1$-morphism in some $\underline{\cC}(\mathtt{i,i})$ equipped with $2$-morphisms
$\delta_{\mathrm{C}}\colon\mathrm{C}\to\mathrm{C}\mathrm{C}$,
called comultiplication, and
$\epsilon_{\mathrm{C}}\colon\mathrm{C}\to\mathbbm{1}_{\mathtt{i}}$, called counit, satisfying
the usual conditions. Dually, one can define algebras as well. A
\emph{right $\mathrm{C}$-comodule} $\mathrm{M}=(\mathrm{M},\delta_{\mathrm{M},\mathrm{C}})$
is a $1$-morphism in some $\underline{\cC}(\mathtt{i},\mathtt{j})$
together with a right coaction $\delta_{\mathrm{M},\mathrm{C}}\colon\mathrm{M}\to\mathrm{M}\mathrm{C}$,
again satisfying the usual coherence conditions.
For each object $\mathtt{j}\in\cC$, these assemble into a category
\begin{gather*}
\mathrm{comod}_{\underline{\ccC}}(\mathrm{C})_{\mathtt{j}}
:=
\{
(\mathrm{M},\delta_{\mathrm{M},\mathrm{C}})\mid\mathrm{M}\in\underline{\cC}(\mathtt{i,j}),\,\,
\mathrm{M}\text{ is a right $\mathrm{C}$-comodule}
\},
\end{gather*}
whose morphisms are called \emph{right $\mathrm{C}$-comodule morphisms} and satisfy the
usual intertwining condition.
All of these notions can be defined, mutatis mutandis, for left coactions as well,
and we put the coalgebras on the left-hand side to stress that we have a left coaction.
Let
\begin{gather*}
\mathrm{inj}_{\underline{\ccC}}(\mathrm{C})_{\mathtt{j}}
:=
\{
(\mathrm{M},\delta_{\mathrm{M},\mathrm{C}})\mid\mathrm{M}
\in\mathrm{comod}_{\underline{\ccC}}(\mathrm{C})_{\mathtt{j}}\text{ injective}
\}
\end{gather*}
denote the full subcategory
of $\mathrm{comod}_{\underline{\ccC}}(\mathrm{C})_{\mathtt{j}}$
consisting of all injective objects. We set
\begin{gather*}
\mathrm{comod}_{\underline{\ccC}}(\mathrm{C})=\coprod_{\mathtt{j}\in\ccC}
\mathrm{comod}_{\underline{\ccC}}(\mathrm{C})_{\mathtt{j}}\quad\text{and}
\quad\mathrm{inj}_{\underline{\ccC}}(\mathrm{C})=\coprod_{\mathtt{j}\in\ccC}
\mathrm{inj}_{\underline{\ccC}}(\mathrm{C})_{\mathtt{j}}.
\end{gather*}

Note that left composition defines a natural $\underline{\cC}$-action on
$\mathrm{comod}_{\underline{\ccC}}(\mathrm{C})$, which restricts to a natural
left $\cC$-action on $\mathrm{inj}_{\underline{\ccC}}(\mathrm{C})$. We use a
bold font whenever we want to consider these as $2$-representations, e.g.
we write $\mathbf{inj}_{\underline{\ccC}}(\mathrm{C})$ instead of $\mathrm{inj}_{\underline{\ccC}}(\mathrm{C})$.

\begin{example}
The identity $1$-morphism $\mathbbm{1}_{\mathtt{i}}$, for any $\mathtt{i}\in\cC$,
has a natural structure of a coalgebra, with both the comultiplication and the counit being
given by the identity $2$-morphism. The $2$-representation $\mathbf{comod}_{\underline{\ccC}}(\mathbbm{1}_{\mathtt{i}})$
of $\cC$ is equivalent to $\underline{\mathbf{P}_{\mathtt{i}}}$
and the $2$-representation $\mathbf{inj}_{\underline{\ccC}}(\mathbbm{1}_{\mathtt{i}})$
of $\cC$ is equivalent to $\mathbf{P}_{\mathtt{i}}$.
\end{example}

Another natural source of coalgebras in $\underline{\cC}$ is provided by
the internal cohom construction:
Let $\mathbf{M}$ be a finitary $2$-representation of $\cC$ and consider
objects $X\in\mathbf{M}(\mathtt{i})$, $Y\in\mathbf{M}(\mathtt{j})$.
Their \emph{internal cohom} $[X,Y]\in\underline{\cC}(\mathtt{i,j})$ is
uniquely determined, up to isomorphism, by the isomorphism
\begin{gather*}
\mathrm{Hom}_{\underline{\mathbf{M}(\mathtt{j})}}(Y,\mathrm{G}\,X)\cong\mathrm{Hom}_{\underline{\ccC}(\mathtt{i,j})}\big([X,Y],\mathrm{G}\big)
\end{gather*}
for all $\mathrm{G}\in\underline{\cC}(\mathtt{i,j})$. If it is not
immediately obvious in which $2$-category or bicategory the internal cohom
is taken, we will use a subscript $[X,Y]_{\underline{\ccC}}$ to clarify. The $1$-morphism $\mathrm{C}^{X}:=[X,X]$
has a natural structure of a coalgebra in $\underline{\cC}$.
Moreover, if $\cC$ is fiat and $X$ generates $\mathbf{M}$ in the sense that $\mathrm{add}\{\mathrm{F}X\mid\mathrm{F}\in\cC\}=\mathcal{M}$,
then $\mathbf{inj}_{\underline{\ccC}}(\mathrm{C}^{X})\simeq\mathbf{M}$. In
this way, we can realize any finitary $2$-representation of a fiat $2$-category
$\cC$, up to equivalence, as
$\mathbf{inj}_{\underline{\ccC}}(\mathrm{C})$ for some coalgebra $\mathrm{C}$ in
$\underline{\cC}$ (or in the injective abelianization of the additive closure
of $\cC$, if no generator in a single $\mathbf{M}(\mathtt{i})$ exists), and simple
transitive $2$-representations of $\cC$ (which always have a generator) correspond
to cosimple coalgebras in $\underline{\cC}$. If $\mathbf{M}$ is a simple transitive
$2$-representation of $\cC$ with apex $\mathcal{J}$ and $\cC$ is $\mathcal{J}$-simple, then
$\mathrm{C}^X$ is in $\cC$ for every $X\in\mathcal{M}$.

Finally, we recall that coalgebras in $\cC$ have an associated Morita--Takeuchi theory. In particular, $\mathrm{C}$
and $\mathrm{C}^{\prime}$ are Morita--Takeuchi equivalent if and only if
$\mathbf{inj}_{\underline{\ccC}}(\mathrm{C})\simeq\mathbf{inj}_{\underline{\ccC}}(\mathrm{C}^{\prime})$.
See also e.g. \cite[Chapter 7]{EGNO}, \cite[Sections 4 and 5]{MMMT} and \cite[Subsection 3.6]{MMMZ}.

For later use, we recall one technical result and record three others. The first
result, Theorem \ref{thm:projapex}, is the content of Theorem \cite[Theorem 2]{KMMZ}, which we will need
repeatedly throughout the paper.

\begin{theorem}\label{thm:projapex}
Let $\mathbf{M}$ be a simple transitive $2$-representation of a fiat $2$-category $\cC$ with apex
$\mathcal{J}$. For every $1$-morphism $\mathrm{F}\in \mathcal{J}$, the
endofunctor $\overline{\mathbf{M}}(\mathrm{F})$ is projective (in the category of
right exact linear endofunctors).
\end{theorem}

In particular, suppose that $\cC$ has only one object $\mathtt{i}$ and that
$\mathsf{B}$ is the basic connected underlying algebra of $\mathbf{M}(\mathtt{i})$, i.e., $\mathcal{M}\simeq \mathsf{B}\text{-}\mathrm{proj}$ and $\overline{\mathcal{M}}\simeq \mathsf{B}\text{-}\mathrm{mod}$.
Theorem \ref{thm:projapex} then implies that, for every $1$-morphism $\mathrm{F}\in \mathcal{J}$, the action of
$\overline{\mathbf{M}}(\mathrm{F})$ on $\mathsf{B}\text{-}\mathrm{mod}$ is given
by tensoring over $\mathsf{B}$ with a projective $\mathsf{B}$-$\mathsf{B}$-bimodule $B_{\mathrm{F}}$. Note that
$B_{\mathrm{F}}\otimes_{\mathsf{B}}M\in \mathsf{B}\text{-}\mathrm{proj}$
for any $M\in \mathsf{B}\text{-}\mathrm{mod}$.

\begin{corollary}\label{cor:projapex}
The algebra $\mathsf{B}$ is Frobenius.
\end{corollary}

\begin{proof}
Let $1=e_{1}+\dots+e_{n}$ be a splitting of the unit into orthogonal primitive idempotents in $\mathsf{B}$.
Then $\{\mathsf{B}e_{i}\otimes_{\Bbbk} e_{j}\mathsf{B}|i,j=1,\dots,n\}$ is a complete and irredundant set of indecomposable
projective $\mathsf{B}$-$\mathsf{B}$-bimodules.

Choose $j\in \{1,\dots n\}$. By Theorem \ref{thm:projapex} and transitivity of $\mathbf{M}$,
for every $i\in\{1,\dots,n\}$ with $i\neq j$, there is an $\mathrm{F}\in\mathrm{add}(\mathcal{J})$ such that $\mathsf{B}e_{i}$ is a direct summand of
$B_{\mathrm{F}}\otimes_{\mathsf{B}}\mathsf{B}e_{j}$, so $B_{\mathrm{F}}$ contains a direct summand of the form
\begin{gather*}
\mathsf{B}e_{i}\otimes_{\Bbbk} e_{k}\mathsf{B},
\end{gather*}
for some $k\in\{1,\dots,n\}$. As in Example \ref{ex:projbimod}, the decomposition of
$B_{\mathrm{F}}\cong B_{\mathrm{F}^{\star\star}}$ thus contains a direct summand of the form
\begin{gather*}
(e_{i}\mathsf{B})^{\star}\otimes_{\Bbbk} e_{l}\mathsf{B},
\end{gather*}
for some $l\in\{1,\dots,n\}$. This shows that $(e_{i}\mathsf{B})^{\star}$ is a projective left $\mathsf{B}$-module, by Theorem \ref{thm:projapex}.

Since this holds for all $i\in\{1,\dots,n\}$, every left injective $\mathsf{B}$-module is also projective, i.e., $\mathsf{B}$ is self-injective. Since every self-injective basic algebra is Frobenius, the result follows.
\end{proof}

Corollary \ref{cor:projapex} shows that Theorem \ref{thm:projapex} can also be formulated
in terms of the injective abelianization $\underline{\mathbf{M}}$, so $\mathcal{M}\simeq \mathsf{B}\text{-}\mathrm{inj}=\mathsf{B}\text{-}\mathrm{proj}$ and $\underline{\mathcal{M}}\cong \mathsf{B}\text{-}\mathrm{mod}$, and the action of $\underline{\mathbf{M}}(\mathrm{F})$ on $\mathsf{B}\text{-}\mathrm{mod}$ is also given by $B_{\mathrm{F}}$, which is injective-projective as an $\mathsf{B}$-$\mathsf{B}$-bimodule. This consequence of Theorem \ref{thm:projapex} and Corollary \ref{cor:projapex} is crucial and will be used repeatedly throughout the paper.

\begin{remark}
The $2$-category $\cC_{\mathsf{B}}$ in Example \ref{ex:projbimod}
is fiat if and only if $\mathsf{B}$ is weakly symmetric, whereas simple transitivity of
$\mathbf{M}$ only implies that $\mathsf{B}=\mathsf{B}^{\mathbf{M}}$ is Frobenius. The reason is that, for any $\mathrm{F}\in \mathcal{J}$, the $\mathsf{B}$-$\mathsf{B}$-bimodule $B_{\mathrm{F}}$ is a direct sum of indecomposable $\mathsf{B}$-$\mathsf{B}$-bimodules.
When $\cC$ is fiat, then $B_{\mathrm{F}}^{\star\star}\cong B_{\mathrm{F}}$, but
${}^{\star\star}$ can still permute the indecomposable direct summands, so the Nakayama
permutation of $\mathsf{B}$ need not be trivial. In Proposition \ref{proposition:weakly-symmetric}, however, we will show that $\mathsf{B}^{\mathbf{M}}$ is weakly symmetric if $\mathbf{M}$ is a (graded) simple transitive $2$-representation of
Soergel bimodules of finite Coxeter type.
\end{remark}

\begin{proposition}\label{prop:ss}
Suppose that $\cC$ is a one-object pivotal fusion $2$-category and
$\mathbf{M}$ a finitary $2$-representation of $\cC$. If $\mathbf{M}$ is simple transitive, then
$\mathcal{M}$ is semisimple.
\end{proposition}

\begin{proof}
Suppose that $\mathbf{M}$ is simple transitive. As explained in Example \ref{rem:fusioncell},
$\cC$ has only one cell, which is thus the apex of $\mathbf{M}$. By Theorem \ref{thm:projapex},
this implies that $\overline{\mathbf{M}}(\mathbbm{1})$ is a projective endofunctor
of $\overline{\mathcal{M}}$, where $\mathbbm{1}$ is the unique identity $1$-morphism of $\cC$. This implies that every simple $L\in\overline{\mathcal{M}}$ is projective, since $L=\mathbf{M}(\mathbbm{1})(L)$. Therefore, $\mathcal{M}$ is semisimple.
\end{proof}

\begin{corollary}\label{cor:ss}
Suppose that $\cC$ is a one-object pivotal fusion $2$-category and
$\mathrm{C}$ a cosimple coalgebra in $\cC$. Then $\mathrm{inj}_{\ccC}(\mathrm{C})=\mathrm{comod}_{\ccC}(\mathrm{C})$ is semisimple.
\end{corollary}

\begin{proof}
If $\mathrm{C}$ is a cosimple coalgebra in $\cC$, then $\mathbf{inj}_{\ccC}(\mathrm{C})$ is simple transitive by \cite[Corollary 12]{MMMZ}. Therefore, its underlying category is semisimple by Proposition \ref{prop:ss}.
\end{proof}


\subsection{$2$-Categories of $2$-representations}\label{s:repcats}


The Classification Problem suggests studying the following $2$-categories.
Given two $2$-categories $\cC$ and $\cD$,
let $[\cC,\cD]$ denote the $2$-category of $2$-functors together with $2$-natural
transformations and modifications.

If $\cC$ is finitary, then let $\cC\text{-}\mathrm{afmod}:=[\cC,\mathfrak{A}^{f}_{\Bbbk}]$
denote the additive, $\Bbbk$-linear $2$-category of its finitary $2$-representations.
We let $\cC\text{-}\mathrm{stmod}$
denote the $1,2$-full $2$-subcategory of $\cC\text{-}\mathrm{afmod}$ consisting of
all simple transitive $2$-representations.
Further, we indicate by the subscript $\mathcal{J}$, where $\mathcal{J}$ is a given
two-sided cell of $\cC$, the $1,2$-full $2$-subcategories
consisting of the finitary $2$-representations of $\cC$ whose weak Jordan--H{\"o}lder
constituents all have apex $\mathcal{J}$.
Finally, we use the superscript $\mathrm{ex}$ to indicate the $2$-full
$2$-subcategories having the same objects, but $1$-morphisms being exact,
in the sense of their component functors extending to exact functors in the (injective) abelianization.

\begin{example}
Suppose that $\cC$ is fiat and $\mathcal{L}$ a left cell inside a two-sided cell $\mathcal{J}$ of
$\cC$. Then $\mathbf{C}_{\mathcal{L}}\in
\cC\text{-}\mathrm{stmod}_{\mathcal{J}}^{\mathrm{ex}}
\subset\cC\text{-}\mathrm{stmod}_{\mathcal{J}}
\subset\cC\text{-}\mathrm{stmod}
\subset\cC\text{-}\mathrm{afmod}$.
\end{example}

As discussed in detail in \cite[Subsection 2.6]{MMMTZ2}, there are various relations
between the above $2$-categories. We recall two of the most crucial ones
here in the form in which we need them. The first relates to all morphisms
between simple transitive $2$-representations with the same apex being exact:

\begin{proposition}\label{prop:exactness}
Suppose that $\cC$ is fiat. For any two-sided cell $\mathcal{J}$
of $\cC$, the $2$-categories $\cC\text{-}\mathrm{stmod}_{\mathcal{J}}^{\mathrm{ex}}$
and $\cC\text{-}\mathrm{stmod}_{\mathcal{J}}$ are equal.
\end{proposition}

The second one is the so-called \emph{strong $\mathcal{H}$-reduction} \cite[Theorem 4.32]{MMMTZ2}:
Suppose that $\cC$ is fiat and let $\mathcal{J}$ be a two-sided cell of $\cC$.
Then ${}^{\star}$ preserves $\mathcal{J}$ and the $\mathcal{J}$-simple subquotient
$\cC_{\mathcal{J}}$, see the end of Subsection \ref{s2.2}, is also fiat and,
by \cite[Theorem 4.28]{MMMTZ2}, there is a
biequivalence
\begin{gather*}
\cC_{\mathcal{J}}\text{-}\mathrm{stmod}_{\mathcal{J}}\simeq
\cC\text{-}\mathrm{stmod}_{\mathcal{J}}.
\end{gather*}
The point of this biequivalence is that the structure of $\cC_{\mathcal{J}}$ is simpler
than that of the whole $\cC$ in general.

However, there are even simpler $2$-categories, which still contain enough information to solve
the Classification Problem. We define \emph{$\mathcal{H}$-cells} as the intersection of left and right
cells. Suppose that $\cC$ is fiat. Then, for every left cell $\mathcal{L}$, we define the associated \emph{diagonal}
$\mathcal{H}$-cell
\begin{gather*}
\mathcal{H}(\mathcal{L}):=\mathcal{L}\cap\mathcal{L}^{\star}.
\end{gather*}
By construction, $\mathcal{H}=\mathcal{H}(\mathcal{L})$
lies in the same two-sided cell $\mathcal{J}$
as $\mathcal{L}$. Further, recall that each $\mathcal{L}$
contains a unique distinguished
$1$-morphism $\mathrm{D}=\mathrm{D}(\mathcal{L})$ called
\textit{Duflo involution} and that, in fact, $\mathrm{D}\in\mathcal{H}(\mathcal{L})$.

Given a left cell $\mathcal{L}$ in some two-sided cell $\mathcal{J}$,
we define $\cC_{\mathcal{H}}$ to be the $2$-full $2$-subcategory of
$\cC_{\mathcal{J}}$ generated by all $1$-morphisms
in $\mathcal{H}:=\mathcal{H}(\mathcal{L})$
together with the identity $1$-morphism $\mathbbm{1}_{\mathtt{i}}$, where $\mathtt{i}$ is
the unique domain and codomain of all $1$-morphisms
in $\mathcal{H}$. The $2$-category $\cC_{\mathcal{H}}$
is fiat,
has $\mathcal{H}$ as its maximal
two-sided cell and is $\mathcal{H}$-simple. The importance of $\cCH$ is expressed
by the following result, called strong $\mathcal{H}$-reduction.

\begin{theorem}\label{theorem:strongH}
Let $\cC$ be a fiat $2$-category and fix a two-sided cell $\mathcal{J}$ of $\cC$. For any
diagonal $\mathcal{H}$-cell $\mathcal{H}\subset\mathcal{J}$, there is a biequivalence
\begin{gather*}
\cC\text{-}\mathrm{stmod}_{\mathcal{J}}\simeq
\cC_{\mathcal{H}}\text{-}\mathrm{stmod}_{\mathcal{H}}.
\end{gather*}
\end{theorem}

\begin{example}
The main purpose of this paper is to study
$\cS\text{-}\mathrm{stmod}$. Theorem \ref{theorem:strongH} tells us that,
by varying over all two-sided cells $\mathcal{J}$ and choosing an arbitrary but fixed
diagonal $\mathcal{H}$-cell in each $\mathcal{J}$, this problem can be
reduced to studying $\cSH\text{-}\mathrm{stmod}_{\mathcal{H}}$. The main result of this paper, which is Theorem \ref{theorem:main},
shows that the latter problem can be reduced even further.
\end{example}

The main reference for the above is \cite[Subsections 2.6, 4.7 and 4.8]{MMMTZ2}.


\subsection{Grading conventions}\label{subsection:grading}


Let $\Bbbk\text{-}\mathrm{fgmod}$ denote the category of finite dimensional ($\mathbb{Z}$-)graded $\Bbbk$-vector spaces. An object in $\Bbbk\text{-}\mathrm{fgmod}$
has the form $V=\bigoplus_{t\in\mathbb{Z}}V_{t}$, where $V_{t}$ denotes
the linear subspace of elements of $V$ which are homogeneous of degree $t$.
Morphisms in $\Bbbk\text{-}\mathrm{fgmod}$ are $\Bbbk$-linear maps
(not necessarily homogeneous, but each morphism is a linear combination
of homogeneous morphisms). The group $\mathbb{Z}$ acts
on $\Bbbk\text{-}\mathrm{fgmod}$ by grading shift $\langle{}_{-}\rangle$ via the rule
$(V\langle t\rangle)_{s}=V_{s+t}$, for all $s,t\in\mathbb{Z}$.
From now on, if
$p(\mathsf{v})=n_{\mone[k]}\mathsf{v}^{\mone[k]}+\dots+n_{l}\mathsf{v}^l\in\mathbb{N}_{0}[\mathsf{v},\mathsf{v}^{\mone}]$
and $V$ is a graded vector space, then we use the notation
\begin{gather*}
V^{\oplus p}:=V\langle k\rangle^{\oplus n_{\mone[k]}}\oplus
\cdots\oplus V\langle -l\rangle^{\oplus n_{l}}.
\end{gather*}
With this notation, we have e.g. $V\langle -t\rangle=V^{\oplus\mathsf{v}^{t}}$.
Further, if $V$ is a finite dimensional graded
vector space, then its graded dimension
$\mathrm{grdim}(V)\in\mathbb{N}_{0}[\mathsf{v},\mathsf{v}^{\mone}]$
is uniquely defined by the property
that the graded vector spaces $V$ and $\Bbbk^{\oplus\mathrm{grdim}(V)}$ are isomorphic by a homogeneous isomorphism of degree $0$, where $\Bbbk$ is concentrated in degree zero.
A finite dimensional algebra $\mathsf{A}$ is called \emph{positively graded} if it
is non-negatively graded, i.e. $\mathrm{grdim}(\mathsf{A})\in\mathbb{N}_{0}[\mathsf{v}]$,
and its degree $0$ component $\mathsf{A}_{0}$ is semisimple.
A finite dimensional positively graded algebra $\mathsf{A}$ is called a \emph{graded Frobenius algebra of length $k$} if $\mathrm{Hom}_{\Bbbk}(\mathsf{A},\Bbbk)\cong\mathsf{A}^{\oplus\mathsf{v}^{\mone[k]}}$
as graded left $\mathsf{A}$-modules, where again the isomorphism is homogeneous of degree zero.

A \emph{graded $\Bbbk$-linear category} $\mathcal{C}$ is
a category enriched over $\Bbbk\text{-}\mathrm{fgmod}$ (in this paper, most categories
will have finite dimensional morphism spaces, but there are some exceptions).
This means that
\begin{gather*}
\mathrm{Hom}_{\mathcal{C}}(X,Y)=\bigoplus_{t\in\mathbb{Z}}\mathrm{Hom}_{\mathcal{C}}(X,Y)_{t}
\end{gather*}
and the composite of two homogeneous morphisms of degrees $t_{1}$ and $t_{2}$ is homogeneous
of degree $t_{1}+t_{2}$. We let
\begin{gather*}
\mathrm{hom}_{\mathcal{C}}(X,Y):=\mathrm{Hom}_{\mathcal{C}}(X,Y)_{0}.
\end{gather*}
By definition, $\mathcal{C}^{(0)}$ is the (non-full) subcategory of $\mathcal{C}$
given by taking the degree zero morphisms between all objects of $\mathcal{C}$, i.e.
\begin{gather*}
\mathrm{Hom}_{\mathcal{C}^{(0)}}(X,Y):=\mathrm{hom}_{\mathcal{C}}(X,Y).
\end{gather*}
The grading on $\mathcal{C}^{(0)}$ is, of course, trivial.

We call a functor $\mathrm{F}\colon\mathcal{C}\to\mathcal{D}$ between
graded $\Bbbk$-linear categories a \emph{graded functor}, provided the assignment
\begin{gather*}
\mathrm{Hom}_{\mathcal{C}}(X,Y)\to\mathrm{Hom}_{\mathcal{D}}\big(\mathrm{F}(X),\mathrm{F}(Y)\big)
\end{gather*}
on morphisms is homogeneous of degree zero.

Following \cite{CiMa}, a $\Bbbk$-linear category $\mathcal{C}$ is a \emph{$\mathbb{Z}$-category}
if it is equipped with
a group homomorphism from $\mathbb{Z}$ to the group of autoequivalences of $\mathcal{C}$.
It is a \emph{free $\mathbb{Z}$-category} if the action of $\mathbb{Z}$ thus obtained is free.
For $t\in\mathbb{Z}$, we again write $X\langle t\rangle=X^{\oplus\mathsf{v}^{\mone[t]}}$ for the action of $t$ on $\mathcal{C}$.

We say that a graded $\Bbbk$-linear category $\mathcal{C}$ \emph{admits translation} if it
is a free $\mathbb{Z}$-category and for every $t\in\mathbb{Z}$ there is a homogeneous
natural isomorphism
$\phi_{t}\colon\langle 0\rangle\xrightarrow{\cong}\langle t\rangle$ of degree $t$,
such that $\phi_{0}=\mathrm{Id}_{\ccC}$ and $\phi_{s}\phi_{t}=\phi_{s+t}$, for $s,t\in\mathbb{Z}$. In particular, we have
\begin{gather*}
\mathrm{Hom}_{\mathcal{C}}(X^{\oplus\mathsf{v}^{\mone[s]}},Y^{\oplus\mathsf{v}^{\mone[t]}})
\cong
\mathrm{Hom}_{\mathcal{C}}(X,Y)^{\oplus\mathsf{v}^{s\mone[t]}}
\end{gather*}
for all objects $X,Y\in\mathcal{C}$ and $s,t\in\mathbb{Z}$. If $\mathcal{C}$
comes with a fixed choice of isomorphisms $\phi_{t}$, we say that $\mathcal{C}$ is a category \emph{with translation}.

A graded functor
$\mathrm{F}\colon\mathcal{C}\to\mathcal{D}$ between two graded categories with translation
automatically preserves translations (up to natural isomorphism), as follows from the
existence of the homogeneous degree-zero natural isomorphism
\begin{gather*}
\mathrm{F}(X^{\oplus\mathsf{v}^{t}})\xrightarrow{\mathrm{F}(\phi^{\mathcal{C}}_{t})}
\mathrm{F}(X)\xrightarrow{\phi^{\mathcal{D}}_{-t}}\mathrm{F}(X)^{\oplus\mathsf{v}^{t}}
\end{gather*}
for any $X\in\mathcal{C}$ and $t\in\mathbb{Z}$.

Suppose that $\mathcal{C}$ admits translation. Then in $\mathcal{C}^{(0)}$
the objects $X$ and $X^{\oplus\mathsf{v}^{t}}$ are
isomorphic if and only if $t=0$, so $\mathcal{C}^{(0)}$ does not admit translation,
although it is a free $\mathbb{Z}$-category. Note also
that $\mathcal{C}^{(0)}$ is never finitary, because
$\mathcal{C}^{(0)}$ has infinitely many isomorphism classes of indecomposable objects.
We say that a free $\mathbb{Z}$-category $\mathcal{X}$
is \emph{$\mathbb{Z}$-finitary} if
there is a finite set of indecomposable objects $X_{i}\in\mathcal{X}$, $i\in I$,
such that every indecomposable object in $\mathcal{X}$ is isomorphic to
$X_{i}^{\oplus\mathsf{v}^{t}}$, for some $i\in I$ and $t\in\mathbb{Z}$, and every object in
$\mathcal{X}$ has an essentially unique decomposition into indecomposables objects.
We will call $\mathcal{C}$ \emph{graded finitary} provided that
$\mathcal{C}^{(0)}$ is $\mathbb{Z}$-finitary.

Given a graded $\Bbbk$-linear category $\mathcal{C}$, we can define a new
graded $\Bbbk$-linear category $\mathcal{C}^{\,\prime}$ with objects
$(X,t)$, where $X\in\mathcal{C}$ and $t\in\mathbb{Z}$,
and
\begin{gather*}
\mathrm{Hom}_{\mathcal{C}^{\prime}}\big((X,s),(Y,t)\big)
:=
\mathrm{Hom}_{\mathcal{C}}(X,Y)^{\oplus\mathsf{v}^{s\mone[t]}},
\end{gather*}
with the evident composition. Then $\mathcal{C}^{\,\prime}$ has a translation, with
$\phi_{t}$ given by the identity on $\mathcal{C}$ shifted by $t$ degrees, and
the natural embedding of $\mathcal{C}$ into $\mathcal{C}^{\,\prime}$ sending $X$ to $(X,0)$
is an equivalence of ungraded categories.

\begin{remark}\label{rem:subtle}
Note that the above embedding is degree-preserving but it need not be
an equivalence of graded categories in general,
e.g. if $\mathcal{C}=\mathcal{C}^{(0)}$, then the functorial
inverse of the embedding would have to be given by $(X,t)\mapsto X$, for all $X\in\mathcal{C}$ and
$t\in\mathbb{Z}$. However, that functor is not degree-preserving, as it maps any homogeneous morphism
of any degree in $\mathcal{C}^{\,\prime}$ to a homogeneous morphism of
degree zero in $\mathcal{C}$. Nonetheless, if $\mathcal{C}$ admits translation, then it is
equivalent to $\mathcal{C}^{\prime}$ as a graded category, because
$(X,t)\mapsto X^{\oplus\mathsf{v}^{\mone[t]}}$ defines a functorial inverse of the above embedding.
\end{remark}

Observe that, given a $\Bbbk$-linear category $\mathcal{C}$ admitting translation,
$\mathcal{C}$ is equal to the skew-category $\mathcal{C}^{(0)}[\mathbb{Z}]$ in the sense of
\cite[Definition 2.3]{CiMa}. Recall that for any $\Bbbk$-linear free $\mathbb{Z}$-category
$\mathcal{D}$, the objects of $\mathcal{D}[\mathbb{Z}]$ coincide with those of $\mathcal{D}$
and the morphism spaces are defined by
\begin{gather*}
\mathrm{Hom}_{\mathcal{D}[\mathbb{Z}]}(X,Y):=\bigoplus_{t\in\mathbb{Z}}\mathrm{Hom}_{\mathcal{D}}(X,Y^{\oplus\mathsf{v}^{\mone[t]}}),
\end{gather*}
with the composition being induced by the one in $\mathcal{D}$ after shifting. The
skew-category $\mathcal{D}[\mathbb{Z}]$ is naturally graded and comes with the natural
translation induced by the identity on $\mathcal{D}$. By construction, the isomorphism classes
of its objects correspond bijectively to the $\mathbb{Z}$-orbits of objects in $\mathcal{D}$.

For future use, we record the following easy lemma.

\begin{lemma}\label{lem:skewcat}
Let $\mathcal{C}=\mathcal{C}^{(0)}$ be a trivially graded $\Bbbk$-linear category,
and let $\mathcal{C}^{\prime}$ be defined as above. Then $(\mathcal{C}^{\prime})^{(0)}$
is equivalent, as a $\Bbbk$-linear free $\mathbb{Z}$-category, to
$\bigoplus_{t\in\mathbb{Z}}\mathcal{C}\langle t\rangle$, where $\mathcal{C}\langle t\rangle$
is the full subcategory of $(\mathcal{C}^{\prime})^{(0)}$ with objects $(X,t)$
for $X\in\mathcal{C}$.

Moreover, let $\mathcal{D}:=(\mathcal{C}^{\prime})^{(0)}[\mathbb{Z}]$. Then $\mathcal{C}^{\prime}$
and $\mathcal{D}$ are equivalent as graded $\Bbbk$-linear categories.
\end{lemma}

\begin{proof}
The first statement is clear and the second is a special case of \cite[Proposition 2.6]{CiMa}.
\end{proof}

Another notion that we will need is that of graded semisimplicity. A graded finitary category
$\mathcal{C}$ is called \emph{graded semisimple} if $\mathcal{C}^{(0)}$ is semisimple.
Note that in the latter case there is a finite number of simple object
$L_{i}\in\mathcal{C}^{(0)}$, $i\in I$, and every simple object in $\mathcal{C}^{(0)}$
is isomorphic to $L_{i}^{\oplus\mathsf{v}^{\mone[t]}}$, for some $i\in I$ and $t\in\mathbb{Z}$.

Graded $\Bbbk$-linear categories admitting translation, together with
graded functors and all natural transformations form a $2$-category, which
we will denote by $\mathfrak{A}_{\Bbbk}^{\mathrm{g}}$. We denote by
$\mathfrak{A}_{\Bbbk}^{\mathrm{gf}}$ the $1,2$-full $2$-subcategory
whose objects are graded finitary categories admitting translation.

As our notation emphasizes, all non-trivially graded $\Bbbk$-linear categories
in this paper will be assumed to admit translation.

References for the above are for example \cite[Section 1]{He} and \cite{CiMa}.


\subsection{Graded $2$-categories of $2$-representations}\label{s:grreps}


Many $2$-categories in this paper will be graded, meaning their
morphism categories are in $\mathfrak{A}_{\Bbbk}^{\mathrm{g}}$ and horizontal composition is a graded bifunctor.

By a \emph{graded finitary} $2$-category, we mean a graded $\Bbbk$-linear $2$-category which has
finitely many objects, on which identity $1$-morphisms are indecomposable, and morphism
categories in $\mathfrak{A}_{\Bbbk}^{\mathrm{gf}}$ with horizontal composition being a graded $\Bbbk$-bilinear bifunctor.

If $\cC$ is a graded $\Bbbk$-linear $2$-category, not necessarily admitting translation,
we define the $2$-category with translation $\cC^{\,\prime}$, with the same objects as $\cC$ and morphism spaces defined by
$\cC^{\,\prime}(\mathtt{i},\mathtt{j}):=\cC(\mathtt{i},\mathtt{j})^{\prime}$ and the
evident horizontal composition.

For a graded $2$-category $\cC$, we denote by $\cC^{\,(0)}$ the $2$-subcategory whose morphism categories are given by $\cC^{\,(0)}(\mathtt{i},\mathtt{j}):=\cC(\mathtt{i},\mathtt{j})^{(0)}$. As in
the case of categories, even if $\cC$ is graded finitary, the $2$-subcategory $\cC^{\,(0)}$ will
never be finitary, since $\mathrm{F}$ and $\mathrm{F}^{\oplus\mathsf{v}^{t}}$ are not
isomorphic in $\cC^{\,(0)}$ unless $t=0$. However, $\cC^{\,(0)}$ is a \emph{$\mathbb{Z}$-finitary
$2$-category}, meaning that its morphism categories are $\mathbb{Z}$-finitary categories
and horizontal composition is compatible with the free $\mathbb{Z}$-actions.
Moreover, the skew $2$-category $\cC^{\,(0)}[\mathbb{Z}]$, whose morphism categories are given by $\cC^{\,(0)}[\mathbb{Z}](\mathtt{i},\mathtt{j}):=\cC^{\,(0)}(\mathtt{i},\mathtt{j})[\mathbb{Z}]$, is biequivalent to $\cC$
(since $\cC^{\,(0)}$ is a free $\mathbb{Z}$-category).

A graded finitary $2$-category $\cC$ is called \emph{locally graded semisimple} if all its
morphism categories are graded semisimple. Equivalently, this means that
$\cC^{\,(0)}$ is locally semisimple.

A \emph{graded finitary} $2$-representation $\mathbf{M}$ of a graded finitary
$2$-category $\cC$ is a graded $\Bbbk$-linear $2$-functor $\cC\to\mathfrak{A}_{\Bbbk}^{\mathrm{gf}}$,
where by a graded $2$-functor we mean that all its component functors are graded.

Similarly, a morphism of graded finitary $2$-representations is a collection of graded functors
together with homogeneous natural isomorphisms of degree zero
providing the compatibility with the $2$-actions.

Graded finitary $2$-representations of a graded finitary $2$-category
$\cC$, together with morphisms of graded finitary $2$-representations and
modifications, form a
graded $2$-category, since spaces of modifications are naturally enriched over
$\Bbbk\text{-}\mathrm{fgmod}$. This $2$-category is denoted by $\cC\text{-}\mathrm{gafmod}$.

The concepts of transitive, simple transitive and cell $2$-representations are defined as in
the ungraded case. In particular, we can define the $1,2$-full $2$-subcategory
$\cC\text{-}\mathrm{gstmod}$, whose objects are graded simple transitive $2$-representations of
$\cC$. As in the ungraded case, a subscript $\mathcal{J}$ indicates that we are restricting
our attention to simple transitive $2$-representations with apex $\mathcal{J}$.

For a graded finitary $2$-representation $\mathbf{M}$
of a graded finitary $2$-category $\cC$, we define
a graded finitary $2$-representation $\mathbf{M}^{\prime}$ of $\cC^{\,\prime}$ by
$\mathbf{M}^{\prime}(\mathtt{i}):=\mathbf{M}(\mathtt{i})^{\prime}$, for every
$\mathtt{i}\in\cC$, with the obvious extension of the action.

Similarly, for a graded finitary $2$-representation $\mathbf{M}$
of a graded finitary $2$-category $\cC$, we define a $2$-representation $\mathbf{M}^{(0)}$ of $\cC^{\,(0)}$ by
$\mathbf{M}^{(0)}(\mathtt{i}):=\mathbf{M}(\mathtt{i})^{(0)}$ for every
$\mathtt{i}\in\cC$. Again, $\mathbf{M}^{(0)}$ is a \emph{$\mathbb{Z}$-finitary $2$-representation},
meaning that $\mathbf{M}^{(0)}(\mathtt{i})$ is a $\mathbb{Z}$-finitary category, for every $\mathtt{i}\in\cC$,
and all other data of the $2$-representation are
compatible with the free $\mathbb{Z}$-actions. Conversely, given a $\mathbb{Z}$-finitary
$2$-representation $\mathbf{N}$ of $\cC^{\,(0)}$, we define a graded finitary $2$-representation
$\mathbf{N}[\mathbb{Z}]$ of $\cC$ by setting
$\mathbf{N}[\mathbb{Z}](\mathtt{i}):=\mathbf{N}(\mathtt{i})[\mathbb{Z}]$
and extending the other data of the $2$-representation in the obvious way.
Finally,
we say that a graded finitary $2$-representation $\mathbf{M}$
is \emph{graded semisimple} if $\mathbf{M}^{(0)}$ is semisimple.

Following \cite[Subsection 6.4]{McPh}, given a graded finitary $2$-representation $\mathbf{M}$, we can associate a coalgebra $\mathrm{C}$ in $\underline{\cC^{\,(0)}}$ to $\mathbf{M}^{(0)}$, and $\mathbf{inj}_{\underline{\ccC}}(\mathrm{C})\simeq \mathbf{M}$.
In other words, we can associate to $\mathbf{M}$ a coalgebra $\mathrm{C}$ whose comultiplication and counit are homogeneous of degree zero. We, therefore, call it a \emph{graded coalgebra}. Now, suppose that $\mathbf{M}$ is simple transitive, then $\mathrm{C}$ is cosimple in $\underline{\cC}$, see Subsection \ref{s2.23}. This implies that $\mathrm{C}$ is also cosimple in
$\underline{\cC^{\,(0)}}$, of course, whence $\mathbf{M}^{\,(0)}\simeq
\mathbf{inj}_{\underline{\ccC^{\,(0)}}}(\mathrm{C})$ is also simple transitive (again, see Subsection \ref{s2.23}). In all this, however,
the definition of the graded injective abelianization
$\underline{\cC}$ requires some care. It is defined as
\begin{gather*}
\underline{\cC}:=(\underline{\cC^{\,(0)}})^{\prime}.
\end{gather*}
This means that the $1$-morphisms in $\underline{\cC}$ are complexes of the form
\begin{gather*}
\mathrm{X}\xrightarrow{\alpha}\mathrm{Y},
\end{gather*}
where $\mathrm{X}$, $\mathrm{Y}$ are $1$-morphisms in $\cC$ and $\alpha\in\mathrm{hom}_{\ccC}(\mathrm{X},\mathrm{Y})$,
and $2$-morphisms in $\underline{\cC}$ are defined in the usual way without the restriction of homogeneity.
With this definition, $\underline{\cC}$ is indeed a graded abelian $2$-category and,
for any graded coalgebra $\mathrm{C}$ in $\underline{\cC(\mathtt{i,i})}$, the category
$\mathrm{inj}_{\underline{\ccC}}(\mathrm{C})$ is indeed a graded subcategory of
$\coprod_{\mathtt{j}\in\ccC}\underline{\cC(\mathtt{i,j})}$.


\section{Soergel bimodules and the asymptotic bicategory}\label{s5}


\subsection{Soergel bimodules}\label{s5.1}


From now on we set $\Bbbk=\mathbb{C}$ (or any other algebraically closed
field of characteristic zero), and all categories and $2$-categories etc. are over $\mathbb{C}$ unless indicated otherwise.

Let $(W,\mathtt{S})$ be a Coxeter system of finite type. We fix a reflection faithful representation
of $W$ and let $\mathsf{C}$ be the corresponding coinvariant algebra with the usual
($\mathbb{Z}$-)grading. We denote by $\cS=\cS_{\mathbb{C}}(W,\mathtt{S})$ the $2$-category
of \emph{Soergel bimodules} over $\mathsf{C}$, by which we mean the $2$-category with
one object $\varnothing$, which we identify with the category of finite dimensional graded $\mathsf{C}$-modules,
whose $1$-morphisms are endofunctors of the latter category isomorphic to tensoring with Soergel bimodules, and whose $2$-morphisms are natural
transformations, which we can identify with bimodule homomorphisms.
The $2$-category $\cS$ is graded fiat (this follows from e.g. \cite{EW}, see also \cite[Subsection 7.1]{MM1} for the case of Weyl groups).

By \cite{So2} (see also \cite{EW3}), there is a $\mathbb{Z}[\mathsf{v},\mathsf{v}^{\mone}]$-linear isomorphism
between the split Grothendieck ring of $\cS^{\,(0)}$ and the Hecke algebra
$\mathsf{H}:=\mathsf{H}_{\mathbb{Z}[\mathsf{v},\mathsf{v}^{\mone}]}(W,\mathtt{S})$ of $W$ such that,
for each $w\in W$, there is a unique (up to homogeneous isomorphism) indecomposable Soergel bimodule
$\mathrm{B}_{w}$ in $\cS^{(0)}$ whose Grothendieck class is sent to the Kazhdan--Lusztig basis element
$b_{w}$ corresponding to $w$, and $[\mathrm{X}^{\oplus\mathsf{v}}]=\mathsf{v}[\mathrm{X}]$
for any $1$-morphism $\mathrm{X}$ in $\cS^{(0)}$. We call this fact the \emph{Soergel--Elias--Williamson
categorification theorem} and we call the isomorphism the \emph{character isomorphism}.
As a consequence, the cell structure of $\cS$ is given by the Kazhdan--Lusztig combinatorics.
In particular, \emph{Lusztig's conjectures} \cite[Conjecture 14.2]{Lu2}, which we will use several times,
hold in our case, see e.g. \cite[Subsection 15.1]{Lu2} or \cite[Corollary 1.4]{duCl}.

We set
\begin{gather*}
\mathrm{C}_{w}:=\mathrm{B}_{w}^{\oplus\mathsf{v}^{\mathbf{a}(w)}}
\quad\text{and}\quad
c_{w}:=\mathsf{v}^{\mathbf{a}(w)}b_{w},
\end{gather*}
where we recall that, to each $w\in W$, Lusztig \cite{Lu-1} assigns a number $\mathbf{a}(w)$,
called its \textit{$\mathbf{a}$-value}, such that the function $\mathbf{a}\colon W\to\mathbb{N}_{0}$ is constant on two-sided cells.
In $\overline{\cS}$ the projective $1$-morphism
$\mathrm{B}_{w}$ has a simple head, and dually in $\underline{\cS}$
the injective $1$-morphism
$\mathrm{B}_{w}$ has a simple socle.

\begin{example}
For $s\in\mathtt{S}$ we have $\mathrm{B}_s\mathrm{B}_s\cong\mathrm{B}_s^{\oplus(\mathsf{v}^{\mone}+\mathsf{v})}$ in $\cS^{\,(0)}$,
but $\mathrm{C}_s\mathrm{C}_s\cong\mathrm{C}_s^{\oplus(1+\mathsf{v}^{2})}$.
The heads of $\mathrm{B}_s$ and $\mathrm{C}_s$, seen as projective
$1$-morphisms in $\overline{\cS}$, are concentrated in
degrees $-1$ and $0$, respectively.
\end{example}

For $x,y,z\in W$, we let $h_{x,y,z}\in\mathbb{N}_{0}[\mathsf{v},\mathsf{v}^{\mone}]$
and $\gamma_{x,y,z^{\mone}}\in\mathbb{N}_{0}$
be given by
\begin{gather}\label{eq:kl-mult}
\mathrm{B}_{x}\mathrm{B}_{y}\cong\bigoplus_{z\in W}\mathrm{B}_{z}^{\oplus h_{x,y,z}}\text{ in }\cS^{\,(0)},\quad
h_{x,y,z}=\mathsf{v}^{\mathbf{a}(z)}\gamma_{x,y,z^{\mone}}
\;\big(\!\bmod\mathsf{v}^{\mathbf{a}(z)-1}\mathbb{N}_{0}[\mathsf{v}^{\mone}]\big),
\end{gather}
where $\mathbf{a}$ is Lusztig's $\mathbf{a}$-function.
Since the $h$ are bar invariant, i.e. invariant with respect to the symmetry
$\mathsf{v}\leftrightarrow\mathsf{v}^{\mone[1]}$, we also have
\begin{gather}\label{eq:kl-coeff}
h_{x,y,z}=\mathsf{v}^{\mone[\mathbf{a}(z)]}\gamma_{x,y,z^{\mone}}
\;\big(\!\bmod\mathsf{v}^{\mone[\mathbf{a}(z)]+1}\mathbb{N}_{0}[\mathsf{v}]\big).
\end{gather}
By the Soergel--Elias--Williamson categorification theorem,
the $h$ are also the structure constants of $\mathsf{H}$ with respect to the basis
$\{b_{w}\mid w\in W\}$ and the $\gamma$ are the structure constants of Lusztig's \emph{asymptotic
Hecke algebra}, denoted $\mathsf{A}$, with respect to the basis $\{a_{w}\mid w\in W\}$.
(Lusztig \cite{Lu1} called the asymptotic Hecke algebra the \emph{$\mathcal{J}$-ring},
denoting its basis elements by $t_{w}$ in \cite{Lu1}.)

Recall from Subsection \ref{s:repcats} that, for every two-sided cell $\mathcal{J}$, there is a
$\mathcal{J}$-simple
subquotient $\cSJ$ of $\cS$ and, for every left cell $\mathcal{L}$ in $\mathcal{J}$, there is a
diagonal $\mathcal{H}$-cell $\mathcal{H}=\mathcal{H}(\mathcal{L})=\mathcal{L}\cap\mathcal{L}^{\star}$
and a $2$-full $2$-subcategory $\cSH$ of $\cSJ$ which is $\mathcal{H}$-simple. Recall also that every
left cell $\mathcal{L}$ contains a unique Duflo involution $d=d_{\mathcal{L}}$
and that $d_{\mathcal{L}}\in\mathcal{H}(\mathcal{L})$. Let $\mathcal{D}(\mathcal{J})$  denote the set of Duflo involutions in $\mathcal{J}$. Both $\cSJ$ and $\cSH$ have asymptotic counterparts too, which we
denote by $\cAJ$ and $\cAH$, respectively, and will recall in Subsection \ref{s5.12}. For every two-sided
cell $\mathcal{J}$ of $W$, the split Grothendieck group of $\cAJ$ is a subalgebra of
$\mathsf{A}$, denoted $\mathsf{A}_{\mathcal{J}}$. The unit of $\mathsf{A}_{\mathcal{J}}$ is equal to
$\sum_{d\in\mathcal{D}(\mathcal{J})}a_{d}$ and there is a decomposition of algebras
\begin{gather*}
\mathsf{A}\cong \prod_{\mathcal{J}\subset W}\mathsf{A}_{\mathcal{J}},
\end{gather*}
where the direct product runs over all two-sided cells of $W$. Note that $\mathsf{A}_{\mathcal{J}}$ is an idempotent subalgebra of $\mathsf{A}$. Similarly, for every diagonal
$\mathcal{H}$-cell $\mathcal{H}\subset\mathcal{J}$, the split Grothendieck group of $\cAH$ is an
idempotent subalgebra of $\mathsf{A}_{\mathcal{J}}$, denoted by $\mathsf{A}_{\mathcal{H}}$.
The unit of $\mathsf{A}_{\mathcal{H}}$ is equal to $a_{d}$,
but the direct product of these $\mathsf{A}_{\mathcal{H}}$ is strictly smaller than
$\mathsf{A}_{\mathcal{J}}$ in general. In fact, $\mathsf{A}_{\mathcal{J}}$ decomposes as
\begin{gather*}
\mathsf{A}_{\mathcal{J}}\cong\bigoplus_{\mathcal{L}_{1},\mathcal{L}_{2}\subset\mathcal{J}}\mathsf{A}_{\mathcal{L}_{1}^{\star}\cap \mathcal{L}_{2}}
\end{gather*}
as a vector space, where the direct sum runs over all pairs of left cells in $\mathcal{J}$ and
$\mathsf{A}_{\mathcal{L}_{1}^{\star}\cap\mathcal{L}_{2}}$ is the direct summand of
$\mathsf{A}_{\mathcal{J}}$ generated by $\{a_{w}\mid w\in\mathcal{L}_{1}^{\star}\cap\mathcal{L}_{2}\}$, which is an $\mathsf{A}_{\mathcal{H}(\mathcal{L}_{1})}$-$\mathsf{A}_{\mathcal{H}(\mathcal{L}_{2})}$-bimodule. Alternatively, we can see
$\mathsf{A}_{\mathcal{J}}$ as a finitary category, whose objects are indexed by the left
cells $\mathcal{L}\subset\mathcal{J}$ and whose morphism categories are defined by
\begin{gather*}
\mathrm{Hom}_{\mathsf{A}_{\mathcal{J}}}(\mathcal{L}_{1},\mathcal{L}_{2}):=\mathsf{A}_{\mathcal{L}_{1}^{\star}\cap\mathcal{L}_{2}}.
\end{gather*}
The identity morphism on $\mathcal{L}$ is given by $a_{d_{\mathcal{L}}}$.

From Soergel's hom formula, see \cite[Theorem 3.6]{EW}, we obtain
\begin{gather}\label{eq:soergel}
\mathrm{dim}\big(
\mathrm{hom}_{\ccS}(\mathrm{B}_{v},\mathrm{B}_{w}^{\oplus\mathsf{v}^k})
\big)
=\delta_{v,w}\delta_{0,k},
\quad v,w\in W,k\in\mathbb{Z}_{\geq 0}.
\end{gather}
By slight abuse of terminology, we will also refer to \eqref{eq:soergel}
as \emph{Soergel's hom formula}. Consequently, for any subset $U\subset W$,
the $2$-endomorphism algebra of $\oplus_{w\in U}\mathrm{B}_{w}$ in $\cS$ is positively graded.
This property is inherited by $\cSH$.

Let $\mathbf{a}(\mathcal{H}):=\mathbf{a}(x)$ for any $x\in\mathcal{H}$.
Finally, the following lemmas are evident
and we state them for later use.

\begin{lemma}\label{lem5.2}
In $\cSHo$,
for all $x,y,z\in\mathcal{H}$, we have that
\begin{gather*}
\mathrm{C}_{x}\mathrm{C}_{y}\cong\bigoplus_{z\in\mathcal{H}}\mathrm{C}_{z}^{\oplus\mathsf{v}^{\mathbf{a}(\mathcal{H})}h_{x,y,z}}.
\end{gather*}
In particular, if $\mathrm{C}_{z}^{\oplus\mathsf{v}^k}$ is isomorphic to
a direct summand of $\mathrm{C}_{x}\mathrm{C}_{y}$, then $k\geq 0$.
\end{lemma}

Lemma \ref{lem5.2} has the following ``negative'' counterpart.

\begin{lemma}\label{lem5.2new}
For $x\in\mathcal{H}$, set $\tilde{\mathrm{C}}_{x}:=\mathrm{C}_{x}^{\oplus\mathsf{v}^{\mone[2\mathbf{a}(\mathcal{H})]}}=
\mathrm{B}_{x}^{\oplus\mathsf{v}^{\mone[\mathbf{a}(\mathcal{H})]}}$.
Then, in $\cSHo$,
for all $x,y,z\in\mathcal{H}$, we have that
\begin{gather*}
\tilde{\mathrm{C}}_{x}\tilde{\mathrm{C}}_{y}\cong
\bigoplus_{z\in\mathcal{H}}\tilde{\mathrm{C}}_{z}^{\oplus\mathsf{v}^{\mone[\mathbf{a}(\mathcal{H})]}h_{x,y,z}}.
\end{gather*}
In particular, if $\tilde{\mathrm{C}}_{z}^{\oplus\mathsf{v}^k}$ is isomorphic to
a direct summand of $\tilde{\mathrm{C}}_{x}\tilde{\mathrm{C}}_{y}$, then $k\leq 0$.
\end{lemma}


\subsection{The asymptotic bicategory}\label{s5.12}


We define two $2$-semicategories inside $\cSHo$:
\begin{gather*}
\mathcal{X}:=
\mathrm{add}\big(\{\mathrm{C}_{w}^{\oplus\mathsf{v}^{k}}\mid w\in
\mathcal{H},k\geq 0\}\big)^{(0)},
\\
\tilde{\mathcal{X}}:=\mathrm{add}
\big(\{\mathrm{C}_{w}^{\oplus\mathsf{v}^{k}}\mid w\in\mathcal{H},k>0\}
\big)^{(0)}.
\end{gather*}
The $2$-semicategory $\mathcal{X}$ is, in fact, a lax monoidal category
with lax identity $1$-morphism $\Cd$
and strict associators. Let us explain this in detail.
From \cite[Subsection 7.6]{MM3} and positivity of the grading on Soergel bimodules,
it follows that there is a unique, up to a non-zero scalar, $2$-morphism
$\epsilon_{d}\colon\Cd\to\mathbbm{1}_{\varnothing}$ in $\cSo$. The
lax structure of the identity $1$-morphism of $\mathcal{X}$ on $\Cd$ is now defined,
for $\mathrm{X}\in\mathcal{X}$, by the two $2$-morphisms in $\cSo$
\begin{gather*}
\ell_{\mathrm{X}}\colon\Cd\mathrm{X}
\xrightarrow{\epsilon_{d}\circ_{\mathrm{h}}\mathrm{id}_{\mathrm{X}}
}\mathrm{X},
\quad
r_{\mathrm{X}}\colon\mathrm{X}\Cd
\xrightarrow{\mathrm{id}_{\mathrm{X}}\circ_{\mathrm{h}}
\epsilon_{d}}\mathrm{X},
\end{gather*}
with the unitality condition expressed, for $\mathrm{X},\mathrm{Y}\in\mathcal{X}$, by the diagram
\begin{gather}\label{eqnn125}
\begin{xy}
\xymatrix{
(\mathrm{X}\Cd)\mathrm{Y}
\ar@{=}[rr]\ar[rd]_{r_{\mathrm{X}}\circ_{\mathrm{h}}\mathrm{id}_{\mathrm{Y}}}
&
& \mathrm{X}(\Cd\mathrm{Y})\ar[dl]^{\,\mathrm{id}_{\mathrm{X}}
\circ_{\mathrm{h}}\ell_{\mathrm{Y}}}
\\
& \mathrm{X}\mathrm{Y}
&
}
\end{xy}
\end{gather}
which commutes by associativity and the interchange law.

The bicategory $\cAH$ is defined as the quotient $\mathcal{X}/(\tilde{\mathcal{X}})$,
where $(\tilde{\mathcal{X}})$ denotes the $2$-ideal of $\mathcal{X}$ generated by $\tilde{\mathcal{X}}$.
We denote by ${\Pi}:\mathcal{X}\to \mathcal{X}/(\tilde{\mathcal{X}})$
the natural projection.

Note that $\cAH$ is only a bicategory as it does not contain any strict identity $1$-morphism,
but, at the same time, the composition in $\cAH$ is strictly associative.
Up to isomorphism, the identity $1$-morphism in $\cAH$ is the image
$\mathrm{A}_{d}:={\Pi}(\Cd)$ of the lax identity $1$-morphism
$\Cd$ of $\mathcal{X}$. Finally, note that there is no non-trivial $\mathbb{Z}$-action on
$\cAH$ and the grading is trivial, i.e. all $2$-morphisms are homogeneous of degree zero.

By the Soergel--Elias--Williamson categorification theorem and \eqref{eq:kl-mult}, $\cAH$
categorifies the asymptotic Hecke
algebra $\mathsf{A}_{\mathcal{H}}$ associated to $\mathcal{H}$
in the sense that $\mathrm{A}_{w}:={\Pi}(\mathrm{C}_{w})$
categorifies the basis element $a_{w}$ for $w\in\mathcal{H}$.
The algebra $\mathsf{A}_{\mathcal{H}}$ is a fusion algebra
and the following categorifies this fact.

\begin{proposition}\label{prop:fusion}
$\cAH$ is a (one-object) pivotal fusion bicategory.
Moreover, if $W$ is a finite Weyl group, then $\cAH$ is biequivalent
to $\mathcal{C}_{\zeta}^{\mathcal{H}}$
from \cite[Section 5]{BFO} and \cite{Os3}.
\end{proposition}

Note that this implies that $\cAH$, $\underline{\cAH}$ and $\overline{\cAH}$ are canonically isomorphic.

\begin{proof}
The bicategory $\cAH$ is fiat by
\cite[Theorem 5.2]{EW2} (see also \cite[Subsection 18.19]{Lu2},
\cite[Subsection 9.3]{Lu3} and \cite[Subsection 4.3]{BFO}). Moreover, due to the definition of $\cAH$ and
Soergel's hom-formula (see \eqref{eq:soergel}), $\cAH$ is locally semisimple.
Finally, the connection to
\cite[Section 5]{BFO} and \cite{Os3} follows from Soergel's
identification of $\cS$ and the category of semisimple perverse sheaves on the
associated flag variety, see \cite[Erweiterungssatz 5]{So}.
\end{proof}

\begin{remark}\label{rem:fusion}
For completeness, we note that there are analogs of the above for two-sided
cells $\mathcal{J}$ instead of diagonal $\mathcal{H}$-cells. In fact, for finite
and affine Weyl groups, Lusztig and Bezrukavnikov--Finkelberg--Ostrik (see \cite{BFO}
and the references therein) showed that the
pivotal fusion bicategories $\cAJ$ corresponding to $\mathcal{J}$ can be described
geometrically as bicategories of vector bundles on a square of a finite set equivariant with respect to an algebraic group.

A small remark about terminology: When $\mathcal{J}$ is finite, e.g. for finite Coxeter type,
one can also see $\cAJ$ as a pivotal multifusion category with decomposable identity object
isomorphic to $\bigoplus_{d\in\mathcal{J}}\Cd$, where the finite direct sum is over all
Duflo involutions in $\mathcal{J}$. In fact, this is Lusztig's point of view, see \cite[Subsection 18.15]{Lu2}. When $\mathcal{J}$ has
infinitely many Duflo involutions (which, of course, can only happen when $\mathcal{J}$
is infinite), there is no identity object, as Lusztig indeed remarks. However,
for all $\mathcal{J}$, finite or infinite, one can consider $\cAJ$ to be a pivotal fusion bicategory,
whose (possibly infinitely many) objects are indexed by the left cells of $\mathcal{J}$,
which correspond bijectively to the Duflo involutions in $\mathcal{J}$. 	

Finally, if $\mathcal{J}$ is finite, by strong $\mathcal{H}$-reduction (which we recalled
in Theorem \ref{theorem:strongH}), Proposition \ref{prop:fusion} and the fact that one-object
fusion bicategories have only one two-sided cell, there is a biequivalence
$\cAJ\text{-}\mathrm{stmod}\simeq\cAH\text{-}\mathrm{stmod}$
for any diagonal $\mathcal{H}$-cell $\mathcal{H}\subseteq \mathcal{J}$.
As we will summarize in Section \ref{appendix}, in each two-sided cell $\mathcal{J}$ of a (finite)
Weyl group one can pick a \emph{nice} diagonal $\mathcal{H}$-cell $\mathcal{H}$, i.e.
one such that $\cAH$ is a well-known pivotal fusion bicategory for which the Classification Problem has
been solved. For (finite) non-Weyl groups this is no longer true in general. However, even for
those groups, roughly half the number of the two-sided cells contains a nice diagonal $\mathcal{H}$-cell.
This is why we restrict our attention to $\cAH$ in this paper.
\end{remark}


\subsection{Going up}\label{s5.2}


This subsection investigates a certain
oplax pseudofunctor from $\cAH$ to $\cSH$ which
we loosely call ``going up''.

By the definition of $\cAH$, the projection $\Pi$ is the identity on
$\mathrm{add}(\{\mathrm{C}_{w}\mid w\in\mathcal{H}\})^{(0)}$ and
zero on $(\tilde{\mathcal{X}})$, so it is a genuine (in contrast to lax) $2$-semifunctor.
It sends the lax identity $\Cd$ of $\mathcal{X}$ to the identity $\mathrm{A}_{d}$ of
$\cAH$ (which is a weak identity, in the sense that there are non-trivial left and right
unitors) and, for any $\mathrm{X},\mathrm{Y}\in\mathcal{X}$, we have
$\Pi(\mathrm{X})\Pi(\mathrm{Y})=\Pi(\mathrm{XY})$, by definition.

Indeed, for any $\mathrm{X}=\Pi(\mathrm{F})\in\cAH$, we can define the $2$-morphism
\begin{gather*}
\lambda_{\mathrm{X}}\colon\mathrm{A}_{d}\mathrm{X}=\Pi(\Cd)\mathrm{X}=\Pi(\Cd)\Pi(\mathrm{F})
=\Pi(\Cd\mathrm{F})
\xrightarrow{\Pi(\ell_{\mathrm{F}})}\Pi(\mathrm{F})=\mathrm{X}.
\end{gather*}
Then $\lambda$ is a natural transformation from
$\mathrm{A}_{d}\circ{}_-$ to the identity $2$-functor on $\cAH$
and we have $\lambda_{\Pi(\mathrm{F})}=\Pi(\ell_{\mathrm{F}})$.
Similarly, one can define a natural transformation $\rho$ from
${}_-\circ\mathrm{A}_{d}$ to the identity $2$-functor on $\cAH$
via $\rho_{\Pi(\mathrm{F})}=\Pi(r_{\mathrm{F}})$.
Here $\lambda$ and $\rho$ are the left and right unitors for the bicategory $\cAH$.
In details, by applying $\Pi$ to \eqref{eqnn125}, we obtain the commutative diagram
\begin{gather*}
\begin{xy}
\xymatrix{
(\mathrm{X}\mathrm{A}_{d})\mathrm{Y}
\ar@{=}[rrrr]\ar@{=}[d]
&&&
& \mathrm{X}(\mathrm{A}_{d}\mathrm{Y})\ar@{=}[d].
\\
\big(\mathrm{X}\Pi(\Cd)\big)\mathrm{Y}\ar[rr]^{\ \ \rho_{\mathrm{X}}\circ_{\mathrm{h}}\mathrm{id}_{\mathrm{Y}}}
&& \mathrm{X}\mathrm{Y}
&&\mathrm{X}\big(\Pi(\Cd)\mathrm{Y}\big),
\ar[ll]_{\,\mathrm{id}_{\mathrm{X}}\circ_{\mathrm{h}}\lambda_{\mathrm{Y}}\ \ }
}
\end{xy}
\end{gather*}
for all $\mathrm{X},\mathrm{Y}\in\cAH$.

Observing that $\lambda_{\Pi(\mathrm{F})}=\Pi(\ell_{\mathrm{F}})$ and $\rho_{\Pi(\mathrm{F})}=\Pi(r_{\mathrm{F}})$, we
obtain the commutative diagrams
\begin{gather*}
\begin{xy}
\xymatrix{
\mathrm{A}_{d}\Pi(\mathrm{F})
\ar@{=}[rr]\ar[d]_{\lambda_{\Pi(\mathrm{F})}}
&& \Pi(\Cd)\Pi(\mathrm{F})
\ar@{=}[d]\\
\Pi(\mathrm{F})
&& \Pi(\Cd\mathrm{F}),
\ar[ll]_{\Pi(\ell_{\mathrm{F}})}
}
\end{xy}\quad
\begin{xy}
\xymatrix{
\Pi(\mathrm{F})\mathrm{A}_{d}
\ar@{=}[rr]\ar[d]_{\rho_{\Pi(\mathrm{F})}}
&& \Pi(\mathrm{F})\Pi(\Cd)
\ar@{=}[d]\\
\Pi(\mathrm{F})
&& \Pi(\mathrm{F}\Cd),
\ar[ll]_{\Pi(r_{\mathrm{F}})}
}
\end{xy}
\end{gather*}
which, together with the fact that $\Pi$ preserves strict associativity, imply that
$\Pi$ is a lax pseudofunctor.

Positivity of the grading on Soergel bimodules shows that the functor underlying $\Pi$ has a left adjoint
\begin{gather}\label{Thetadef}
\Theta\colon\cAH\to\mathcal{X},
\quad
\mathrm{Hom}_{\mathcal{X}}\big(\Theta(\mathrm{F}),\mathrm{G}\big)\cong
\mathrm{Hom}_{\ccA_{\mathcal{H}}}\big(\mathrm{F},\Pi(\mathrm{G})\big),
\end{gather}
which is unique up to natural $2$-isomorphisms. Up to isomorphism, $\Theta$ is determined by
$\Theta(\mathrm{A}_{w})\cong\mathrm{C}_{w}$, in particular, $\Theta$ is an embedding.

Since $\Pi$ is lax, the doctrinal adjunction \cite{Ke}
(see also \cite[Formula (3.5)]{SS}) implies that $\Theta$ is an oplax pseudofunctor.
Note that, for each $\mathrm{X}\in\mathcal{X}$, the object $\Pi(\mathrm{X})$ is isomorphic
to $\Pi(\mathrm{Y})$, where $\mathrm{Y}$ is the subobject of $\mathrm{X}$ generated in degree $0$.
As $\mathrm{Y}\cong\Theta(\mathrm{F})$, for some $\mathrm{F}\in\cAH$,
the adjunction morphisms guarantee that $\Pi\Theta\cong\mathrm{id}_{\ccA_{\mathcal{H}}}$.

We can further use $\epsilon_{d}\colon
\Cd\to\mathbbm{1}_{\varnothing}$ to extend $\Theta$ to an oplax pseudofunctor from
$\cAH$ to $\cSH$ via the embeddings
$\mathcal{X}\hookrightarrow\cSHo\hookrightarrow\cSH$.

We use the notation $\mathbbm{1}_{\ccA_{\mathcal{H}}}:=\mathrm{A}_{d}$, where $d$ is the Duflo involution in $\mathcal{H}$.
Note that $\mathbbm{1}_{\ccA_{\mathcal{H}}}$ is only a weak identity $1$-morphism.
Then $\Theta(\mathbbm{1}_{\ccA_{\mathcal{H}}})=\Cd$.

\begin{lemma}\label{lem5.1}
For any coalgebra $\mathrm{A}$ in $\cAH$, the oplax $2$-functor
$\Theta$ induces the structure of a coalgebra in
$\cSHo$ on $\Theta(\mathrm{A})$.
Moreover, a (left or right) $\mathrm{A}$-comodule $X$ in $\cAH$ is sent to a
(left or right) $\Theta(\mathrm{A})$-comodule $\Theta(X)$ in $\cSHo$, which can, of course, be viewed as a comodule in $\cSH$.
\end{lemma}

\begin{proof}
Mutatis mutandis as in \cite[Proposition 5.5]{JS}.
\end{proof}

\begin{proposition}\label{prop5.4}
If $\mathrm{A}$ as in Lemma \ref{lem5.1}
is cosimple in $\cAH$, then $\Theta(\mathrm{A})$ is cosimple
in $\underline{\cSH}$.
\end{proposition}

Note that Proposition \ref{prop5.4} rules out the existence of
ungraded subcoalgebras of $\Theta(\mathrm{A})$ as well.

\begin{proof}
By construction, $\Theta(\mathrm{A})$ belongs to $\mathrm{add}(\mathcal{H})$.
Let $\mathrm{B}$ be a cosimple subcoalgebra
of $\Theta(\mathrm{A})$ in $\underline{\cSH}$. By
\cite[Corollary 12]{MMMZ}, the corresponding $2$-representation $\mathbf{inj}_{\underline{\ccS_{\mathcal{H}}}}(\mathrm{B})$ is simple transitive.
By \cite[Example 4.20]{MMMTZ2}, we have
$\mathrm{B}\cong [\mathrm{B},\mathrm{B}]_{\underline{\ccSH}}$, which, by
\cite[Theorem 4.19]{MMMTZ2}, implies that $\mathrm{B}$
is a direct summand of $\Theta(\mathrm{A})$ in $\cSH$.
Hence, $\mathrm{B}$ is of the form $\Theta(\widetilde{\mathrm{A}})$, for some
$1$-morphism $\widetilde{\mathrm{A}}$ in $\cAH$. It is easy to check
that $\widetilde{\mathrm{A}}$ is a subcoalgebra of $\mathrm{A}$ in $\cAH$
and is thus isomorphic to $\mathrm{A}$. The claim follows.
\end{proof}

\begin{example}\label{ex:Cd-alg-coalg}
Being an identity $1$-morphism, $\mathrm{A}_{d}$ is a cosimple coalgebra in $\cAH$. By
Lemma \ref{lem5.1} and Proposition \ref{prop5.4}, this implies that
$\Cd$ has the structure of a cosimple coalgebra in $\underline{\cSH}$ whose structure
$2$-morphisms are homogeneous of degree zero. By duality, this implies that $\tilde{\mathrm{C}}_{d}$
has the structure of a simple algebra in $\overline{\cSH}$, again with structure $2$-morphisms homogeneous of degree zero.

In Subsection \ref{s7.5}, we will
additionally see that $\Cd$ has the structure of a Frobenius algebra in $\cSH$.
\end{example}

We denote by $\square$ the cotensor product, see \cite[Section 0]{Ta}.

\begin{lemma}\label{lem5.25-1}
Let $\mathrm{A}$ be a coalgebra in $\cAH$, $\mathrm{X}$ a right
$\mathrm{A}$-comodule and $\mathrm{Y}$ a left $\mathrm{A}$-comodule. Then
there is a $2$-isomorphism in $\cSHo$
\begin{gather*}
\Theta(\mathrm{X}\square_{\mathrm{A}}\mathrm{Y})\cong\Theta(\mathrm{X})\square_{\Theta(\mathrm{A})}\Theta(\mathrm{Y})
\end{gather*}
that is natural in $\mathrm{X}$ and $\mathrm{Y}$.
\end{lemma}

\begin{proof}
Consider the commutative diagram in $\cSHo$
\begin{gather*}
\xymatrix{
\Theta(\mathrm{X}\square_{\mathrm{A}}\mathrm{Y})\ar@{^{(}->}[rr]\ar@{^{(}-->}[d]
&&\Theta(\mathrm{X}\mathrm{Y})\ar[rr]\ar@{^{(}->}[d]
&&\Theta(\mathrm{X}\,\mathrm{A}\,\mathrm{Y})\ar@{^{(}->}[d]\\
\Theta(\mathrm{X})\square_{\Theta(\mathrm{A})}\Theta(\mathrm{Y})\ar@{^{(}->}[rr]&&
\Theta(\mathrm{X})\Theta(\mathrm{Y})\ar[rr]^{\varphi}
&&\Theta(\mathrm{X})\Theta(\mathrm{A})\Theta(\mathrm{Y}),
}
\end{gather*}
where the top row is given by applying $\Theta$ to the definition of $\mathrm{X}\square_{\mathrm{A}}\mathrm{Y}$,
the bottom row is the definition of $\Theta(\mathrm{X})\square_{\Theta(\mathrm{A})}\Theta(\mathrm{Y})$
and the two solid inclusions are given by the oplax structure of $\Theta$.
The dashed inclusion is induced by commutativity of the solid square and the universality
property of the cotensor product. As all solid arrows are natural in $\mathrm{X}$ and $\mathrm{Y}$, the same holds for the dashed
inclusion.

Note that the vertical arrows in the solid square restrict to $2$-isomorphisms in degree
$0$. Consequently, the dashed arrow is also $2$-isomorphism in degree $0$. To prove the lemma,
we need to show that $\varphi$ is injective when restricted to summands of
$\Theta(\mathrm{X})\Theta(\mathrm{Y})$ generated in positive degrees (as $1$-morphisms in $\cSHo$).
We will verify this by passing to the cell $2$-representation $\mathbf{C}_{\mathcal{H}}$.

Let $\mathsf{B}$ be the underlying algebra of $\mathbf{C}_{\mathcal{H}}$. Then
$\mathsf{B}$ is naturally a positively graded algebra. The bimodules representing
$\Theta(\mathrm{A})$, $\Theta(\mathrm{X})$ and $\Theta(\mathrm{Y})$ are all projective by Theorem \ref{thm:projapex}
and generated in degree $0$, by construction. Let $M$ be an indecomposable summand of
$\mathbf{C}_{\mathcal{H}}\big(\Theta(\mathrm{X})\Theta(\mathrm{Y})\big)$ not generated in degree $0$. Being a summand in the tensor product over
$\mathsf{B}$ of two projective $\mathsf{B}$-$\mathsf{B}$-bimodules generated in
degree $0$, we can view $M$ as a summand of
$(\mathsf{B}\otimes_{\mathbb{C}}\mathsf{B})\otimes_{\mathsf{B}}
(\mathsf{B}\otimes_{\mathbb{C}}\mathsf{B})\cong
\mathsf{B}\otimes_{\mathbb{C}}\mathsf{B}\otimes_{\mathbb{C}}\mathsf{B}$.
Writing any non-zero element $u$ of $M$ in a fixed homogeneous basis of
$\mathsf{B}\otimes_{\mathbb{C}}\mathsf{B}\otimes_{\mathbb{C}}\mathsf{B}$, the element $u$
will have a non-zero coefficient at some $b\otimes c\otimes d$, where $b$, $c$ and $d$
are elements from a fixed homogeneous basis of $\mathsf{B}$ and $c$ has positive degree.
Under $\varphi$, such $b\otimes c\otimes d$ is sent to
\begin{gather*}
q:=b\otimes e\otimes c\otimes d - b\otimes c\otimes f\otimes d,
\end{gather*}
where $e$ and $f$ are non-zero elements of degree $0$ in $\mathsf{B}$.
The element $q$ is clearly non-zero in
$\mathbf{C}_{\mathcal{H}}\big(\Theta(\mathrm{X})\Theta(\mathrm{A})\Theta(\mathrm{Y})\big)$.
All other basis elements appearing with non-zero coefficient in $u$ will be mapped to
elements with zero coefficients at $b\otimes e\otimes c\otimes d$ and
$b\otimes c\otimes f\otimes d$. The claim follows.
\end{proof}

\begin{proposition}\label{prop5.25-2}
Let $\mathrm{A}$ and $\mathrm{B}$ be Morita--Takeuchi equivalent coalgebras in $\cAH$. Then
$\Theta(\mathrm{A})$ and $\Theta(\mathrm{B})$ are
Morita--Takeuchi equivalent in $\cSHo$.
\end{proposition}

\begin{proof}
By \cite[Theorem 5.1]{MMMT}, there exist a
$\mathrm{B}$-$\mathrm{A}$-bicomodule $\mathrm{X}$ and an
$\mathrm{A}$-$\mathrm{B}$-bicomodule $\mathrm{Y}$ such that
\begin{gather*}
{}_{\mathrm{B}}\mathrm{X}\square_{\mathrm{A}}\mathrm{Y}_{\mathrm{B}}\cong
{}_{\mathrm{B}}\mathrm{B}_{\mathrm{B}}\quad\text{ and }\quad
{}_{\mathrm{A}}\mathrm{Y}\square_{\mathrm{B}}\mathrm{X}_{\mathrm{A}}\cong{}_{\mathrm{A}}\mathrm{A}_{\mathrm{A}}.
\end{gather*}
Applying Lemmas \ref{lem5.1} and \ref{lem5.25-1}, we obtain
\begin{gather*}
{}_{\Theta(\mathrm{B})}\Theta(\mathrm{X})\square_{\Theta(\mathrm{A})}
\Theta(\mathrm{Y})_{\Theta(\mathrm{B})}\cong
{}_{\Theta(\mathrm{B})}\Theta(\mathrm{B})_{\Theta(\mathrm{B})},\\
{}_{\Theta(\mathrm{A})}\Theta(\mathrm{Y})\square_{\Theta(\mathrm{B})}
\Theta(\mathrm{X})_{\Theta(\mathrm{A})}\cong {}_{\Theta(\mathrm{A})}\Theta(\mathrm{A})_{\Theta(\mathrm{A})}.
\end{gather*}
The claim of the proposition now follows from \cite[Theorem 5.1]{MMMT}.
\end{proof}

\begin{proposition}\label{prop5.41}
If $\mathrm{A}$ as in Lemma \ref{lem5.1} is cosimple and
$\mathrm{X}\in\mathrm{inj}_{\ccA_{\mathcal{H}}}(\mathrm{A})$, then
$\Theta(\mathrm{X})$ is in $\mathrm{inj}_{\underline{\ccSH}}\big(\Theta(\mathrm{A})\big)$.
\end{proposition}

Note that Lemma \ref{lem5.1} shows that the right coaction
of $\Theta(\mathrm{A})$ on $\Theta(\mathrm{X})$ is, in fact, homogeneous of degree zero.

\begin{proof}
By additivity, it suffices to prove that $\Theta(\mathrm{A}_{w}\mathrm{A})$ is
in $\mathrm{inj}_{\underline{\ccSH}}\big(\Theta(\mathrm{A})\big)$, for any
$w\in\mathcal{H}$.

Set $\mathrm{B}=[\mathrm{A}_{w}\mathrm{A},\mathrm{A}_{w}\mathrm{A}]_{\ccA_{\mathcal{H}}}$. By \cite[Corollary 4.17]{MMMTZ2},
the Morita--Takeuchi equivalence between
$\mathrm{A}$ and $\mathrm{B}$ is given by the $\mathrm{A}\text{-}\mathrm{B}$-bicomodule
$\mathrm{A}\mathrm{A}_{w}^{\star}$ and the $\mathrm{B}\text{-}\mathrm{A}$-bicomodule
$\mathrm{A}_{w}\mathrm{A}$.
By Lemma \ref{lem5.25-1} and Proposition \ref{prop5.25-2},
the $\Theta(\mathrm{A})\text{-}\Theta(\mathrm{B})$-bicomodule
$\Theta(\mathrm{A}\mathrm{A}_{w}^{\star})$ and the $\Theta(\mathrm{B})\text{-}\Theta(\mathrm{A})$-bicomodule
$\Theta(\mathrm{A}_{w}\mathrm{A})$ provide a Morita--Takeuchi equivalence between
$\Theta(\mathrm{A})$ and $\Theta(\mathrm{B})$ in $\cSHo$. In particular,
$\Theta(\mathrm{A}_{w}\mathrm{A})$ is in
$\mathrm{inj}_{\underline{\ccSH}}\big(\Theta(\mathrm{A})\big)$.
\end{proof}

\begin{corollary}\label{cor5.42}
If $\mathrm{A}$ as in Lemma \ref{lem5.1} is cosimple,
then the ranks of $\mathrm{inj}_{\ccA_{\mathcal{H}}}(\mathrm{A})$ and
$\mathrm{inj}_{\underline{\ccS_{\mathcal{H}}}}\big(\Theta(\mathrm{A})\big)$ are equal.
\end{corollary}

\begin{proof}
The fact that $\Pi\Theta\cong\mathrm{id}_{\ccA_{\mathcal{H}}}$ and
Proposition \ref{prop5.41} imply that application of $\Theta$ induces an
injection from the set of isomorphism classes of indecomposable objects in
$\mathrm{inj}_{\ccA_{\mathcal{H}}}(\mathrm{A})$ to the
set of isomorphism classes of indecomposable objects in
$\mathrm{inj}_{\underline{\ccS_{\mathcal{H}}}}\big(\Theta(\mathrm{A})\big)$.
Let us now argue that this is also surjective.

Let $M\in\mathrm{inj}_{\underline{\ccS_{\mathcal{H}}}}\big(\Theta(\mathrm{A})\big)$ be
indecomposable. Up to an overall grading shift, we may assume that there is a decomposition
in $\cSHo$ of the form
\begin{gather*}
M\cong \bigoplus_{w\in\mathcal{H}}
\mathrm{C}_{w}^{\oplus p_{w}},
\end{gather*}
where the $p_{w}$ belong to $\mathbb{N}_{0}[\mathsf{v}]$ and are such that
$p_{w}(0)$ is non-zero for at least one $w\in\mathcal{H}$. Here $p_{w}(0)$ means the
evaluation of $p_{w}(\mathsf{v})$ at $\mathsf{v}=0$. Let
\begin{gather*}
M^{0}:=\bigoplus_{w\in\mathcal{H}}\mathrm{C}_{w}^{\oplus p_{w}(0)}.
\end{gather*}
By positivity of the grading on $\cSH$ and
Lemma \ref{lem5.2}, $M^{0}$ is a $\Theta(\mathrm{A})$-subcomodule of $M$ in $\cSHo$.

Note that ${\Pi}\big(\Theta(\mathrm{A})\big)\cong\mathrm{A}$ by construction.
Similarly, ${\Pi}\big(M^{0}\Theta(\mathrm{A})\big)\cong
{\Pi}(M^{0}){\Pi}\big(\Theta(\mathrm{A})\big)$.
Hence, there is an induced $\mathrm{A}$-comodule structure on ${\Pi}(M^{0})$. By Corollary \ref{cor:ss}, ${\Pi}(M^{0})$ is injective as an $\mathrm{A}$-comodule since $A$ is cosimple and $\cA_{\mathcal{H}}$ is fusion.
Now, $\Theta\big({\Pi}(M^{0})\big)\cong M^{0}$ is an injective
$\Theta(\mathrm{A})$-comodule in $\cSHo$
by Proposition \ref{prop5.41}. Consequently, $M^{0}$ is isomorphic to $M$
in $\cSHo$.
\end{proof}


\subsection{Going down}\label{s5.3}

In this subsection
we describe, see Lemma \ref{lem5.3}, and investigate
what we loosely call ``going down''. As the name suggests,
this is, in some sense, the opposite of going up.

Let $\mathbf{M}$ be a graded simple transitive
$2$-representation of $\cSH$
with apex $\mathcal{H}$. Let $\mathrm{C}$ be a
graded coalgebra in ${\cSH}$ such that $\mathbf{M}$ is equivalent to
$\mathbf{inj}_{\underline{\ccS_{\mathcal{H}}}}(\mathrm{C})$.
Let $X_{1},\dots,X_n$ be a complete and irredundant list of representatives of
isomorphism classes (up to grading shift)
of indecomposable objects in $\mathrm{inj}_{\underline{\ccS_{\mathcal{H}}}}(\mathrm{C})$,
normalized such that as $1$-morphisms in $\cSH$ they are concentrated in non-negative degrees
with non-zero degree zero part. Set $X:=X_{1}\oplus X_{2}\oplus\dots\oplus X_n$.

\begin{lemma}\label{lem5.3}
The quotient
\begin{gather*}
\mathrm{add}\big(\{X^{\oplus\mathsf{v}^{k}}\mid k\geq 0\}\big)^{(0)}/
\big(\mathrm{add}\big(\{X^{\oplus\mathsf{v}^{k}}\mid k>0\}\big)^{(0)}\big)
\end{gather*}
carries an induced action of $\cAH$. Here
$\big(\mathrm{add}\big(\{X^{\oplus\mathsf{v}^{k}}\mid k>0\}\big)^{(0)}\big)$
is the ideal generated by $\mathrm{add}\big(\{X^{\oplus\mathsf{v}^{k}}\mid k>0\}\big)^{(0)}$.
\end{lemma}

\begin{proof}
Lemma \ref{lem5.2} implies that, if $X_m^{\oplus\mathsf{v}^{k}}$ is isomorphic to a direct summand
of $\mathrm{C}_{w}\,X_j$ in $\cSHo$,
where $w\in\mathcal{H}$, then $k\geq 0$. Therefore
$\mathrm{add}\big(\{X^{\oplus\mathsf{v}^{k}}\mid k\geq 0\}\big)^{(0)}$ is stable under the action of
$\mathrm{add}\big(\{\mathrm{C}_{w}^{\oplus\mathsf{v}^{l}}\mid w\in W,l\geq 0\}\big)^{(0)}$. Moreover,
$\big(\mathrm{add}\big(\{\mathrm{C}_{w}^{\oplus\mathsf{v}^{l}}\mid w\in W,l>0\}\big)^{(0)}
\big)$ maps
$\mathrm{add}\big(\{X^{\oplus\mathsf{v}^{k}}\mid k\geq 0\}\big)^{(0)}$ to
$\big(\mathrm{add}\big(\{X^{\oplus\mathsf{v}^{k}}\mid k>0\}\big)^{(0)}\big)$.
The claim follows.
\end{proof}

We state and prove the following lemma for $\cSH$ or $\cAH$, although we only need it
in the latter case, since it is a useful result in its own right.

\begin{lemma}\label{lem5.3-0}
Let $\cC=\cSH$ or $\cC=\cAH$. Given any graded simple transitive $2$-representation $\mathbf{M}$ of $\cC$
with apex $\mathcal{H}$, there exists a coalgebra $\mathrm{C}$ in $\cC$ such that
$\mathbf{inj}_{\underline{\ccC}}(\mathrm{C})$ is equivalent to
$\mathbf{M}$ and $\mathrm{C}$ is the image of a multiplicity free direct
sum of representatives of isomorphism classes of indecomposable objects in
$\mathrm{inj}_{\underline{\ccC}}(\mathrm{C})$
under the forgetful functor to $\cC$.
\end{lemma}

\begin{proof}
Let $\mathrm{C}^{\prime}$ be any coalgebra in $\cC$ and $Y$ a multiplicity
free direct sum of representatives of isomorphism classes of indecomposable objects in
$\mathrm{inj}_{\underline{\ccC}}(\mathrm{C}^{\prime})$. Then the coalgebra
$\mathrm{C}:=\mathrm{C}^{Y}$ in $\cC$ is Morita--Takeuchi equivalent
to $\mathrm{C}^{\prime}$, and the equivalence between $\mathrm{inj}_{\underline{\ccC}}(\mathrm{C}^{\prime})$ and
$\mathrm{inj}_{\underline{\ccC}}(\mathrm{C})$ identifies $Y$ with $\mathrm{C}$, so $\mathrm{C}$ is the image of a
multiplicity free direct sum of representatives of isomorphism classes of indecomposable objects in
$\mathrm{inj}_{\underline{\ccC}}(\mathrm{C})$ under the forgetful functor to $\cC$, as required.
\end{proof}

\begin{proposition}\label{prop5.3-1}
There is an injection $\hat{\Theta}$ from the set of equivalence classes of simple transitive
$2$-representations of $\cAH$ to the set of equivalence classes of graded
simple transitive $2$-representations of $\cSH$ with apex $\mathcal{H}$.
\end{proposition}

\begin{proof}
Let $\mathbf{N}$ be a simple transitive $2$-representation of $\cAH$.
Let $\mathrm{A}$ be a coalgebra in $\cAH$ such that
$\mathbf{N}$ is equivalent to $\mathbf{inj}_{\ccA_{\mathcal{H}}}(\mathrm{A})$.
Then $\mathbf{inj}_{\underline{\ccS_{\mathcal{H}}}}\big(\Theta(\mathrm{A})\big)$ is a
graded simple transitive $2$-representation of $\cSH$ (with apex $\mathcal{H}$) due to
Proposition \ref{prop5.4} and \cite[Corollary 12]{MMMZ}.
By Proposition \ref{prop5.25-2}, this yields a well-defined map $\hat{\Theta}$.

Now assume $\mathrm{A}$ is chosen such that $\mathrm{A}$ is the image of a
multiplicity free direct sum
of representatives of isomorphism classes of indecomposable objects in
$\mathrm{inj}_{\ccAH}(\mathrm{A})$ under the forgetful functor to $\cAH$,
see Lemma \ref{lem5.3-0}.
Set $\mathrm{C}:=\Theta(\mathrm{A})$.
Then, by Proposition \ref{prop5.41} and Corollary \ref{cor5.42},
the object $X$ defined in the going down procedure is isomorphic to
$\mathrm{C}$. The $2$-representation of $\cAH$ obtained by going down
is now, clearly, equivalent to  $\mathrm{inj}_{\ccAH}(\mathrm{A})$.
Therefore, the map defined in the previous paragraph is injective, as claimed.
\end{proof}

As we will see later, Theorem \ref{theorem:main} actually implies that $\hat{\Theta}$ is a bijection.


\section{The role of the Duflo involution}\label{s7}


Throughout this section we set $\mathbf{a}:=\mathbf{a}(\mathcal{H})$ for
our fixed diagonal $\mathcal{H}$-cell, and let $d$ be the Duflo involution in $\mathcal{H}$.


\subsection{Cell $2$-representations and Duflo involutions}\label{s7.1}


Let $w\in\mathcal{H}$ and $\tilde{L}_{w}$ be the corresponding simple object in
$\underline{\mathbf{C}_{\mathcal{H}}}(\varnothing)$, concentrated in degree zero.
By \cite[Section 7]{MM3} (see also \cite[Subsection 4.5]{MM1}),
$\tilde{\mathrm{C}}_{w}\,\tilde{L}_{d}$ is an indecomposable
injective object in $\underline{\mathbf{C}_{\mathcal{H}}}(\varnothing)$ with simple socle
$\tilde{L}_{w}$ concentrated in degree $0$.

Dually, let $L_{w}$ be the simple object in
$\overline{\mathbf{C}_{\mathcal{H}}}(\varnothing)$, corresponding to $w$, concentrated in degree zero. Then
$\mathrm{C}_{w}\,L_{d}$ is an indecomposable projective object
in $\overline{\mathbf{C}_{\mathcal{H}}}(\varnothing)$ with simple
head $L_{w}$ concentrated in degree $0$.

\begin{lemma}\label{lem7.1-0}
For any $x,w\in\mathcal{H}$,
\begin{enumerate}[label=(\roman*)]

\item\label{lem7.1-0i} the injective objects $\tilde{\mathrm{C}}_{x}\,\tilde{L}_{w}$ and $\mathrm{C}_{x}\,\tilde{L}_{w}$ in
$\underline{\mathbf{C}_{\mathcal{H}}}(\varnothing)$ are concentrated
between the degrees $-2\mathbf{a}$ and $0$ and between
the degrees $0$ and $2\mathbf{a}$, respectively;

\item\label{lem7.1-0ii} the projective objects $\tilde{\mathrm{C}}_{x}\,L_{w}$
and $\mathrm{C}_{x}\,L_{w}$ in $\overline{\mathbf{C}_{\mathcal{H}}}(\varnothing)$
are concentrated between the degrees $-2\mathbf{a}$ and $0$ and between
the degrees $0$ and $2\mathbf{a}$, respectively.

\end{enumerate}
\end{lemma}

\begin{proof}
Let us prove the first statement in \ref{lem7.1-0i}. As the $2$-category of Soergel bimodules is positively graded, the fact
that $\tilde{\mathrm{C}}_{x}\tilde{\mathrm{C}}_{w}\,\tilde{L}_{d}$ is concentrated in non-positive degrees follows from Lemma \ref{lem5.2new}.
This in turn implies that $\tilde{\mathrm{C}}_{x}\,\tilde{L}_{w}$ is concentrated in non-positive degrees as well.

By adjunction, we have
\begin{gather*}
\mathrm{hom}_{\underline{\mathbf{C}_{\mathcal{H}}}}(\tilde{\mathrm{C}}_{x}\,\tilde{L}_{w},\tilde{L}_{y}^{\oplus\mathsf{v}^k})
\cong\mathrm{hom}_{\underline{\mathbf{C}_{\mathcal{H}}}}(\tilde{L}_{w}, \tilde{\mathrm{C}}_{x^{\mone}}\,\tilde{L}_{y}^{\oplus\mathsf{v}^{(2\mathbf{a}+k)}}).
\end{gather*}
As the right-hand side of the above isomorphism is zero for $2\mathbf{a}+k<0$, we deduce that
$\mathrm{hom}_{\underline{\mathbf{C}_{\mathcal{H}}}}(\tilde{\mathrm{C}}_{x}\,\tilde{L}_{w},\tilde{L}_{y}^{\oplus \mathsf{v}^k})$ is also zero
if $k<-2\mathbf{a}$.

The second statement in \ref{lem7.1-0i} follows from the first one and the fact that
$\tilde{\mathrm{C}}_{x}=\mathrm{C}_{x}^{\oplus\mathsf{v}^{\mone[2\mathbf{a}]}}$.

The dual statements in \ref{lem7.1-0ii} are proved in exactly the same way,
using Lemma \ref{lem5.2} instead of Lemma \ref{lem5.2new}.
\end{proof}

Let $\mathbf{P}:=\mathbf{P}_{\varnothing}$ be the principal
$2$-representation of $\cSH$ and $\underline{\mathbf{P}}$ its injective abelianization, see Example \ref{example:yoneda}. Denote by $I_{w}$ and $I_e$ the corresponding injective
object in $\underline{\mathbf{P}}(\varnothing)$ with respect to $\tilde{\mathrm{C}}_{w}$
and $\tilde{\mathrm{C}}_e=\mathbbm{1}_{\varnothing}$, respectively, see \cite[Subsection 3.1]{MMMT} for details.
Note that $I_{w}$ has simple socle $\tilde{L}_{w}$ concentrated in degree 0 and $I_{e}$ has
simple socle $\tilde{L}_e$ concentrated in degree $\mathbf{a}$.

\begin{lemma}\label{lem7.1-3} For any $x\in\mathcal{H}$, the following hold.
\begin{enumerate}[label=(\roman*)]

\item\label{lem7.1-3i}
The injective object $\tilde{\mathrm{C}}_{x}\,\tilde{L}_{d}$ in $\underline{\mathbf{C}_{\mathcal{H}}}(\varnothing)$
has simple head $\tilde{L}_{x}$ concentrated in degree $-2\mathbf{a}$.

\item\label{lem7.1-3ii}
The projective object $\mathrm{C}_{x}\,L_{d}$ in $\overline{\mathbf{C}_{\mathcal{H}}}(\varnothing)$
has simple socle $L_{x}$ concentrated in degree $2\mathbf{a}$.

\end{enumerate}
\end{lemma}

\begin{proof}
We only prove statement \ref{lem7.1-3i} since the dual statement \ref{lem7.1-3ii} follows by similar arguments.
As $\cSH$ is fiat, the underlying algebra of $\underline{\mathbf{C}_{\mathcal{H}}}$ is self-injective, by Corollary \ref{cor:projapex}, which implies that the indecomposable injective object $\tilde{\mathrm{C}}_{x}\,\tilde{L}_{d}$
is also projective and thus, has a simple head.
Therefore, it suffices to prove that $\mathrm{hom}_{\underline{\mathbf{C}_{\mathcal{H}}}}
(\tilde{\mathrm{C}}_{x}\,\tilde{L}_{d}^{\oplus\mathsf{v}^{2\mathbf{a}}},\tilde{L}_{x})$ is non-zero.
By adjunction, this is equivalent to proving that
\begin{gather}\label{eq7.1-3}
\mathrm{hom}_{\underline{\mathbf{C}_{\mathcal{H}}}}(\tilde{L}_{d},\tilde{\mathrm{C}}_{x^{\mone}}\,\tilde{L}_{x})\not\cong 0,
\end{gather}
which holds if and only if $\tilde{\mathrm{C}}_{d}\,\tilde{L}_{d}$ appears as a direct
summand of $\tilde{\mathrm{C}}_{x^{\mone}}\,\tilde{L}_{x}$ in $\underline{\mathbf{C}_{\mathcal{H}}}(\varnothing)$.

Note that the latter holds if and only if the same is true in $\underline{\mathbf{P}}(\varnothing)$.
(Recall the two equivalent constructions of cell $2$-representations,
see \cite[Subsection 4.5]{MM1}.)
In the principal $2$-representation, we can use the following fact.
Since $\tilde{\mathrm{C}}_{x}\,I_e\cong I_{x}$, we have
\begin{gather*}
\mathrm{hom}_{\underline{\mathbf{P}}}(\tilde{\mathrm{C}}_{x}\,\tilde{L}_{d},I_e^{\oplus\mathsf{v}^{k}})\cong
\mathrm{hom}_{\underline{\mathbf{P}}}(\tilde{L}_{d},I_{x^{\mone}}^{\oplus\mathsf{v}^{(2\mathbf{a}+k)}}),
\end{gather*}
whose right-hand side is zero unless $x=d$ and $k=-2\mathbf{a}$. In other words, the
only $x\in\mathcal{H}$
such that $\tilde{\mathrm{C}}_{x}\,\tilde{L}_{d}$ has  a composition factor isomorphic to $\tilde{L}_e$,
up to a shift, is $x=d$. Therefore, to prove \eqref{eq7.1-3}, it is enough to show that
$\mathrm{hom}_{\underline{\mathbf{P}}}(\tilde{\mathrm{C}}_{x^{\mone}}\,\tilde{L}_{x},I_{e}^{\oplus\mathsf{v}^{\mone[2\mathbf{a}]}})$ is non-zero.
By adjunction, we have
\begin{gather*}
\mathrm{hom}_{\underline{\mathbf{P}}}(\tilde{\mathrm{C}}_{x^{\mone}}\,\tilde{L}_{x},I_{e}^{\oplus\mathsf{v}^{\mone[2\mathbf{a}]}})\cong
\mathrm{hom}_{\underline{\mathbf{P}}}(\tilde{L}_{x},I_{x}),
\end{gather*}
where the right-hand side is non-zero.
Hence, $\tilde{L}_{x}^{\oplus\mathsf{v}^{\mone[2\mathbf{a}]}}$ appears in the
head of $\tilde{\mathrm{C}}_{x}\,\tilde{L}_{d}$ in $\underline{\mathbf{C}_{\mathcal{H}}}(\varnothing)$.
\end{proof}

Assume that $\mathsf{B}=\mathsf{B}^{\,\underline{\mathbf{C}_{\mathcal{H}}}}$ is
the underlying basic algebra of $\underline{\mathbf{C}_{\mathcal{H}}}$.
Then the indecomposable objects in $\mathbf{C}_{\mathcal{H}}(\varnothing)$ are
identified with indecomposable injective $\mathsf{B}$-modules.

\begin{proposition}\label{prop:Frobenius}
The algebra $\mathsf{B}$ is a finite dimensional
positively graded weakly symmetric Frobenius algebra of graded length $2\mathbf{a}$.
\end{proposition}

\begin{proof}
By construction, the algebra $\mathsf{B}$ is non-negatively graded and its degree $0$ part, which is
isomorphic to $\mathbb{C}^{\#\mathcal{H}}$, is semisimple, so $\mathsf{B}$ is positively graded.

Since $\cSH$ is fiat, the algebra $\mathsf{B}$ is Frobenius, see Corollary \ref{cor:projapex}.

The fact that $\mathsf{B}$ is weakly symmetric follows from Lemma \ref{lem7.1-3}\ref{lem7.1-3i}.
Together with Lemma \ref{lem7.1-0}\ref{lem7.1-0ii}, this implies $\mathsf{B}$ is of graded length $2\mathbf{a}$.
\end{proof}

\begin{remark}\label{remark:symmetric}
In some cases we know that $\mathsf{B}$ is symmetric:
\begin{enumerate}[label=(\roman*)]
\item Let $\mathcal{L}(\mathcal{H})$ be the left cell for $\mathcal{H}$.
From \cite[Theorem 4.6]{MS} we know that $\mathsf{B}$ is symmetric if $W$
is a finite Weyl group and
\begin{gather*}
\text{there exists }\mathtt{I}\subset\mathtt{S}\text{ such that }
w_{0}w_{0}^{\mathtt{I}}\in\mathcal{L}(\mathcal{H}).
\end{gather*}

\item For any finite Coxeter group $W$, $\mathsf{B}$ is symmetric if $\mathcal{H}$ is a diagonal
$\mathcal{H}$-cell in the subregular two-sided cell of $\cS$, see \cite[Corollary 5, Proposition 14
and the comment below it, and Remarks 29, 32 and 38]{KMMZ} and \cite[Theorem I]{MT}.

\end{enumerate}
However, we do not know if this holds in general.
\end{remark}

Since $\cSH$ is fiat, the action of $\mathrm{C}_{w}$ on the category of $\mathsf{B}$-modules is
exact. Further, by Theorem \ref{thm:projapex}, we know that the action of $\mathrm{C}_{w}$ via $\mathbf{C}_{\mathcal{H}}$
is given by tensoring with a projective-injective bimodule. It follows from Lemma
\ref{lem7.1-0}\ref{lem7.1-0i}
that the bimodule representing $\mathrm{C}_{w}$ is isomorphic to a direct sum of
bimodules of the form $\mathsf{B}e_{u}\otimes_{\mathbb{C}}e_{v}\mathsf{B}$, possibly with multiplicities
but without grading shifts, where $e_{u},e_{v}$ are some primitive idempotents of $\mathsf{B}$. By Proposition \ref{prop:Frobenius}
and the fact that $\mathbf{C}_{\mathcal{H}}$ is a faithful $2$-functor which is degree-preserving on $2$-morphisms,
this implies that the $1$-morphism $\mathrm{C}_{w}$ in $\cSH$ has graded length at most $4\mathbf{a}$.

Recall from Example \ref{ex:Cd-alg-coalg} that $\Cd$ is a cosimple coalgebra $1$-morphism in $\underline{\cSH}$.
By \cite[Corollary 12]{MMMZ},
$\mathbf{M}:=\mathbf{inj}_{\underline{\ccS_{\mathcal{H}}}}(\Cd)$ is a graded simple transitive $2$-representation of $\cSH$
with apex $\mathcal{H}$. We denote by $\mathsf{B}^{\mathbf{M}}$ the basic algebra underlying $\mathbf{M}$.

\begin{proposition}\label{prop:frobenius21}
The algebra $\mathsf{B}^{\mathbf{M}}$ is a positively graded Frobenius algebra of graded length $2\mathbf{a}$.
\end{proposition}

\begin{proof}
The algebra $\mathsf{B}^{\mathbf{M}}$ is graded, by definition, and Frobenius by Corollary \ref{cor:projapex}.

By Lemma \ref{lem5.3-0}, we can choose a coalgebra $\mathrm{A}$ in $\cAH$ such that
$\mathbf{inj}_{\ccA_{\mathcal{H}}}(\mathrm{A})$ is equivalent to
$\mathbf{inj}_{\ccA_{\mathcal{H}}}(\mathrm{A}_{d})$ and $\mathrm{A}$ is the image of a multiplicity free direct sum
of representatives of isomorphism classes of indecomposable objects in $\mathrm{inj}_{\ccA_{\mathcal{H}}}(\mathrm{A})$ under
the forgetful functor to $\cAH$. This implies
that $A\cong\oplus_{w\in\mathcal{H}}\mathrm{A}_{w}$ as $1$-morphisms in $\cAH$, because $A_{d}$ is the identity $1$-morphism of $\cAH$.

Define $\mathrm{C}:=\Theta(\mathrm{A})$ in $\cSH^{(0)}$. Since $\Theta(\mathrm{A}_{w})\cong\mathrm{C}_{w}$, for all $w\in\mathcal{H}$, there is an isomorphism
$\mathrm{C}\cong\oplus_{w\in\mathcal{H}}\mathrm{C}_{w}$ in $\cSH^{(0)}$ and, by Proposition \ref{prop5.25-2}, there is an equivalence $\mathbf{inj}_{\underline{\ccS_{\mathcal{H}}}}(\mathrm{C})\simeq\mathbf{M}$. From the definition of $\mathrm{C}$, for all $k\in\mathbb{Z}$, we have
\begin{gather}\label{eq:marco11}
\mathrm{hom}_{{\ccS_{\mathcal{H}}}}(\mathrm{C},\mathbbm{1}_{\ccS_{\mathcal{H}}}^{\oplus\mathsf{v}^{k}})
\cong
\mathrm{hom}_{\mathbf{M}}(\mathrm{C},\mathbbm{1}_{\ccS_{\mathcal{H}}}^{\oplus\mathsf{v}^{k}}\mathrm{C}).
\end{gather}
If $k>0$, then Soergel's hom formula \cite[Theorem 3.6]{EW} (here we need the full formula, not just the restricted version which we recalled in \eqref{eq:soergel}) implies that the left-hand side of \eqref{eq:marco11} is zero. If $k=0$, then Soergel's hom formula implies that the left-hand side of \eqref{eq:marco11} has dimension one, because
\begin{gather*}
\dim\big(\mathrm{hom}_{{\ccS_{\mathcal{H}}}}
(\mathrm{C}_{w},\mathbbm{1}_{\ccS_{\mathcal{H}}})\big)=
\begin{cases}
1 &\text{if }w=d,\\
0 &\text{else}.
\end{cases}
\end{gather*}
This implies that $\mathsf{B}^{\mathbf{M}}$ is positively graded.

Since $\mathrm{C}\cong\oplus_{w\in\mathcal{H}}\mathrm{C}_{w}$ in $\cSHo$,
the left-hand side of \eqref{eq:marco11} is zero
if $k<-2\mathbf{a}$ and hence, the graded length of $\mathsf{B}^{\mathbf{M}}$ is at
most $2\mathbf{a}$. As the algebra underlying the cell $2$-representation has graded length $2\mathbf{a}$, see Proposition \ref{prop:Frobenius},
we know that there exists a $\mathrm{C}_{w}$ such that
$\mathrm{hom}_{\ccS_{\mathcal{H}}}
(\mathrm{C}_{w},\mathbbm{1}_{\ccS_{\mathcal{H}}}^{\oplus\mathsf{v}^{\mone[2\mathbf{a}]}})\neq 0$, which
implies that the left-hand side of \eqref{eq:marco11} is non-zero for $k=-2\mathbf{a}$. Thus, the graded length of
$\mathsf{B}^{\mathbf{M}}$ is exactly $2\mathbf{a}$. This completes the proof.
\end{proof}

\begin{lemma}\label{lem:gr}
For all $w\in\mathcal{H}$, the $1$-morphism $\mathrm{C}_{w}$ in $\cSH$ is of graded length $4\mathbf{a}$.
\end{lemma}

\begin{proof}
Recall from \cite[Example 4.20]{MMMTZ2} that $\mathrm{C}_{d}\cong[\Cd,\Cd]$ in $\cSHo$, where the internal cohom-construction is with respect to $\mathbf{M}=\mathbf{inj}_{\underline{\ccS_{\mathcal{H}}}}(\Cd)$.
We first claim that
\begin{gather}\label{eq:hom-1}
\mathrm{hom}_{{\ccS_{\mathcal{H}}}}(\Cd,\mathrm{C}^{\oplus\mathsf{v}^{\mone[4\mathbf{a}]}})
\cong
\mathrm{hom}_{\mathbf{M}}(\Cd,\mathrm{C}\Cd^{\oplus\mathsf{v}^{\mone[4\mathbf{a}]}})
\end{gather}
is non-zero, where $\mathrm{C}$ is as in the proof of Proposition \ref{prop:frobenius21}.
By Lemma \ref{lem5.2}, we have
\begin{gather*}
\mathrm{C}_{w}\Cd\cong\bigoplus_{z\in\mathcal{H}}\mathrm{C}_{z}^{\oplus\mathsf{v}^{\mathbf{a}}h_{w,d,z}}
\end{gather*}
in $\cSHo$. Noting that \eqref{eq:kl-mult}, \eqref{eq:kl-coeff} and \cite[Conjecture 14.2.P2, P5 and P7]{Lu2} imply that
\begin{gather*}
\gamma_{w,d,z^{\mone}}=\gamma_{z^{\mone},w,d}=\delta_{z,w},
\end{gather*}
we deduce that
\begin{gather*}
\mathsf{v}^{\mathbf{a}}h_{w,d,z}\in
\begin{cases}
1+\dots+\mathsf{v}^{2\mathbf{a}}& \text{if }z=w,
\\
\mathsf{v}\mathbb{N}_{0}[\mathsf{v}]\cap\mathsf{v}^{2\mathbf{a}-1}\mathbb{N}_{0}[\mathsf{v}^{\mone}]&\text{if }z\neq w.
\end{cases}
\end{gather*}
Therefore, the object $\mathrm{C}^{\oplus\mathsf{v}^{\mone[2\mathbf{a}]}}$ is
isomorphic to a direct summand of $\mathrm{C}\Cd^{\oplus\mathsf{v}^{\mone[4\mathbf{a}]}}$ in $\mathbf{M}^{(0)}(\varnothing)$.
The head of the indecomposable injective object $\Cd$
in $\mathbf{M}$ is isomorphic to a direct summand of the socle of
$\mathrm{C}^{\oplus\mathsf{v}^{\mone[2\mathbf{a}]}}$ in $\underline{\mathbf{M}^{(0)}}(\varnothing)$,
see Proposition \ref{prop:frobenius21} and its proof.
Hence, the right-hand side of \eqref{eq:hom-1} is non-zero, which implies that
the left-hand side is non-zero. This shows that $\mathrm{C}$ has a direct summand
in $\cSHo$ isomorphic to $\mathrm{C}_{v}$ with graded length at least $4\mathbf{a}$.
Therefore $\mathrm{C}_{v}$ must have graded length exactly $4\mathbf{a}$, as we already
know that it is at most $4\mathbf{a}$.
By adjunction, we have
\begin{gather}\label{eq:grl}
\begin{aligned}
0\neq\mathrm{hom}_{{\ccS_{\mathcal{H}}}}(\Cd,\mathrm{C}_{v}^{\oplus\mathsf{v}^{\mone[4\mathbf{a}]}})
&\cong
\mathrm{hom}_{{\ccS_{\mathcal{H}}}}(\tilde{\mathrm{C}}_{v^{\mone}},\tilde{\mathrm{C}}_{d}^{\oplus\mathsf{v}^{\mone[4\mathbf{a}]}})
\cong\mathrm{hom}_{{\ccS_{\mathcal{H}}}}(\mathrm{C}_{v^{\mone}},\Cd^{\oplus\mathsf{v}^{\mone[4\mathbf{a}]}}),
\end{aligned}
\end{gather}
yielding that the graded length of $\Cd$ is at
least $4\mathbf{a}$. Thus, as above, it must be equal to $4\mathbf{a}$.
Note that for any $w\in\mathcal{H}$ we have
$\mathsf{v}^{\mathbf{a}}h_{w,w^{\mone},d}\in 1+\dots+\mathsf{v}^{2\mathbf{a}}$.
Therefore each $\tilde{\mathrm{C}}_{w^{\mone}}\mathrm{C}_{w}^{\oplus\mathsf{v}^{\mone[4\mathbf{a}]}}$
contains a direct summand $\Cd^{\oplus\mathsf{v}^{\mone[4\mathbf{a}]}}$ in $\cSHo$. By \eqref{eq:grl}, we have
\begin{gather*}
\begin{aligned}
0\neq\mathrm{hom}_{{\ccS_{\mathcal{H}}}}(\mathrm{C}_{v^{\mone}},\tilde{\mathrm{C}}_{w^{\mone}}
\mathrm{C}_{w}^{\oplus\mathsf{v}^{\mone[4\mathbf{a}]}})\cong
\mathrm{hom}_{{\ccS_{\mathcal{H}}}}(\mathrm{C}_{w}\mathrm{C}_{v^{\mone}},
\mathrm{C}_{w}^{\oplus\mathsf{v}^{\mone[4\mathbf{a}]}}).
\end{aligned}
\end{gather*}
By Lemma \ref{lem5.2}, the direct summands of
$\mathrm{C}_{w}\mathrm{C}_{v^{\mone}}$ in $\cSHo$ have non-negative shift.
Again as before, this shows that the graded length of $\mathrm{C}_{w}$ is at
least $4\mathbf{a}$, which implies that it must be equal to $4\mathbf{a}$.
\end{proof}

Recall from Example \ref{ex:Cd-alg-coalg} that $\tilde{\mathrm{C}}_{d}$ is a simple
algebra $1$-morphism in $\overline{\cSH}$.

\begin{proposition}\label{prop711}
\leavevmode
\begin{enumerate}[label=(\roman*)]

\item\label{prop711i} The $2$-representation $\mathbf{inj}_{\underline{\ccS_{\mathcal{H}}}}(\Cd)$
is equivalent to the cell $2$-representa\-tion $\mathbf{C}_{\mathcal{H}}$.

\item\label{prop711ii} The $2$-representation $\boldsymbol{\mathrm{proj}}_{\overline{\ccS_{\mathcal{H}}}}(\tilde{\mathrm{C}}_{d})$
is also equivalent to the cell $2$-representation $\mathbf{C}_{\mathcal{H}}$.

\end{enumerate}
\end{proposition}

\begin{proof}
Again, we will prove the statement \ref{prop711i} and the dual statement
\ref{prop711ii} follows verbatim.
Consider $\Cd$ as an object of $\mathbf{C}_{\mathcal{H}}(\varnothing)$
and set $\mathrm{C}:=[\Cd,\Cd]$ (note that here the internal cohom-construction is with respect to $\mathbf{C}_{\mathcal{H}}$ and not $\mathbf{M}$ as in the proof of Lemma \ref{lem:gr}). As a $1$-morphism in $\cSHo$ we have
\begin{gather*}
\mathrm{C}\cong\bigoplus_{w\in\mathcal{H}}
\mathrm{C}_{w}^{\oplus p_{w}},
\end{gather*}
with $p_{w}\in\mathbb{N}_{0}[\mathsf{v},\mathsf{v}^{\mone}]$. Furthermore,
\begin{gather}\label{eq7.1-2}
\mathrm{hom}_{{\ccS_{\mathcal{H}}}}(\mathrm{C},\mathrm{C}_{w}^{\oplus\mathsf{v}^k})
=
\mathrm{hom}_{{\ccS_{\mathcal{H}}}}\big([\Cd,\Cd],\mathrm{C}_{w}^{\oplus\mathsf{v}^k}\big)
\cong
\mathrm{hom}_{\mathbf{C}_{\mathcal{H}}}(\Cd,\mathrm{C}_{w}^{\oplus\mathsf{v}^k}\Cd).
\end{gather}
By positivity of the grading on $\cSH$ and
Lemma \ref{lem5.2}, we see that $p_{w}\in\mathbb{N}_{0}[\mathsf{v}^{\mone}]$.

If $k<-4\mathbf{a}$, then the right-hand side of \eqref{eq7.1-2} is zero, because
$\Cd$ is an indecomposable injective object of graded length
$2\mathbf{a}$ by Lemma \ref{lem7.1-0}, and the action of $\mathrm{C}_{w}$
increases the graded length by at most $2\mathbf{a}$ by Lemma \ref{lem5.2}. On the left-hand side,
the indecomposable injective object $\mathrm{C}_{w}$ has graded length $4\mathbf{a}$, see Lemma \ref{lem:gr},
which implies that $\mathrm{C}$ lives in non-negative degrees, that is, $p_{w}=p_{w}(0)$.

For $k=0$, the right-hand side of \eqref{eq7.1-2} is one dimensional if $w=d$, and zero otherwise.
This implies $\mathrm{C}\cong\Cd$ in $\cSHo$. Since the degree $0$ maps
$\Cd\to\mathbbm{1}_{\ccS_{\mathcal{H}}}$ and $\Cd\to\Cd\Cd$ are unique up to
scalar, it follows that $\mathrm{C}\cong\Theta(\mathbbm{1}_{\ccA_{\mathcal{H}}})$ as
coalgebra $1$-morphisms in $\cSHo$.
\end{proof}

Let
\begin{gather}\label{def:Cdbicomodules}
\cBH:=(\Cd)\mathrm{biinj}_{\underline{\mathrm{add}(\mathcal{H})}}(\Cd)
\end{gather}
denote the graded one-object bicategory of biinjective $\Cd$-$\Cd$-bicomodules in $\underline{\mathrm{add}(\mathcal{H})}$,
where horizontal composition is given by the cotensor product ${}_{-}\square_{\Cd}{}_{-}$ and the identity $1$-morphism is
$\Cd$. By definition, a biinjective $\Cd$-$\Cd$-bicomodule is injective as a left and
as a right $\Cd$-comodule, but not necessarily as a $\Cd$-$\Cd$-bicomodule, e.g. $\Cd$ is biinjective
but not injective as a $\Cd$-$\Cd$-bicomodule when $d\neq e, w_{0}$. The bicategory $\cBH$
will play an important role in this article, c.f. Proposition \ref{proposition:bicomodules} and
Theorem \ref{theorem:main}. For now, it is important to record the following corollary.

\begin{corollary}\label{cor:End-Bicomod}
There is a graded biequivalence
\begin{gather*}
\cBHop\simeq\cEnd_{\ccSH}(\CH).
\end{gather*}
\end{corollary}

\begin{proof}
By Proposition \ref{prop711}\ref{prop711i}, there is a graded biequivalence
\begin{gather*}
\cEnd_{\ccSH}(\CH)\simeq\cEnd_{\ccSH}\big(\mathbf{inj}_{\underline{\ccSH}}(\Cd)\big).
\end{gather*}
Using this identification, the assignment
\begin{gather*}
\cBHop\xrightarrow{}\cEnd_{\ccSH}(\CH)
,\quad
\mathrm{X}\mapsto{}_{-}\,\sd\mathrm{X}
\end{gather*}
defines a biequivalence, as follows from Proposition \ref{prop:exactness} and
\cite[Theorems 4.26 and 4.31]{MMMTZ2}.
\end{proof}


\subsection{The categorified bar involution}\label{s7.2}


\begin{proposition}\label{prop:catbar}
There exists a functorial involution $\vee\colon\cS\to\cS$,
which is covariant on $1$-morphisms and contravariant and degree-preserving on $2$-morphisms, such that we have
\begin{gather}\label{eq:catbar}
(\mathrm{B}_{w}^{\oplus\mathsf{v}^k})^{\vee}\cong\mathrm{B}_{w}^{\oplus\mathsf{v}^{\mone[k]}}
\end{gather}
in $\cSo$ for all $w\in W$ and $k\in\mathbb{Z}$.
\end{proposition}

\begin{proof}
Let $\cD$ be the diagrammatic (one-object) Soergel $2$-category as in \cite{EW3},
whose $1$-morphisms are of the form $\underline{w}\langle k\rangle$, where
$\underline{w}$ is a word in the simple reflections of $W$ and $k\in\mathbb{Z}$ indicates a formal
degree shift, and whose $2$-morphisms are Soergel diagrams, defined by generators
and relations. By \cite[Theorem 6.30]{EW3}, we can identify $\cS$ with $\mathrm{add}(\cD)$.
(Strictly speaking, we have to quotient $\mathrm{add}(\cD)$ by the $2$-ideal
generated by the totally invariant polynomials
with no constant term in the base ring $R$ in that paper, because
$\cS$ was defined over the coinvariant algebra. This is a technical detail, which we will suppress from now on.) Now, $\vee\colon\cD\to\cD$
is defined by $(\underline{w}\langle -k\rangle)^{\vee}:=
\underline{w}\langle k\rangle$, for all words $\underline{w}$ and $k\in\mathbb{Z}$, and by flipping the Soergel diagrams
upside-down. By definition, $\vee$ is covariant on $1$-morphisms, contravariant and
degree-preserving on $2$-morphisms. For any word $\underline{w}$ in the simple
reflections of $W$, let $\mathrm{BS}(\underline{w})$ be the corresponding Bott--Samelson
bimodule in $\cS$. Then
$\big(\mathrm{BS}(\underline{w})^{\oplus\mathsf{v}^k}\big)^{\vee}\cong\mathrm{BS}(\underline{w})^{\oplus\mathsf{v}^{\mone[k]}}$
under the identification $\cS\simeq\mathrm{add}(\cD)$.

Extend $\vee$ to $\cS\simeq\mathrm{add}(\cD)$. To show \eqref{eq:catbar},
we use induction on the length $\ell(w)$ of $w$, the case $\ell(w)=0$ being immediate.
Assume that $\ell(w)>0$ and that \eqref{eq:catbar} holds
for all $v\in W$ with $\ell(v)<\ell(w)$ and all $k\in\mathbb{Z}$. By \cite[Satz 6.14]{So3}
(see also \cite[Theorem 3.16]{EW3} and the text around it)
and \cite[Corollary 6.27]{EW3}, we have
\begin{gather*}
\mathrm{BS}(\underline{w})\cong\mathrm{B}_{w}\oplus\bigoplus_{v\prec w}\mathrm{B}_{v}^{\oplus p_{w,v}}\;\text{in}\; \cSo
\end{gather*}
for every $w\in W$, where $\underline{w}$ is an arbitrary
reduced expression for $w$, the $p_{w,v}\in\mathbb{N}_{0}[\mathsf{v},\mathsf{v}^{\mone}]$
are invariant under the bar involution (which follows from \cite[Chapter 4]{Lu2} and
the Soergel--Elias--Williamson
categorification theorem) and $\prec$ is the Bruhat order.
As remarked above, we have $\mathrm{BS}(\underline{w})^{\vee}\cong\mathrm{BS}(\underline{w})$ in $\cSo$.
By induction, we also have $(\mathrm{B}_{v}^{\oplus p_{w,v}})^{\vee}\cong \mathrm{B}_{v}^{\oplus p_{w,v}}$
in $\cSo$ for all $v\prec w$, using the bar invariance
of $p_{w,v}$. Since $\cSo$ is Krull--Schmidt, we deduce that $B_{w}^{\vee}\cong B_{w}$ in $\cSo$.
This implies that \eqref{eq:catbar} holds for all $k\in\mathbb{Z}$.
\end{proof}

Recall that the bar involution on the Hecke algebra is
uniquely determined by the fact that it is $\mathbb{Z}$-linear, sends $\mathsf{v}^{k}\mapsto\mathsf{v}^{\mone[k]}$
for all $k\in\mathbb{Z}$ and fixes the Kazhdan--Lusztig basis
elements, see e.g. \cite[Chapters 4 and 5]{Lu2}. By Proposition \ref{prop:catbar} and \cite[Corollary 6.27 and Theorem 6.30]{EW3},
the duality $\vee$ thus categorifies the bar involution on the Hecke algebra. We will therefore refer to it
as the \emph{categorified bar involution}.

\begin{remark}
The origin of the categorified bar involution is \cite{So3}. In the diagrammatic setting
$\vee$ appears in \cite[Definition 6.22]{EW3}, where it is denoted $\iota$ and gives an antiinvolution on double light leaves.
\end{remark}

\begin{proposition}\label{prop:duality}
The categorified bar involution on $\cS$ induces a functorial involution on $\cSH$, also
denoted $\vee$, which is covariant on $1$-morphisms and
contravariant and degree-preserving on $2$-morphisms,
such that $(\mathrm{B}_{x}^{\oplus\mathsf{v}^{k}})^{\vee}\cong\mathrm{B}_{x}^{\oplus\mathsf{v}^{\mone[k]}}$ in $\cSHo$,
for all $x\in\mathcal{H}$ and $k\in\mathbb{Z}$. This functorial involution
extends to an equivalence between $\underline{\cSH}$ and
$\overline{\cSH}$ which sends injective $1$-morphisms in the first $2$-category to projective $1$-morphisms in the second.
\end{proposition}

\begin{proof}
Since $\vee$ is covariant on $1$-morphisms and
$(\mathrm{B}_{w}^{\oplus\mathsf{v}^k})^{\vee}\cong\mathrm{B}_{w}^{\oplus\mathsf{v}^{\mone[k]}}$ in $\cSo$,
for $w\in W$ and $k\in\mathbb{Z}$,
it preserves left, right and two-sided cells.

As $\vee$ sends identity $2$-morphisms to identity $2$-morphisms,
it also preserves the maximal $2$-ideal in $\cS$ that does not contain the identity $2$-morphisms
on the $1$-morphisms in $\mathrm{add}(\mathcal{J})$, whence it descends to the
$\mathcal{J}$-simple quotient $\cS_{\leq\mathcal{J}}$. Since $\vee$ also preserves the
two-sided cell $\mathcal{J}$ and its diagonal $\mathcal{H}$-cells, it restricts to the $2$-full
$2$-subcategories $\cSJ$ and $\cSH$, which proves the first claim.

Finally, since $\vee$ is contravariant on $2$-morphisms, it extends to an equivalence between
$\overline{\cSH}$ and
$\underline{\cSH}$.
\end{proof}

\begin{corollary}\label{cor:duality}
The functorial involution $\vee$ induces a functorial
involution on $\mathbf{C}_{\mathcal{H}}$, also denoted $\vee$, which is contravariant and degree-preserving
on morphisms and satisfies $(\mathrm{B}_{x}^{\oplus\mathsf{v}^{k}})^{\vee}
\cong\mathrm{B}_{x}^{\oplus\mathsf{v}^{\mone[k]}}$ in $\mathbf{C}^{(0)}_{\mathcal{H}}(\varnothing)$,
for all $x\in\mathcal{H}$ and $k\in\mathbb{Z}$.
This functorial involution extends to an equivalence between
$\underline{\mathbf{C}_{\mathcal{H}}}$ and $\overline{\mathbf{C}_{\mathcal{H}}}$
which sends injective objects in the category underlying the first to projective
objects in the category underlying the second.
\end{corollary}

\begin{proof}
As already remarked, the functorial involution $\vee$ also preserves left cells.
The rest now follows as in the proof of Proposition \ref{prop:duality}.
\end{proof}

\begin{remark}\label{remark:7}
The existence of $\vee$ implies that any statement in Subsection \ref{s7.1} has a dual counterpart. In particular, the equivalence
$\underline{\mathbf{C}_{\mathcal{H}}}\simeq \overline{\mathbf{C}_{\mathcal{H}}}$
induces a degree-preserving isomorphism of algebras $\mathsf{B}^\vee\cong\mathsf{B}^{\mathrm{op}}$, see Proposition \ref{prop:Frobenius}.
\end{remark}


\subsection{Explicit bimodules for the cell $2$-representation}\label{s7.4}


In the following, we will use projective abelianizations instead of injective ones.
As we are in the fiat setup, the difference does not play an essential role on an abstract level,
but with this choice we can conveniently
describe the action of $\mathrm{C}_{w}$ by projective bimodules and
their composition by tensoring over the underlying algebra, instead of injective bicomodules
and cotensoring over the underlying coalgebra.

Denote by
\begin{gather*}
\mathsf{B}:=
\mathrm{End}_{\overline{\mathbf{C}_{\mathcal{H}}}}
\big(
\bigoplus_{w\in\mathcal{H}}\mathrm{C}_{w}
\big)^{\mathrm{op}}
\end{gather*}
the algebra underlying the cell $2$-representation.
Fix a complete set of orthogonal primitive idempotents $e_{w}\in\mathsf{B}$, for $w\in\mathcal{H}$,
corresponding to the indecomposable projective objects $\mathrm{C}_{w}\in\mathbf{C}_{\mathcal{H}}(\varnothing)$.
Set $Q_{w}:=\mathsf{B}e_{w}$ and let $L_{w}$
(recall the notation in Subsection \ref{s7.1}) be the simple head of $Q_{w}$ in $\overline{\mathbf{C}_{\mathcal{H}}}$. Note that $L_{w}$ is one-dimensional and isomorphic to $\mathbb{C}e_{w}$ as a vector space.
Note that $\mathrm{C}_{w}\,L_{d}\cong Q_{w}$ for $w\in\mathcal{H}$, see
Subsection \ref{s2.2} and \cite[Section 7]{MM3}.
Lemma \ref{lem7.1-3}\ref{lem7.1-3ii} implies that the socle of $Q_{w}$ is isomorphic
to $L_{w}^{\oplus\mathsf{v}^{2\mathbf{a}}}$.

For every pair $x,y\in\mathcal{H}$, we have
\begin{gather}\label{eq:grdimHom}
\begin{aligned}
\mathrm{grdim}\big(\mathrm{Hom}_{\mathsf{B}}(Q_{x},Q_{y})\big)
&=\mathrm{grdim}\big(\mathrm{Hom}_{\mathsf{B}}(\mathrm{C}_{x}\,L_{d},\mathrm{C}_{y}\,L_{d})\big)
\\
&=\mathrm{grdim}\big(\mathrm{Hom}_{\mathsf{B}}(\mathrm{C}_{y^{\mone}}\mathrm{C}_{x}\,L_{d}^{\oplus\mathsf{v}^{\mone[2\mathbf{a}]}},L_{d})\big)
\\
&=\sum_{z\in\mathcal{H}}h_{y^{\mone},x,z}\,\mathrm{grdim}\big(\mathrm{Hom}_{\mathsf{B}}(\mathrm{C}_{z}\,L_{d}^{\oplus\mathsf{v}^{\mone[\mathbf{a}]}},L_{d})\big)
\\
&=\sum_{z\in\mathcal{H}}h_{y^{\mone},x,z}\,\mathrm{grdim}\big(\mathrm{Hom}_{\mathsf{B}}(Q_{z}^{\oplus\mathsf{v}^{\mone[\mathbf{a}]}},L_{d})\big)
\\
&=\mathsf{v}^{\mathbf{a}}h_{y^{\mone},x,d}.
\end{aligned}
\end{gather}
By \cite[Subsection 13.6]{Lu2}, we know that
\begin{gather*}
\mathsf{v}^{\mathbf{a}}h_{y^{\mone},x,d}\in
\begin{cases}
1+\dots+\mathsf{v}^{2\mathbf{a}}& \text{if }x=y;
\\
\mathsf{v}\mathbb{N}_{0}[\mathsf{v}]\cap\mathsf{v}^{2\mathbf{a}-1}\mathbb{N}_{0}[\mathsf{v}^{\mone}]&\text{if }x\neq y.
\end{cases}
\end{gather*}
Recall that, by definition of the Kazhdan--Lusztig basis, the $h_{y^{\mone},x,d}$ are invariant under the bar involution.

\begin{remark}
By \cite[13.1(e)]{Lu2}, we have $h_{v^{\mone},x,d}=h_{x^{\mone},v,d}$, which corresponds to the fact that
\begin{gather*}
\mathrm{grdim}\big(\mathrm{Hom}_{\mathsf{B}}(Q_{x},Q_{v})\big)
=
\mathrm{grdim}\big(\mathrm{Hom}_{\mathsf{B}}(Q_{v},Q_{x})\big)
\end{gather*}
for all $x,v\in\mathcal{H}$.
\end{remark}

For the formulation of the next proposition, we recall
that the $\gamma_{w,v,u^{\mone}}$ are defined in Subsection \ref{s5.1}.

\begin{proposition}\label{prop:2-action}
For any $w\in\mathcal{H}$, the action of $\mathrm{C}_{w}$ on the category of
finite dimensional graded $\mathsf{B}$-modules is isomorphic to tensoring (over $\mathsf{B}$)
with the
graded projective $\mathsf{B}$-$\mathsf{B}$-bimodule
\begin{gather*}
\bigoplus_{u,v\in\mathcal{H}}\big(\mathsf{B}e_{u}\otimes_{\mathbb{C}} e_{v}\mathsf{B}\big)^{\oplus\gamma_{w,v,u^{\mone}}}.
\end{gather*}
\end{proposition}

\begin{proof}
By Theorem \ref{thm:projapex}, we know that the
action of $\mathrm{C}_{w}$ is given by tensoring with a $\mathsf{B}$-$\mathsf{B}$-bimodule of the form
\begin{gather*}
\bigoplus_{u,v\in\mathcal{H}}\big(\mathsf{B}e_{u}\otimes_{\mathbb{C}} e_{v}\mathsf{B}\big)^{\oplus c_{w,v,u}},
\end{gather*}
for certain $c_{w,v,u}\in\mathbb{N}_{0}[\mathsf{v},\mathsf{v}^{\mone}]$.
We also know that, for any $x\in\mathcal{H}$, we must have
\begin{gather*}
\mathrm{C}_{w}\,Q_{x}\cong\mathrm{C}_{w}\mathrm{C}_{x}\,L_{d}\cong
\bigoplus_{u\in\mathcal{H}}
Q_{u}^{\oplus\mathsf{v}^{\mathbf{a}}h_{w,x,u}}.
\end{gather*}
On the other hand, using \eqref{eq:grdimHom}, we obtain
\begin{align*}
\bigoplus_{u,v\in\mathcal{H}}\big(\mathsf{B} e_{u}\otimes_{\mathbb{C}} e_{v} \mathsf{B}\big)^{\oplus c_{w,v,u}}
\otimes_{\mathsf{B}}\mathsf{B}e_{x}\cong & \bigoplus_{u,v\in\mathcal{H}}
\big(\mathsf{B}e_{u}\otimes_{\mathbb{C}}e_{v}\mathsf{B}e_{x}\big)^{\oplus c_{w,v,u}}
\\
\cong & \bigoplus_{u,v\in\mathcal{H}}\mathsf{B} e_{u}^{\oplus\mathrm{grdim}(e_{v} \mathsf{B}e_{x})c_{w,v,u}}
\\
\cong & \bigoplus_{u,v\in\mathcal{H}}\mathsf{B} e_{u}^{\oplus\mathsf{v}^{\mathbf{a}}h_{x^{\mone},v,d}c_{w,v,u} },
\end{align*}
and hence, deduce that the $c_{w,v,u}$ have to satisfy
\begin{gather}\label{eq:magic1}
\sum_{v\in\mathcal{H}}h_{x^{\mone},v,d}c_{w,v,u}=h_{w,x,u},
\end{gather}
for all $w,x,u,v\in\mathcal{H}$.

For every fixed pair $w,u\in\mathcal{H}$, this is a system
of $\#\mathcal{H}$ linear equations, indexed by $x\in\mathcal{H}$, in $\#\mathcal{H}$ variables,
indexed by $v\in\mathcal{H}$.
We claim that $c_{w,v,u}=\gamma_{w,v,u^{\mone}}$, for $v\in\mathcal{H}$, is the unique solution of \eqref{eq:magic1}.

Let us first show that $c_{w,v,u}=\gamma_{w,v,u^{\mone}}$ is a solution of \eqref{eq:magic1}, i.e. that we have
\begin{gather*}
\sum_{v\in\mathcal{H}}h_{x^{\mone},v,d}\gamma_{w,v,u^{\mone}}=h_{w,x,u}.
\end{gather*}
This equation is similar to one in \cite[Subsection 18.8]{Lu2}
and can be proved in the same way, using:
\begin{enumerate}[\textbullet]

\item the equation at the beginning of the proof of Theorem 18.9(b) in \cite{Lu2}, i.e.
\begin{gather}\label{eq:Lmagic}
\sum_{z\in W}h_{x_{1},x_{2},z}\gamma_{z,x_{3},y^{\mone}}
=
\sum_{z\in W}h_{x_{1},z,y}\gamma_{x_{2},x_{3},z^{\mone}};
\end{gather}

\item the symmetries in \cite[13.1(e)]{Lu2};

\item \cite[Proposition 13.9(b) and Conjecture 14.2.P7]{Lu2}, i.e.
\begin{gather}\label{eq:Lsym}
h_{a,b,c}=h_{b^{\mone},a^{\mone},c^{\mone}},
\quad
\gamma_{a,b,c}=\gamma_{b^{\mone},a^{\mone},c^{\mone}},
\quad
\gamma_{a,b,c}=\gamma_{c,a,b};
\end{gather}

\item \cite[Conjectures 14.2.P2 and 14.2.P5]{Lu2}, i.e.
\begin{gather}\label{eq:Ldgamma}
\gamma_{v,u,d}=\delta_{v,u^{\mone}}.
\end{gather}

\end{enumerate}

By \eqref{eq:Lmagic}, \eqref{eq:Lsym} and \eqref{eq:Ldgamma}, we have
\begin{align*}
\sum_{v\in\mathcal{H}}h_{x^{\mone},v,d}\gamma_{w,v,u^{\mone}}=&\sum_{v\in\mathcal{H}}h_{x^{\mone},v,d}\gamma_{v^{\mone},w^{\mone},u}
\\
=& \sum_{v\in\mathcal{H}}h_{x^{\mone},v,d}\gamma_{w^{\mone},u,v^{\mone}}
\\
=& \sum_{v\in\mathcal{H}}h_{x^{\mone},w^{\mone},v}\gamma_{v,u,d}
\\
=& h_{x^{\mone},w^{\mone},u^{\mone}}
\\
=& h_{w,x,u}.
\end{align*}

Finally, note
\begin{gather*}
h_{x^{\mone},v,d}=\mathsf{v}^{\mone[\mathbf{a}]}\gamma_{x^{\mone},v,d}
\;\big(\!\bmod\mathsf{v}^{\mone[\mathbf{a}]+1}\mathbb{N}_{0}[\mathsf{v}]\big),
\end{gather*}
so the determinant of the matrix
\begin{gather*}
\big(h_{x^{\mone},v,d}\big)_{x,v\in\mathcal{H}}
\end{gather*}
belongs to $v^{\mone[\mathbf{a}](\#\mathcal{H})}(1+\mathsf{v}\mathbb{N}_{0}[\mathsf{v}])$,
and the matrix is hence invertible over $\mathbb{C}(\mathsf{v})$. Our system of linear
equations in \eqref{eq:magic1} therefore has a unique solution and the statement of the proposition follows.
\end{proof}


\subsection{The Frobenius structure on the Duflo involution}\label{s7.5}


In this subsection, we describe the structure
of a Frobenius algebra on the Duflo involution
in $\cSH$ explicitly. As $\Cd\cong\Theta(\mathrm{A}_{d})$ is a
graded coalgebra, the comultiplication and counit $2$-morphisms are homogeneous of degree $0$.
By duality, $\tilde{\mathrm{C}}_{d}=\Cd^{\oplus\mathsf{v}^{\mone[2]\mathbf{a}}}$
is a graded algebra. This implies that the multiplication and unit $2$-morphisms of $\Cd$ are homogeneous of
degree $-2\mathbf{a}$ and $2\mathbf{a}$, respectively. It remains to show
that the multiplication and the comultiplication of $\Cd$ satisfy the Frobenius conditions.

Recall the algebra $\mathsf{B}$ constructed in Subsection \ref{s7.4}.
Proposition \ref{prop:2-action}, combined with
\eqref{eq:Lsym} and \eqref{eq:Ldgamma}, implies
that $\Cd$ acts via the $\mathsf{B}$-$\mathsf{B}$-bimodule
\begin{gather*}
\bigoplus_{u\in\mathcal{H}}\mathsf{B}e_{u}\otimes_{\mathbb{C}}e_{u}\mathsf{B}.
\end{gather*}
The comultiplication on this bimodule is given by
\begin{gather*}
\delta_{d}\colon
\bigoplus_{u\in\mathcal{H}}
\mathsf{B}e_{u}\otimes_{\mathbb{C}}e_{u}\mathsf{B}\to\bigoplus_{u,v\in\mathcal{H}}
\mathsf{B}e_{u}\otimes_{\mathbb{C}}e_{u}\mathsf{B}e_{v}\otimes_{\mathbb{C}}e_{v}\mathsf{B},\;
e_{u}\otimes e_{u}\mapsto \delta_{d}(u)e_{u}\otimes e_{u}\otimes e_{u}
\end{gather*}
and the counit by
\begin{gather*}
\epsilon_{d}\colon
\bigoplus_{u\in\mathcal{H}}\mathsf{B}e_{u}\otimes_{\mathbb{C}} e_{u}\mathsf{B}\to
\mathsf{B},\quad ae_{u}\otimes e_{u}b\mapsto \delta_{d}(u)^{\mone} ae_{u}b
\end{gather*}
for certain $\delta_{d}(u)\in\mathbb{C}^\times$. In general we do not have a presentation of $\cSH$ in terms of
generating $2$-morphisms and relations, so we can not compute the $\delta_{d}(u)$ explicitly. In the specific
case of dihedral Soergel bimodules such a presentation does exist and the $\delta_{d}(u)$ were computed
explicitly in \cite[Subsection 4.2]{MT}. In general, all we can say is that if the bimodule map representing $\delta_{d}$
involves $\delta_{d}(u)$, then the bimodule map representing $\epsilon_{d}$ involves $\delta_{d}(u)^{\mone}$ by counitality.

To describe the algebra structure, consider the Frobenius trace
$\mathrm{tr}\colon\mathsf{B}\to\mathbb{C}$ and note that
$\mathrm{tr}(e_{u}be_{v})=0$ for a homogeneous element $e_{u}be_{v}$
unless $u=v$ (because $\mathsf{B}$ is weakly symmetric) and the degree of
$e_{u}be_{v}$ is $2\mathbf{a}$. Furthermore, for any $u,v\in\mathcal{H}$, let $\{b_{i,v,u}\mid i=1,\dots,m_{u,v}\}$ and
$\{b^{i,u,v}\mid i=1,\dots,m_{u,v}\}$ be dual bases of $e_{v}\mathsf{B}e_{u}$ and
$e_{u}\mathsf{B}e_{v}$, respectively, with respect to the Frobenius trace $\mathrm{tr}$, i.e.
\begin{gather*}
\mathrm{tr}(b_{i,v,u}b^{j,u,v})=
\begin{cases}
1 &\text{if }i=j;
\\
0 &\text{if }i\neq j.
\end{cases}
\end{gather*}
Here $m_{u,v}=h_{u^{\mone},v,d}(1)=h_{v^{\mone},u,d}(1)$, with $h$ defined as
in \eqref{eq:kl-mult}. Following the same convention, the dual of $e_{u}$ is denoted by $e^{u}$,
for $u\in\mathcal{H}$. Let $\sigma$ be the Nakayama
automorphism of $\mathsf{B}$ defined by $\mathrm{tr}(ab)=\mathrm{tr}(b\sigma(a))$, for
any $a,b\in\mathsf{B}$. Note that, since $\mathsf{B}$ is weakly symmetric, we have $\sigma(e_{u})=e_{u}$
and $\sigma(e^{u})=e^{u}$ for all $u\in\mathcal{H}$. Without loss of generality, we therefore
assume that $e_{u}=b_{1,u,u}$ and $e^{u}=b^{1,u,u}$.
The algebra structure is then given by the multiplication
\begin{gather*}
\mu_{d}\colon
\bigoplus_{u,v\in\mathcal{H}}\big(\mathsf{B}e_{u}\otimes_{\mathbb{C}} e_{u}\mathsf{B} e_{v}\otimes_{\mathbb{C}} e_{v}
\mathsf{B}\big)^{\oplus\mathsf{v}^{\mone[2\mathbf{a}]}}
\hspace*{-0.1cm}\to
\bigoplus_{u\in\mathcal{H}}
\mathsf{B}e_{u}\otimes_{\mathbb{C}} e_{u}\mathsf{B},\\
\hspace*{1.5cm}e_{u}\otimes e_{u}ae_{v}\otimes e_{v}\mapsto \delta_{u,v}\mu_{d}(u)\mathrm{tr}(a) e_{u}\otimes e_{u}
\end{gather*}
and the unit
\begin{gather*}
\iota_{d}\colon\mathsf{B}\to\bigoplus_{u\in\mathcal{H}}
\big(\mathsf{B}e_{u}\otimes_{\mathbb{C}} e_{u}\mathsf{B}\big)^{\oplus\mathsf{v}^{\mone[2\mathbf{a}]}},
\quad e_{u}\mapsto\sum_{v\in\mathcal{H}} \mu_{d}(v)^{\mone}\sum_{i=1}^{m_{u,v}}b^{i,u,v}\otimes b_{i,v,u},
\end{gather*}
for certain $\mu_{d}(u)\in\mathbb{C}^{\times}$. As for the comultiplication and counit $2$-morphisms, we cannot determine the
$\mu_{d}(u)$ explicitly in general. For the specific case of dihedral Soergel bimodules, those scalars were computed explicitly in
\cite[Subsection 4.2]{MT}. In the definition of $\mu_{d}$, note that $\mathrm{tr}_{\mathsf{B}}(a)=0$ unless $u=v$.

We already know that $\Cd$ is a coalgebra and an algebra.
One can immediately check that $\delta_{d}$ and $\mu_{d}$ satisfy the Frobenius conditions.
Since the cell $2$-representation is faithful on $2$-morphisms, this proves that $\Cd$ is a
Frobenius algebra in $\cSH$.

The result in this subsection is weaker than the Klein--Elias--Hogancamp conjecture
\cite[Subsection 5.2]{Kl}, \cite[Conjecture 4.40]{EH}, which conjectures the existence
of a Frobenius structure on $\Cd$ in $\cS$ itself. Indeed, we do not know how to ``lift'' the Frobenius structure
on $\Cd$ from the $\mathcal{H}$-simple quotient $\cSH$ to the whole of $\cS$.
The problem is that $\Cd\Cd$,
as noted in \cite[Conjecture 4.40]{EH}, may contain indecomposable direct summands in $\cSo$ isomorphic
to $\mathrm{C}_{w}^{\oplus\mathsf{v}^{t}}$, with $w>_{\mathcal{J}}\mathcal{H}$ and either $t<\mathbf{a}-\mathbf{a}(w)$ or $3\mathbf{a}-\mathbf{a}(w)<t$.
Note that, in terms of $\mathrm{C}_{z}$, Soergel's hom formula in \eqref{eq:soergel} becomes:
\begin{gather*}
\mathrm{dim}\big(\mathrm{hom}_{\ccS}(\mathrm{C}_{v},\mathrm{C}_{w}^{\oplus\mathsf{v}^{(k+\mathbf{a}(v)-\mathbf{a}(w))}})\big)
=\delta_{v,w}\delta_{0,k},
\quad v,w\in W,k\in\mathbb{Z}_{\geq 0}.
\end{gather*}
In particular, this implies that the dimension of
\begin{gather*}
\mathrm{hom}_{\ccS}(\Cd,\Cd\Cd)
\quad\text{or}\quad
\mathrm{hom}_{\ccS}
(\Cd\Cd,\Cd^{\oplus \mathsf{v}^{2\mathbf{a}}})
\end{gather*}
need not be one, in general. Let us give one simple example.

\begin{example}\label{exfrob}
For rank $2$ or lower $\Cd\Cd$ never contains such direct summands,
but for higher ranks it frequently does. For an explicit and minimal example, let $W$ be of type $A_{3}$
with simple reflections $s_{1}$, $s_{2}$, $s_{3}$, where we write $i=s_{i}$ for short, and Coxeter diagram
\begin{gather*}
\begin{xy}
\xymatrix{
1 \ar@{-}[r] & 2 \ar@{-}[r] & 3
}
\end{xy}.
\end{gather*}
Set $d=12321$.
Then $\mathbf{a}=3$. Consider also the longest element $w_{0}=121321$ of $W$ (whose $\mathbf{a}$-value is $6$),
which is strictly greater than $d$ in the two-sided order. We have
\begin{gather*}
\Cd\Cd\cong
\Cd^{\oplus 1\oplus 3\mathsf{v}^2\oplus 3\mathsf{v}^4\oplus\mathsf{v}^6}
\oplus
\mathrm{C}_{w_{0}}^{\oplus\mathsf{v}^{\mone[4]}\oplus 4\mathsf{v}^{\mone[2]}\oplus 6\oplus 4\mathsf{v}^2\oplus\mathsf{v}^4}\; \text{in}\; \cSo,
\end{gather*}
where the minimal shift of $\mathrm{C}_{w_{0}}$ is strictly smaller than $\mathbf{a}-\mathbf{a}(w_{0})=-3$ and
the maximal shift is strictly bigger than $3\mathbf{a}-\mathbf{a}(w_{0})=3$.
\end{example}


\section{Lifted simple transitive 2-representations}\label{s7.7}


\subsection{The underlying algebra}\label{s7.7.1}


Suppose $(\mathrm{A},\delta_{\mathrm{A}},\epsilon_{\mathrm{A}})$ is a cosimple coalgebra in
$\cAH$. By \cite[Corollary 12]{MMMZ} and Corollary \ref{cor:ss}, $\mathbf{N}:=\mathbf{inj}_{\ccA_{\mathcal{H}}}(\mathrm{A})$ is a simple transitive birepresentation of $\cAH$ and $\mathcal{N}=\mathrm{inj}_{\ccA_{\mathcal{H}}}(\mathrm{A})=\mathrm{comod}_{\ccA_{\mathcal{H}}}(\mathrm{A})$
is semisimple. By Lemma \ref{lem5.1} and Proposition \ref{prop5.4},
$\Theta(\mathrm{A})$ is also a cosimple coalgebra,
which implies that $\mathbf{M}:=\mathbf{inj}_{\underline{\ccS_{\mathcal{H}}}}\big(\Theta(\mathrm{A})\big)$ is a
graded simple transitive
$2$-representation of $\cSH$ with apex $\mathcal{H}$, using \cite[Corollary 12]{MMMZ} again.

Since $\cAH$ is locally semisimple, see
Proposition \ref{prop:fusion}, $\mathrm{A}$ must contain a direct summand
isomorphic to $\mathrm{A}_{d}$, which is the identity $1$-morphism in $\cAH$, and
$\epsilon_{\mathrm{A}}\colon\mathrm{A}\to\mathrm{A}_{d}$ is the projection,
which is a morphism of coalgebras.
Hence, we obtain a faithful morphism of birepresentations of $\cAH$
\begin{gather*}
\Phi_{\ccA_{\mathcal{H}}}\colon\mathbf{inj}_{\ccA_{\mathcal{H}}}(\mathrm{A})
=\mathbf{N}\to
\mathbf{inj}_{\ccA_{\mathcal{H}}}(\mathrm{\mathrm{A}_{d}}) \simeq \mathbf{A}_{\mathcal{H}},
\end{gather*}
which is the identity on morphisms and sends $(N,\delta_{N,\mathrm{A}})$ in
$\mathbf{N}$ to $(N,(\mathrm{id}_{N}\circ_{\mathrm{h}}\epsilon_{\mathrm{A}})\circ_{\mathrm{v}}\delta_{N, \mathrm{A}})$.
Here $\mathbf{A}_{\mathcal{H}}$ denotes the cell birepresentation of $\cAH$.

Since $\Theta$ is $\mathbb{C}$-linear, the above implies that
$\Theta(\mathrm{A})$ contains a direct summand in $\cSHo$ isomorphic to $\Theta(\mathrm{A}_{d})=\Cd$
and the counit $\epsilon_{\Theta(\mathrm{A})}\colon\Theta(\mathrm{A})\to\mathbbm{1}_{\ccS_{\mathcal{H}}}$
is the composite of the projection
$\pi_{d}:=\Theta(\epsilon_{\mathrm{A}})\colon\Theta(\mathrm{A})\to\Cd$
and $\epsilon_{d}\colon\Cd\to\mathbbm{1}_{\ccS_{\mathcal{H}}}$.
In particular, we obtain a faithful, degree-preserving morphism
of $2$-representations of $\cSH$
\begin{gather*}
\Phi_{\ccS_{\mathcal{H}}}\colon\mathbf{inj}_{\underline{\ccS_{\mathcal{H}}}}
\big(\Theta(\mathrm{A})\big)=\mathbf{M}\to
\mathbf{inj}_{\underline{\ccS_{\mathcal{H}}}}(\mathrm{\Cd}) \simeq\mathbf{C}_{\mathcal{H}},
\end{gather*}
which is the identity on morphisms and sends the comodule
$(M,\delta_{M,\Theta(\mathrm{A})})$ in $\mathbf{M}(\varnothing)$ to the comodule
$\big(M,(\mathrm{id}_{M}\circ_{\mathrm{h}}\pi_{d})\circ_{\mathrm{v}}\delta_{M,\Theta(\mathrm{A})}\big)$.

Altogether, this yields a commuting square
\begin{gather*}
\begin{xy}
\xymatrix{
\mathbf{M}
\ar[rr]^{\Phi_{\ccS_{\mathcal{H}}}}
&& \mathbf{C}_{\mathcal{H}}
\\
\mathbf{N}
\ar[rr]_{\Phi_{\ccA_{\mathcal{H}}}}\ar[u]^{\Theta}
&& \mathbf{A}_{\mathcal{H}}.\ar[u]_{\Theta}
}
\end{xy}
\end{gather*}

Let $N_{1},\dots,N_r$ be a complete set of pairwise non-isomorphic
simple ($=$indecomposable) objects in
$\mathbf{N}(\varnothing)$. For every $i=1,\dots,r$, we have
\begin{gather*}
\Phi_{\ccA_{\mathcal{H}}}(N_i)\cong
\bigoplus_{w\in\mathcal{H}}\mathrm{A}_{w}^{\oplus p_{i,w}},
\end{gather*}
for certain $p_{i,w}\in\mathbb{N}_{0}$.

Let $M_i:=\Theta(N_i)$, for $i=1,\dots,r$. Then $M_{1},\dots,M_r$ is a
complete and irredundant set of indecomposable objects of $\mathbf{M}(\varnothing)$ up to isomorphism and grading shift, and
\begin{gather}\label{eq:forgetfuldecomp}
\Phi_{\ccS_{\mathcal{H}}}(M_i)\cong\bigoplus_{w\in\mathcal{H}}
\mathrm{C}_{w}^{\oplus p_{i,w}}\;\text{in}\; \CHo(\varnothing),
\end{gather}
for every $i=1,\dots,r$.

\begin{lemma}\label{lemma:2repcoef}
For $w\in\mathcal{H}$ and $i,j=1,\dots,r$,
define $\tilde{h}_{w,i,j}\in\mathbb{N}[\mathsf{v},\mathsf{v}^{\mone}]$ by
\begin{gather}\label{eq:M}
\mathrm{C}_{w}\,M_i
\cong
\bigoplus_{j=1}^{r}
M_j^{\oplus\tilde{h}_{w,i,j}}\;\text{in}\; \mathbf{M}^{(0)}(\varnothing).
\end{gather}
Then
\begin{gather*}
\tilde{h}_{w,i,j}\in\mathsf{v}^{2\mathbf{a}}\mathbb{N}_{0}[\mathsf{v}^{\mone}]\cap\mathbb{N}_{0}[\mathsf{v}].
\end{gather*}
\end{lemma}

\begin{proof}
The proof follows from comparing two decompositions in $\CHo(\varnothing)$.
On one hand, by \eqref{eq:forgetfuldecomp} and \eqref{eq:M}, we have
\begin{gather}\label{eq:2repcoef1}
\Phi_{\ccS_{\mathcal{H}}}(\mathrm{C}_{w}\,M_i)\cong\bigoplus_{j=1}^{r}\Phi_{\ccS_{\mathcal{H}}}
\big(M_j^{\oplus\tilde{h}_{w,i,j}}\big)\cong\bigoplus_{j=1}^{r}\bigoplus_{v\in\mathcal{H}}
\mathrm{C}_{v}^{\oplus\tilde{h}_{w,i,j}p_{j,v}}.
\end{gather}
On the other hand, by \eqref{eq:forgetfuldecomp} and the fact that
$\Phi_{\ccS_{\mathcal{H}}}$ is a morphism of $2$-representations, we have
\begin{gather}\label{eq:2repcoef2}
\Phi_{\ccS_{\mathcal{H}}}(\mathrm{C}_{w}\,M_i)\cong\mathrm{C}_{w}\,\Phi_{\ccS_{\mathcal{H}}}(M_i)\cong
\bigoplus_{u\in\mathcal{H}}\mathrm{C}_{w}\mathrm{C}_{u}^{\oplus p_{i,u}}\cong
\bigoplus_{u,v\in\mathcal{H}}\mathrm{C}_{v}^{\oplus\mathsf{v}^{\mathbf{a}}p_{i,u}h_{w,u,v}}.
\end{gather}
Comparing \eqref{eq:2repcoef1} and \eqref{eq:2repcoef2} for a fixed $v$, we obtain the equation
\begin{gather*}
\sum_{j=1}^{r}\tilde{h}_{w,i,j}p_{j,v}=\mathsf{v}^{\mathbf{a}}\sum_{u\in\mathcal{H}}p_{i,u}h_{w,u,v}.
\end{gather*}
The facts that
$p_{i,u},p_{j,v}\in\mathbb{N}_{0}$ and $\mathsf{v}^{\mathbf{a}}h_{w,u,v}\in
\mathsf{v}^{2\mathbf{a}}\mathbb{N}_{0}[\mathsf{v}^{\mone}]\cap\mathbb{N}_{0}[\mathsf{v}]$ now imply the result, as for every
$j=1,\dots,r$ there exists at least one $v\in\mathcal{H}$ such that $p_{j,v}\neq 0$.
\end{proof}

Define
\begin{gather*}
\mathsf{B}^{\mathbf{M}}:=\mathrm{End}_{\mathbf{M}}\big(\bigoplus_{i=1}^{r} M_i\big)^{\mathrm{op}}.
\end{gather*}
Then $\mathbf{M}(\varnothing)$ is equivalent to the category of finite dimensional graded injective $\mathsf{B}^{\mathbf{M}}$-modules.

\begin{proposition}\label{prop:frobenius2}
The algebra $\mathsf{B}^{\mathbf{M}}$ is a finite dimensional positively graded Frobenius algebra of graded length $2\mathbf{a}$.
\end{proposition}

\begin{proof}
The case of $\mathbf{M}$ being the cell $2$-representation $\mathbf{C}_{\mathcal{H}}$
is discussed in Proposition \ref{prop:Frobenius} (or in Proposition \ref{prop:frobenius21} because of Proposition \ref{prop711}).
In the case when $\mathbf{M}$ is not necessarily the cell $2$-representation, the proposition follows from similar arguments
as in the proof of Proposition \ref{prop:frobenius21}.
\end{proof}

\begin{remark}\label{remark:frobenius2}
In contrast to the situation in Subsection \ref{s7.1},
we do not know a priori that $\mathsf{B}^{\mathbf{M}}$ is weakly symmetric. Therefore, we have to
include a possible Nakayama permutation in Subsection \ref{s7.7.3} below.
Only at the end of that section, we will be able to show that it is trivial.
\end{remark}


\subsection{A characterization of $2$-representations in the image of $\hat{\Theta}$}\label{s7.7.2}


We note the following characterization of the graded simple transitive $2$-representations
in the image of the map $\hat{\Theta}$ from Proposition \ref{prop5.3-1}.
As we will see in Theorem \ref{theorem:main}, these exhaust the graded simple transitive $2$-representation of $\cSH$
with apex $\mathcal{H}$, up to equivalence.

\begin{theorem}\label{thm7.1-2}
Let $\mathbf{M}$ be a graded simple transitive $2$-representation of $\cSH$ with apex $\mathcal{H}$.
Then $\mathbf{M}$ is in the image of $\hat{\Theta}$ if and only if the
following conditions are satisfied:
\begin{enumerate}[label=(\roman*)]

\item \label{thm7.1-2.1} there is a choice $\{M_i\mid i\in I\}$ of representatives of isomorphism
classes of indecomposable objects in $\mathbf{M}(\varnothing)$
such that the endomorphism algebra $\mathsf{B}$ of $M:=\displaystyle\bigoplus_{i\in I} M_i$
in $\mathbf{M}(\varnothing)$ is positively graded, and, additionally;

\item \label{thm7.1-2.2} for every $w\in\mathcal{H}$ and $i\in I$, the object
$\mathrm{C}_{w} M_i \in\mathbf{M}^{(0)}(\varnothing)$ decomposes into a direct sum whose summands, up to degree-preserving isomorphism,
are of the form $M_j^{\oplus\mathsf{v}^l}$, where $j\in I$ and $0\leq l\leq 2\mathbf{a}$;

\item \label{thm7.1-2.3} the graded length of $\mathsf{B}$ is not greater than $2\mathbf{a}$.
\end{enumerate}
\end{theorem}

\begin{proof}
For the ``only if'' part observe that, by Proposition \ref{prop5.41}, we can pick a choice
of representatives of isomorphism classes of indecomposable objects in
$\mathbf{M}$, which are in the image of $\Theta$. The condition in \ref{thm7.1-2.2} then follows from
Lemma \ref{lemma:2repcoef}. Conditions \ref{thm7.1-2.1} and \ref{thm7.1-2.3} hold by Proposition \ref{prop:frobenius2}.

The ``if'' direction follows the proof of Proposition \ref{prop711} closely.
Set $\mathrm{C}:=[M,M]$. Then
\begin{gather}\label{eq7.1-21}
\mathrm{hom}_{{\ccS_{\mathcal{H}}}}(\mathrm{C},\mathrm{C}_{w}^{\oplus\mathsf{v}^k})\cong
\mathrm{hom}_{\mathbf{M}}(M,\mathrm{C}_{w}^{\oplus\mathsf{v}^k}M).
\end{gather}
Conditions \ref{thm7.1-2.1} and \ref{thm7.1-2.2} imply that the right-hand side
is zero if $k>0$. Hence, writing
\begin{gather*}
\mathrm{C}\cong\bigoplus_{w\in\mathcal{H}}
\mathrm{C}_{w}^{\oplus p_{w}} \;\text{in}\;\cSHo,
\end{gather*}
we obtain $p_{w}\in\mathbb{N}_{0}[\mathsf{v}^{\mone}]$.

Next we want to establish an analog of Lemma \ref{lem7.1-0}.
Namely, we claim that, for any simple object $L$ in $\underline{\mathbf{M}}(\varnothing)$
concentrated in degree $0$, and for any $w\in\mathcal{H}$, the injective object
$\mathrm{C}_{w}\,L$ is concentrated between the degrees $0$ and $2\mathbf{a}$.
Similarly to the proof of Lemma \ref{lem7.1-0}, the fact that
$\mathrm{C}_{w}\,L$ is concentrated  in positive degrees follows
from Conditions \ref{thm7.1-2.1} and \ref{thm7.1-2.2}.
The fact that $\mathrm{C}_{w}\,L$ is concentrated  in degrees below
$2\mathbf{a}$ follows from conditions \ref{thm7.1-2.2} and \ref{thm7.1-2.3}.

Now, if $k<-4\mathbf{a}$, then the right-hand side of \eqref{eq7.1-21}
is zero since, by condition \ref{thm7.1-2.3},
$M$ is a projective-injective object of graded length at most
$2\mathbf{a}$ and the action of $\mathrm{C}_{w}$ is given by projective functors
which increase the graded length by at most $2\mathbf{a}$ (see the previous paragraph).
Given the graded length of $\mathrm{C}_{w}$ in $\cSH$, see Lemma \ref{lem:gr}, this again shows
that $p_{w}=p_{w}(0)\in\mathbb{N}_0$ for all $w\in\mathcal{H}$, so $\mathrm{C}$ is in the image of $\Theta$,
as claimed.
\end{proof}


\subsection{Explicit bimodules for the $2$-action}\label{s7.7.3}


The degree-zero part of $\mathsf{B}^{\mathbf{M}}$ is isomorphic to
$\bigoplus_{i=1}^{r}\mathbb{C}e_i$, where $e_{1},\dots,e_r$ is a complete and irredundant set of primitive,
orthogonal idempotents corresponding
to $M_{1},\dots,M_r$ respectively. Due to \eqref{eq:forgetfuldecomp},
every $M_i$ is concentrated between degrees $0$ and $2\mathbf{a}$, whence
\begin{gather*}
M_i\cong
\mathrm{Hom}_{\mathbb{C}}(e_i\mathsf{B}^{\mathbf{M}},\mathbb{C})^{\oplus\mathsf{v}^{2\mathbf{a}}}
\cong \mathsf{B}^{\mathbf{M}}e_{\sigma(i)},
\end{gather*}
where $\sigma$ is the Nakayama permutation of $\mathsf{B}^{\mathbf{M}}$.
By Theorem \ref{thm:projapex}, the action of $\mathrm{C}_{w}$ on
the category of finite dimensional graded injective $\mathsf{B}^{\mathbf{M}}$-modules, for $w\in\mathcal{H}$,
is given by tensoring over $\mathsf{B}^{\mathbf{M}}$ with a
$\mathsf{B}^{\mathbf{M}}$-$\mathsf{B}^{\mathbf{M}}$-bimodule of the form
\begin{gather*}
\bigoplus_{i,j=1}^{r}\big(\mathsf{B}^{\mathbf{M}}e_{\sigma(j)}
\otimes_{\mathbb{C}} e_{\sigma(i)}
\mathsf{B}^{\mathbf{M}}\big)^{\oplus\tilde{\gamma}_{w,i,j}},
\end{gather*}
for certain $\tilde{\gamma}_{w,i,j}\in\mathbb{N}_0[\mathsf{v},\mathsf{v}^{\mone}]$.

\begin{proposition}\label{prop:2-action2}
We have
\begin{gather*}
\tilde{\gamma}_{w,k,j}=\tilde{h}_{w,k,j}(0)\in\mathbb{N}_0,
\end{gather*}
for all $w\in\mathcal{H}$ and $j,k=1,\dots r$.
\end{proposition}

\begin{proof}
For $w\in\mathcal{H}$ and $1\leq k\leq r$, we obtain two different expressions for
$\mathrm{C}_{w}\mathsf{B}^{\mathbf{M}}e_{\sigma(k)}$ in
$\mathsf{B}^{\mathbf{M}}\text{-}\mathrm{fgproj}^{(0)}$, the category of finite dimensional
graded projective $\mathsf{B}^{\mathbf{M}}$-modules. On one hand, by \eqref{eq:M} and the fact that
$M_l\cong \mathsf{B}^{\mathbf{M}}e_{\sigma(l)}$, we have
\begin{gather}\label{eq:2-action21}
\mathrm{C}_{w}\mathsf{B}^{\mathbf{M}}e_{\sigma(k)}\cong\bigoplus_{j=1}^{r}
\mathsf{B}^{\mathbf{M}}e_{\sigma(j)}^{\oplus\tilde{h}_{w,k,j}}.
\end{gather}
On the other hand, we have
\begin{gather}\label{eq:2-action22}
\mathrm{C}_{w} \mathsf{B}^{\mathbf{M}}e_{\sigma(k)}\cong\bigoplus_{i,j=1}^{r}
\mathsf{B}^{\mathbf{M}}
e_{\sigma(j)}^{\oplus\tilde{\gamma}_{w,i,j}\mathrm{grdim}(e_{\sigma(i)}\mathsf{B}^{\mathbf{M}}e_{\sigma(k)})}.
\end{gather}
Comparing the terms in \eqref{eq:2-action21} and \eqref{eq:2-action22} for a fixed $j$, shows that
\begin{gather}\label{eq:2-action23}
\sum_{i=1}^{r}\tilde{\gamma}_{w,i,j}\mathrm{grdim}\big(e_{\sigma(i)}\mathsf{B}^{\mathbf{M}}e_{\sigma(k)}\big)=
\tilde{h}_{w,k,j}.
\end{gather}
Suppose that $\tilde{\gamma}_{w,\sigma^{\mone}(k),j}$ has a non-zero term belonging to $\mathsf{v}\mathbb{N}_0[\mathsf{v}]$
for some $w,i,j$.
By \eqref{eq:2-action23} and the fact that
$\mathrm{grdim}(e_k\mathsf{B}^{\mathbf{M}}e_{\sigma(k)})$
has highest term $\mathsf{v}^{2\mathbf{a}}$, see Proposition \ref{prop:frobenius2},
this implies that $\tilde{h}_{w,k,j}$ has a non-zero term
belonging to $\mathsf{v}^{2\mathbf{a}+1}\mathbb{N}_0[\mathsf{v}]$.
However, this contradicts Lemma \ref{lemma:2repcoef}.

Since $\mathrm{grdim}(e_{\sigma(i)}\mathsf{B}^{\mathbf{M}} e_{\sigma(k)})\in\delta_{i,k}+
\mathsf{v}\mathbb{N}_0[\mathsf{v}]$, the equation in \eqref{eq:2-action23} implies
that $\tilde{\gamma}_{w,k,j}$ cannot have non-zero terms belonging to
$\mathsf{v}^{\mone}\mathbb{N}_0[\mathsf{v}^{\mone}]$ either, whence
\begin{gather*}
\tilde{\gamma}_{w,k,j}=\tilde{h}_{w,k,j}(0)\in\mathbb{N}_{0},
\end{gather*}
for all $w\in\mathcal{H}$ and $j,k=1,\dots,r$.
\end{proof}

In particular, note that the fact that the constant term in $\mathrm{grdim}\big(e_{\sigma(i)}\mathsf{B}^{\mathbf{M}}e_{\sigma(k)}\big)$
is $1$ if $i=k$,
and $0$ otherwise, implies that
\begin{gather*}
\Cd\mathsf{B}^{\mathbf{M}}e_{\sigma(k)}\cong\bigoplus_{j=1}^r\mathsf{B}^{\mathbf{M}}e_{\sigma(j)}^{\oplus\tilde{\gamma}_{d,k,j}}\oplus R,
\end{gather*}
where all summands of $R$ have coefficients in $\mathsf{v}\mathbb{N}_0[\mathsf{v}]$.
Since the first summand descends to the action of $\mathrm{A}_{d}$, which is the
identity $1$-morphism in $\cAH$, on $\mathbf{N}(\varnothing)$, by Lemma \ref{lem5.3},
we see that $\tilde{\gamma}_{d,k,j}=\delta_{k,j}$.
By Proposition \ref{prop:2-action2} and equation \eqref{eq:2-action23},
this shows that the action of $\Cd=\Theta(\mathrm{A}_{d})$ is given by tensoring with the bimodule
\begin{gather}\label{eq:Duflobimodule}
\bigoplus_{i=1}^{r}
\mathsf{B}^{\mathbf{M}}e_i\otimes_{\mathbb{C}} e_i\mathsf{B}^{\mathbf{M}}
\end{gather}
and that
\begin{gather}\label{eq:grdim}
\mathrm{grdim}\big(\mathrm{Hom}_{\mathsf{B}^{\mathbf{M}}}(\mathsf{B}^{\mathbf{M}}e_i,
\mathsf{B}^{\mathbf{M}}e_k)\big)=\mathrm{grdim}\big(e_i\mathsf{B}^{\mathbf{M}}e_k\big)
=\tilde{h}_{d,\sigma^{\mone}(k),\sigma^{\mone}(i)}.
\end{gather}

In \cite[Theorem 18.9]{Lu2}, Lusztig defined a homomorphism
$\phi\colon\mathsf{H}\to\mathsf{A}\otimes_{\mathbb{Z}}\mathbb{Z}[\mathsf{v},\mathsf{v}^{\mone}]$ of $\mathbb{Z}[\mathsf{v},\mathsf{v}^{\mone}]$-algebras.
Its restriction to $\mathcal{H}$ is given by
\begin{gather*}
\phi_{\mathcal{H}}(c_{w})=\sum_{u\in\mathcal{H}}\mathsf{v}^{\mathbf{a}}h_{w,d,u}a_{u},
\end{gather*}
where $c_{w}:=[\mathrm{C}_{w}]$ in the split Grothendieck group
$[\cSHo]_{\oplus}$ (which should not be confused with Lusztig's
$c_{w}$) and $a_{u}:=[\mathrm{A}_{u}]$ in $[\cAH]_{\oplus}$.
Let $\phi_{\mathcal{H}}^{!}$ denote the pullback of $\phi$.

\begin{proposition}\label{prop:catphi}
We have
\begin{gather*}
[\mathbf{M}]_{\oplus}\cong
\phi_{\mathcal{H}}^{!}
\big([\mathbf{N}]_{\oplus}\big).
\end{gather*}
\end{proposition}

\begin{proof}
By \eqref{eq:Duflobimodule}, there is a degree-preserving isomorphism
\begin{gather*}
\mathrm{C}_{d}\mathbb{C}e_{\sigma(i)}\cong\mathsf{B}^{\mathbf{M}}e_{\sigma(i)}.
\end{gather*}
Using this, we obtain two expressions for $\mathrm{C}_{w}\mathsf{B}^{\mathbf{M}}e_{\sigma(i)}$
in $\mathsf{B}^{\mathbf{M}}\text{-}\mathrm{fgproj}^{(0)}$. On one hand,
\begin{gather}\label{eq:catphi1}
\mathrm{C}_{w}\mathsf{B}^{\mathbf{M}}e_{\sigma(i)}\cong\bigoplus_{j=1}^{r}\mathsf{B}^{\mathbf{M}}e_{\sigma(j)}^{\oplus\tilde{h}_{w,i,j}}.
\end{gather}
On the other hand,
\begin{gather}\label{eq:catphi2}
\begin{aligned}
\mathrm{C}_{w}\mathsf{B}^{\mathbf{M}}e_{\sigma(i)}\cong
\mathrm{C}_{w}\Cd\mathbb{C}e_{\sigma(i)}&\cong
\bigoplus_{u\in\mathcal{H}}\mathrm{C}_{u}^{\oplus\mathsf{v}^{\mathbf{a}}h_{w,d,u}}\mathbb{C}e_{\sigma(i)}
\\
&\cong
\bigoplus_{u\in\mathcal{H}}\bigoplus_{j=1}^{r}\mathsf{B}^{\mathbf{M}}e_{\sigma(j)}^{\oplus\mathsf{v}^{\mathbf{a}}h_{w,d,u}\tilde{\gamma}_{u,i,j}}.
\end{aligned}
\end{gather}
Comparing \eqref{eq:catphi1} and \eqref{eq:catphi2} for a fixed $j$ yields
\begin{gather}\label{eq:catphi3}
\tilde{h}_{w,i,j}=\sum_{u\in\mathcal{H}}\mathsf{v}^{\mathbf{a}}h_{w,d,u}\tilde{\gamma}_{u,i,j},
\end{gather}
which is precisely what we had to prove.
\end{proof}

\begin{corollary}\label{cor:bar-invariance}
$\mathsf{v}^{\mone[\mathbf{a}]}\tilde{h}_{w,i,j}$ is bar invariant.
\end{corollary}

\begin{proof}
Equation \eqref{eq:catphi3} implies that $\mathsf{v}^{\mone[\mathbf{a}]}\tilde{h}_{w,i,j}$
is bar invariant,
since the $h_{w,d,u}$ and $\tilde{\gamma}_{u,i,j}$ are bar invariant. This completes the proof.
\end{proof}

\begin{proposition}\label{proposition:weakly-symmetric}
The algebra $\mathsf{B}^{\mathbf{M}}$ is weakly symmetric.
\end{proposition}

\begin{proof}
Recalling \eqref{eq:kl-mult}, \eqref{eq:kl-coeff} and \eqref{eq:Ldgamma},
we know that
\begin{gather*}
\mathsf{v}^{\mathbf{a}}h_{d,d,u}\in
\begin{cases}
1+\dots+\mathsf{v}^{2\mathbf{a}}& \text{if }u=d;
\\
\mathsf{v}\mathbb{N}_{0}[\mathsf{v}]\cap\mathsf{v}^{2\mathbf{a}-1}\mathbb{N}_{0}[\mathsf{v}^{\mone}]&\text{if }u\neq d.
\end{cases}
\end{gather*}
By \eqref{eq:catphi3} and the equality $\tilde{\gamma}_{d,i,j}=\delta_{i,j}$,
\eqref{eq:grdim} then shows that
\begin{gather*}
\mathrm{grdim}\big(e_i\mathsf{B}^{\mathbf{M}}e_i\big)
=\tilde{h}_{d,\sigma^{\mone}(i),\sigma^{\mone}(i)}
\in 1+\dots+\mathsf{v}^{2\mathbf{a}}.
\end{gather*}
By Proposition \ref{prop:frobenius2}, the highest non-zero term of $\mathrm{grdim}\big(e_i\mathsf{B}^{\mathbf{M}}e_j\big)$ is $\mathsf{v}^{2\mathbf{a}}$ if and only if $j=\sigma(i)$. Therefore, we obtain $\sigma(i)=i$ for all $i$ and the claim follows.
\end{proof}


\section{Separability and semisimplicity}\label{section:separability}


Recall the definition of $\cBH$ from \eqref{def:Cdbicomodules}.
The main point of this section is to prove that the coalgebra $\Cd$ in $\cSH$
is separable, see Propositions \ref{proposition:dot diagram} and
\ref{proposition:dotdiagram-left}, which in turn enables us to show that $\cBH$
is locally graded semisimple, see Proposition \ref{proposition:bicomodules}.


\subsection{Some (diagrammatic) preliminaries}\label{subsection:diagrams}


Let $\mathsf{B}=\mathsf{B}^{\mathbf{C}_{\mathcal{H}}}$ be the basic underlying algebra
of the cell $2$-representation $\mathbf{C}_{\mathcal{H}}$ of
$\cSH$, see Subsection \ref{s7.1}.
This in particular means that there is an equivalence of $2$-representations of
$\cSH$ on
\begin{gather*}
\mathbf{C}_{\mathcal{H}}(\varnothing)\simeq\mathsf{B}\text{-}\mathrm{fgproj}
\end{gather*}
such that the $2$-action on $\mathsf{B}\text{-}\mathrm{fgproj}$, the category of
finite dimensional graded projective $\mathsf{B}$-modules, is given by a graded pseudofunctor from $\cSH$ to
$\cC_{\mathsf{B}}$
(projective $\mathsf{B}$-$\mathsf{B}$-bimodules, see Example \ref{ex:projbimod})
which is $2$-faithful, i.e. faithful on $2$-morphisms. Recall that $\mathsf{B}\cong\oplus_{i=0}^{2\mathbf{a}}\mathsf{B}_{i}$
is a finite dimensional positively graded weakly symmetric Frobenius algebra of
graded length $2\mathbf{a}$, where $\mathbf{a}=\mathbf{a}(\mathcal{H})$ is
the value of Lusztig's $\mathbf{a}$-function on $\mathcal{H}$, see
Proposition \ref{prop:Frobenius}. Note that therefore the objects in $\mathsf{B}\text{-}\mathrm{fgproj}$ are projective-injective and thus, $\mathsf{B}\text{-}\mathrm{fgproj}=\mathsf{B}\text{-}\mathrm{fginj}$, the latter
denoting the category of finite dimensional graded injective $\mathsf{B}$-modules.

By Subsection \ref{s7.5}, $\Cd$ is a graded Frobenius algebra in $\cSH$
with homogeneous comultiplication $\delta_{d}\colon\Cd\to\Cd\Cd$ and counit $\epsilon_{d}\colon\Cd\to\mathrm{C}_e$
of degree $0$, and homogeneous multiplication $\mu_{d}\colon\Cd\Cd\to\Cd$ and unit
$\iota_{d}\colon\mathrm{C}_e\to\Cd$ of degree $-2\mathbf{a}$ and $2\mathbf{a}$,
respectively.
Recall that $\gamma_{d,u,v}=\delta_{u,v^{\mone}}$, so Proposition \ref{prop:2-action}
shows that $\Cd$ acts on $\mathsf{B}\text{-}\mathrm{fgproj}$ by tensoring over $\mathsf{B}$
with the projective $\mathsf{B}$-$\mathsf{B}$-bimodule
\begin{gather*}
\bigoplus_{x\in\mathcal{H}}
\mathsf{B}e_{x}\otimes e_{x}\mathsf{B},
\end{gather*}
and, consequently, the action of $\Cd\Cd$ is given by tensoring with
\begin{gather*}
\bigoplus_{x,y\in\mathcal{H}}\mathsf{B}e_{x}\otimes e_{x}\mathsf{B}e_{y}\otimes e_{y}\mathsf{B}.
\end{gather*}
The $\mathsf{B}$-$\mathsf{B}$-bimodule maps
corresponding to $\delta_{d}$, $\mu_{d}$, $\epsilon_{d}$ and $\iota_{d}$, which we defined in
Subsection \ref{s7.5}, are graphically illustrated as follows (for convenience, we recall the
formulas defining them):
\begin{gather}\label{eq:delta}
\delta_{d}\colon e_{x}
\otimes e_{x}
\mapsto\delta_{d}(x)e_{x}\otimes e_{x}\otimes e_{x}
\leftrightsquigarrow
\begin{tikzpicture}[anchorbase,scale=0.4,smallnodes]
\draw[cstrand] (0,0) node[below,black,yshift=-2pt]{$d$} to (0,1);
\draw[cstrand] (0,1) to[out=180,in=270] (-1,2) node[above,black,yshift=-2pt]{$d$};
\draw[cstrand] (0,1) to[out=0,in=270] (1,2) node[above,black,yshift=-2pt]{$d$};
\end{tikzpicture}
\;(\text{degree }0)
,
\\
\label{eq:mu}
\mu_{d}\colon e_{x}\otimes a\otimes e_{y}\mapsto \delta_{x,y}\mu_{d}(x)\mathrm{tr}(a)e_{x}\otimes e_{x}
\leftrightsquigarrow
\begin{tikzpicture}[anchorbase,scale=0.4,smallnodes]
\draw[cstrand] (0,0) node[above,black]{$d$} to (0,-1);
\draw[cstrand] (0,-1) to[out=180,in=90] (-1,-2) node[below,black,yshift=-2pt]{$d$};
\draw[cstrand] (0,-1) to[out=0,in=90] (1,-2) node[below,black,yshift=-2pt]{$d$};
\end{tikzpicture}
\;(\text{degree} -2\mathbf{a})
,
\\
\label{eq:epsilon}
\epsilon_{d}\colon ae_{x}\otimes e_{x}a^{\prime}\mapsto \delta_{d}(x)^{\mone}aa^{\prime}
\leftrightsquigarrow
\begin{tikzpicture}[anchorbase,scale=0.4,smallnodes]
\draw[white] (0,2) to (0,1);
\draw[cstrand,marked=1] (0,0) node[below,black,yshift=-2pt]{$d$} to (0,1);
\end{tikzpicture}
\;(\text{degree }0)
,
\\
\label{eq:iota}
\iota_{d}\colon e_{x}\mapsto\sum_{y\in\mathcal{H}}\mu_{d}(y)^{\mone}\sum_{i=1}^{m_{x,y}}b^{i,x,y}\otimes b_{i,y,x}
\leftrightsquigarrow
\begin{tikzpicture}[anchorbase,scale=0.4,smallnodes]
\draw[white] (0,0) to (0,1);
\draw[cstrand,marked=1] (0,2) node[above,black]{$d$} to (0,1);
\end{tikzpicture}
\;(\text{degree } 2\mathbf{a})
.
\end{gather}
Here the source (bottom) and target (top) of the diagrams in \eqref{eq:delta},
\eqref{eq:mu}, \eqref{eq:epsilon} and \eqref{eq:iota} are composites of unshifted copies of $\Cd$.

\begin{remark}
Unlike the degrees in e.g. \cite[Conjecture 4.40]{EH}), the ones above are ``unbalanced'',
i.e. $\deg(\mu_{d})\neq\deg(\delta_{d})$ and $\deg(\iota_{d})\neq\deg(\epsilon_{d})$, because
we are working with $\mathrm{C}_{d}$ instead of $\mathrm{B}_{d}$.
\end{remark}

To avoid overloading the diagrams with shifts, we will also sometimes use a diagram of a
$2$-morphism $f\colon X\to Y$ of degree $d$ to depict the $2$-morphism $f\langle d\rangle\colon X\to Y^{\oplus\mathsf{v}^{\mone[d]}}$.
Note that the dual $\Cd^{\star}$ of $\Cd$ is  isomorphic to $\Cd^{\oplus\mathsf{v}^{\mone[2]\mathbf{a}}}$ in $\cSHo$, so we can
see e.g. $\mu_{d}\langle-2\mathbf{a}\rangle$ also as a $2$-morphism of degree zero from $\Cd\Cd^{\star}$
or $\Cd^{\star}\Cd$, which are isomorphic in $\cSHo$, to $\Cd$. The coevaluation
and evaluation of $\Cd$ in $\cSHo$ are given by
\begin{gather}\label{eq:BB-coevev}
\begin{aligned}
\mathrm{coev}_{d}
&=\delta_{d}\circ_{\mathrm{v}}\iota_{d}\langle 2\mathbf{a}\rangle\colon\mathrm{C}_e\to \Cd^{\oplus\mathsf{v}^{\mone[2]\mathbf{a}}}
\Cd,
\\
\mathrm{ev}_{d}
&=\epsilon_{d}\circ_{\mathrm{v}}\mu_{d}\langle-2\mathbf{a}\rangle\colon\Cd\Cd^{\oplus\mathsf{v}^{\mone[2]\mathbf{a}}}\to\mathrm{C}_e,
\end{aligned}
\end{gather}
which we draw as
\begin{align*}
\mathrm{coev}_{d}
\leftrightsquigarrow
\begin{tikzpicture}[anchorbase,scale=0.4,smallnodes]
\draw[white] (0,-2) to (0,0);
\draw[cstrand] (0,0) node[above,black]{$d$} to[out=270,in=180] (1,-1) to [out=0,in=270] (2,0) node[above,black]{$d$};
\end{tikzpicture}
&:=
\begin{tikzpicture}[anchorbase,scale=0.4,smallnodes]
\draw[cstrand,marked=1] (0,1) to (0,0);
\draw[cstrand] (0,1) to[out=180,in=270] (-1,2) node[above,black]{$d$};
\draw[cstrand] (0,1) to[out=0,in=270] (1,2) node[above,black]{$d$};
\end{tikzpicture}
\;(\text{degree }2\mathbf{a}),
\\
\mathrm{ev}_{d}
\leftrightsquigarrow
\begin{tikzpicture}[anchorbase,scale=0.4,smallnodes]
\draw[white] (0,2) to (0,0);
\draw[cstrand] (0,0) node[below,black,yshift=-2pt]{$d$} to[out=90,in=180] (1,1) to [out=0,in=90] (2,0) node[below,black,yshift=-2pt]{$d$};
\end{tikzpicture}
&:=
\begin{tikzpicture}[anchorbase,scale=0.4,smallnodes]
\draw[cstrand,marked=1] (0,-1) to (0,0);
\draw[cstrand] (0,-1) to[out=180,in=90] (-1,-2) node[below,black,yshift=-2pt]{$d$};
\draw[cstrand] (0,-1) to[out=0,in=90] (1,-2) node[below,black,yshift=-2pt]{$d$};
\end{tikzpicture}
\;(\text{degree }-2\mathbf{a}).
\end{align*}
Note that ${}^{\star}\Cd\cong\Cd^{\star}$ in $\cSHo$
and we could also define $\mathrm{coev}_{d}^{\prime}\colon
\mathrm{C}_e\to\Cd \Cd^{\star}$ and $\mathrm{ev}_{d}^{\prime}\colon\Cd^{\star}\Cd\to\mathrm{C}_e$ in the evident way.
Since $\Cd^{\star\star}\cong \Cd$, the diagrams for the evaluation and the coevaluations
of $\Cd$ and $\Cd^{\star}$ are the same, which justifies the lack of arrows on the diagrams.

Recall that being a Frobenius algebra in $\cSH$ means that the
(co)multiplication and the (co)unit satisfy the diagrammatic equations
\begin{gather}\label{eq:coass}
\begin{tikzpicture}[anchorbase,scale=0.4,smallnodes]
\draw[cstrand] (-.25,0) node[below,black,yshift=-2pt]{$d$} to (-.25,1);
\draw[cstrand] (-.25,1) to[out=180,in=270] (-1.5,2);
\draw[cstrand] (-.25,1) to[out=0,in=270] (1,2);
\draw[cstrand] (1,2) to (1,3);
\draw[cstrand] (1,3) to[out=180,in=270] (0,4) node[above,black]{$d$};
\draw[cstrand] (1,3) to[out=0,in=270] (2,4) node[above,black]{$d$};
\draw[cstrand] (-1.5,2) to (-1.5,4) node[above,black]{$d$};
\end{tikzpicture}
=
\begin{tikzpicture}[anchorbase,scale=0.4,smallnodes]
\draw[cstrand] (.25,0) node[below,black,yshift=-2pt]{$d$} to (.25,1);
\draw[cstrand] (.25,1) to[out=0,in=270] (1.5,2);
\draw[cstrand] (.25,1) to[out=180,in=270] (-1,2);
\draw[cstrand] (-1,2) to (-1,3);
\draw[cstrand] (-1,3) to[out=0,in=270] (0,4) node[above,black]{$d$};
\draw[cstrand] (-1,3) to[out=180,in=270] (-2,4) node[above,black]{$d$};
\draw[cstrand] (1.5,2) to (1.5,4) node[above,black]{$d$};
\end{tikzpicture}
,\quad
\begin{tikzpicture}[anchorbase,scale=0.4,smallnodes]
\draw[cstrand] (-.25,0) node[above,black]{$d$} to (-.25,-1);
\draw[cstrand] (-.25,-1) to[out=180,in=90] (-1.5,-2);
\draw[cstrand] (-.25,-1) to[out=0,in=90] (1,-2);
\draw[cstrand] (1,-2) to (1,-3);
\draw[cstrand] (1,-3) to[out=180,in=90] (0,-4) node[below,black,,yshift=-2pt]{$d$};
\draw[cstrand] (1,-3) to[out=0,in=90] (2,-4) node[below,black,yshift=-2pt]{$d$};
\draw[cstrand] (-1.5,-2) to (-1.5,-4) node[below,black,yshift=-2pt]{$d$};
\end{tikzpicture}
=
\begin{tikzpicture}[anchorbase,scale=0.4,smallnodes]
\draw[cstrand] (.25,0) node[above,black]{$d$} to (.25,-1);
\draw[cstrand] (.25,-1) to[out=0,in=90] (1.5,-2);
\draw[cstrand] (.25,-1) to[out=180,in=90] (-1,-2);
\draw[cstrand] (-1,-2) to (-1,-3);
\draw[cstrand] (-1,-3) to[out=0,in=90] (0,-4) node[below,black,yshift=-2pt]{$d$};
\draw[cstrand] (-1,-3) to[out=180,in=90] (-2,-4) node[below,black,yshift=-2pt]{$d$};
\draw[cstrand] (1.5,-2) to (1.5,-4) node[below,black,yshift=-2pt]{$d$};
\end{tikzpicture}
,\\
\label{eq:counit}
\begin{tikzpicture}[anchorbase,scale=0.4,smallnodes]
\draw[cstrand] (0,0) node[below,black,yshift=-2pt]{$d$} to (0,1);
\draw[cstrand] (0,1) to[out=180,in=270] (-1,2);
\draw[cstrand] (0,1) to[out=0,in=270] (1,2);
\draw[cstrand,marked=1] (1,2) to (1,3);
\draw[cstrand] (-1,2) to (-1,4) node[above,black]{$d$};
\end{tikzpicture}
=
\begin{tikzpicture}[anchorbase,scale=0.4,smallnodes]
\draw[cstrand] (0,0) node[below,black,yshift=-2pt]{$d$} to (0,4) node[above,black]{$d$};
\end{tikzpicture}
=
\begin{tikzpicture}[anchorbase,scale=0.4,smallnodes]
\draw[cstrand] (0,0) node[below,black,yshift=-2pt]{$d$} to (0,1);
\draw[cstrand] (0,1) to[out=180,in=270] (-1,2);
\draw[cstrand] (0,1) to[out=0,in=270] (1,2);
\draw[cstrand,marked=1] (-1,2) to (-1,3);
\draw[cstrand] (1,2) to (1,4) node[above,black]{$d$};
\end{tikzpicture}
,\quad
\begin{tikzpicture}[anchorbase,scale=0.4,smallnodes]
\draw[cstrand] (0,0) node[above,black]{$d$} to (0,-1);
\draw[cstrand] (0,-1) to[out=180,in=90] (-1,-2);
\draw[cstrand] (0,-1) to[out=0,in=90] (1,-2);
\draw[cstrand,marked=1] (1,-2) to (1,-3);
\draw[cstrand] (-1,-2) to (-1,-4) node[below,black,yshift=-2pt]{$d$};
\end{tikzpicture}
=
\begin{tikzpicture}[anchorbase,scale=0.4,smallnodes]
\draw[cstrand] (0,0) node[below,black,yshift=-2pt]{$d$} to (0,4) node[above,black]{$d$};
\end{tikzpicture}
=
\begin{tikzpicture}[anchorbase,scale=0.4,smallnodes]
\draw[cstrand] (0,0) node[above]{$d$} to (0,-1);
\draw[cstrand] (0,-1) to[out=180,in=90] (-1,-2);
\draw[cstrand] (0,-1) to[out=0,in=90] (1,-2);
\draw[cstrand,marked=1] (-1,-2) to (-1,-3);
\draw[cstrand] (1,-2) to (1,-4) node[below,black,yshift=-2pt]{$d$};
\end{tikzpicture}
,\\
\label{eq:frob}
\begin{tikzpicture}[anchorbase,scale=0.4,smallnodes]
\draw[cstrand] (0,0) node[below,black,yshift=-2pt]{$d$} to (0,3) node[above,black]{$d$};
\draw[cstrand] (2,0) node[below,black,yshift=-2pt]{$d$} to (2,3) node[above,black]{$d$};
\draw[cstrand] (0,1) to (2,2);
\end{tikzpicture}
=
\begin{tikzpicture}[anchorbase,scale=0.4,smallnodes]
\draw[cstrand] (0,0) node[below,black,yshift=-2pt]{$d$} to[out=90,in=180] (1,1);
\draw[cstrand] (2,0) node[below,black,yshift=-2pt]{$d$} to[out=90,in=0] (1,1);
\draw[cstrand] (1,1) to (1,2);
\draw[cstrand] (1,2) to[out=0,in=270] (2,3) node[above,black]{$d$};
\draw[cstrand] (1,2) to[out=180,in=270] (0,3) node[above,black]{$d$};
\end{tikzpicture}
=
\begin{tikzpicture}[anchorbase,scale=0.4,smallnodes]
\draw[cstrand] (0,0) node[below,black,yshift=-2pt]{$d$} to (0,3) node[above,black]{$d$};
\draw[cstrand] (2,0) node[below,black,yshift=-2pt]{$d$} to (2,3) node[above,black]{$d$};
\draw[cstrand] (2,1) to (0,2);
\end{tikzpicture}
.
\end{gather}
Recall further that \eqref{eq:coass}, \eqref{eq:counit} and \eqref{eq:frob} imply
\begin{gather}\label{eq:isotopies}
\begin{gathered}
\begin{tikzpicture}[anchorbase,scale=0.4,smallnodes]
\draw[cstrand] (0,0) node[below,black,yshift=-2pt]{$d$} to (0,2) to [out=90,in=180] (1,3) to [out=0,in=90] (2,2) to [out=270,in=180] (3,1) to[out=0,in=270] (4,2) to (4,4) node[above,black]{$d$};
\end{tikzpicture}
=
\begin{tikzpicture}[anchorbase,scale=0.4,smallnodes]
\draw[cstrand] (0,0) node[below,black,yshift=-2pt]{$d$} to (0,4) node[above,black]{$d$};
\end{tikzpicture}
=
\begin{tikzpicture}[anchorbase,scale=0.4,smallnodes]
\draw[cstrand] (0,0) node[below,black,yshift=-2pt]{$d$} to (0,2) to [out=90,in=0] (-1,3) to [out=180,in=90] (-2,2) to [out=270,in=0] (-3,1) to[out=180,in=270] (-4,2) to (-4,4) node[above,black]{$d$};
\end{tikzpicture}
,\\
\begin{tikzpicture}[anchorbase,scale=0.4,smallnodes]
\draw[cstrand] (0,0) node[below,black,yshift=-2pt]{$d$} to (0,1);
\draw[cstrand] (0,1) to[out=180,in=270] (-1,2) to (-1,3) node[above,black]{$d$};
\draw[cstrand] (0,1) to[out=0,in=270] (1,2) to[out=90,in=180] (2,3) to[out=0,in=90] (3,2) to (3,0) node[below,black,yshift=-2pt]{$d$};
\end{tikzpicture}
{=}
\begin{tikzpicture}[anchorbase,scale=0.4,smallnodes]
\draw[cstrand] (0,.5) node[above,black]{$d$} to (0,0) to (0,-1);
\draw[cstrand] (0,-1) to[out=180,in=90] (-1,-2) to (-1,-2.5) node[below,black,yshift=-2pt]{$d$};
\draw[cstrand] (0,-1) to[out=0,in=90] (1,-2) to (1,-2.5) node[below,black,yshift=-2pt]{$d$};
\end{tikzpicture}
{=}
\begin{tikzpicture}[anchorbase,scale=0.4,smallnodes]
\draw[cstrand] (0,0) node[below,black,yshift=-2pt]{$d$} to (0,1);
\draw[cstrand] (0,1) to[out=0,in=270] (1,2) to (1,3) node[above,black]{$d$};
\draw[cstrand] (0,1) to[out=180,in=270] (-1,2) to[out=90,in=0] (-2,3) to[out=180,in=90] (-3,2) to (-3,0) node[below,black,yshift=-2pt]{$d$};
\end{tikzpicture}
,
\begin{tikzpicture}[anchorbase,scale=0.4,smallnodes]
\draw[cstrand] (0,0) node[above,black]{$d$} to (0,-1);
\draw[cstrand] (0,-1) to[out=180,in=90] (-1,-2) to (-1,-3) node[below,black,yshift=-2pt]{$d$};
\draw[cstrand] (0,-1) to[out=0,in=90] (1,-2) to[out=270,in=180] (2,-3) to[out=0,in=270] (3,-2) to (3,0) node[above,black]{$d$};
\end{tikzpicture}
{=}
\begin{tikzpicture}[anchorbase,scale=0.4,smallnodes]
\draw[cstrand] (0,-.5) node[below,black,yshift=-2pt]{$d$} to (0,0) to (0,1);
\draw[cstrand] (0,1) to[out=180,in=270] (-1,2) to (-1,2.5) node[above,black]{$d$};
\draw[cstrand] (0,1) to[out=0,in=270] (1,2) to (1,2.5) node[above,black]{$d$};
\end{tikzpicture}
{=}
\begin{tikzpicture}[anchorbase,scale=0.4,smallnodes]
\draw[cstrand] (0,0) node[above,black]{$d$} to (0,-1);
\draw[cstrand] (0,-1) to[out=0,in=90] (1,-2) to (1,-3) node[below,black,yshift=-2pt]{$d$};
\draw[cstrand] (0,-1) to[out=180,in=90] (-1,-2) to[out=270,in=0] (-2,-3) to[out=180,in=270] (-3,-2) to (-3,0) node[above,black]{$d$};
\end{tikzpicture}
,\\
\begin{tikzpicture}[anchorbase,scale=0.4,smallnodes]
\draw[white] (0,-2) to (0,.5);
\draw[cstrand,marked=1] (0,.5) node[above,black]{$d$} to (0,0)  to[out=270,in=180] (1,-1) to[out=0,in=270] (2,0) to (2,.5);
\end{tikzpicture}
=
\begin{tikzpicture}[anchorbase,scale=0.4,smallnodes]
\draw[white] (0,-2) to (0,.5);
\draw[cstrand,marked=1] (0,.5) node[above,black]{$d$} to (0,-1);
\end{tikzpicture}
=
\begin{tikzpicture}[anchorbase,scale=0.4,smallnodes]
\draw[white] (0,-2) to (0,.5);
\draw[cstrand,marked=1] (0,.5) node[above,black]{$d$} to (0,0)  to[out=270,in=0] (-1,-1) to[out=180,in=270] (-2,0) to (-2,.5);
\end{tikzpicture}
,\quad
\begin{tikzpicture}[anchorbase,scale=0.4,smallnodes]
\draw[white] (0,2) to (0,-.5);
\draw[cstrand,marked=1] (0,-.5) node[below,black,yshift=-2pt]{$d$} to (0,0)  to[out=90,in=180] (1,1) to[out=0,in=90] (2,0) to (2,-.5);
\end{tikzpicture}
=
\begin{tikzpicture}[anchorbase,scale=0.4,smallnodes]
\draw[white] (0,2) to (0,-.5);
\draw[cstrand,marked=1] (0,-.5) node[below,black,yshift=-2pt]{$d$} to (0,1);
\end{tikzpicture}
=
\begin{tikzpicture}[anchorbase,scale=0.4,smallnodes]
\draw[white] (0,2) to (0,-.5);
\draw[cstrand,marked=1] (0,-.5) node[below,black,yshift=-2pt]{$d$} to (0,0)  to[out=90,in=0] (-1,1) to[out=180,in=90] (-2,0) to (-2,-.5);
\end{tikzpicture}
.
\end{gathered}
\end{gather}
By \eqref{eq:isotopies} we only need to consider isotopy classes
of diagrams, and we simplify our diagrams by drawing them in a more topological fashion.
However, some diagrams are not the same, as the
calculus is not necessarily cyclic for general $\mathsf{B}$-$\mathsf{B}$-bimodule maps. In particular, for any
\begin{gather}\label{eq:alpha}
\alpha\in
\mathrm{hom}_{\ccS_{\mathcal{H}}}
(\Cd\dots\Cd,
\Cd\dots\Cd^{\oplus\mathsf{v}^{\mone[t]+s}})
\leftrightsquigarrow
\begin{tikzpicture}[anchorbase,scale=0.4,smallnodes]
\draw[cstrand] (0,0) node[below,black,yshift=-2pt]{$d\dots d$} to (0,1) node[above,black,box]{$\alpha$};
\draw[cstrand] (0,2) to (0,3) node[above,black]{$d\dots d$};
\end{tikzpicture}
\;(\text{degree }t-s)
,
\end{gather}
where the dots indicate multiple strands,
we define the \emph{right mate} $\alpha^{\star}$ and the \emph{left mate} ${}^{\star}\alpha$ of $\alpha$
via
\begin{gather*}
\begin{tikzpicture}[anchorbase,scale=0.4,smallnodes]
\draw[cstrand] (0,0) node[below,black,yshift=-2pt]{$d\dots d$} to (0,1) node[above,black,box]{$\alpha$};
\draw[cstrand] (0,2) node[right,black,xshift=2pt]{$\star$} to (0,3) node[above,black]{$d\dots d$};
\end{tikzpicture}
:=
\begin{tikzpicture}[anchorbase,scale=0.4,smallnodes]
\draw[cstrand] (-1,3) node[above,black]{$d\dots d$} to (-1,1) to [out=270,in=180] (0,0) to [out=0,in=270] (1,1) node[above,black,box]{$\alpha$};
\draw[cstrand] (1,2) to [out=90,in=180] (2,3) to [out=0,in=90] (3,2) to (3,0) node[below,black,yshift=-2pt]{$d\dots d$};
\end{tikzpicture}
,\quad
\begin{tikzpicture}[anchorbase,scale=0.4,smallnodes]
\draw[cstrand] (0,0) node[below,black,yshift=-2pt]{$d\dots d$} to (0,1) node[above,black,box]{$\alpha$};
\draw[cstrand] (0,2) node[left,black,xshift=-2pt]{$\star$} to (0,3) node[above,black]{$d\dots d$};
\end{tikzpicture}
:=
\begin{tikzpicture}[anchorbase,scale=0.4,smallnodes]
\draw[cstrand] (1,3) node[above,black]{$d\dots d$} to (1,1) to [out=270,in=0] (0,0) to [out=180,in=270] (-1,1) node[above,black,box]{$\alpha$};
\draw[cstrand] (-1,2) to [out=90,in=0] (-2,3) to [out=180,in=90] (-3,2) to (-3,0) node[below,black,yshift=-2pt]{$d\dots d$};
\end{tikzpicture}
,
\end{gather*}
which have the same degree as the diagram in \eqref{eq:alpha} itself.
To be more concrete, let us compute
the mates of $\alpha\in\mathrm{hom}_{\ccS_{\mathcal{H}}}
(\Cd,\Cd^{\oplus\mathsf{v}^{\mone[t]+s}})$
in the cell $2$-representation. Suppose that the $\mathsf{B}$-$\mathsf{B}$-bimodule map
$\alpha$ is given by
\begin{gather}\label{eq:alpha-bimod}
\alpha\colon
e_{x}\otimes e_{x}\mapsto\sum_{y\in\mathcal{H}}\sum_{i=1}^{n_{x,y}}\alpha^{(1)}_{i,x,y}\otimes\alpha^{(2)}_{i,y,x},
\end{gather}
for $n_{x,y}\in\mathbb{N}_0$ and some homogeneous
$\alpha^{(1)}_{i,x,y}\in e_{x}\mathsf{B}e_{y}$ and $\alpha^{(2)}_{i,y,x}\in
e_{y}\mathsf{B}e_{x}$ such that $\deg(\alpha^{(1)}_{i,x,y})+\deg(\alpha^{(2)}_{i,y,x})=t-s$.
By \eqref{eq:BB-coevev}, the
$\mathsf{B}$-$\mathsf{B}$-bimodule maps corresponding to
$\mathrm{coev}_{d}\langle -2\mathbf{a}\rangle$ and $\mathrm{ev}_{d}\langle 2\mathbf{a}\rangle$ are given by
\begin{gather*}
\delta_{d}\circ_{\mathrm{v}}\iota_{d}\colon e_{x}\mapsto\sum_{y\in\mathcal{H}} \sum_{i=1}^{m_{x,y}}
\mu_{d}(y)^{\mone}\delta_{d}(y) b^{i,x,y}\otimes e_{y}\otimes b_{i,y,x}
,\\
\epsilon_{d}\circ_{\mathrm{v}}\mu_{d}\colon e_{x}\otimes a\otimes e_{y}\mapsto\delta_{x,y}\mu_{d}(x)\delta_{d}(x)^{\mone}\mathrm{tr}(a)e_{x},
\end{gather*}
for $x,y\in\mathcal{H}$. Then $\alpha^{\star}$ and ${}^{\star}\alpha$ are given by
\begin{gather*}
\alpha^{\star}\colon
e_{x}\otimes e_{x}\mapsto
\sum_{y\in\mathcal{H}}\sum_{i=1}^{n_{y,x}}\mu_{d}(x)\delta_{d}(x)^{\mone}\mu_{d}(y)^{\mone}\delta_{d}(y)\sigma(\alpha^{(2)}_{i,x,y})\otimes\alpha^{(1)}_{i,y,x},
\\
{}^{\star}\alpha\colon
e_{x}\otimes e_{x}\mapsto\sum_{y\in\mathcal{H}}\sum_{i=1}^{n_{y,x}}\mu_{d}(x)\delta_{d}(x)^{\mone}\mu_{d}(y)^{\mone}\delta_{d}(y)\alpha^{(2)}_{i,x,y}\otimes\sigma^{\mone}(\alpha^{(1)}_{i,y,x}),
\end{gather*}
where $\sigma$ is the Nakayama automorphism. In other words, the $\mathsf{B}$-$\mathsf{B}$-bimodule map associated to $\alpha^{\star}$
can be obtained from the one associated to ${}^{\star}\alpha$ by applying
$\sigma\otimes \sigma$. To indicate this diagrammatically, we write
\begin{gather}\label{eq:nakayama}
\begin{tikzpicture}[anchorbase,scale=0.4,smallnodes]
\draw[cstrand] (0,0) node[below,black,yshift=-2pt]{$d$} to (0,1) node[above,black,box]{$\alpha$};
\draw[cstrand] (0,2) node[right,black,xshift=2pt]{$\star$} to (0,3) node[above,black]{$d$};
\end{tikzpicture}
=
\sigma
\Bigg(
\begin{tikzpicture}[anchorbase,scale=0.4,smallnodes]
\draw[cstrand] (0,0) node[below,black,yshift=-2pt]{$d$} to (0,1) node[above,black,box]{$\alpha$};
\draw[cstrand] (0,2) node[left,black,xshift=-2pt]{$\star$} to (0,3) node[above,black]{$d$};
\end{tikzpicture}
\Bigg)
.
\end{gather}
Note that the equations in \eqref{eq:isotopies} imply that the left and
right mates of the structural $2$-morphisms of $\Cd$ agree, i.e. $\mu_{d}=\delta_{d}^{\star}={}^{\star}\delta_{d}$ and
$\iota_{d}=\epsilon_{d}^{\star}={}^{\star}\epsilon_{d}$.


\subsection{Dot diagrams}\label{subsection:dot diagrams}


Next, recall that $\mathbf{C}_{\mathcal{H}}\simeq
\mathbf{inj}_{\underline{\ccS_{\mathcal{H}}}}(\Cd)$.
For short, we denote the morphism spaces in the latter category by
$\mathrm{Hom}_{\Cd}$ (for the enriched version) and $\mathrm{hom}_{\Cd}$ (for morphism spaces of degree zero).
A specific object in this category is $\Cd$, since
the coalgebra structure of $\Cd$ gives it the structure of an injective right comodule in $\cSH$ over itself.

\begin{lemma}\label{lemma:BB-intertwiners}
Let $\alpha$ be a homogeneous endomorphism of $\Cd$ in $\cSH$ of
degree $t$, for some $t\in\mathbb{Z}$. Then
we have $\alpha\in\mathrm{hom}_{\Cd}
(\Cd,\Cd^{\oplus \mathsf{v}^{\mone[t]}})$ if and only if the
corresponding $\mathsf{B}$-$\mathsf{B}$-bimodule map is of the form
\begin{gather*}
e_{x}\otimes e_{x}\mapsto
\alpha_{x}\otimes e_{x},\quad\forall x\in\mathcal{H},
\end{gather*}
where $\alpha_{x}$ is some homogeneous element in $e_{x}\mathsf{B}e_{x}$ of degree
$t$.
\end{lemma}

\begin{proof}
Let $x\in\mathcal{H}$ and suppose that the $\mathsf{B}$-$\mathsf{B}$-bimodule map associated to $\alpha$ is given by \eqref{eq:alpha-bimod}.
Now compare the $\mathsf{B}$-$\mathsf{B}$-bimodule maps associated to $\delta_{d}\circ_{\mathrm{v}}\alpha$
and $(\alpha\circ_{\mathrm{h}}\mathrm{id}_{\Cd})\circ_{\mathrm{v}}\delta_{d}$, respectively
\begin{gather*}
e_{x}\otimes e_{x}\mapsto\sum_{y\in\mathcal{H}}\sum_{i=1}^{n_{x,y}}\delta_{d}(y)\alpha_{i,x,y}^{(1)}
\otimes e_{y}\otimes\alpha_{i,y,x}^{(2)},
\\
e_{x}\otimes e_{x}\mapsto\sum_{y\in\mathcal{H}}\sum_{i=1}^{n_{x,y}}\delta_{d}(x)\alpha_{i,x,y}^{(1)}\otimes\alpha_{i,y,x}^{(2)}\otimes e_{x}.
\end{gather*}
By $2$-faithfulness of the $2$-functor $\cSH\to\cC_{\mathsf{B}}$ defining the cell
$2$-representation (see the beginning of this section), we see that the
equality $\delta_{d}\circ_{\mathrm{v}}\alpha=(\alpha\circ_{\mathrm{h}}\mathrm{id}_{\Cd})\circ_{\mathrm{v}}\delta_{d}$
holds if and only if
\begin{gather*}
\alpha_{i,y,x}^{(2)}\in
\begin{cases}
\mathbb{C}e_{x}
&\text{if }x=y;
\\
\{0\}
&\text{otherwise},
\end{cases}
\end{gather*}
which proves the lemma.
\end{proof}

Before the next statement, which is a version of separability, recall the following
special cases of Soergel's hom formula \eqref{eq:soergel}:
\begin{gather}\label{eq:soergel-hom}
\mathrm{hom}_{\Cd}(\Cd,\Cd)\cong\mathbb{C}\cong\mathrm{hom}_{\Cd}(\Cd,\Cd^{\oplus\mathsf{v}^{\mone[2]\mathbf{a}}}).
\end{gather}

\begin{proposition}\label{proposition:dot diagram}
There is a unique $2$-morphism in $\mathrm{hom}_{\Cd}(\Cd,\Cd^{\oplus\mathsf{v}^{\mone[2]\mathbf{a}}})$,
which we call \emph{the dot diagram} and which we depict as
\begin{gather*}
\text{the dot diagram: }
\begin{tikzpicture}[anchorbase,scale=0.4,smallnodes]
\draw[cstrand] (0,0) node[below,black,yshift=-2pt]{$d$} to (0,1) node[above,black,box]{\raisebox{-.025cm}{$\bullet$}};
\draw[cstrand] (0,2) to (0,3) node[above,black]{$d$};
\end{tikzpicture}
\;\;(\text{degree } 2\mathbf{a}),
\end{gather*}
such that
\begin{gather}\label{eq:dotdiagram}
\begin{tikzpicture}[anchorbase,scale=0.4,smallnodes]
\draw[cstrand] (0,0) node[below,black,yshift=-2pt]{$d$} to (0,1);
\draw[cstrand] (0,1) to [out=0,in=270] (1,2) node[above,black,box]{\raisebox{-.025cm}{$\bullet$}};
\draw[cstrand] (0,1) to [out=180,in=270] (-1,2.5) to [out=90,in=180] (0,4);
\draw[cstrand] (1,3) to [out=90,in=0] (0,4);
\draw[cstrand] (0,4) to (0,5) node[above,black]{$d$};
\end{tikzpicture}
\;=
\begin{tikzpicture}[anchorbase,scale=0.4,smallnodes]
\draw[cstrand] (0,0) node[below,black,yshift=-2pt]{$d$} to (0,5) node[above,black]{$d$};
\end{tikzpicture}.
\end{gather}
\end{proposition}

\begin{proof}
Choose any non-zero $2$-morphism $\alpha\in\mathrm{hom}_{\Cd}(\Cd, \Cd^{\oplus\mathsf{v}^{\mone[2]\mathbf{a}}})$,
which we depict as in \eqref{eq:alpha}.
By \eqref{eq:soergel-hom}, we have
\begin{gather}\label{eq:alpha-diagram}
\begin{tikzpicture}[anchorbase,scale=0.4,smallnodes]
\draw[cstrand] (0,0) node[below,black,yshift=-2pt]{$d$} to (0,1);
\draw[cstrand] (0,1) to [out=0,in=270] (1,2) node[above,black,box]{$\alpha$};
\draw[cstrand] (0,1) to [out=180,in=270] (-1,2.5) to [out=90,in=180] (0,4);
\draw[cstrand] (1,3) to [out=90,in=0] (0,4);
\draw[cstrand] (0,4) to (0,5) node[above,black]{$d$};
\end{tikzpicture}
=\lambda\cdot
\begin{tikzpicture}[anchorbase,scale=0.4,smallnodes]
\draw[cstrand] (0,0) node[below,black,yshift=-2pt]{$d$} to (0,5) node[above,black]{$d$};
\end{tikzpicture}
\end{gather}
for some scalar $\lambda\in\mathbb{C}$.
We claim that $\lambda\neq 0$.

To prove this claim, we only need to compute
the $\mathsf{B}$-$\mathsf{B}$-bimodule map corresponding to the morphism on the left-hand side
in \eqref{eq:alpha-diagram}. By Lemma \ref{lemma:BB-intertwiners}, we have
\begin{gather*}
\begin{tikzpicture}[anchorbase,scale=0.4,smallnodes]
\draw[cstrand] (0,0) node[below,black,yshift=-2pt]{$d$} to (0,1) node[above,black,box]{$\alpha$};
\draw[cstrand] (0,2) to (0,3) node[above,black]{$d$};
\end{tikzpicture}
\rightsquigarrow
e_{x}\otimes e_{x}\mapsto\alpha_{x}\otimes e_{x},
\end{gather*}
where $\alpha_{x}=\alpha(x)e^x$ for some scalar $\alpha(x)\in\mathbb{C}$ depending on $x\in\mathcal{H}$.
Hence,
\eqref{eq:delta} and \eqref{eq:mu} imply that
\begin{gather*}
\begin{tikzpicture}[anchorbase,scale=0.4,smallnodes]
\draw[cstrand] (0,0) node[below,black,yshift=-2pt]{$d$} to (0,1);
\draw[cstrand] (0,1) to [out=0,in=270] (1,2) node[above,black,box]{$\alpha$};
\draw[cstrand] (0,1) to [out=180,in=270] (-1,2.5) to [out=90,in=180] (0,4);
\draw[cstrand] (1,3) to [out=90,in=0] (0,4);
\draw[cstrand] (0,4) to (0,5) node[above,black]{$d$};
\end{tikzpicture}
\rightsquigarrow
e_{x}\otimes e_{x}\mapsto\delta_{d}(x)\mu_{d}(x) \alpha(x)  e_{x}\otimes e_{x}.
\end{gather*}
We see that $\lambda=\delta_{d}(x)\mu_{d}(x)\alpha(x)$
for all $x\in\mathcal{H}$. In particular, the scalar $\delta_{d}(x)\mu_{d}(x)\alpha(x)$
does not depend on the choice of $x$. Since $\alpha$ is not zero,
there exists an $x\in\mathcal{H}$ such that $\alpha(x)\neq 0$, which implies that
$\lambda\neq 0$ because $\delta_{d}(x)\mu_{d}(x)\neq 0$ for all $x\in\mathcal{H}$.
Note that this also means that $\alpha(x)\neq 0$ for all $x\in\mathcal{H}$.

Now define
\begin{gather*}
\begin{tikzpicture}[anchorbase,scale=0.4,smallnodes]
\draw[cstrand] (0,0) node[below,black,yshift=-2pt]{$d$} to (0,1) node[above,black,box]{\raisebox{-.025cm}{$\bullet$}};
\draw[cstrand] (0,2) to (0,3) node[above,black]{$d$};
\end{tikzpicture}
:= \lambda^{\mone}\cdot
\begin{tikzpicture}[anchorbase,scale=0.4,smallnodes]
\draw[cstrand] (0,0) node[below,black,yshift=-2pt]{$d$} to (0,1) node[above,black,box]{$\alpha$};
\draw[cstrand] (0,2) to (0,3) node[above,black]{$d$};
\end{tikzpicture}
\,.
\end{gather*}
Clearly, this is the unique $2$-morphism in
$\mathrm{hom}_{\Cd}(\Cd,\Cd^{\oplus\mathsf{v}^{\mone[2]\mathbf{a}}})\cong\mathbb{C}$ satisfying \eqref{eq:dotdiagram}.
\end{proof}

\begin{remark}\label{remark:soergel-calculus}
One might be tempted to use the diagrammatic calculus to define the dot diagram
as a multiple of
\begin{gather*}
\begin{tikzpicture}[anchorbase,scale=0.4,smallnodes]
\draw[cstrand,marked=1, marked=-1] (-.5,1) to (-.5,1.5) node[left,black]{$d$} to (-.5,2);
\draw[cstrand] (0,0) node[below,black,yshift=-2pt]{$d$} to (0,3) node[above, black]{$d$};
\end{tikzpicture},
\end{gather*}
but we do not know whether this $2$-morphism in
$\mathrm{hom}_{\Cd}(\Cd,\Cd^{\oplus\mathsf{v}^{\mone[2]\mathbf{a}}})$ is non-zero in
general.
\end{remark}

For later use, note that
\begin{gather}\label{eq:dotdiagram-image}
\begin{tikzpicture}[anchorbase,scale=0.4,smallnodes]
\draw[cstrand] (0,0) node[below,black,yshift=-2pt]{$d$} to (0,1) node[above,black,box]{\raisebox{-.025cm}{$\bullet$}};
\draw[cstrand] (0,2) to (0,3) node[above,black]{$d$};
\end{tikzpicture}
\rightsquigarrow
e_{x}\otimes e_{x} \mapsto\nu_{d}(x)^{\mone}e^{x}\otimes e_{x},
\end{gather}
for all $x\in\mathcal{H}$, where $\nu_{d}(x):=\delta_{d}(x)\mu_{d}(x)\in\mathbb{C}^{\times}$.
Note also that the fact that the dot diagram is a homomorphism of right $\Cd$-comodules in
$\cSH$ translates to
\begin{gather}\label{eq:dotdiagram-slide}
\begin{tikzpicture}[anchorbase,scale=0.4,smallnodes]
\draw[cstrand] (0,0) node[below,black,yshift=-2pt]{$d$} to (0,1) node[above,black,box]{\raisebox{-.025cm}{$\bullet$}};
\draw[cstrand] (0,2) to (0,3) node[above,black]{$d$};
\draw[cstrand] (0,0.5) to [out=0,in=270] (1,3) node[above,black]{$d$};
\end{tikzpicture}
=
\begin{tikzpicture}[anchorbase,scale=0.4,smallnodes]
\draw[cstrand] (0,0) node[below,black,yshift=-2pt]{$d$} to (0,1) node[above,black,box]{\raisebox{-.025cm}{$\bullet$}};
\draw[cstrand] (0,2) to (0,3) node[above,black]{$d$};
\draw[cstrand] (0,2.5) to [out=0,in=270] (1,3) node[above,black]{$d$};
\end{tikzpicture}
.
\end{gather}

Since $\sigma$ fixes $e_{x}$ and $e^x$, the equality in \eqref{eq:nakayama} and
the $\mathsf{B}$-$\mathsf{B}$-bimodule map in
\eqref{eq:dotdiagram-image} show that the right and left mates of the dot diagram agree, i.e.
\begin{gather*}
\begin{tikzpicture}[anchorbase,scale=0.4,smallnodes]
\tikzstyle{box} = [minimum height=0.4cm,draw, rounded corners, draw,rectangle]
\draw[cstrand] (0,0) node[below,black,yshift=-2pt]{$d$} to (0,1) node[above,black,box]{\raisebox{-.025cm}{$\bullet$}};
\draw[cstrand] (0,2) node[right,black,xshift=2pt]{$\star$} to (0,3) node[above,black]{$d$};
\end{tikzpicture}
=
\begin{tikzpicture}[anchorbase,scale=0.4,smallnodes]
\tikzstyle{box} = [minimum height=0.4cm,draw, rounded corners, draw,rectangle]
\draw[cstrand] (0,0) node[below,black,yshift=-2pt]{$d$} to (0,1) node[above,black,box]{\raisebox{-.025cm}{$\bullet$}};
\draw[cstrand] (0,2) node[left,black,xshift=-2pt]{$\star$} to (0,3) node[above,black]{$d$};
\end{tikzpicture}
\,.
\end{gather*}
Either of these diagrams we therefore call the \emph{dual dot diagram}, and the following
is an almost immediate consequence of Proposition \ref{proposition:dot diagram},
which can be proved diagrammatically using \eqref{eq:isotopies}
and \eqref{eq:dotdiagram-slide}.

\begin{proposition}\label{proposition:dotdiagram-left}
The dual dot diagram is a homomorphism of left $\Cd$-comodules in $\cSH$ and
satisfies
\begin{gather}\label{eq:dualdotdiagram}
\begin{tikzpicture}[anchorbase,scale=0.4,smallnodes]
\draw[cstrand] (0,0) node[below,black,yshift=-2pt]{$d$} to (0,1);
\draw[cstrand] (0,1) to [out=180,in=270] (-1,2) node[above,black,box]{\raisebox{-.025cm}{$\bullet$}};
\draw[cstrand] (0,1) to [out=0,in=270] (1,2.5) to [out=90,in=0] (0,4);
\draw[cstrand] (-1,3) node[left,black,xshift=-2pt]{$\star$} to [out=90,in=180] (0,4);
\draw[cstrand] (0,4) to (0,5) node[above,black]{$d$};
\end{tikzpicture}
\;=
\begin{tikzpicture}[anchorbase,scale=0.4,smallnodes]
\draw[cstrand] (0,0) node[below,black,yshift=-2pt]{$d$} to (0,5) node[above,black]{$d$};
\end{tikzpicture}
.
\end{gather}
\end{proposition}

Diagrammatically, this means that the dual dot diagram satisfies a
sliding relation analogous to \eqref{eq:dotdiagram-slide}. Furthermore,
as in \eqref{eq:dotdiagram-image}, the image of the dual dot diagram in $\cC_{\mathsf{B}}$
is given by
\begin{gather*}
\begin{tikzpicture}[anchorbase,scale=0.4,smallnodes]
\draw[cstrand] (0,0) node[below,black,yshift=-2pt]{$d$} to (0,1) node[above,black,box]{\raisebox{-.025cm}{$\bullet$}};
\draw[cstrand] (0,2) node[right,black,xshift=2pt]{$\star$} to (0,3) node[above,black]{$d$};
\end{tikzpicture}
\rightsquigarrow
e_{x}\otimes e_{x}\mapsto\nu_{d}(x)^{\mone}e_{x}\otimes e^x
\end{gather*}
for all $x\in\mathcal{H}$.

A straightforward calculation in $\cC_{\mathsf{B}}$ proves the following
diagrammatic equations, which will be very useful in the following.

\begin{lemma}\label{lemma:idempotent}
We have
\begin{gather}\label{eq:idempotent}
\begin{tikzpicture}[anchorbase,scale=0.4,smallnodes]
\draw[cstrand] (0,0) to (0,1);
\draw[cstrand] (0,1) to [out=0,in=270] (1,2)  node[above,black,box]{\raisebox{-.025cm}{$\bullet$}};
\draw[cstrand] (1,3) to (1,4) node[above,black]{$d$};
\draw[cstrand] (0,1) to [out=180,in=270] (-1,4) node[above,black]{$d$};
\draw[cstrand] (-1,-3) node[below,black,yshift=-2pt]{$d$} to [out=90,in=180] (0,0);
\draw[cstrand] (1,-3) node[below,black,yshift=-2pt]{$d$} to (1,-2) node[above,black,box]{\raisebox{-.025cm}{$\bullet$}};
\draw[cstrand] (1,-1) to [out=90,in=0] (0,0);
\end{tikzpicture}
=
\begin{tikzpicture}[anchorbase,scale=0.4,smallnodes]
\draw[cstrand] (-1,-3) node[below,black,yshift=-2pt]{$d$} to (-1,4) node[above,black]{$d$};
\draw[cstrand] (1,-3) node[below,black,yshift=-2pt]{$d$} to (1,0) node[above,black,box]{\raisebox{-.025cm}{$\bullet$}};
\draw[cstrand] (1,1) to (1,4) node[above,black]{$d$};
\end{tikzpicture}
\,,\quad
\begin{tikzpicture}[anchorbase,scale=0.4,smallnodes]
\draw[cstrand] (0,0) to (0,1);
\draw[cstrand] (0,1) to [out=180,in=270] (-1,2)  node[above,black,box]{\raisebox{-.025cm}{$\bullet$}};
\draw[cstrand] (-1,3) node[left,black,xshift=-2pt]{$\star$} to (-1,4) node[above,black]{$d$};
\draw[cstrand] (0,1) to [out=0,in=270] (1,4) node[above,black]{$d$};
\draw[cstrand] (1,-3) node[below,black,yshift=-2pt]{$d$} to [out=90,in=0] (0,0);
\draw[cstrand] (-1,-3) node[below,black,yshift=-2pt]{$d$} to (-1,-2) node[above,black,box]{\raisebox{-.025cm}{$\bullet$}};
\draw[cstrand] (-1,-1) node[left,black,xshift=-2pt]{$\star$} to [out=90,in=180] (0,0);
\end{tikzpicture}
=
\begin{tikzpicture}[anchorbase,scale=0.4,smallnodes]
\draw[cstrand] (1,-3) node[below,black,yshift=-2pt]{$d$} to (1,4) node[above,black]{$d$};
\draw[cstrand] (-1,-3) node[below,black,yshift=-2pt]{$d$} to (-1,0) node[above,black,box]{\raisebox{-.025cm}{$\bullet$}};
\draw[cstrand] (-1,1) node[left,black,xshift=-2pt]{$\star$} to (-1,4) node[above,black]{$d$};
\end{tikzpicture}
.
\end{gather}
\end{lemma}


\subsection{Graded semisimple bicomodules}\label{subsection:bh-semisimple}


We are now ready to determine the structure of $\cBH$, in particular, that it is locally graded semisimple.

\begin{lemma}\label{lemma:unique-bicomodule}
The $1$-morphism $\mathrm{C}_{x}^{\oplus\mathsf{v}^{t}}$ has a unique
$\Cd$-$\Cd$-bicomodule structure in $\cSHo$, for every $x\in\mathcal{H}$ and $t\in\mathbb{Z}$
\end{lemma}

\begin{proof}
For every $x\in\mathcal{H}$, we have $\mathrm{C}_{x}\cong \Theta(\mathrm{A}_{x})$, see \eqref{Thetadef} and the text below it. Since $\mathrm{A}_{d}$ is the identity $1$-morphism in $\cAH$, the $1$-morphism $\mathrm{A}_{x}$ is naturally an
$\mathrm{A}_{d}$-$\mathrm{A}_{d}$-bicomodule.
This, together with the fact that $\Theta$ is an oplax pseudofunctor, implies that $\mathrm{C}_{x}$ is
a $\Cd$-$\Cd$-bicomodule in $\cSHo$.

Of course, this implies that
$\mathrm{C}_{x}^{\oplus\mathsf{v}^{t}}$ is a $\Cd$-$\Cd$-bicomodule in $\cSHo$ for all
$t\in\mathbb{Z}$, with the same (degree zero) structure $2$-morphisms.
We denote the left and right
$\Cd$-coactions on $\mathrm{C}_{x}^{\oplus\mathsf{v}^{t}}$ by
$\delta_{\Cd,\mathrm{C}_{x}^{\oplus\mathsf{v}^{t}}}$ and $\delta_{\mathrm{C}_{x}^{\oplus\mathsf{v}^{t}},\Cd}$,
respectively. Note that the bicomodule structure on each $\mathrm{C}_{x}^{\oplus\mathsf{v}^{t}}$ in  $\cSHo$ is unique
because Soergel's hom formula \eqref{eq:soergel} gives
\begin{gather*}
\dim\big(\mathrm{hom}_{\ccS_{\mathcal{H}}}
(\mathrm{C}_{x}^{\oplus\mathsf{v}^{t}}, \mathrm{C}_{x}^{\oplus\mathsf{v}^{t}}\Cd)\big)
=1=
\dim\big(\mathrm{hom}_{\ccS_{\mathcal{H}}}
(\mathrm{C}_{x}^{\oplus\mathsf{v}^{t}},\Cd
\mathrm{C}_{x}^{\oplus\mathsf{v}^{t}})\big),
\end{gather*}
and the choice of $\delta_{\Cd,\mathrm{C}_{x}^{\oplus\mathsf{v}^{t}}}$ and
$\delta_{\mathrm{C}_{x}^{\oplus\mathsf{v}^{t}},\Cd}$ is fixed by counitality.
\end{proof}

\begin{lemma}\label{lemma:BB-intertwiners-2}
For any $x,y\in\mathcal{H}$ and $t\in\mathbb{Z}$, we have
\begin{gather*}
\dim\big(\mathrm{hom}_{\ccBH}
(\mathrm{C}_{x},\mathrm{C}_{y}^{\oplus\mathsf{v}^{t}})\big)
=\delta_{x,y}\delta_{t,0}.
\end{gather*}
\end{lemma}

\begin{proof}
By arguments analogous to those in the proof of Lemma \ref{lemma:BB-intertwiners}, but this time
applied to the left and the right
$\Cd$-coactions,
we see that
\begin{gather*}
\dim\big(\mathrm{hom}_{\ccBH}
(\mathrm{C}_{x},\mathrm{C}_{y}^{\oplus\mathsf{v}^{t}})\big)
=0\;\,\text{unless $t= 0$.}
\end{gather*}
The result for $t=0$ follows from \eqref{eq:soergel}.
\end{proof}

Since $\Cd$ is a graded Frobenius algebra in $\cSH$ (c.f. Subsection \ref{s7.5}),
any right, left or bicomodule of $\Cd$ has a compatible right, left or bimodule
structure over $\Cd$ up to a degree shift of $-2\mathbf{a}$. This is well-known, but
for convenience and later use, we briefly repeat the main arguments and constructions. The shortest way to do
this is diagrammatically. Let $\mathrm{X}$ be a right $\Cd$-comodule in $\cSHo$ with right coaction
$\delta_{\mathrm{X},\Cd}
\colon\mathrm{X}\to\mathrm{X}\Cd$.
We depict $\delta_{\mathrm{X},\Cd}$ as
\begin{gather*}
\delta_{\mathrm{X},\Cd}
\leftrightsquigarrow
\begin{tikzpicture}[anchorbase,scale=0.4,smallnodes]
\draw[dstrand] (0,0) node[below,black,yshift=-2pt]{$\mathrm{X}$} to (0,1);
\draw[dstrand] (0,1) to[out=180,in=270] (-1,2) node[above,black,yshift=-2pt]{$\mathrm{X}$};
\draw[cstrand] (0,1) to[out=0,in=270] (1,2) node[above,black,yshift=-2pt]{$d$};
\end{tikzpicture}
\;(\text{degree }0).
\end{gather*}
The coassociativity and counitality of $\delta_{\mathrm{X},\Cd}$ are expressed by the diagrammatic equations
\begin{gather}\label{eq:xcoass-xcounit}
\begin{tikzpicture}[anchorbase,scale=0.4,smallnodes]
\draw[dstrand] (-.25,0) node[below,black,yshift=-2pt]{$\mathrm{X}$} to (-.25,1);
\draw[dstrand] (-.25,1) to[out=180,in=270] (-1.5,2);
\draw[dstrand] (-1.5,2) to (-1.5,4) node[above,black]{$\mathrm{X}$};
\draw[cstrand] (-.25,1) to[out=0,in=270] (1,2);
\draw[cstrand] (1,2) to (1,3);
\draw[cstrand] (1,3) to[out=180,in=270] (0,4) node[above,black]{$d$};
\draw[cstrand] (1,3) to[out=0,in=270] (2,4) node[above,black]{$d$};
\end{tikzpicture}
=
\begin{tikzpicture}[anchorbase,scale=0.4,smallnodes]
\draw[dstrand] (.25,0) node[below,black,yshift=-2pt]{$\mathrm{X}$} to (.25,1);
\draw[dstrand] (.25,1) to[out=180,in=270] (-1,2);
\draw[dstrand] (-1,2) to (-1,3);
\draw[dstrand] (-1,3) to[out=180,in=270] (-2,4) node[above,black]{$\mathrm{X}$};
\draw[cstrand] (.25,1) to[out=0,in=270] (1.5,2);
\draw[cstrand] (-1,3) to[out=0,in=270] (0,4) node[above,black]{$d$};
\draw[cstrand] (1.5,2) to (1.5,4) node[above,black]{$d$};
\end{tikzpicture}
,\quad
\begin{tikzpicture}[anchorbase,scale=0.4,smallnodes]
\draw[dstrand] (0,0) node[below,black,yshift=-2pt]{$\mathrm{X}$} to (0,1);
\draw[dstrand] (0,1) to[out=180,in=270] (-1,2);
\draw[dstrand] (-1,2) to (-1,4) node[above,black]{$\mathrm{X}$};
\draw[cstrand] (0,1) to[out=0,in=270] (1,2);
\draw[cstrand,marked=1] (1,2) to (1,3);
\end{tikzpicture}
=
\begin{tikzpicture}[anchorbase,scale=0.4,smallnodes]
\draw[dstrand] (0,0) node[below,black,yshift=-2pt]{$\mathrm{X}$} to (0,4) node[above,black]{$X$};
\end{tikzpicture}
.
\end{gather}
We can then define the right $\Cd$-action $\mu_{\mathrm{X},\Cd}\colon
\mathrm{X}\Cd\to\mathrm{X}$ of degree $-2\mathbf{a}$ as
\begin{gather}\label{eq:x-dd-horizontal}
\mu_{\mathrm{X},\Cd}
\leftrightsquigarrow
\begin{tikzpicture}[anchorbase,scale=0.4,smallnodes]
\draw[dstrand] (0,.5) node[above,black]{$\mathrm{X}$} to (0,0) to (0,-1);
\draw[dstrand] (0,-1) to[out=180,in=90] (-1,-2) to (-1,-2.5) node[below,black,yshift=-2pt]{$\mathrm{X}$};
\draw[cstrand] (0,-1) to[out=0,in=90] (1,-2) to (1,-2.5) node[below,black,yshift=-2pt]{$d$};
\end{tikzpicture}
:=
\begin{tikzpicture}[anchorbase,scale=0.4,smallnodes]
\draw[dstrand] (0,0) node[below,black,yshift=-2pt]{$\mathrm{X}$} to (0,1);
\draw[dstrand] (0,1) to[out=180,in=270] (-1,2) to (-1,3) node[above,black]{$\mathrm{X}$};
\draw[cstrand] (0,1) to[out=0,in=270] (1,2) to[out=90,in=180] (2,3) to[out=0,in=90] (3,2) to (3,0) node[below,black,yshift=-2pt]{$d$};
\end{tikzpicture}
\;(\text{degree }-2\mathbf{a}).
\end{gather}
Note also that \eqref{eq:isotopies} implies that
\begin{gather}\label{eq:x-d-horizontal}
\begin{tikzpicture}[anchorbase,scale=0.4,smallnodes]
\draw[dstrand] (0,0) node[above,black]{$\mathrm{X}$} to (0,-1);
\draw[dstrand] (0,-1) to[out=180,in=90] (-1,-2) to (-1,-3) node[below,black,yshift=-2pt]{$\mathrm{X}$};
\draw[cstrand] (0,-1) to[out=0,in=90] (1,-2) to[out=270,in=180] (2,-3) to[out=0,in=270] (3,-2) to (3,0) node[above,black]{$d$};
\end{tikzpicture}
=
\begin{tikzpicture}[anchorbase,scale=0.4,smallnodes]
\draw[dstrand] (0,0) node[below,black,yshift=-2pt]{$\mathrm{X}$} to (0,1);
\draw[dstrand] (0,1) to[out=180,in=270] (-1,2) to (-1,3) node[above,black]{$\mathrm{X}$};
\draw[cstrand] (0,1) to[out=0,in=270] (1,2) to[out=90,in=180] (2,3) to[out=0,in=90] (3,2) to[out=270,in=180] (4,1) to[out=0,in=270] (5,2) to
(5,3) node[below,above]{$d$};
\end{tikzpicture}
=
\begin{tikzpicture}[anchorbase,scale=0.4,smallnodes]
\draw[dstrand] (0,-.5) node[below,black,yshift=-2pt]{$\mathrm{X}$} to (0,0) to (0,1);
\draw[dstrand] (0,1) to[out=180,in=270] (-1,2) to (-1,2.5) node[above,black]{$\mathrm{X}$};
\draw[cstrand] (0,1) to[out=0,in=270] (1,2) to (1,2.5) node[above,black]{$d$};
\end{tikzpicture}.
\end{gather}
Similarly, it is easy to see that
\begin{gather}\label{eq:xfrob}
\begin{gathered}
\begin{tikzpicture}[anchorbase,scale=0.4,smallnodes]
\draw[dstrand] (-.25,0) node[above,black]{$\mathrm{X}$} to (-.25,-1);
\draw[dstrand] (-.25,-1) to[out=180,in=90] (-1.5,-2);
\draw[dstrand] (-1.5,-2) to (-1.5,-4) node[below,black,yshift=-2pt]{$\mathrm{X}$};
\draw[cstrand] (-.25,-1) to[out=0,in=90] (1,-2);
\draw[cstrand] (1,-2) to (1,-3);
\draw[cstrand] (1,-3) to[out=180,in=90] (0,-4) node[below,black,yshift=-2pt]{$d$};
\draw[cstrand] (1,-3) to[out=0,in=90] (2,-4) node[below,black,yshift=-2pt]{$d$};
\end{tikzpicture}
=
\begin{tikzpicture}[anchorbase,scale=0.4,smallnodes]
\draw[dstrand] (.25,0) node[above,black]{$\mathrm{X}$} to (.25,-1);
\draw[dstrand] (.25,-1) to[out=180,in=90] (-1,-2);
\draw[dstrand] (-1,-2) to (-1,-3);
\draw[dstrand] (-1,-3) to[out=180,in=90] (-2,-4) node[below,black,yshift=-2pt]{$\mathrm{X}$};
\draw[cstrand] (.25,-1) to[out=0,in=90] (1.5,-2);
\draw[cstrand] (-1,-3) to[out=0,in=90] (0,-4) node[below,black,yshift=-2pt]{$d$};
\draw[cstrand] (1.5,-2) to (1.5,-4) node[below,black,yshift=-2pt]{$d$};
\end{tikzpicture}
,\quad
\begin{tikzpicture}[anchorbase,scale=0.4,smallnodes]
\draw[dstrand] (0,0) node[above,black]{$\mathrm{X}$} to (0,-1);
\draw[dstrand] (0,-1) to[out=180,in=90] (-1,-2);
\draw[dstrand] (-1,-2) to (-1,-4) node[below,black,yshift=-2pt]{$\mathrm{X}$};
\draw[cstrand] (0,-1) to[out=0,in=90] (1,-2);
\draw[cstrand,marked=1] (1,-2) to (1,-3);
\end{tikzpicture}
=
\begin{tikzpicture}[anchorbase,scale=0.4,smallnodes]
\draw[dstrand] (0,0) node[above,black]{$\mathrm{X}$} to (0,-4) node[below,black,yshift=-2pt]{$d$};
\end{tikzpicture}
,
\\
\begin{tikzpicture}[anchorbase,scale=0.4,smallnodes]
\draw[dstrand] (0,0) node[below,black,yshift=-2pt]{$\mathrm{X}$} to (0,3) node[above,black]{$\mathrm{X}$};
\draw[cstrand] (2,0) node[below,black,yshift=-2pt]{$d$} to (2,3) node[above,black]{$d$};
\draw[cstrand] (0,1) to (2,2);
\end{tikzpicture}
=
\begin{tikzpicture}[anchorbase,scale=0.4,smallnodes]
\draw[dstrand] (0,0) node[below,black,yshift=-2pt]{$\mathrm{X}$} to[out=90,in=180] (1,1);
\draw[dstrand] (1,1) to (1,2);
\draw[dstrand] (1,2) to[out=180,in=270] (0,3) node[above,black]{$\mathrm{X}$};
\draw[cstrand] (2,0) node[below,black,yshift=-2pt]{$d$} to[out=90,in=0] (1,1);
\draw[cstrand] (1,2) to[out=0,in=270] (2,3) node[above,black]{$d$};
\end{tikzpicture}
=
\begin{tikzpicture}[anchorbase,scale=0.4,smallnodes]
\draw[dstrand] (0,0) node[below,black,yshift=-2pt]{$\mathrm{X}$} to (0,3) node[above,black]{$\mathrm{X}$};
\draw[cstrand] (2,0) node[below,black,yshift=-2pt]{$d$} to (2,3) node[above,black]{$d$};
\draw[cstrand] (2,1) to (0,2);
\end{tikzpicture}
,
\end{gathered}
\end{gather}
i.e. $\mu_{\mathrm{X},\Cd}$ is associative and unital and there is a Frobenius-type
compatibility between $\mu_{\mathrm{X},\Cd}$ and $\delta_{\mathrm{X},\Cd}$. Thus,
\eqref{eq:xcoass-xcounit}, \eqref{eq:x-dd-horizontal}, \eqref{eq:x-d-horizontal}
and \eqref{eq:xfrob} imply that the diagrammatic calculus is again
topological in nature, which we will use to simplify our diagrammatic arguments.

By symmetry, similar results hold for left coactions and left actions of $\Cd$ and
one can check that any $\Cd$-$\Cd$-bicomodule in $\cSH$ has a
compatible $\Cd$-$\Cd$-bimodule structure and vice versa
(which diagrammatically translates into a height exchange, see \eqref{eq:height} below).
To simplify our diagrams for bimodules and bicomodules, we also allow four valent vertices, e.g.
\begin{gather}\label{eq:height}
\begin{tikzpicture}[anchorbase,scale=0.4,smallnodes]
\draw[dstrand] (0,0) node[below,black,yshift=-2pt]{$\mathrm{X}$} to (0,3) node[above,black]{$\mathrm{X}$};
\draw[cstrand] (1,3) node[above,black]{$d$} to[out=270,in=0] (0,1.5);
\draw[cstrand] (-1,3) node[above,black]{$d$} to[out=270,in=180] (0,1.5);
\end{tikzpicture}
:=
\begin{tikzpicture}[anchorbase,scale=0.4,smallnodes]
\draw[dstrand] (0,0) node[below,black,yshift=-2pt]{$\mathrm{X}$} to (0,3) node[above,black]{$\mathrm{X}$};
\draw[cstrand] (1,3) node[above,black]{$d$} to[out=270,in=0] (0,2.25);
\draw[cstrand] (-1,3) node[above,black]{$d$} to[out=270,in=180] (0,0.75);
\end{tikzpicture}
=
\begin{tikzpicture}[anchorbase,scale=0.4,smallnodes]
\draw[dstrand] (0,0) node[below,black,yshift=-2pt]{$\mathrm{X}$} to (0,3) node[above,black]{$\mathrm{X}$};
\draw[cstrand] (1,3) node[above,black]{$d$} to[out=270,in=0] (0,0.75);
\draw[cstrand] (-1,3) node[above,black]{$d$} to[out=270,in=180] (0,2.25);
\end{tikzpicture}
,\quad
\begin{tikzpicture}[anchorbase,scale=0.4,smallnodes]
\draw[dstrand] (0,0) node[below,black,yshift=-2pt]{$\mathrm{X}$} to (0,3) node[above,black]{$\mathrm{X}$};
\draw[cstrand] (1,0) node[below,black,yshift=-2pt]{$d$} to[out=90,in=0] (0,1.5);
\draw[cstrand] (-1,0) node[below,black,yshift=-2pt]{$d$} to[out=90,in=180] (0,1.5);
\end{tikzpicture}
:=
\begin{tikzpicture}[anchorbase,scale=0.4,smallnodes]
\draw[dstrand] (0,0) node[below,black,yshift=-2pt]{$\mathrm{X}$} to (0,3) node[above,black]{$\mathrm{X}$};
\draw[cstrand] (1,0) node[below,black,yshift=-2pt]{$d$} to[out=90,in=0] (0,2.25);
\draw[cstrand] (-1,0) node[below,black,yshift=-2pt]{$d$} to[out=90,in=180] (0,0.75);
\end{tikzpicture}
=
\begin{tikzpicture}[anchorbase,scale=0.4,smallnodes]
\draw[dstrand] (0,0) node[below,black,yshift=-2pt]{$\mathrm{X}$} to (0,3) node[above,black]{$\mathrm{X}$};
\draw[cstrand] (1,0) node[below,black,yshift=-2pt]{$d$} to[out=90,in=0] (0,0.75);
\draw[cstrand] (-1,0) node[below,black,yshift=-2pt]{$d$} to[out=90,in=180] (0,2.25);
\end{tikzpicture}
.
\end{gather}

The following lemma is the analog of Lemma \ref{lemma:idempotent}.

\begin{lemma}\label{lemma:x-d-idempotent}
We have
\begin{gather}\label{eq:x-d-idempotent}
\begin{tikzpicture}[anchorbase,scale=0.4,smallnodes]
\draw[dstrand] (0,0) to (0,1);
\draw[dstrand] (0,1) to [out=180,in=270] (-1,4) node[above,black]{$\mathrm{X}$};
\draw[dstrand] (-1,-3) node[below,black,yshift=-2pt]{$\mathrm{X}$} to [out=90,in=180] (0,0);
\draw[cstrand] (0,1) to [out=0,in=270] (1,2)  node[above,black,box]{\raisebox{-.025cm}{$\bullet$}};
\draw[cstrand] (1,3) to (1,4) node[above,black]{$d$};
\draw[cstrand] (1,-3) node[below,black,yshift=-2pt]{$d$} to (1,-2) node[above,black,box]{\raisebox{-.025cm}{$\bullet$}};
\draw[cstrand] (1,-1) to [out=90,in=0] (0,0);
\end{tikzpicture}
=
\begin{tikzpicture}[anchorbase,scale=0.4,smallnodes]
\draw[dstrand] (-1,-3) node[below,black,yshift=-2pt]{$\mathrm{X}$} to (-1,4) node[above,black]{$\mathrm{X}$};
\draw[cstrand] (1,-3) node[below,black,yshift=-2pt]{$d$} to (1,0) node[above,black,box]{\raisebox{-.025cm}{$\bullet$}};
\draw[cstrand] (1,1) to (1,4) node[above,black]{$d$};
\end{tikzpicture}
\,,\quad
\begin{tikzpicture}[anchorbase,scale=0.4,smallnodes]
\draw[dstrand] (0,0) to (0,1);
\draw[dstrand] (0,1) to [out=0,in=270] (1,4) node[above,black]{$\mathrm{X}$};
\draw[dstrand] (1,-3) node[below,black,yshift=-2pt]{$\mathrm{X}$} to [out=90,in=0] (0,0);
\draw[cstrand] (0,1) to [out=180,in=270] (-1,2)  node[above,black,box]{\raisebox{-.025cm}{$\bullet$}};
\draw[cstrand] (-1,3) node[left,black,xshift=-2pt]{$\star$} to (-1,4) node[above,black]{$d$};
\draw[cstrand] (-1,-3) node[below,black,yshift=-2pt]{$d$} to (-1,-2) node[above,black,box]{\raisebox{-.025cm}{$\bullet$}};
\draw[cstrand] (-1,-1) node[left,black,xshift=-2pt]{$\star$} to [out=90,in=180] (0,0);
\end{tikzpicture}
=
\begin{tikzpicture}[anchorbase,scale=0.4,smallnodes]
\draw[dstrand] (1,-3) node[below,black,yshift=-2pt]{$\mathrm{X}$} to (1,4) node[above,black]{$\mathrm{X}$};
\draw[cstrand] (-1,-3) node[below,black,yshift=-2pt]{$d$} to (-1,0) node[above,black,box]{\raisebox{-.025cm}{$\bullet$}};
\draw[cstrand] (-1,1) node[left,black,xshift=-2pt]{$\star$} to (-1,4) node[above,black]{$d$};
\end{tikzpicture}
\,.
\end{gather}
\end{lemma}

\begin{proof}
The trick is to reduce the statement to the one from Lemma \ref{lemma:idempotent}.
We show how to do this for the first diagrammatic equation.
\begin{gather*}
\begin{tikzpicture}[anchorbase,scale=0.4,smallnodes]
\draw[dstrand] (0,0) to (0,1);
\draw[dstrand] (0,1) to [out=180,in=270] (-1,4) node[above,black]{$\mathrm{X}$};
\draw[dstrand] (-1,-3) node[below,black,yshift=-2pt]{$\mathrm{X}$} to [out=90,in=180] (0,0);
\draw[cstrand] (0,1) to [out=0,in=270] (1,2)  node[above,black,box]{\raisebox{-.025cm}{$\bullet$}};
\draw[cstrand] (1,3) to (1,4) node[above,black]{$d$};
\draw[cstrand] (1,-3) node[below,black,yshift=-2pt]{$d$} to (1,-2) node[above,black,box]{\raisebox{-.025cm}{$\bullet$}};
\draw[cstrand] (1,-1) to [out=90,in=0] (0,0);
\end{tikzpicture}
=
\begin{tikzpicture}[anchorbase,scale=0.4,smallnodes]
\draw[dstrand] (-1,-3) node[below,black,yshift=-2pt]{$\mathrm{X}$} to
(-1,4) node[above,black]{$\mathrm{X}$};
\draw[cstrand] (1,0) to (1,2) node[above,black,box]{\raisebox{-.025cm}{$\bullet$}};
\draw[cstrand] (1,3) to (1,4) node[above,black]{$d$};
\draw[cstrand] (1,-3) node[below,black,yshift=-2pt]{$d$} to (1,-2) node[above,black,box]{\raisebox{-.025cm}{$\bullet$}};
\draw[cstrand] (1,-1) to (1,0);
\draw[cstrand] (-1,0.5) to (1,.5);
\end{tikzpicture}
=
\begin{tikzpicture}[anchorbase,scale=0.4,smallnodes]
\draw[dstrand] (-1,-3) node[below,black,yshift=-2pt]{$\mathrm{X}$} to
(-1,4) node[above,black]{$\mathrm{X}$};
\draw[cstrand] (1,0) to (1,2) node[above,black,box]{\raisebox{-.025cm}{$\bullet$}};
\draw[cstrand] (1,3) to (1,4) node[above,black]{$d$};
\draw[cstrand] (1,-3) node[below,black,yshift=-2pt]{$d$} to (1,-2) node[above,black,box]{\raisebox{-.025cm}{$\bullet$}};
\draw[cstrand] (1,-1) to (1,0);
\draw[cstrand] (-1,3) to[out=0,in=180] (1,.5);
\draw[cstrand,marked=1] (1,0) to[out=180,in=90] (0,-1) to (0,-2.5);
\end{tikzpicture}
\stackrel{(\ref{eq:idempotent})}{=}
\begin{tikzpicture}[anchorbase,scale=0.4,smallnodes]
\draw[dstrand] (-1,-3) node[below,black,yshift=-2pt]{$\mathrm{X}$} to (-1,4)
node[above,black]{$\mathrm{X}$};
\draw[cstrand] (1,-3) node[below,black,yshift=-2pt]{$d$} to (1,0) node[above,black,box]{\raisebox{-.025cm}{$\bullet$}};
\draw[cstrand] (1,1) to (1,4) node[above,black]{$d$};
\draw[cstrand,marked=1] (-1,3) to[out=0,in=90] (0,2) to (-0,-2.5);
\end{tikzpicture}
=
\begin{tikzpicture}[anchorbase,scale=0.4,smallnodes]
\draw[dstrand] (-1,-3) node[below,black,yshift=-2pt]{$\mathrm{X}$} to (-1,4)
node[above,black]{$\mathrm{X}$};
\draw[cstrand] (1,-3) node[below,black,yshift=-2pt]{$d$} to (1,0) node[above,black,box]{\raisebox{-.025cm}{$\bullet$}};
\draw[cstrand] (1,1) to (1,4) node[above,black]{$d$};
\end{tikzpicture}
.
\end{gather*}
The proof of the second diagrammatic equation is analogous.
\end{proof}

The following lemma is the analog of \eqref{eq:dotdiagram} and \eqref{eq:dualdotdiagram}.

\begin{lemma}\label{lemma:x-dot-and-dual-dot}
We have
\begin{gather}\label{eq:x-dot-and-dual-dot}
\begin{tikzpicture}[anchorbase,scale=0.4,smallnodes]
\draw[dstrand] (0,0) node[below,black,yshift=-2pt]{$\mathrm{X}$} to (0,5) node[above,black]{$\mathrm{X}$};
\draw[cstrand] (0,1) to [out=0,in=270] (2,2) node[above,black,box]{\raisebox{-.025cm}{$\bullet$}};
\draw[cstrand] (2,3) to [out=90,in=0] (0,4);
\end{tikzpicture}
=
\begin{tikzpicture}[anchorbase,scale=0.4,smallnodes]
\draw[dstrand] (0,0) node[below,black,yshift=-2pt]{$\mathrm{X}$} to (0,5) node[above,black]{$\mathrm{X}$};
\end{tikzpicture}
=
\begin{tikzpicture}[anchorbase,scale=0.4,smallnodes]
\draw[dstrand] (0,0) node[below,black,yshift=-2pt]{$\mathrm{X}$} to (0,5) node[above,black]{$\mathrm{X}$};
\draw[cstrand] (0,1) to [out=180,in=270] (-2,2) node[above,black,box]{\raisebox{-.025cm}{$\bullet$}};
\draw[cstrand] (-2,3) node[left,black,xshift=-2pt]{$\star$} to [out=90,in=180] (0,4);
\end{tikzpicture}
.
\end{gather}
\end{lemma}

\begin{proof}
Let us prove the first equality.
\begin{gather*}
\begin{tikzpicture}[anchorbase,scale=0.4,smallnodes]
\draw[dstrand] (0,0) node[below,black,yshift=-2pt]{$\mathrm{X}$} to (0,5) node[above,black]{$\mathrm{X}$};
\draw[cstrand] (0,1) to [out=0,in=270] (2,2) node[above,black,box]{\raisebox{-.025cm}{$\bullet$}};
\draw[cstrand] (2,3) to [out=90,in=0] (0,4);
\end{tikzpicture}
\stackrel{(\ref{eq:xcoass-xcounit})}{=}
\begin{tikzpicture}[anchorbase,scale=0.4,smallnodes]
\draw[dstrand] (0,0) node[below,black,yshift=-2pt]{$\mathrm{X}$} to (0,5) node[above,black]{$\mathrm{X}$};
\draw[cstrand] (0,1) to [out=0,in=270] (2,2) node[above,black,box]{\raisebox{-.025cm}{$\bullet$}};
\draw[cstrand] (2,3) to[out=90,in=0] (1.5,3.5)
to[out=180,in=90] (1,3) to (1,1);
\end{tikzpicture}
=
\begin{tikzpicture}[anchorbase,scale=0.4,smallnodes]
\draw[dstrand] (0,0) node[below,black,yshift=-2pt]{$\mathrm{X}$} to (0,5) node[above,black]{$\mathrm{X}$};
\draw[cstrand] (0,1) to [out=0,in=270] (2,2) node[above,black,box]{\raisebox{-.025cm}{$\bullet$}};
\draw[cstrand] (2,3) to[out=90,in=0] (1.5,3.5)
to[out=180,in=90] (1,3) to (1,1);
\draw[cstrand,marked=1] (1.5,3.5) to (1.5,4);
\end{tikzpicture}
\stackrel{(\ref{eq:dotdiagram})}{=}
\begin{tikzpicture}[anchorbase,scale=0.4,smallnodes]
\draw[dstrand] (0,0) node[below,black,yshift=-2pt]{$\mathrm{X}$} to (0,5) node[above,black]{$\mathrm{X}$};
\draw[cstrand,marked=1] (0,1) to[out=0,in=270] (1,2) to (1,4);
\end{tikzpicture}
\stackrel{(\ref{eq:xcoass-xcounit})}{=}
\begin{tikzpicture}[anchorbase,scale=0.4,smallnodes]
\draw[dstrand] (0,0) node[below,black,yshift=-2pt]{$\mathrm{X}$} to (0,5) node[above,black]{$\mathrm{X}$};
\end{tikzpicture}
.
\end{gather*}
The second equation can be proved by symmetry.
\end{proof}

For any $t\in\mathbb{Z}$, let $\cAH\langle t\rangle$ be  the category defined as
\begin{gather*}
\cAH\langle t\rangle:=
\mathrm{add}\big(\{\mathrm{C}_{w}^{\oplus\mathsf{v}^{k}}\mid w\in
\mathcal{H},k\geq t\}\big)^{(0)}
/
\big(\mathrm{add}
\big(\{\mathrm{C}_{w}^{\oplus\mathsf{v}^{k}}\mid w\in\mathcal{H},k>t\}
\big)^{(0)}\big).
\end{gather*}
By definition, we will write
\begin{gather*}
\mathrm{A}_{w}^{\oplus\mathsf{v}^{\mone[t]}}
=
{\Pi}(\mathrm{C}_{w}^{\oplus\mathsf{v}^{\mone[t]}})
\in\cAH\langle t\rangle,
\end{gather*}
for any $w\in\mathcal{H}$. Note that horizontal composition in $\cSH$ descends to
$\circ_{\mathrm{h}}\colon\cAH\langle s\rangle\times\cAH\langle t\rangle\to\cAH\langle s+t\rangle$, so the action of $\cAH$ on
$\cAH\langle t\rangle$ gives rise to a semisimple simple transitive $2$-representation of
$\cAH$. In particular, $\bigoplus_{t\in\mathbb{Z}}\cAH\langle t\rangle$
is a $\mathbb{Z}$-finitary locally semisimple bicategory, as defined in Subsection \ref{s:grreps}.

Recall also from Subsection \ref{s:grreps} the definition of the
graded finitary $2$-category $\cC^{\,\prime}$ with translation for a given graded $\Bbbk$-linear $2$-category $\cC$.

\begin{proposition}\label{proposition:bicomodules}
We have an equivalence of $\mathbb{Z}$-finitary locally semisimple bicategories
\begin{gather*}
\cBHo\simeq\bigoplus_{t\in\mathbb{Z}}\cAH\langle t\rangle.
\end{gather*}
Moreover, this implies an equivalence of locally graded semisimple bicategories
\begin{gather}\label{eq:bicomodules2}
\cBH\simeq\cAHp.
\end{gather}
\end{proposition}

\begin{proof}
Let $\mathrm{X}$ be a $1$-morphism in $\cBHo$. Suppose that $\mathrm{C}_{x}^{\oplus\mathsf{v}^{t}}$
is a direct summand of $\mathrm{X}$ in $\cSHo$
for some $x\in\mathcal{H}$ and $t\in\mathbb{Z}$. We  first claim that
$\mathrm{C}_{x}^{\oplus\mathsf{v}^{t}}$ is also a direct
summand of $\mathrm{X}$ in $\cBHo$.

To this end, let
$\alpha\in\mathrm{hom}_{\ccS_{\mathcal{H}}}(\mathrm{C}_{x}^{\oplus\mathsf{v}^{t}},\mathrm{X})$
and $\beta\in\mathrm{hom}_{\ccS_{\mathcal{H}}}(\mathrm{X},\mathrm{C}_{x}^{\oplus\mathsf{v}^{t}})$
be the embedding and the projection respectively, i.e.
$\beta\circ_{\mathrm{v}}\alpha=\mathrm{id}_{\mathrm{C}_{x}^{\oplus\mathsf{v}^{t}}}$.
Depict them using our usual conventions, with dotted strands of a different color for
$\mathrm{C}_{x}$ (which we again denote by $x$ in the diagrams) and $X$:
\begin{gather*}
\alpha
\leftrightsquigarrow
\begin{tikzpicture}[anchorbase,scale=0.4,smallnodes]
\draw[dstrand] (0,2) to (0,3) node[above,black]{$\mathrm{X}$};
\draw[xstrand] (0,0) node[below,black,yshift=-2pt]{$x$} to (0,1) node[above,black,box]{$\alpha$};
\end{tikzpicture}
,\quad
\beta
\leftrightsquigarrow
\begin{tikzpicture}[anchorbase,scale=0.4,smallnodes]
\draw[dstrand] (0,0) node[below,black,yshift=-2pt]{$\mathrm{X}$}
to (0,1);
\draw[xstrand] (0,2) to (0,3) node[above,black]{$x$};
\draw[xstrand] (0,.99) to (0,1) node[above,black,box]{$\beta$};
\end{tikzpicture}
.
\end{gather*}
Define $\alpha^{\prime}\in\mathrm{hom}_{\ccBH}(\mathrm{C}_{x}^{\oplus\mathsf{v}^{t}},\mathrm{X})$ and $\beta^{\prime}\in\mathrm{hom}_{\ccBH}(\mathrm{X},\mathrm{C}_{x}^{\oplus\mathsf{v}^{t}})$ as
\begin{gather*}
\alpha^{\prime}:=
\begin{tikzpicture}[anchorbase,scale=0.4,smallnodes]
\draw[dstrand] (0,3) to (0,5) node[above,black]{$\mathrm{X}$};
\draw[xstrand] (0,0) node[below,black,yshift=-2pt]{$x$} to (0,2) node[above,black,box]{$\alpha$};
\draw[cstrand] (0,1) to [out=0,in=270] (1.5,2) node[above,black,box]{\raisebox{-.025cm}{$\bullet$}};
\draw[cstrand] (1.5,3) to [out=90,in=0] (0,4);
\draw[cstrand] (0,1) to [out=180,in=270] (-1.5,2) node[above,black,box]{\raisebox{-.025cm}{$\bullet$}};
\draw[cstrand] (-1.5,3) node[left,black,xshift=-2pt]{$\star$} to [out=90,in=180] (0,4);
\end{tikzpicture}
\,,\quad
\beta^{\prime}:=
\begin{tikzpicture}[anchorbase,scale=0.4,smallnodes]
\draw[dstrand] (0,0) node[below,black,yshift=-2pt]{$\mathrm{X}$} to (0,2);
\draw[xstrand] (0,3) to (0,5) node[above,black]{$x$};
\draw[cstrand] (0,1.99) to (0,2) node[above,black,box]{$\beta$};
\draw[cstrand] (0,1) to[out=0,in=270] (1.5,2) node[above,black,box]{\raisebox{-.025cm}{$\bullet$}};
\draw[cstrand] (1.5,3) to[out=90,in=0] (0,4);
\draw[cstrand] (0,1) to[out=180,in=270] (-1.5,2) node[above,black,box]{\raisebox{-.025cm}{$\bullet$}};
\draw[cstrand] (-1.5,3) node[left,black,xshift=-2pt]{$\star$} to [out=90,in=180] (0,4);
\end{tikzpicture}
\,.
\end{gather*}
Note that $\alpha^{\prime}$ and $\beta^{\prime}$ are indeed $2$-morphisms in $\cBHo$, e.g.
\begin{gather*}
\begin{tikzpicture}[anchorbase,scale=0.4,smallnodes]
\draw[dstrand] (0,3) to (0,5) node[above,black]{$\mathrm{X}$};
\draw[xstrand] (0,0) node[below,black,yshift=-2pt]{$x$} to (0,2) node[above,black,box]{$\alpha$};
\draw[cstrand] (0,1) to [out=0,in=270] (1.5,2) node[above,black,box]{\raisebox{-.025cm}{$\bullet$}};
\draw[cstrand] (1.5,3) to [out=90,in=0] (0,4);
\draw[cstrand] (0,1) to [out=180,in=270] (-1.5,2) node[above,black,box]{\raisebox{-.025cm}{$\bullet$}};
\draw[cstrand] (-1.5,3) node[left,black,xshift=-2pt]{$\star$} to [out=90,in=180] (0,4);
\draw[cstrand] (0,.5) to [out=0,in=270] (3,5) node[above,black]{$d$};
\end{tikzpicture}
=
\begin{tikzpicture}[anchorbase,scale=0.4,smallnodes]
\draw[dstrand] (0,3) to (0,5) node[above,black]{$\mathrm{X}$};
\draw[xstrand] (0,0) node[below,black,yshift=-2pt]{$x$} to (0,2) node[above,black,box]{$\alpha$};
\draw[cstrand] (0,1) to [out=0,in=270] (1.5,2) node[above,black,box]{\raisebox{-.025cm}{$\bullet$}};
\draw[cstrand] (1.5,3) to [out=90,in=0] (0,4);
\draw[cstrand] (0,1) to [out=180,in=270] (-1.5,2) node[above,black,box]{\raisebox{-.025cm}{$\bullet$}};
\draw[cstrand] (-1.5,3) node[left,black,xshift=-2pt]{$\star$} to [out=90,in=180] (0,4);
\draw[cstrand] (1.25,1.5) to [out=0,in=270] (3,5) node[above,black]{$d$};
\end{tikzpicture}
\stackrel{(\ref{eq:dotdiagram-slide})}{=}
\begin{tikzpicture}[anchorbase,scale=0.4,smallnodes]
\draw[dstrand] (0,3) to (0,5) node[above,black]{$\mathrm{X}$};
\draw[xstrand] (0,0) node[below,black,yshift=-2pt]{$x$} to (0,2) node[above,black,box]{$\alpha$};
\draw[cstrand] (0,1) to [out=0,in=270] (1.5,2) node[above,black,box]{\raisebox{-.025cm}{$\bullet$}};
\draw[cstrand] (1.5,3) to [out=90,in=0] (0,4);
\draw[cstrand] (0,1) to [out=180,in=270] (-1.5,2) node[above,black,box]{\raisebox{-.025cm}{$\bullet$}};
\draw[cstrand] (-1.5,3) node[left,black,xshift=-2pt]{$\star$} to [out=90,in=180] (0,4);
\draw[cstrand] (1.25,3.5) to [out=0,in=270] (3,5) node[above,black]{$d$};
\end{tikzpicture}
\stackrel{(\ref{eq:xfrob})}{=}
\begin{tikzpicture}[anchorbase,scale=0.4,smallnodes]
\draw[dstrand] (0,3) to (0,5) node[above,black]{$\mathrm{X}$};
\draw[xstrand] (0,0) node[below,black,yshift=-2pt]{$x$} to (0,2) node[above,black,box]{$\alpha$};
\draw[cstrand] (0,1) to [out=0,in=270] (1.5,2) node[above,black,box]{\raisebox{-.025cm}{$\bullet$}};
\draw[cstrand] (1.5,3) to [out=90,in=0] (0,4);
\draw[cstrand] (0,1) to [out=180,in=270] (-1.5,2) node[above,black,box]{\raisebox{-.025cm}{$\bullet$}};
\draw[cstrand] (-1.5,3) node[left,black,xshift=-2pt]{$\star$} to [out=90,in=180] (0,4);
\draw[cstrand] (0,4.5) to (2,4.5) to [out=0,in=270] (3,5) node[above,black]{$d$};
\end{tikzpicture}
,
\end{gather*}
which shows that $\alpha^{\prime}$ commutes with the right $\Cd$-coaction.
Next, by \eqref{eq:x-d-idempotent},
\eqref{eq:x-dot-and-dual-dot}, the Frobenius properties
and $\beta\circ_{\mathrm{v}}\alpha=\mathrm{id}_{\mathrm{C}_{x}^{\oplus\mathsf{v}^{t}}}$,
we have
$\beta^{\prime}\circ_{\mathrm{v}}
\alpha^{\prime}=
\mathrm{id}_{\mathrm{C}_{x}^{\oplus\mathsf{v}^{t}}}$. Indeed,
\begin{align*}
\begin{tikzpicture}[anchorbase,scale=0.4,smallnodes]
\draw[dstrand] (0,3) to (0,6);
\draw[xstrand] (0,7) to (0,9) node[above,black]{$x$};
\draw[xstrand] (0,0) node[below,black,yshift=-2pt]{$x$} to (0,2) node[above,black,box]{$\alpha$};
\draw[cstrand] (0,1) to [out=0,in=270] (1.5,2) node[above,black,box]{\raisebox{-.025cm}{$\bullet$}};
\draw[cstrand] (1.5,3) to [out=90,in=0] (0,4);
\draw[cstrand] (0,1) to [out=180,in=270] (-1.5,2) node[above,black,box]{\raisebox{-.025cm}{$\bullet$}};
\draw[cstrand] (-1.5,3) node[left,black,xshift=-2pt]{$\star$} to [out=90,in=180] (0,4);
\draw[cstrand] (0,5.99) to (0,6) node[above,black,box]{$\beta$};
\draw[cstrand] (0,5) to[out=0,in=270] (1.5,6) node[above,black,box]{\raisebox{-.025cm}{$\bullet$}};
\draw[cstrand] (1.5,7) to[out=90,in=0] (0,8);
\draw[cstrand] (0,5) to[out=180,in=270] (-1.5,6) node[above,black,box]{\raisebox{-.025cm}{$\bullet$}};
\draw[cstrand] (-1.5,7) node[left,black,xshift=-2pt]{$\star$} to [out=90,in=180] (0,8);
\end{tikzpicture}
&=
\begin{tikzpicture}[anchorbase,scale=0.4,smallnodes]
\draw[dstrand] (0,3) to (0,6);
\draw[xstrand] (0,7) to (0,9) node[above,black]{$x$};
\draw[xstrand] (0,0) node[below,black,yshift=-2pt]{$x$} to (0,2) node[above,black,box]{$\alpha$};
\draw[cstrand] (0,1) to [out=0,in=270] (1.5,2) node[above,black,box]{\raisebox{-.025cm}{$\bullet$}};
\draw[cstrand] (1.5,3) to [out=90,in=0] (0,4.25);
\draw[cstrand] (0,1) to [out=180,in=270] (-1.5,2) node[above,black,box]{\raisebox{-.025cm}{$\bullet$}};
\draw[cstrand] (-1.5,3) node[left,black,xshift=-2pt]{$\star$} to [out=90,in=180] (0,4);
\draw[cstrand] (0,5.99) to (0,6) node[above,black,box]{$\beta$};
\draw[cstrand] (0,4.75) to[out=0,in=270] (1.5,6) node[above,black,box]{\raisebox{-.025cm}{$\bullet$}};
\draw[cstrand] (1.5,7) to[out=90,in=0] (0,8);
\draw[cstrand] (0,5) to[out=180,in=270] (-1.5,6) node[above,black,box]{\raisebox{-.025cm}{$\bullet$}};
\draw[cstrand] (-1.5,7) node[left,black,xshift=-2pt]{$\star$} to [out=90,in=180] (0,8);
\end{tikzpicture}
\stackrel{\eqref{eq:x-d-idempotent}}{=}
\begin{tikzpicture}[anchorbase,scale=0.4,smallnodes]
\draw[dstrand] (0,3) to (0,6);
\draw[xstrand] (0,7) to (0,9) node[above,black]{$x$};
\draw[xstrand] (0,0) node[below,black,yshift=-2pt]{$x$} to (0,2) node[above,black,box]{$\alpha$};
\draw[cstrand] (0,1) to [out=0,in=270] (1.5,2);
\draw[cstrand] (1.5,2) to (1.5,4) node[above,black,box]{\raisebox{-.025cm}{$\bullet$}};
\draw[cstrand] (0,1) to [out=180,in=270] (-1.5,2) node[above,black,box]{\raisebox{-.025cm}{$\bullet$}};
\draw[cstrand] (-1.5,3) node[left,black,xshift=-2pt]{$\star$} to [out=90,in=180] (0,4);
\draw[cstrand] (0,5.99) to (0,6) node[above,black,box]{$\beta$};
\draw[cstrand] (1.5,5) to (1.5,7);
\draw[cstrand] (1.5,7) to[out=90,in=0] (0,8);
\draw[cstrand] (0,5) to[out=180,in=270] (-1.5,6) node[above,black,box]{\raisebox{-.025cm}{$\bullet$}};
\draw[cstrand] (-1.5,7) node[left,black,xshift=-2pt]{$\star$} to [out=90,in=180] (0,8);
\end{tikzpicture}
\stackrel{\eqref{eq:x-d-idempotent}}{=}
\begin{tikzpicture}[anchorbase,scale=0.4,smallnodes]
\draw[dstrand] (0,3) to (0,6);
\draw[xstrand] (0,7) to (0,9) node[above,black]{$x$};
\draw[xstrand] (0,0) node[below,black,yshift=-2pt]{$x$} to (0,2) node[above,black,box]{$\alpha$};
\draw[cstrand] (0,1) to [out=0,in=270] (1.5,2);
\draw[cstrand] (0,1) to [out=180,in=270] (-1.5,2);
\draw[cstrand] (1.5,2) to (1.5,4) node[above,black,box]{\raisebox{-.025cm}{$\bullet$}};
\draw[cstrand] (-1.5,2) to (-1.5,4) node[above,black,box]{\raisebox{-.025cm}{$\bullet$}};
\draw[cstrand] (0,5.99) to (0,6) node[above,black,box]{$\beta$};
\draw[cstrand] (1.5,5) to (1.5,7);
\draw[cstrand] (-1.5,5) to (-1.5,7);
\draw[cstrand] (1.5,7) to[out=90,in=0] (0,8);
\draw[cstrand] (-1.5,7) to [out=90,in=180] (0,8);
\node at (-2.25,5) {$\star$};
\end{tikzpicture}
\\
&=
\begin{tikzpicture}[anchorbase,scale=0.4,smallnodes]
\draw[xstrand] (0,0) node[below,black,yshift=-2pt]{$x$} to (0,9) node[above,black]{$x$};
\draw[cstrand] (0,1) to [out=0,in=270] (1.5,2);
\draw[cstrand] (0,1) to [out=180,in=270] (-1.5,2);
\draw[cstrand] (1.5,2) to (1.5,4) node[above,black,box]{\raisebox{-.025cm}{$\bullet$}};
\draw[cstrand] (-1.5,2) to (-1.5,4) node[above,black,box]{\raisebox{-.025cm}{$\bullet$}};
\draw[cstrand] (1.5,5) to (1.5,7);
\draw[cstrand] (-1.5,5) to (-1.5,7);
\draw[cstrand] (1.5,7) to[out=90,in=0] (0,8);
\draw[cstrand] (-1.5,7) to [out=90,in=180] (0,8);
\node at (-2.25,5) {$\star$};
\end{tikzpicture}
=
\begin{tikzpicture}[anchorbase,scale=0.4,smallnodes]
\draw[xstrand] (0,0) node[below,black,yshift=-2pt]{$x$} to (0,9) node[above,black]{$x$};
\draw[cstrand] (0,1.25) to [out=0,in=270] (1.5,2);
\draw[cstrand] (0,1) to [out=180,in=270] (-1.5,2);
\draw[cstrand] (1.5,2) to (1.5,4) node[above,black,box]{\raisebox{-.025cm}{$\bullet$}};
\draw[cstrand] (-1.5,2) to (-1.5,4) node[above,black,box]{\raisebox{-.025cm}{$\bullet$}};
\draw[cstrand] (1.5,5) to (1.5,7);
\draw[cstrand] (-1.5,5) to (-1.5,7);
\draw[cstrand] (1.5,7) to[out=90,in=0] (0,7.75);
\draw[cstrand] (-1.5,7) to [out=90,in=180] (0,8);
\node at (-2.25,5) {$\star$};
\end{tikzpicture}
\stackrel{\eqref{eq:x-dot-and-dual-dot}}{=}
\begin{tikzpicture}[anchorbase,scale=0.4,smallnodes]
\draw[xstrand] (0,0) node[below,black,yshift=-2pt]{$x$} to (0,9) node[above,black]{$x$};
\draw[cstrand] (0,1) to [out=180,in=270] (-1.5,2);
\draw[cstrand] (-1.5,2) to (-1.5,4) node[above,black,box]{\raisebox{-.025cm}{$\bullet$}};
\draw[cstrand] (-1.5,5) to (-1.5,7);
\draw[cstrand] (-1.5,7) to [out=90,in=180] (0,8);
\node at (-2.25,5) {$\star$};
\end{tikzpicture}
\stackrel{\eqref{eq:x-dot-and-dual-dot}}{=}
\begin{tikzpicture}[anchorbase,scale=0.4,smallnodes]
\draw[xstrand] (0,0) node[below,black,yshift=-2pt]{$x$} to (0,9) node[above,black]{$x$};
\end{tikzpicture}
,
\end{align*}
which proves our claim. Hence,
the decomposition of $\mathrm{X}$ into indecomposables is the same
in $\cBHo$ as in $\cSHo$.

Together with Lemma \ref{lemma:BB-intertwiners-2}, this implies that
\begin{gather*}
\cBHo(\varnothing,\varnothing)\simeq\bigoplus_{t\in\mathbb{Z}}\cAH(\varnothing,\varnothing)\langle t\rangle
\end{gather*}
as $\mathbb{Z}$-finitary categories.

To show that this is an equivalence of bicategories, recall the oplax pseudofunctor $\Theta\colon\cAH\to\cSHo$ from \eqref{Thetadef}.
Since each $\mathrm{X}\in\cAH$ has a canonical $\mathrm{A}_{d}$-$\mathrm{A}_{d}$-bicomodule structure,
$\Theta(\mathrm{X})$ has an induced $\Cd\cong\Theta(\mathrm{A}_{d})$-bicomodule structure
in $\cSHo$ by the
oplax condition of $\Theta$. By Lemma \ref{lem5.25-1}, there are homogeneous
$2$-isomorphisms of degree zero
\begin{gather*}
\Theta(\mathrm{XY})\cong\Theta(\mathrm{X}\square_{\mathrm{A}_{d}}\mathrm{Y})\cong
\Theta(\mathrm{X})\square_{\Theta(\mathrm{A}_{d})}\Theta(\mathrm{Y})\cong
\Theta(\mathrm{X})\square_{\Cd}\Theta(\mathrm{Y})
\end{gather*}
that are natural in $\mathrm{X}$ and $\mathrm{Y}$. This shows that $\Theta$ gives rise to a $\mathbb{C}$-linear pseudofunctor
\begin{gather*}
\Theta\colon\cAH\to\cBHo,
\end{gather*}
for which we use the same notation.

Observe that $\Theta$ extends uniquely
to a $\mathbb{C}$-linear pseudofunctor
\begin{gather*}
\Theta^{\prime}\colon\bigoplus_{t\in\mathbb{Z}}\cAH\langle t\rangle\to\cBHo
\end{gather*}
which is compatible with the $\mathbb{Z}$-actions, and again $2$-full by Lemma
\ref{lemma:BB-intertwiners-2} and $2$-faithful by semisimplicity. By the claim at
the beginning of the proof and Lemma \ref{lemma:unique-bicomodule}, each $\mathrm{X}\in\cSHo$
has therefore a unique $\Cd$-$\Cd$-bicomodule structure and is isomorphic to $\Theta^{\prime}(\mathrm{Z})$
for some $\mathrm{Z}\in\bigoplus_{t\in\mathbb{Z}}\cAH\langle t\rangle$. This shows
that $\Theta^{\prime}$ is essentially surjective on $1$-morphisms.
Hence
\begin{gather*}
\cBHo\simeq\bigoplus_{t\in\mathbb{Z}}\cAH\langle t\rangle
\end{gather*}
as free $\mathbb{Z}$-bicategories.

Since $\bigoplus_{t\in\mathbb{Z}}\cAH\langle t\rangle\simeq\cAHpzero$ by Lemma \ref{lem:skewcat}, this implies that
there are biequivalences of free $\mathbb{Z}$-bicategories
\begin{gather*}
\cAHpzero\simeq\bigoplus_{t\in\mathbb{Z}}\cAH\langle t\rangle\simeq\cBHo.
\end{gather*}
Moreover, by the paragraph above Lemma \ref{lem:skewcat}, $\cBH\simeq\cBHo[\mathbb{Z}]$ and, therefore, we obtain biequivalences of locally graded bicategories
\begin{gather*}
\cBH\simeq \cBHo[\mathbb{Z}]\simeq\cAHpzero[\mathbb{Z}]\simeq
\cAHp,
\end{gather*}
where the third biequivalence follows from Lemma \ref{lem:skewcat} and the second from the previous chain of biequivalences.
\end{proof}

The following corollary follows immediately from Corollary
\ref{cor:End-Bicomod} and Proposition \ref{proposition:bicomodules}.

\begin{corollary}\label{cor:ssendo}
There is a biequivalence of graded bicategories
\begin{gather*}
\cEnd_{\ccSH}(\mathbf{C}_{\mathcal{H}})\simeq\cAHpop.
\end{gather*}
\end{corollary}

\begin{remark}\label{remark:dual-bh}
If we forget the gradings, then the biequivalence \eqref{eq:bicomodules2} becomes a
biequivalence of (ungraded) locally semisimple bicategories
\begin{gather*}
\cBH\simeq\cAH.
\end{gather*}
This induces the structure of a pivotal fusion bicategory on $\cBH$, by Proposition \ref{prop:fusion}.
Abusing the terminology slightly, we therefore say that $\cBHo$ is
\emph{pivotal $\mathbb{Z}$-fusion}, in particular, it is a semisimple pivotal monoidal category.

Notice that the dual of $\mathrm{C}_{w}^{\oplus\mathsf{v}^{t}}$ in
$\cBHo$ is isomorphic to $\mathrm{C}_{w^{\mone}}^{\oplus\mathsf{v}^{\mone[t]}}$,
for any $w\in\mathcal{H}$ and $t\in\mathbb{Z}$.
This in contrast to the
dual of $\mathrm{C}_{w}^{\oplus\mathsf{v}^{t}}$ in $\cSHo$, which is isomorphic to
$\mathrm{C}_{w^{\mone}}^{\oplus\mathsf{v}^{\mone[t]\mone[2]\mathbf{a}}}$.
\end{remark}

\begin{remark}
The proof of Proposition \ref{proposition:bicomodules} also shows that
\begin{gather*}
\cBH\simeq (\Cd)\mathrm{comod}_{\mathrm{add}(\mathcal{H})}(\Cd).
\end{gather*}
\end{remark}

By the graded version of Theorem \ref{theorem:strongH}, there is a pair
of graded biequivalences
\begin{gather*}
\cSH\text{-}\mathrm{gstmod}_{\mathcal{H}}
\simeq
\cSJ\text{-}\mathrm{gstmod}_{\mathcal{J}}
\simeq
\cS\text{-}\mathrm{gstmod}_{\mathcal{J}},
\end{gather*}
where $\mathcal{J}$ is the two-sided cell containing $\mathcal{H}$.
Among other things, the first biequivalence uses the fact that
the additive closure of $\mathcal{H}$ in $\cSH$ is equal to the
additive closure of $\mathcal{H}$ in $\cSJ$. The composite of both biequivalences sends
$\CH\simeq\mathbf{inj}_{\underline{\ccSH}}(\Cd)$ to
$\CL\simeq\mathbf{inj}_{\underline{\ccS}}(\Cd)$, where
$\mathcal{L}$ is the left cell such that $\mathcal{H}=\mathcal{L}\cap\mathcal{L}^{\star}$.
It also provides a graded equivalence between
$\cEnd_{\ccSH}(\CH)$ and
$\cEnd_{\ccS}(\CL)$, both being equivalent to
$\cBHop\simeq\cAHpop$.

We can now explain a generalization of Proposition \ref{proposition:bicomodules}.
Let $\cS\text{-}\mathrm{gcell}_{\mathcal{J}}$ be the $2$-category whose
\begin{enumerate}[\textbullet]

\item objects are left cells $\mathcal{L}$ in $\mathcal{J}$, identified with the cell $2$-representation $\CL$ of $\cS$;

\item morphism categories $\cS\text{-}\mathrm{gcell}_{\mathcal{J}}(\mathcal{L}_{1},\mathcal{L}_{2}):=
\mathrm{Hom}_{\ccS}(\mathbf{C}_{\mathcal{L}_{1}},\mathbf{C}_{\mathcal{L}_{2}})$ with horizontal
composition given by composition of morphisms of $2$-representations.

\end{enumerate}

Recalling that every left cell $\mathcal{L}$ contains a unique Duflo involution
$d_{\mathcal{L}}$ and the equivalence $\CL\simeq\mathbf{inj}_{\underline{\ccS}}(\mathrm{C}_{d_{\mathcal{L}}})$,
$\cS\text{-}\mathrm{gcell}_{\mathcal{J}}$ is biequivalent to the bicategory $\cBJop$,
where $\cBJ$ has the same objects but with morphism categories
\begin{gather*}
\cBJ(\mathcal{L}_{1},\mathcal{L}_{2})=(\mathrm{C}_{d_{\mathcal{L}_{1}}})\mathrm{biinj}_{\underline{\mathrm{add}(\mathcal{J})}}(\mathrm{C}_{d_{\mathcal{L}_{2}}})
\end{gather*}
where horizontal composition is given by taking cotensor products over the relevant
Duflo involution and the identity $1$-morphism on $\mathcal{L}$ is $\mathrm{C}_{d_{\mathcal{L}}}$.

As already recalled in Remark \ref{rem:fusion}, we can define the bicategory $\cAJ$ in
analogy to $\cAH$ using the perverse filtration of $\mathrm{add}(\mathcal{J})$.

\begin{theorem}\label{theorem:local-ss}
The $2$-category $\cS\text{-}\mathrm{gcell}_{\mathcal{J}}$ is locally graded semisimple.
\end{theorem}

\begin{proof}
As in Subsection \ref{s7.5} and Proposition \ref{proposition:dot diagram},
the proof follows from the fact that the Duflo involution $\mathrm{C}_{d_{\mathcal{L}}}$ is
a separable Frobenius algebra in $\cSJ$ for every left cell $\mathcal{L}$ in $\mathcal{J}$.
Recall that $\cAH$ is a $2$-full one-object subbicategory (i.e. a full monoidal subcategory)
of the bicategory $\cAJ$. As we already mentioned in Remark \ref{rem:fusion}, the
latter can be regarded as a pivotal multifusion bicategory with one object (whose identity
$1$-morphism is decomposable in general) or, as we will do in this paper, as a pivotal
fusion bicategory whose objects are the left cells $\mathcal{L}\subseteq\mathcal{J}$ and
whose morphism categories $\cAJ(\mathcal{L}_{1},\mathcal{L}_{2})$ consist of $1$-morphisms
in $\cA_{\mathcal{L}^{\star}_{1}\cap\mathcal{L}_{2}}$, for the $\mathcal{H}$-cell
$\mathcal{L}^{\star}_{1}\cap\mathcal{L}_{2}$.

Analogously to Proposition \ref{proposition:bicomodules}, one can show that $\cBJ$ is locally
graded semisimple and that there is a biequivalence of $\mathbb{Z}$-finitary bicategories
\begin{gather*}
\cBJo\simeq\bigoplus_{t\in\mathbb{Z}}\cAJ\langle t\rangle,
\end{gather*}
which implies that there is a graded biequivalence
\begin{gather*}
\cBJ\simeq\cAJp.
\end{gather*}
Since there is also a graded biequivalence $\cS\text{-}\mathrm{gcell}_{\mathcal{J}}\simeq\cBJop$,
the result follows.
\end{proof}

\begin{remark}
The previous theorem shows that $\cS\text{-}\mathrm{gcell}_{\mathcal{J}}$ is graded
biequivalent to
$\cAJpop$ and hence locally graded semisimple. However the graded cell $2$-representations
of $\cS$ with apex $\mathcal{J}$, which form the objects of $\cS\text{-}\mathrm{gcell}_{\mathcal{J}}$,
are not semisimple themselves in general, as follows from Proposition \ref{prop:Frobenius}.

At the same time, we can analogously define the graded $2$-category $\cAJp\text{-}\mathrm{gcell}_{\mathcal{J}}$
as having objects $\mathcal{L}$ for left cells $\mathcal{L}\subset \mathcal{J}$, which
we now identify with the corresponding graded cell $2$-representations of $\cAJp$ (which are
graded semisimple), and with morphism categories given by the corresponding morphisms of
$2$-representations. Again, this is graded biequivalent to the bicategory with the same
objects, and the graded morphism category from $\mathcal{L}_{1}$ to $\mathcal{L}_{2}$ given by
$(\mathrm{A}_{d_{\mathcal{L}_{1}}})\mathrm{biinj}_{\ccAJp}
(\mathrm{A}_{d_{\mathcal{L}_{2}}}) $. Since the $\mathrm{A}_{d_{\mathcal{L}_{1}}}$ are
idempotents with the trivial coalgebra structure, this implies that
$\cAJp\text{-}\mathrm{gcell}_{\mathcal{J}}$ is graded biequivalent to
$(\cAJp)^{\mathrm{op}}$ and hence there is a graded biequivalence
\begin{gather}\label{eq:cell-case}
\cS\text{-}\mathrm{gcell}_{\mathcal{J}}\simeq
\cAJp\text{-}\mathrm{gcell}_{\mathcal{J}}.
\end{gather}
In Theorem \ref{theorem:main-local-ss} we will prove that
$\cS\text{-}\mathrm{gstmod}_{\mathcal{J}}$ is locally graded semisimple,
which is a generalization of Theorem \ref{theorem:local-ss}. The reason for including the latter theorem here is that the proof above, which does not depend on the results in Section \ref{section:main-theorem}, shows that both $2$-categories in \eqref{eq:cell-case} are graded biequivalent to $\cAJpop$, which is a nice result in its own right. The analogous result for $\cS\text{-}\mathrm{gstmod}_{\mathcal{J}}$ is more involved and, since we do not need it in this paper, we have omitted it.
\end{remark}


\section{The main theorem}\label{section:main-theorem}


In this section we will prove our main theorem:

\begin{theorem}\label{theorem:main}
There are biequivalences of graded $2$-categories
\begin{gather*}
\cSH\text{-}\mathrm{gstmod}_{\mathcal{H}}\simeq\cBH\text{-}\mathrm{gstmod}\simeq\cAH\text{-}\mathrm{stmod}^{\prime}.
\end{gather*}
\end{theorem}

\begin{remark}
To avoid confusion about the gradings,
recall that $\cBH\simeq \cAHp$ by Proposition \ref{proposition:bicomodules}. The
last biequivalence in Theorem \ref{theorem:main} then means that,
for every $\mathbf{M}\in\cAHp\text{-}\mathrm{gstmod}$, there is an $\mathbf{N}\in\cAH\text{-}\mathrm{stmod}$
such that
$\mathbf{M}\simeq\mathbf{N}^{\prime}$ in $\cAHp\text{-}\mathrm{gstmod}$, and this correspondence establishes
a bijection between the equivalence classes of objects in $\cAHp\text{-}\mathrm{gstmod}$ and $\cAH\text{-}\mathrm{stmod}$.
Moreover, for every pair $\mathbf{M}_{1},\mathbf{M}_{2}\in
\cAHp\text{-}\mathrm{gstmod}$, there is an
equivalence of graded categories $\mathrm{Hom}_{\ccAHp}
(\mathbf{M}_{1},\mathbf{M}_{2})\simeq
\mathrm{Hom}_{\ccA_{\mathcal{H}}}(\mathbf{N}_{1},\mathbf{N}_{2})^{\prime}$.

Note that this does \emph{not} mean
that $\cAHp\text{-}\mathrm{gstmod}$ and
$\cAH\text{-}\mathrm{stmod}$ are biequivalent as graded $2$-categories, because the morphism
categories of the former $2$-category are non-trivially graded, whereas those of the latter $2$-category
are trivially graded (meaning that all $2$-morphisms live in degree zero),
c.f. Remark \ref{rem:subtle}.
\end{remark}

The proof of Theorem \ref{theorem:main} is the content of Propositions \ref{prop:main-theorem-prop1}
and \ref{prop:main-theorem-prop2} in the next two subsections, each of
which proves one of the two biequivalences above, but let us first state and
prove a very important corollary.

\begin{corollary}\label{cor:gradedfiniteness}
For any finite Coxeter type, there are finitely many equivalence classes
of graded simple transitive $2$-representations of $\cS$.
\end{corollary}

\begin{proof}
Let $W$ be any finite Coxeter group and $\cS=\cS(W)$. By the composite of the biequivalences in
Theorem \ref{theorem:main}, there is a bijection between the set of
equivalence classes of graded simple transitive $2$-representations of $\cSH$ with
apex $\mathcal{H}$ and the set of equivalence classes of simple transitive birepresentations of
$\cAH$,
for any diagonal $\mathcal{H}$-cell of $\cS$. Since every $\cAH$ is pivotal fusion, the corollary now
follows from \cite[Corollary 9.1.6 and Proposition 3.4.6]{EGNO}, strong $\mathcal{H}$-reduction, see \cite[Theorem 15]{MMMZ} and \cite[Theorem 4.32]{MMMTZ2}),
and the fact that $\cS$ has only finitely many two-sided cells.
\end{proof}


\subsection{The right biequivalence in Theorem \ref{theorem:main}}\label{subsection:bh-reps}


\begin{lemma}\label{lemma:bijection2-1}
Every graded simple transitive $2$-representa\-tion of $\cBH$ is graded semisimple.
\end{lemma}

\begin{proof}
Let $\mathbf{M}$ be a graded simple transitive $2$-representation of $\cBH$.
We want to show that the radical of $\mathcal{M}$ is trivial, which
implies the claim. Recall that the radical $\mathcal{R}$ is the unique maximal nilpotent ideal of $\mathcal{M}$, which can be defined as the union of the radicals of all full
subcategories of $\mathcal{M}$ with finitely many objects. Since such a subcategory can also be
seen as a finite dimensional algebra, its radical can be defined as in \cite[Section 3.1]{DoKi}.
By \cite[Theorem 3.1.1]{DoKi}, this also implies that $\mathcal{M}/\mathcal{R}$ is semisimple. We first consider the $\cBHo$ $2$-representation $\mathbf{M}^{(0)}$,
and note that this does not contain any proper ideals, see the last paragraph of Subsection \ref{s:grreps}.

We then use Proposition \ref{proposition:bicomodules} to restrict
$\mathbf{M}^{(0)}$ to a $2$-representation of $\cAH$, and let $\mathbf{N}$
denote the transitive $2$-subrepresentation of $\cAH$ generated by a non-zero
indecomposable object $X$ in $\mathbf{M}^{(0)}(\varnothing)$.

Recall that $\cAH$ is pivotal fusion and therefore only has one two-sided cell,
namely $\mathcal{H}$, which contains the identity $1$-morphism $\mathrm{A}_{d}$.
Thus, Theorem \ref{thm:projapex} implies that $\mathrm{A}_{d}$ acts projectively, so
any simple transitive $2$-representation of $\cAH$ is semisimple.
This in turn implies that the radical $\mathrm{rad}(\mathbf{N}(\varnothing))$
is the unique
maximal ideal of $\mathbf{N}$.
By Proposition \ref{proposition:bicomodules}, the direct sums
\begin{gather*}
\bigoplus_{t\in\mathbb{Z}}\mathrm{rad}\big(\mathbf{N}(\varnothing)\langle t\rangle\big),
\quad
\bigoplus_{s\neq t\in\mathbb{Z}}\mathrm{Hom}\big(\mathbf{N}(\varnothing),\mathbf{N}(\varnothing)\langle t-s\rangle\big)
\end{gather*}
are both $\cBHo$-invariant, and hence so is their sum, which equals
$\mathrm{rad}(\mathbf{M}^{(0)}(\varnothing))$. Since $\mathbf{M}^{(0)}$ does not
contain any proper ideals, the latter radical is zero. Thus $\mathbf{M}^{(0)}$ is semisimple, and hence
$\mathbf{M}$ is graded semisimple.
\end{proof}

\begin{proposition}\label{prop:main-theorem-prop1}
There is a biequivalence of graded $2$-categories
\begin{gather*}
\cBH\text{-}\mathrm{gstmod}\simeq
\cAH\text{-}\mathrm{stmod}^{\prime}.
\end{gather*}
\end{proposition}

\begin{proof}
Given a simple transitive $2$-representation $\mathbf{N}$ of $\cAH$, we can choose a simple object $\mathrm{X}\in\mathbf{N}(\varnothing)$ to obtain
$\mathbf{N}\simeq\mathbf{inj}_{\ccA_{\mathcal{H}}}\big([\mathrm{X},\mathrm{X}]\big)$.
By Proposition \ref{proposition:bicomodules} and \cite[Theorem 9]{MMMT} and
\cite[Corollary 12]{MMMZ},
$\mathbf{M}:=\mathbf{inj}_{\ccBH}\big([\mathrm{X},\mathrm{X}]\big)\simeq\mathbf{N}^{\prime}$
is a graded simple transitive $2$-representation of $\cBH$. (To simplify
notation, we use the identification $[\mathrm{X},\mathrm{X}]=\Theta\big([\mathrm{X},
\mathrm{X}]\big)$.) This provides an assignment
$\cAH\text{-}\mathrm{stmod}^{\prime}\to\cBH\text{-}\mathrm{gstmod}$ on objects, and we claim that it is essentially surjective.

Indeed, let $\mathbf{M}$ be a graded simple transitive $2$-representation of $\cBH$. As
before, we consider the $\cBHo$ $2$-representation $\mathbf{M}^{(0)}$, which again does
not contain any proper ideals, see Subsection \ref{s:grreps}, and restrict it to $\cAH$.
Letting $X$ be a non-zero indecomposable object in $\mathbf{M}^{(0)}(\varnothing)$, we
let $\mathbf{N}$ denote the transitive $2$-subrepresentation of $\cAH$ generated by $X$.
By virtue of $\cAH$ being pivotal
fusion, $\mathbf{N}$ is simple transitive and, moreover, the subcategory
$\mathbf{N}(\varnothing)$ of $\mathbf{M}^{(0)}(\varnothing)$ is also semisimple by Lemma \ref{lemma:bijection2-1}.

Also by Lemma \ref{lemma:bijection2-1}, we have
\begin{gather}\label{eq:bijection2-1}
\mathbf{M}^{(0)}\simeq\bigoplus_{t\in\mathbb{Z}}\mathbf{N}\langle t\rangle
\end{gather}
as $2$-representations of $\cBHo$.
By \cite[Theorem 4.7]{MMMT} and \cite[Corollary 12]{MMMZ} (which in the
context of finitary semisimple categories are
the dual versions of \cite[Corollary 7.10.5 and Exercise 7.10.6]{EGNO})
and local semisimplicity of $\cAH$, we have
\begin{gather}\label{eq:bijection2-2}
\mathbf{N}\simeq\mathbf{inj}_{\ccAH}\big([\mathrm{X},\mathrm{X}]\big),
\end{gather}
where $[\mathrm{X},\mathrm{X}]\in\cAH$ is the cosimple coalgebra given by the
internal cohom construction, see Subsection \ref{s2.23}. Proposition \ref{proposition:bicomodules}
implies that it is also cosimple as a coalgebra in
$\cBHo$ and from \eqref{eq:bijection2-1} and \eqref{eq:bijection2-2} it follows
that
\begin{gather*}
\mathbf{M}^{(0)}\simeq\bigoplus_{t\in\mathbb{Z}}\mathbf{inj}_{\ccAH}\big([\mathrm{X},\mathrm{X}]\big)\langle t\rangle\simeq\mathbf{inj}_{\ccBHo}([\mathrm{X},\mathrm{X}])
\end{gather*}
as $2$-representations of $\cBHo$.
Moreover, $\cBH\simeq(\cBHo)^{\prime}$ and, by semisimplicity, we have
\begin{gather*}
\mathbf{M}\simeq\mathbf{N}^{\prime}\simeq\mathbf{inj}_{\ccBH}\big([\mathrm{X},\mathrm{X}]\big)
\end{gather*}
as graded finitary $2$-representations, and our assignment is essentially surjective

We further claim that any morphism $\Phi$ of $2$-representations of $\cAH$
between $\mathbf{N}_{1}$ and $\mathbf{N}_{2}$ extends in a unique way
to a morphism of graded $2$-representations of $\cBH$ between
$\mathbf{M}_{1}=\mathbf{N}_{1}^{\prime}$ and $\mathbf{M}_{2}=\mathbf{N}_{2}^{\prime}$.
To see this, note that any object of $\mathbf{M}_{1}(\varnothing)^{(0)}$
is a direct sum of objects of the form $X^{\oplus\mathsf{v}^{t}}$ for some
$X\in\mathbf{N}_{1}(\varnothing)$. Since $X^{\oplus\mathsf{v}^{t}}=\mathbf{M}_{1}(\mathbbm{1}^{\oplus\mathsf{v}^{t}}_{\varnothing})\,X$,
we obtain that
\begin{gather*}
\Phi(X^{\oplus\mathsf{v}^{t}})=
\Phi\big(\mathbf{M}_{1}(\mathbbm{1}^{\oplus\mathsf{v}^{t}}_\varnothing)X \big)\cong\mathbf{M}_{2}(\mathbbm{1}^{\oplus\mathsf{v}^{t}}_{\varnothing})\Phi(X)=\Phi(X)^{\oplus\mathsf{v}^{t}}\;\text{in}\;\mathbf{M}_{2}(\varnothing)^{(0)}.
\end{gather*}
Similarly, all components of a morphism $f$ in $\mathbf{M}_{1}(\varnothing)$ are scalar multiples of $\mathrm{id}_{X}^{\oplus\mathsf{v}^{t}}$
for indecomposable $X$, again by semisimplicity. Since
$\mathrm{id}_{X}^{\oplus\mathsf{v}^{t}}=\mathbf{M}_{1}(\mathrm{id}_{\mathbbm{1}_\varnothing}^{\oplus\mathsf{v}^{t}})_{X}$, we obtain
\begin{gather*}
\Phi(\mathrm{id}_{X}^{\oplus\mathsf{v}^{t}})
=\Phi\big(\mathbf{M}_{1}(\mathrm{id}_{\mathbbm{1}_\varnothing}^{\oplus\mathsf{v}^{t}})_{X}\big)
=\mathbf{M}_{2}(\mathrm{id}_{\mathbbm{1}_\varnothing}^{\oplus\mathsf{v}^{t}})_{\Phi(X)}
=\mathbf{M}_{2}(\mathrm{id}_{\mathbbm{1}_\varnothing})_{\Phi(X)}^{\oplus\mathsf{v}^{t}}
=\mathrm{id}_{\Phi(X)}^{\oplus\mathsf{v}^{t}},
\end{gather*}
which completes the proof of our claim. Recall that $1$-morphisms in
$\cAH\text{-}\mathrm{stmod}^{\prime}$ are of the form $(\Phi,s)$, where $\Phi$ is a morphism of $2$-representations of $\cAH$
between $\mathbf{N}_{1}$ and $\mathbf{N}_{2}$ and $s\in\mathbb{Z}$.
Extending $\Phi$ as above, we send $(\Phi,s)$ to $\Phi\langle s\rangle$.
Modifications extend similarly, and we obtain a graded pseudofunctor
$\cAH\text{-}\mathrm{stmod}^{\prime}\to\cBH\text{-}\mathrm{gstmod}$.

To check that the above graded functor
\begin{gather*}
\mathrm{Hom}_{\ccAH}(\mathbf{N}_{1},\mathbf{N}_{2})^{\prime}\to\mathrm{Hom}_{\ccBH}(\mathbf{M}_{1},\mathbf{M}_{2})
\end{gather*}
is an equivalence, let $\Phi\colon\mathbf{M}_{1}\to\mathbf{M}_{2}$ be a morphism
of graded $2$-representations for $\cBH$. Suppose we have defined $\mathbf{N}_{1}\simeq\mathbf{inj}_{\ccAH}\big([\mathrm{X}_{1},\mathrm{X}_{1}]\big),\mathbf{N}_{2}\simeq\mathbf{inj}_{\ccAH}\big([\mathrm{X}_{2},\mathrm{X}_{2}]\big)$
via a choice of objects $X_i\in\mathbf{M}_{i}(\varnothing)$. If $\Phi(X_{1})\in\mathbf{N}_{2}(\varnothing)\langle t\rangle$,
it is easy to see that $\Phi\langle -t\rangle$ restricts to a morphism of $2$-representations $\Psi\colon\mathbf{N}_{1}\to\mathbf{N}_{2}$.
Modifications restrict naturally. It is immediate that these assignments provide a quasi-inverse to the functor $\mathrm{Hom}_{\ccAH}(\mathbf{N}_{1},\mathbf{N}_{2})^{\prime}\to\mathrm{Hom}_{\ccBH}(\mathbf{M}_{1},\mathbf{M}_{2})$.
\end{proof}


\subsection{The left biequivalence in Theorem \ref{theorem:main}}\label{subsection:a-theorem}


A functor between two finitary categories is called \emph{injective} if its
extension to the injective abelianizations of those categories is an injective object in the
category of left exact functors between the said abelianizations.
Further, a morphism of finitary  $2$-representations is injective if its underlying functors are
injective functors. For a finitary $2$-category
$\cC$ and two finitary $2$-representations $\mathbf{M}$, $\mathbf{N}$ of $\cC$, we denote
by $\mathrm{Hom}_{\ccC}^{\mathrm{inj}}(\mathbf{M},\mathbf{N})$ the category of injective
morphisms of $2$-representations. Moreover, considering injective endomorphisms of $\mathbf{M}$, we obtain the $2$-semicategory $\cEnd^{\mathrm{inj}}_{\ccC}(\mathbf{M})$.
We further denote by $\mathrm{End}^{\mathrm{inj},r}_{\ccC}(\mathbf{M})$
the additive closure of the injective endomorphisms and the identity morphism
of $\mathbf{M}$. This category also defines the $2$-category $\cEnd^{\mathrm{inj},r}_{\ccC}(\mathbf{M})$
whose single object is $\mathbf{M}$.

For future use, we record the following lemma.

\begin{lemma}\label{injadj}
Let $\mathrm{F}\colon\mathcal{B}\to\mathcal{C}$ be an injective functor and assume that both $ \mathcal{B}$ and $\mathcal{C}$ are finitary
Frobenius categories. Then $\mathrm{F}$ is isomorphic to tensoring with a projective bimodule and the collection of such injective functor is closed under adjunctions.
\end{lemma}

\begin{proof}
Assume without loss of generality that $\underline{\mathcal{B}}\simeq\mathsf{B}\text{-}\mathrm{mod}$ and  $\underline{\mathcal{C}}\simeq\mathsf{C}\text{-}\mathrm{mod}$ for self-injective algebras $\mathsf{B}$ and $\mathsf{C}$, respectively.
Any left exact functor from $\mathsf{B}\text{-}\mathrm{mod}$ to  $\mathsf{C}\text{-}\mathrm{mod}$ is isomorphic to $\mathrm{Hom}_{\mathsf{B}} (X,{}_{-})$ for some $\mathsf{B}$-$\mathsf{C}$-bimodules $X$. The morphism space between such functors $\mathrm{Hom}_{\mathsf{B}}(X^{\prime},{}_{-})$ and $\mathrm{Hom}_{\mathsf{B}}(X,{}_{-})$ is given by $\mathrm{Hom}_{\mathsf{B}\text{-}\mathsf{C}}(X,X^{\prime})$, thus $\mathrm{Hom}_{\mathsf{B}}(X,{}_{-})$ is injective if and only if $X$ is a projective bimodule. In this case, using
$\mathrm{Hom}_{\mathsf{B}}(\mathsf{B}\otimes_{\Bbbk}\mathsf{C},{}_{-})\cong \mathsf{C}^{\star}\otimes_{\Bbbk}\mathsf{B}\otimes_{\mathsf{B}}{}_{-},$
together with additivity of $\mathrm{Hom}$ and self-injectivity of $\mathsf{C}$, we see that injective functors correspond to tensoring with projective bimodules, and that these are, moreover, closed under taking adjoints.
\end{proof}

Before we can prove the left biequivalence in Theorem \ref{theorem:main}, we need to recall the
appropriate application of the double centralizer theorem \cite[Theorem 5.2]{MMMTZ2} to the
case of the cell $2$-representation $\mathbf{C}_\mathcal{H}$ of $\cSH$.
The $2$-actions of $\cSH$ and
$\cEnd_{\ccSH}(\CH)$ on $\CH$ commute by definition, this implies that there is the canonical
(graded) $2$-functor
\begin{gather*}
\mathrm{can}\colon\cSH\to\cEnd^{\mathrm{ex}}_{\ccEnd_{\cccSH}(\CH)}(\CH).
\end{gather*}

The following theorem is a special, but graded, case of \cite[Theorem 5.2]{MMMTZ2}.

\begin{theorem}\label{thm:double-centralizer}
The canonical $2$-functor $\mathrm{can}$ is fully faithful on $2$-morphisms and essentially surjective on
$1$-morphisms when restricted to
$\mathrm{add}(\mathcal{H})$ and corestricted to
$\cEnd^{\mathrm{inj}}_{\ccEnd_{\cccSH}(\CH)}(\CH)$.
\end{theorem}

We are now ready to prove the existence of the left graded biequivalence in Theorem \ref{theorem:main}.

\begin{proposition}\label{prop:main-theorem-prop2}
There is a biequivalence of graded $2$-categories
\begin{gather*}
\cSH\text{-}\mathrm{gstmod}_{\mathcal{H}}\simeq\cBH\text{-}\mathrm{gstmod}.
\end{gather*}
\end{proposition}

\begin{proof}
Consider the following graded fiat $2$-category $\hat{\cSH}$ with two objects, denoted
$\mathtt{c}$ and $\mathtt{b}$. We identify the object $\mathtt{c}$ with
the graded category $\mathcal{C}:=\mathbf{C}_{\mathcal{H}}(\varnothing)$, i.e.
the underlying graded category of the cell $2$-representation of $\cSH$ with
apex $\mathcal{H}$, and the object $\mathtt{b}$ with the graded category
$\mathcal{B}:=\cBH(\varnothing,\varnothing)$, i.e. the graded category underlying
the principal birepresentation of $\cBH$.
Recall that, due to the graded equivalence $\cEnd_{\ccSH}(\CH)\simeq\cBHop$ in Corollary \ref{cor:End-Bicomod}, the $2$-category $\cEnd_{\ccSH}(\CH)$ acts, by graded endofunctors and homogeneous degree-zero natural transformations, on both
$\mathcal{C}$ and $\mathcal{B}$.

The graded morphism categories of $\hat{\cSH}$ are defined as
\begin{align*}
\hat{\cSH}(\mathtt{b},\mathtt{b})
&:=\mathrm{Hom}_{\ccEnd_{\cccSH}(\CH)}(\mathcal{B},\mathcal{B}),
\\
\hat{\cSH}(\mathtt{b},\mathtt{c})
&:=\mathrm{Hom}^{\mathrm{inj}}_{\ccEnd_{\cccSH}(\CH)}(\mathcal{B},\mathcal{C}),
\\
\hat{\cSH}(\mathtt{c},\mathtt{b})
&:=\mathrm{Hom}^{\mathrm{inj}}_{\ccEnd_{\cccSH}(\CH)}(\mathcal{C},\mathcal{B}),
\\
\hat{\cSH}(\mathtt{c},\mathtt{c})
&:=\mathrm{Hom}^{\mathrm{inj},r}_{\ccEnd_{\cccSH}(\CH)}(\mathcal{C},\mathcal{C}).
\end{align*}
Note that
\begin{gather*}
\mathrm{Hom}_{\ccEnd_{\cccSH}(\CH)}(\mathcal{B},\mathcal{B})=  \mathrm{Hom}^{\mathrm{inj}}_{\ccEnd_{\cccSH}(\CH)}(\mathcal{B},\mathcal{B})
=
\mathrm{Hom}^{\mathrm{inj},r}_{\ccEnd_{\cccSH}(\CH)}(\mathcal{B},\mathcal{B})
\end{gather*}
since $\mathcal{B}$ is graded semisimple by Proposition \ref{proposition:bicomodules}.
The graded $2$-categorical structure on
$\hat{\cSH}$ is given by composition in the same way as in e.g.
\cite[Subsection 2.3]{MM3}, and the fiat structure is defined by taking adjoints, which
preserves injective functors, since the underlying categories are Frobenius (see Lemma \ref{injadj}).

Next, note that the regular birepresentation of $\cBH$ defines a natural graded embedding
\begin{gather}\label{eq:grademb}
\mathbf{B}_{\mathcal{H}}\colon \cBH\to\hat{\cSH}(\mathtt{b},\mathtt{b}),
\end{gather}
as it clearly commutes with the right $2$-action of $\cBH$ on $\mathcal{B}$. Locally
graded semisimplicity of $\cBH$ implies that this graded embedding is a graded biequivalence.

By Theorem \ref{thm:double-centralizer}, the natural graded embedding of $\cSH$
into $\hat{\cSH}(\mathtt{c},\mathtt{c})$ is fully faithful on $2$-morphisms
and essentially surjective on $1$-morphisms when restricted to $\mathrm{add}(\mathcal{H})$
and corestricted to $\mathrm{Hom}^{\mathrm{inj}}_{\ccEnd_{\cccSH}(\CH)}(\mathcal{C},\mathcal{C})$.
In particular, the endomorphism $2$-category of the object $\mathtt{c}$ has two $\mathcal{H}$-cells:
that containing the identity and that consisting of indecomposable objects in
$\mathrm{Hom}^{\mathrm{inj}}_{\ccEnd_{\cccSH}(\CH)}(\mathcal{C},\mathcal{C})$,
which we can identify with $\mathcal{H}$.
Theorem \ref{thm:double-centralizer} then implies that there is a biequivalence of graded
$2$-categories
\begin{gather*}
\cSH\text{-}\mathrm{gstmod}_{\mathcal{H}}\simeq\hat{\cSH}(\mathtt{c},\mathtt{c})\text{-}
\mathrm{gstmod}_{\mathcal{H}}.
\end{gather*}

The statement now follows by
strong $\mathcal{H}$-reduction, see Theorem \ref{theorem:strongH}. To see this, note that
$\hat{\cSH}$ has only one non-trivial two-sided cell (provided $\mathcal{H}$
is not trivial, of course), which we denote by
$\mathcal{J}$ and is formed by the isomorphism classes of all indecomposable $1$-morphisms that are not
isomorphic to $\mathbbm{1}_{\mathtt{c}}$. This two-sided cell contains two diagonal $\mathcal{H}$-cells, the one
formed by the isomorphism classes of the indecomposable $1$-morphisms in
$\hat{\cSH}(\mathtt{b},\mathtt{b})$ and the other formed by the isomorphism classes of the indecomposable $1$-morphisms
in $\mathrm{Hom}^{\mathrm{inj}}_{\ccEnd_{\cccSH}(\CH)}(\mathcal{C},\mathcal{C})$, which can both be identified with
$\mathcal{H}$ due to Proposition \ref{proposition:bicomodules} and Theorem \ref{thm:double-centralizer}.
By the above and Theorem \ref{theorem:strongH}, there are biequivalences of graded $2$-categories
\begin{gather*}
\cSH\text{-}\mathrm{gstmod}_{\mathcal{H}}
\simeq
\hat{\cSH}\text{-}\mathrm{gstmod}_{\mathcal{J}}
\simeq
\cBH\text{-}\mathrm{gstmod}_{\mathcal{H}}=\cBH\text{-}\mathrm{gstmod},
\end{gather*}
where the middle biequivalence is induced by $\mathbf{B}_{\mathcal{H}}$
in \eqref{eq:grademb} and the last equality uses that $\cBH$ only has one $\mathcal{H}$-cell, which follows
from Remark \ref{rem:fusioncell} and Proposition \ref{proposition:bicomodules}.
This completes the proof.
\end{proof}

\begin{remark}
The construction of the $2$-category $\hat{\cSH}$ in the proof of Proposition
\ref{prop:main-theorem-prop2} is in the spirit of what is known as a \emph{Morita context}
in the tensor category literature, see \cite[Remark 3.18 and Section 4]{Mu}.
\end{remark}


\subsection{Local graded semisimplicity}


In this subsection we generalize Theorem \ref{theorem:local-ss}.

\begin{theorem}\label{theorem:main-local-ss}
The $2$-category $\cS\text{-}\mathrm{gstmod}_{\mathcal{J}}$ is locally graded semisimple.
\end{theorem}

\begin{proof}
By Theorem \ref{theorem:strongH} and Theorem \ref{theorem:main},
it suffices to prove that $\cAH\text{-}\mathrm{stmod}$ is locally semisimple.
Noting that $\cAH$ is pivotal fusion and that, therefore, its simple transitive $2$-representations
coincide with its so-called \emph{indecomposable exact module categories}, this is
a consequence of \cite[Theorems 2.15 and 2.16]{ENO}.
\end{proof}


\subsection{The ungraded case}\label{subsection:ungraded}


In this section, let $\cS=\cS(W,S)$ be the ungraded category of Soergel
bimodules for a finite Coxeter system $(W,S)$. This category is defined just as its graded counterpart, but
ignores the grading. The indecomposables in $\cS$
are still indexed by $W$ and will be denoted by
$\mathrm{C}_{x}$ as well. This follows from the fact that for any object $X\in\cS$, the
endomorphism ring $\mathrm{End}_{\ccS}(\mathrm{X})$ is local in the graded setting if and only
if it is local in the ungraded setting. The character isomorphism in the
Soergel--Elias--Williamson categorification theorem then has a $\mathbb{Z}$-linear counterpart
in the ungraded setting which sends $[\mathrm{C}_{x}]\in[\cS]_{\oplus}$ to
$c_{x}\in\mathbb{Z}[W]$, where $c_{x}$ is the
image of the Kazhdan--Lusztig basis element under the $\mathbb{Z}$-linear map
$\mathsf{H}(W)\to\mathbb{Z}[W]$ which sends the standard basis element
$T_{w}$ to $w$, for $w\in W$, and $\mathsf{v}$ to $1$.

Recall that $\cSH\text{-}\mathrm{stmod}_{\mathcal{H}}$ denotes the $2$-category
of all simple transitive $2$-representa\-tions of $\cSH$ with apex $\mathcal{H}$
without any grading assumptions. Any simple transitive graded $2$-representation of the
graded $\cSH$ with apex $\mathcal{H}$ remains simple transitive in the ungraded setting,
but there is no a priori reason why the forgetful $2$-functor
\begin{gather*}
F\colon\cSH\text{-}\mathrm{gstmod}_{\mathcal{H}}\to
\cSH\text{-}\mathrm{stmod}_{\mathcal{H}}
\end{gather*}
should be essentially surjective. In principle, there could be ungradable
simple transitive $2$-representations of $\cSH$ with apex $\mathcal{H}$.
This was actually an open question until now,
e.g. for finite dihedral Coxeter type it was unknown if the bicolored ADE classification of the
graded simple transitive $2$-representations of
$\cS$ with subregular apex in \cite{KMMZ}
and \cite{MT} remained valid in the ungraded case. We can now answer that
question affirmatively, because $F$ is indeed a biequivalence for any finite Coxeter type.

This follows from the fact that in the ungraded case the double centralizer property
from Theorem \ref{thm:double-centralizer} still holds (note that \cite[Theorem 5.2]{MMMTZ2}
was formulated and proved in the ungraded setting)
and that there is a biequivalence
\begin{gather}\label{eq:ungraded1}
\cEnd_{\ccSH}(\CH)\simeq\cAHop.
\end{gather}
The proof of Theorem \ref{thm:double-centralizer} does not use the grading in an
essential way, so that theorem remains true in the ungraded setup. The arguments which
prove the existence of the biequivalence in \eqref{eq:ungraded1} in the ungraded case, are
exactly the same as the ones which prove Proposition \ref{proposition:bicomodules}.
The only difference is that there is no translation in the ungraded setting,
so the trivial $\mathbb{Z}$-cover of
$\cAH$ ``collapses'' to $\cAH$ itself.
Just as in the proof of Theorem \ref{theorem:main}, the double centralizer theorem and
the biequivalence in \eqref{eq:ungraded1} imply that there is a biequivalence
\begin{gather}\label{eq:ungradedmain}
\cSH\text{-}\mathrm{stmod}_{\mathcal{H}}\simeq\cAH\text{-}\mathrm{stmod},
\end{gather}
by Theorem \ref{theorem:strongH} (which also holds both in the graded and the ungraded
setting). This shows that the forgetful $2$-functor $F$ is indeed a biequivalence.

As in Corollary \ref{cor:gradedfiniteness}, the biequivalence in \eqref{eq:ungradedmain} implies
that, for any finite Coxeter type, there are finitely many equivalence classes
of simple transitive $2$-representations of $\cS$.


\section{Classification results}\label{appendix}


\subsection*{The asymptotic bicategory and its $2$-representations}\label{appendix.1}


Recall that, by Proposition \ref{prop:fusion}, $\cAH$ is a pivotal fusion bicategory
and thus, $\mathcal{H}$
is its only cell and all of its simple transitive $2$-representations are semisimple.
Further, the notion of a simple transitive $2$-representation agrees with the notion
of an indecomposable exact module category, which is often used as the terminology in \cite{BFO}, \cite{EGNO}, \cite{KO},
\cite{Os4}, \cite{Os2}, \cite{Os3} below.

Up to a handful of exceptions,
the asymptotic bicategory $\cAH$ comes in three flavors and for all of them
a classification of simple transitive $2$-representations is known, as we will summarize now (giving more details below).
Recall that $\cR\mathrm{ep}(G)=\cR\mathrm{ep}(G,\mathbb{C})$ is the pivotal fusion $2$-category of finite dimensional $G$-modules.
Let $\cV\mathrm{ect}(G)=\cV\mathrm{ect}_{\mathbb{C}}(G)$ denote the pivotal fusion $2$-category of finite dimensional $G$-graded $\mathbb{C}$-vector spaces
($\cV\mathrm{ect}=\cV\mathrm{ect}(1)$ are plain finite dimensional vector spaces), and
$\cS\cO(3)_{k}$ the pivotal fusion $2$-category of complex,
finite dimensional representations
of quantum $\mathfrak{so}_{3}$ semisimplified at level $k$ and without
twists, see e.g. \cite[Examples 2.3.4 and 8.18.5]{EGNO}.

\begin{enumerate}[label=(\alph*)]

\item \textit{Finite Weyl type (excluding $G_{2}$): generic case.}
Up to three exceptions in types $E_{7}$ and $E_{8}$, see $($b$)$,
for each two-sided cell $\mathcal{J}$ there exists a diagonal $\mathcal{H}$-cell
$\mathcal{H}$ such that
$\cAH\simeq\cV\mathrm{ect}\big((\mathbb{Z}/2\mathbb{Z})^{k}\big)$ for some $k\in\mathbb{N}$,
or $\cAH\simeq\cR\mathrm{ep}(G)$ for $G$ being $S_{3}$, $S_{4}$ or $S_{5}$.
This follows from (the arguments in the proof of) \cite[Theorem 4]{BFO}.

For all of these, the classification of the associated simple transitive $2$-representa\-tion works as follows.
Let $G$ be a finite group and let $\Omega(G)$ denote the set of subgroups of $G$ up to conjugacy,
$K$ a choice of representative of $[K]\in\Omega(G)$, and $H^{2}(K,\mathbb{C}^{\times})$ the second
group cohomology of $K$ with values in $\mathbb{C}^{\times}=\mathbb{C}\setminus\{0\}$, whose
non-trivial generators are called \textit{Schur multipliers}.
By e.g. \cite[Example 7.4.10 and Corollary 7.12.20]{EGNO}, we have
\begin{gather*}
\scalebox{0.95}{$
\left\{
\begin{gathered}
\text{equivalence classes of simple transitive}
\\
\text{$2$-representations of $\cV\mathrm{ect}(G)$ or $\cR\mathrm{ep}(G)$}
\end{gathered}
\right\}
\stackrel{1{:}1}{\longleftrightarrow}
\left\{
\begin{gathered}
\big([K],\varpi\big)\mid[K]\in\Omega(G),
\\
\varpi\in H^2(K,\mathbb{C}^{\times})
\end{gathered}
\right\}$}
.
\end{gather*}
The simple transitive $2$-representations of
$\cV\mathrm{ect}(G)$ have rank $\#G/\#K$ and the ones for $\cR\mathrm{ep}(G)$ are
the $\varpi$-twisted representation categories $\cR\mathrm{ep}^{\varpi}(K)$ (in particular,
their rank is equal to the rank of the character ring of $K$ for trivial $\varpi$).
Hence, we need to analyze the simple transitive $2$-representations of $\cV\mathrm{ect}(G)$ or
$\cR\mathrm{ep}(G)$, which are given
by (conjugacy classes of) subgroups of $K\subset G$, their number $\#$,
and the Schur multipliers in $H^2(K,\mathbb{C}^{\times})$ of these subgroups. We additionally
list their rank $\mathrm{rk}$. Listing the data that we need is easy
(calculating the subgroups and their numbers for $(\mathbb{Z}/2\mathbb{Z})^{k}$ is a pleasant
exercise, while the
Schur multipliers of these subgroups were already determined by Schur, see e.g. \cite[Theorem 4]{Ber}
for a more modern reference; the data for the other three cases, $S_{3}$, $S_{4}$ and $S_{5}$,
can be calculated by computer). In the row $\mathrm{rk}$ in \eqref{tables}, two entries
correspond to two different simple transitive 2-representations, one for each
Schur multiplier with the one for the trivial Schur multiplier listed first.
\begin{gather}\label{tables}
\begin{gathered}
\,
\xy
(0,0)*{
\renewcommand*{\arraystretch}{1.5}
\begin{tabular}{C||C}
K & (\mathbb{Z}/2\mathbb{Z})^l
\\
\hline
\text{\#} & \binom{k}{l}
\\
\hline
H^2 & (\mathbb{Z}/2\mathbb{Z})^{l(l-1)/2}
\\
\hline
\mathrm{rk} & k/l
\end{tabular}};
(0,17)*{\scalebox{0.75}{\fcolorbox{black}{gray!10}{$\cV\mathrm{ect}\big((\mathbb{Z}/2\mathbb{Z})^{k}\big)$}}};
(0,15)*{\phantom{.}};
(0,-15)*{\phantom{.}};
\endxy
\;,\qquad\quad
\xy
(0,0)*{
\renewcommand*{\arraystretch}{1.5}
\begin{tabular}{C||C|C|C|C}
K & \phantom{.}1\phantom{.} & \mathbb{Z}/2\mathbb{Z} & \mathbb{Z}/3\mathbb{Z} & S_{3}
\\
\hline
\text{\#} & 1 & 1 & 1 & 1
\\
\hline
H^2 & 1 & 1 & 1 & 1
\\
\hline
\mathrm{rk}  & 1 & 2 & 3 & 3
\end{tabular}};
(0,17)*{\text{\tiny\fcolorbox{black}{gray!10}{$\ccR\mathrm{ep}(S_{3})$}}};
(0,15)*{\phantom{.}};
(0,-15)*{\phantom{.}};
\endxy
,
\\
\xy
(0,0)*{
\renewcommand*{\arraystretch}{1.5}
\begin{tabular}{C||C|C|C|C|C|C|C|C|C}
K & \phantom{.}1\phantom{.} & \mathbb{Z}/2\mathbb{Z} & \mathbb{Z}/3\mathbb{Z} & \mathbb{Z}/4\mathbb{Z}
& (\mathbb{Z}/2\mathbb{Z})^2 & S_{3} & D_{4} & A_{4} & S_{4}
\\
\hline
\text{\#} & 1 & 2 & 1 & 1 & 2 & 1 & 1 & 1 & 1
\\
\hline
H^2 & 1 & 1 & 1 & 1 & \mathbb{Z}/2\mathbb{Z} & 1 & \mathbb{Z}/2\mathbb{Z} & \mathbb{Z}/2\mathbb{Z} & \mathbb{Z}/2\mathbb{Z}
\\
\hline
\mathrm{rk}  & 1 & 2 & 3 & 4 & 4,1 & 3 & 5,2 & 4,3 & 5,3
\end{tabular}};
(0,17)*{\text{\tiny\fcolorbox{black}{gray!10}{$\ccR\mathrm{ep}(S_{4})$}}};
(0,15)*{\phantom{.}};
(0,-15)*{\phantom{.}};
\endxy
,
\\
\scalebox{0.55}{$\xy
(0,0)*{
\renewcommand*{\arraystretch}{1.5}
\begin{tabular}{C||C|C|C|C|C|C|C|C|C|C|C|C|C|C|C|C}
K & \phantom{.}1\phantom{.} & \mathbb{Z}/2\mathbb{Z} & \mathbb{Z}/3\mathbb{Z} & \mathbb{Z}/4\mathbb{Z}
& (\mathbb{Z}/2\mathbb{Z})^2 & \mathbb{Z}/5\mathbb{Z} & S_{3} & \mathbb{Z}/6\mathbb{Z} & D_{4} & D_{5}
& A_{4} & D_{6} & GA(1,5) & S_{4} & A_{5} & S_{5}
\\
\hline
\text{\#} & 1 & 2 & 1 & 1 & 2 & 1 & 2 & 1 & 1 & 1 & 1 & 1 & 1 & 1 & 1 & 1
\\
\hline
H^2 & 1 & 1 & 1 & 1 & \mathbb{Z}/2\mathbb{Z} & 1 & 1 & 1 & \mathbb{Z}/2\mathbb{Z} & \mathbb{Z}/2\mathbb{Z} & \mathbb{Z}/2\mathbb{Z}
& \mathbb{Z}/2\mathbb{Z} & 1 & \mathbb{Z}/2\mathbb{Z} & \mathbb{Z}/2\mathbb{Z} & \mathbb{Z}/2\mathbb{Z}
\\
\hline
\mathrm{rk}  & 1 & 2 & 3 & 4 & 4,1 & 5 & 3 & 6 & 5,2 & 4,2 & 4,3
& 6,3 & 5 & 5,3 & 5,4 & 7,5
\end{tabular}};
(0,17)*{\text{\tiny\fcolorbox{black}{gray!10}{$\ccR\mathrm{ep}(S_{5})$}}};
(0,15)*{\phantom{.}};
(0,-15)*{\phantom{.}};
\endxy$}
.
\end{gathered}
\end{gather}
(Here $GA(1,5)$ is the general affine group of rank one over $\mathbb{F}_{5}$.)

\item \textit{Finite Weyl type: exceptional case.} Type $E_{7}$ contains one
and type $E_{8}$ two so-called \emph{exceptional cells}. For these, by \cite[Theorem 1.1]{Os3}, we have
$\cAH\simeq\cV\mathrm{ect}^{\varsigma}(\mathbb{Z}/2\mathbb{Z})$,
having its $2$-structure twisted by the non-trivial element
$\varsigma$ in the third group cohomology group
$H^3(\mathbb{Z}/2\mathbb{Z},\mathbb{C}^{\times})\cong\mathbb{Z}/2\mathbb{Z}$. In this case $\cAH$ has only one associated simple
transitive $2$-representation, which is the regular $2$-representation of rank $2$, see e.g. \cite[Theorem 3.1]{Os4}.

\item \textit{Finite dihedral type (including $G_{2}$).} We have $\cAH\simeq\cV\mathrm{ect}$
for the cells containing the identity element or the longest element, or
$\cAH\simeq\cS\cO(3)_{k}$ for the middle cell, by
\cite[Theorem 2.15]{El1}. By e.g. \cite[Theorem 6.1]{KO} and \cite[Theorem 6.1]{Os2}, we have
\begin{gather}\label{dihedral}\hspace*{0.5cm}
\Bigg\{
\begin{gathered}
\text{equivalence classes of simple transitive}
\\
\text{$2$-representations of $\cS\cO(3)_{k}$}
\end{gathered}
\Bigg\}
\stackrel{1{:}1}{\longleftrightarrow}
\Bigg\{
\begin{gathered}
\text{bicolored ADE diagrams}
\\
\text{with Coxeter number $k+2$}
\end{gathered}
\Bigg\}
.
\end{gather}
The corresponding simple transitive $2$-representations have rank equal
to the number of vertices of the associated ADE diagram.

\item \textit{Types $H_{3}$ and $H_{4}$.} We do not know what $\cAH$ is in general.
For details see below.

\end{enumerate}


\subsection*{What the main theorem covers}\label{appendix.last}


Recall that any (graded) simple transitive $2$-repre\-sentation of $\cS$ has an apex $\mathcal{J}$ in $W$.
Then, Theorem \ref{theorem:main} together with Theorem \ref{theorem:strongH}, implies that
\begin{gather*}
\cS\text{-}\mathrm{gstmod}_{\mathcal{J}}\simeq
\cAH\text{-}\mathrm{stmod}^{\prime}
\end{gather*}
holds for any diagonal $\mathcal{H}$-cell $\mathcal{H}\subset\mathcal{J}$. Thus, we get
\begin{gather*}
\left\{
\begin{gathered}
\text{equivalence classes of (graded)}
\\
\text{simple transitive $2$-representations}
\\
\text{of $\cS$ with apex $\mathcal{J}$}
\end{gathered}
\right\}
\stackrel{1{:}1}{\longleftrightarrow}
\left\{
\begin{gathered}
\text{equivalence classes of}
\\
\text{simple transitive $2$-representations}
\\
\text{of $\cAH$ (including grading shifts)}
\end{gathered}
\right\}
.
\end{gather*}
(Note that Corollary \ref{cor5.42} also gives us the rank of
the simple transitive $2$-representa\-tions of $\cSH$
associated to the ones from $\cAH$. However, the corresponding
simple transitive $2$-representations
for $\cS$ will have greater rank in general.)
Thus, the above shows that
only certain cells in Coxeter types $H_{3}$ and $H_{4}$ -- most prominently,
the cell \eqref{eq:h4} in type $H_{4}$ given below -- would remain open
with respect to a complete classification of (graded) simple transitive $2$-representations of $\cS$.
For all other cases, Theorem \ref{theorem:main} gives a complete classification and parametrization of the (graded) simple
transitive $2$-representations of $\cS$, as we will summarize now.


\subsection*{What we cover}\label{appendix.0}


Let us give some details of what is covered by the results in this paper.
We say that a
type is \emph{``done''} if we can identify $\cAH$
and parameterize its simple transitive $2$-representations for at least one diagonal
$\mathcal{H}$-cell $\mathcal{H}\subset\mathcal{J}$.

For this purpose, we use what we call a \emph{full cell matrix}:
\begin{gather}\label{eq:cellmatrix}
\renewcommand{\arraystretch}{1.25}
\scalebox{0.75}{$\begin{tabular}{|C|C|C|C|C|}
\hline
\cellcolor{blue!25} 4_{5,5} & 1_{5,5} & 1_{5,20} & 2_{5,25} & 2_{5,25}
\\
\hline
1_{5,5} & \cellcolor{blue!25} 4_{5,5} & 1_{5,20} & 2_{5,25} &  2_{5,25}
\\
\hline
1_{20,5} & 1_{20,5} & \cellcolor{blue!25} 4_{20,20} & 2_{20,25} & 2_{20,25}
\\
\hline
2_{25,5} & 2_{25,5} & 2_{25,20} & \cellcolor{blue!25} 4_{25,25} & 1_{25,25}
\\
\hline
2_{25,5} & 2_{25,5} & 2_{25,20} & 1_{25,25} & \cellcolor{blue!25} 4_{25,25}
\\
\hline
\end{tabular}$}.
\end{gather}
Here we indicate the number of elements in
left or right cells, where e.g. $2_{20,25}$ is to be understood as
a $20$-by-$25$ matrix containing only the entry $2$ (thus, having $1000$ elements),
i.e. it is $2Id_{20,25}$, where $Id_{20,25}$ is the $20$-by-$25$ identity matrix.
The shaded boxes are (matrices of) diagonal $\mathcal{H}$-cells.

The full cell matrices are block matrices, but we also view them as matrices
containing only the scalars $n$ in $nId_{20,25}$ and call them \emph{cell matrices}.
The difference between the two is that
the cell matrix actually encodes equivalence classes of cell $2$-representation:
each column of the full cell matrix corresponds to a (left) cell $2$-representation,
while the columns of the cell matrix itself correspond to (left) cell $2$-representation up to equivalence.

Here are a few more examples, which have appeared under various names in the
literature; the first and second cases are $k=0$ respectively $k=1$ below.
\begin{gather}\label{eq:matrix-ab}
\begin{aligned}
\text{strongly regular:}\quad&
\scalebox{0.75}{$\fcolorbox{black}{blue!25}{$1_{a,a}$}$},
&\cAH&\simeq\cV\mathrm{ect},
\\[2pt]
\hline
\\[-13pt]
\text{nice:}\quad&
\renewcommand{\arraystretch}{1.25}
\scalebox{0.75}{$\begin{tabular}{|C|C|}
\hline
\cellcolor{blue!25} 2_{b,b} & 1_{c,b}
\\
\hline
1_{b,c} & \cellcolor{blue!25} 2_{c,c}
\\
\hline
\end{tabular}$},
&\cAH&\simeq\cV\mathrm{ect}(\mathbb{Z}/2\mathbb{Z}),
\\[2pt]
\hline
\\[-13pt]
\text{exceptional:}\quad&
\scalebox{0.75}{$\fcolorbox{black}{blue!25}{$2_{d,d}$}$},
&\cAH&\simeq\cV\mathrm{ect}^{\varsigma}(\mathbb{Z}/2\mathbb{Z}),
\\[2pt]
\hline
\\[-13pt]
\text{dihedral:}\quad&
\begin{aligned}
&\renewcommand{\arraystretch}{1.25}
\scalebox{0.75}{$\begin{tabular}{|C|C|}
\hline
\cellcolor{blue!25} \tfrac{m{-}1}{2} & \tfrac{m{-}1}{2}
\\
\hline
\tfrac{m{-}1}{2} & \cellcolor{blue!25} \tfrac{m{-}1}{2}
\\
\hline
\end{tabular}$},
&m\text{ odd},
\\[0pt]
&\renewcommand{\arraystretch}{1.25}
\scalebox{0.75}{$\begin{tabular}{|C|C|}
\hline
\cellcolor{blue!25} \tfrac{m}{2} & \tfrac{m{-}2}{2}
\\
\hline
\tfrac{m{-}2}{2} & \cellcolor{blue!25} \tfrac{m}{2}
\\
\hline
\end{tabular}$},
&m\text{ even},
\end{aligned}
&\cAH&\simeq\cS\cO(3)_{m-2},
\end{aligned}
\end{gather}
where $a^2$, $2(b^2+c^2+bc)$, $2d^2$ or $2(m-1)$, respectively, is the size of the cell in question.
(Knowing the size of the cells, one can recover $a,b,c,d$ since there is always a unique solution
in positive integers.)
Note for example that in the nice case there are only two columns in the
cell matrix, thus, only two cell $2$-representations up to equivalence.
However, there are $b+c$ cell $2$-representations as encoded in the columns of the full cell matrix.

\subsubsection*{Type $A_n$}\label{appendix.2}

This type is \textbf{done} for all $n$, since all
cells are strongly regular. In this case there is one equivalence class of simple transitive $2$-representation per
two-sided cell, all of which are cell $2$-representations.

\subsubsection*{Type $B_n$}\label{appendix.3}

This type is \textbf{done} for all $n$:

\begin{enumerate}[label=(\roman*)]

\item For all $\mathcal{H}$, we have $G=(\mathbb{Z}/2\mathbb{Z})^k$ for
some $k\in\mathbb{N}$ with $k(k+1)\leq n$. The classification in this case is given by \eqref{tables}.

\item The diagonal of the cell matrix is $2^k$, all other entries are $2^{l}$ for $l<k$.

\item Type $B_{5}$ is the first case not covered by previous classification results.

\item $B_{6}$ is the smallest example (i.e. the example of smallest rank)
in classical type where we have a non-cell simple transitive
$2$-representation; see below.

\end{enumerate}

Type $B_{5}$ is:
\begin{gather*}
\renewcommand{\arraystretch}{1.25}
\scalebox{0.75}{$\begin{tabular}{|C||C|C|C|C|C|C|C|C|C|C|C|C|C|C|C|C|}
\hline
\text{cell} & 0 & 1 & 2 & 3 & 4 & 5 & 6 & 7 & 7' & 6'
& 5' & 4' & 3' & 2' & 1' & 0' \\
\hline
\hline
\text{size} & 1 & 42 & 150 & 100 & 225 & 152 & 600 & 650 & 650 & 600 & 152 & 225 & 100 & 150 & 42 & 1 \\
\hline
\hline
\mathbf{a} & 0 & 1 & 2 & 3 & 3 & 4 & 4 & 5 & 6 & 7 & 9 & 10 & 10 & 11 & 16 & 25 \\
\hline
\hline
\cAH & 0 & 1 & 1 & 0 & 0 & 1 & 1
& 1 & 1 & 1 & 1 & 0 & 0
& 1 & 1 & 0 \\
\hline
\end{tabular}$}
\end{gather*}
Here and throughout: from left to right, we have listed the numbered cells,
paired $\mathcal{J}\leftrightsquigarrow\mathcal{J}^{\prime}=\mathcal{J}w_{0}$ (with $0$ being the minimal containing $1$
and $0'$ the maximal cell containing $w_{0}$). From top to bottom, we have listed their sizes, the $\mathbf{a}$-values, and $\cAH$, where
the number $k$ means that $\cAH\simeq\cV\mathrm{ect}(G)$ for $G=(\mathbb{Z}/2\mathbb{Z})^{k}$.

Type $B_{6}$ is:
\begin{gather*}
\renewcommand{\arraystretch}{1.25}
\scalebox{0.475}{$\begin{tabular}{|C||C|C|C|C|C|C|C|C|C|C|C|C|C|C|C|C|C|C|C|C|C|C|C|C|C|C|C|}
\hline
\text{cell} & 0 & 1 & 2 & 3 & 4 & 5 & 6 & 7
& 8 & 9 & 10 & 11 & 12{=}12' & 13{=}13'
& 11' & 10' & 9' & 8' & 7' & 6'
& 5' & 4' & 3' & 2' & 1' & 0' \\
\hline
\hline
\text{size} & 1 & 62 & 342 & 576 & 650 & 3150 & 350 & 1600 & 2432 & 3402
& 900 & 2025 & 14500 & 600 & 2025 & 900 & 3402 & 2432 & 1600 & 350
& 3150 & 650 & 576 & 342 & 62 & 1 \\
\hline
\hline
\mathbf{a} & 0 & 1 & 2 & 3 & 3 & 4 & 4 & 5 & 5 & 6
& 6 & 6 & 7 & 9 & 10 & 10 & 10 & 11 & 11 & 16
& 12 & 15 & 17 & 18 & 25 & 36 \\
\hline
\hline
\cAH & 0 & 1 & 1 & 0 & 1
& 1 & 1 & 0 & 1 & 1 & 0
& 0 & 2 & 1 & 0 & 0 & 1 & 1 & 0 & 1
& 1 & 1 & 0 & 1 & 1 & 0
\\
\hline
\end{tabular}$}
\end{gather*}

The cell $12$ is displayed in \eqref{eq:cellmatrix}. In this case, we have
$G=(\mathbb{Z}/2\mathbb{Z})^2$, which has the (non-conjugate) subgroups $1,K_{1},K_{2},K_{3},G$.
The subgroups $1$ and $K_{1}\cong K_{2}\cong K_{3}\cong\mathbb{Z}/2\mathbb{Z}$ all have trivial
second group cohomology, but
$H^2(G,\mathbb{C}^{\times})\cong\mathbb{Z}/2\mathbb{Z}$. Thus, we have six
equivalence classes of (graded) simple transitive $2$-representations of ranks $1$, $1$,
$2$, $2$, $2$ and $4$, respectively, see \eqref{tables}.
It follows from Theorem \ref{theorem:main}
that this case gives a non-cell simple transitive $2$-representation. The same happens repeatedly for higher ranks.

\subsubsection*{Type $D_n$}\label{appendix.4}

This type is \textbf{done} for all $n$:

\begin{enumerate}[label=(\roman*)]

\item[$($i$)${,}$($ii$)$] As for type $B_n$, but with $(k+1)^2\leq n$.

\item[$($iii$)$] Type $D_{4}$ is the first case not covered by previous results.

\item[$($iv$)$] In type $D_9$ the first non-cell simple transitive $2$-representations appear.

\end{enumerate}

\subsubsection*{Type $I_{2}(m)$}\label{appendix.5}

This type is \textbf{done} for all $m>2$:

\begin{enumerate}[label=(\roman*)]

\item Every $\mathcal{J}$ contains a $w_{0}^{\mathtt{I}}$; there are only three two-sided cells.

\item The bottom and top cell are strongly regular.

\item The middle cell has a dihedral cell matrix, see \eqref{eq:matrix-ab}, and the classification is given in \eqref{dihedral}.
This is the smallest example with
non-cell simple transitive $2$-representations, starting from type $I_{2}(6)=G_{2}$.

\end{enumerate}

\subsubsection*{Type $E_{6}$}\label{appendix.9}

This type is \textbf{done}:
\begin{gather*}
\renewcommand{\arraystretch}{1.25}
\scalebox{0.55}{$\begin{tabular}{|C||C|C|C|C|C|C|C|C|C|C|C|C|C|C|C|C|C|}
\hline
\text{cell} & 0 & 1 & 2 & 3 & 4 & 5 & 6 & 7
& 8{=}8' & 7' & 6' & 5' & 4' & 3'
& 2' & 1' & 0' \\
\hline
\hline
\text{size} & 1 & 36 & 400 & 1350 & 4096 & 3600 & 6561 & 576 & 18600 & 576 & 6561 & 3600 & 4096 & 1350 & 400 & 36 & 1 \\
\hline
\hline
\mathbf{a} & 0 & 1 & 2 & 3 & 4 & 5 & 6 & 6
& 7 & 12 & 10 & 11 & 13 & 15
& 20 & 25 & 36 \\
\hline
\hline
\cAH & 0 & 0 & 0 & 1 & 0
& 0 & 0 & 0 & \eqref{eq:e6appendix} & 0 & 0
& 0 & 0 & 1 & 0 & 0 & 0 \\
\hline
\end{tabular}$}
.
\end{gather*}
Here we write $k$ for $\cAH\simeq\cV\mathrm{ect}(G)$ with $G=(\mathbb{Z}/2\mathbb{Z})^{k}$.
The corresponding cells are strongly regular or nice, except
\begin{gather}\label{eq:e6appendix}
\renewcommand{\arraystretch}{1.25}
\scalebox{0.75}{$\begin{tabular}{|C|C|C|}
\hline
\cellcolor{blue!25} 3_{10,10} & 2_{50,10} & 1_{20,10}
\\
\hline
2_{10,50} & \cellcolor{blue!25} 3_{50,50} & 3_{20,50}
\\
\hline
1_{10,20} & 3_{50,20} & \cellcolor{blue!25} 6_{20,20}
\\
\hline
\end{tabular}$}
,\quad
\fcolorbox{black}{blue!25}{$3_{10,10}$}\colon
\cAH\simeq\cR\mathrm{ep}(S_{3}).
\end{gather}
Again, we get non-cell simple transitive $2$-representations, see \eqref{tables}.

\subsubsection*{Type $E_{7}$}\label{appendix.10}

This type is \textbf{done}, and it is quite similar to type $E_{6}$:
\begin{gather*}
\renewcommand{\arraystretch}{1.25}
\scalebox{0.55}{$\begin{tabular}{|C||C|C|C|C|C|C|C|C|C|C|C|C|C|C|C|C|C|C|}
\hline
\text{cell} & 0 & 1 & 2 & 3 & 4 & 5 & 6 & 7
& 8 & 9 & 10 & 11 & 12 & 13
& 14 & 15 & 16 & 17 \\
\hline
\hline
\text{size} & 1 & 49 & 729 & 4802 & 441 & 25650 & 35721 & 44100 & 11025 & 28224 & 35721 & 262150 & 246402 & 142884 & 44100 & 296352 & 11025 & 524288 \\
\hline
\hline
\mathbf{a} & 0 & 1 & 2 & 3 & 3 & 4 & 5 & 6 & 6 & 6 & 7 & 7 & 8 & 9 & 10 & 10 & 12 & 11 \\
\hline
\hline
\cAH & 0 & 0 & 0 & 1 & 0
& 1 & 0 & 0 & 0 & 0 & 0
& \eqref{eq:e6appendix} & 1 & 0 & 0 & 1 & 0 & \varsigma \\
\hline
\end{tabular}$}
\end{gather*}
\begin{gather*}
\scalebox{0.55}{$\begin{tabular}{|C||C|C|C|C|C|C|C|C|C|C|C|C|C|C|C|C|C|}
\hline
\text{cell} & 16' & 15' & 14' & 13' & 12' & 11' & 10'
& 9' & 8' & 7' & 6' & 5' & 4'
& 3' & 2' & 1' & 0' \\
\hline
\hline
\text{size} & 11025 & 296352 & 44100 & 142884 & 246402 & 262150 & 35721 & 28224 & 11025 & 44100 & 35721 & 25650 & 441 & 4802 & 729 & 49 & 1\\
\hline
\hline
\mathbf{a} & 15 & 13 & 13 & 14 & 15 & 16 & 20 & 21 & 21 & 21 & 22 & 25 & 36 & 30 & 37 & 46 & 63\\
\hline
\hline
\cAH & 0 & 1 & 0 & 0 & 1
& \eqref{eq:e6appendix} & 0 & 0 & 0 & 0 & 0
& 1 & 0 & 1 & 0 & 0 & 0 \\
\hline
\end{tabular}$}
.
\end{gather*}
Cell $17$ is exceptional with
$\cAH\simeq\cV\mathrm{ect}^{\varsigma}(\mathbb{Z}/2\mathbb{Z})$, see \cite[Theorem 1.1]{Os3}.
The cells $11$ and $11'$ are as in \eqref{eq:e6appendix} (with diagonals $3_{70,70}$, $3_{210,210}$
and $6_{35,35}$), giving
non-cell simple transitive $2$-representations. All
other cells are strongly regular or nice.

\subsubsection*{Type $E_{8}$}\label{appendix.11}

This type is \textbf{done}, and it is similar to type $E_{7}$:
\begin{gather*}
\renewcommand{\arraystretch}{1.25}
\scalebox{0.35}{$\begin{tabular}{|C||C|C|C|C|C|C|C|C|C|C|C|C|C|C|C|C|C|C|C|C|C|C|C|}
\hline
\text{cell} & 0 & 1 & 2 & 3 & 4 & 5 & 6 & 7
& 8 & 9 & 10 & 11 & 12 & 13
& 14 & 15 & 16 & 17 & 18 & 19 & 20 & 21 & 22{=}22' \\
\hline
\hline
\text{size} & 1 & 64 & 1225 & 20384 & 72200 & 313600 & 321489 & 740000 & 4986240 & 5696250
& 10497600 & 7768224 & 7683200 & 33554432 & 275625 & 29635200
& 12740000 & 20575296 & 8037225 & 36905625 & 17640000 & 47360000 & 4410000 \\
\hline
\hline
\mathbf{a} & 0 & 1 & 2 & 3 & 4 & 5 & 6 & 6 & 7 & 8
& 9 & 10 & 10 & 11 & 12 & 12
& 13 & 13 & 14 & 14 & 15 & 15 & 20 \\
\hline
\hline
\cAH & 0 & 0 & 0 & 1 & 1
& 0 & 0 & 1 & \eqref{eq:e6appendix} & \eqref{eq:e6appendix} & 0
& 1 & 1 & \varsigma & 0 & 1 & 1 & 0
& 0 & 0 & 0 & 1 & 0 \\
\hline
\end{tabular}$}
\end{gather*}
\begin{gather*}
\renewcommand{\arraystretch}{1.25}
\scalebox{0.35}{$\begin{tabular}{|C||C|C|C|C|C|C|C|C|C|C|C|C|C|C|C|C|C|C|C|C|C|C|C|}
\hline
\text{cell} & 23{=}23' & 21' & 20' & 19' & 18' & 17' & 16' & 15' & 14'
& 13' & 12' & 11' & 10' & 9' & 8'
& 7' & 6' & 5' & 4' & 3' & 2' & 1' & 0' \\
\hline
\hline
\text{size} & 202671840 & 47360000 & 17640000 & 36905625 & 8037225 & 20575296 & 12740000
& 29635200 & 275625 & 33554432 & 7683200 & 7768224 & 10497600 & 5696250 & 4986240 & 740000 & 321489 & 313600
& 72200 & 20384 & 1225 & 64 & 1 \\
\hline
\hline
\mathbf{a} & 16 & 21 & 21 & 22 & 22 & 23 & 25
& 24 & 36 & 26 & 28 & 30 & 31 & 32 & 37 & 42 & 46 & 47
& 52 & 63 & 74 & 91 & 120 \\
\hline
\hline
\cAH & \eqref{eq:biggestcell} & 1 & 0 & 0 & 0
& 0 & 1 & 1 & 0 & \varsigma & 1
& 1 & 0 & \eqref{eq:e6appendix} & \eqref{eq:e6appendix} & 1 & 0 & 0
& 1 & 1 & 0 & 0 & 0 \\
\hline
\end{tabular}$}
,
\end{gather*}
As before, the majority of cells are strongly regular or nice.
We also have exceptional cells with
$\cAH\simeq\cV\mathrm{ect}^{\varsigma}(\mathbb{Z}/2\mathbb{Z})$, see \cite[Theorem 1.1]{Os3},
and also cells as in \eqref{eq:e6appendix} (with diagonals $3_{448,448}$, $3_{896,896}$
and $6_{56,56}$, or $3_{175,175}$, $3_{875,875}$
and $6_{350,350}$), giving
non-cell, simple transitive $2$-representations. There is one remaining cell
with $\cAH\simeq\cR\mathrm{ep}(S_{5})$ for appropriate $\mathcal{H}$,
giving again non-cell simple transitive $2$-representations. Its cell matrix is
\begin{gather}\label{eq:biggestcell}
\begin{gathered}
\scalebox{0.8}{$\begin{tabular}{|C|C|C|C|C|C|C|}
\hline
\cellcolor{blue!25}\mathbf{7}_{420,420} & \mathbf{5}_{756,420} & \mathbf{6}_{1596,420} &
\mathbf{5}_{168,420} & \mathbf{3}_{378,420} & \mathbf{4}_{1092,420} & \mathbf{2}_{70,420}
\\
\hline
\mathbf{5}_{420,756} & \cellcolor{blue!25}\mathbf{8}_{756,756} & \mathbf{7}_{1596,756} &
\mathbf{7}_{168,756} & \mathbf{8}_{378,756} & \mathbf{8}_{1092,756} & \mathbf{7}_{70,756}
\\
\hline
\mathbf{6}_{420,1596} & \mathbf{7}_{756,1596} & \cellcolor{blue!25}\mathbf{12}_{1596,1596} &
\mathbf{8}_{168,1596} & \mathbf{9}_{378,1596} & \mathbf{13}_{1092,1596} & \mathbf{11}_{70,1596}
\\
\hline
\mathbf{5}_{420,168} & \mathbf{7}_{756,168} & \mathbf{8}_{1596,168} &
\cellcolor{blue!25}\mathbf{12}_{168,168} & \mathbf{7}_{378,168} & \mathbf{12}_{1092,168} & \mathbf{12}_{70,168}
\\
\hline
\mathbf{3}_{420,378} & \mathbf{8}_{756,378} & \mathbf{9}_{1596,378} &
\mathbf{7}_{168,378} & \cellcolor{blue!25}\mathbf{15}_{378,378} & \mathbf{14}_{1092,378} & \mathbf{19}_{70,378}
\\
\hline
\mathbf{4}_{420,1092} & \mathbf{8}_{756,1092} & \mathbf{13}_{1596,1092} &
\mathbf{12}_{168,1092} & \mathbf{14}_{378,1092} & \cellcolor{blue!25}\mathbf{21}_{1092,1092} & \mathbf{24}_{70,1092}
\\
\hline
\mathbf{2}_{420,70} & \mathbf{7}_{756,70} & \mathbf{11}_{1596,70} &
\mathbf{12}_{168,70} & \mathbf{19}_{378,70} & \mathbf{24}_{1092,70} & \cellcolor{blue!25}\mathbf{39}_{70,70}
\\
\hline
\end{tabular}$}
,
\\
\fcolorbox{black}{blue!25}{$7_{420,420}$}\colon
\cAH\cong\cR\mathrm{ep}(S_{5}).
\end{gathered}
\end{gather}
The simple transitive $2$-representations of $\cR\mathrm{ep}(S_{5})$ can be obtained from \eqref{tables}.

\subsubsection*{Type $F_{4}$}\label{appendix.7}

This type is \textbf{done}:
\begin{gather*}
\renewcommand{\arraystretch}{1.25}
\scalebox{0.75}{$\begin{tabular}{|C||C|C|C|C|C|C|C|C|C|C|C|}
\hline
\text{cell} & 0 & 1 & 2 & 3 & 4 & 5{=}5' & 4' & 3' & 2' & 1' & 0' \\
\hline
\hline
\text{size} & 1 & 24 & 81 & 64 & 64 & 684 & 64 & 64 & 81 & 24 & 1 \\
\hline
\hline
\mathbf{a} & 0 & 1 & 2 & 3 & 3 & 4 & 9 & 9 & 10 & 13 & 24 \\
\hline
\hline
\cAH & 0 & 1 & 0 & 0 & 0 & \eqref{eq:f4} & 0
& 0 & 0 & 1 & 0 \\
\hline
\end{tabular}$}
\end{gather*}
where we write $k$
for $\cAH\simeq\cV\mathrm{ect}(G)$ with $G=(\mathbb{Z}/2\mathbb{Z})^{k}$, as before,
and the cell matrices are strongly regular or nice, see
\eqref{eq:matrix-ab}.
In the remaining case we have (for appropriate $\mathcal{H}$):
\begin{gather}\label{eq:f4}
\renewcommand{\arraystretch}{1.25}
\scalebox{0.75}{$\begin{tabular}{|C|C|C|C|C|}
\hline
\cellcolor{blue!25} 5_{3,3} & 3_{3,3} & 4_{3,4} & 5_{3,1} & 2_{3,1}\\
\hline
3_{3,3} & \cellcolor{blue!25} 5_{3,3} & 4_{3,4} & 2_{3,1} & 5_{3,1}\\
\hline
4_{4,3} & 4_{4,3} & \cellcolor{blue!25} 9_{4,4} & 6_{4,1} & 6_{4,1}\\
\hline
5_{1,3} & 2_{1,3} & 6_{1,4} & \cellcolor{blue!25} 9_{1,1} & 3_{1,1}\\
\hline
2_{1,3} & 5_{1,3} & 6_{1,4} & 3_{1,1} & \cellcolor{blue!25} 9_{1,1}\\
\hline
\end{tabular}$}
,\quad
\fcolorbox{black}{blue!25}{$5_{3,3}$}\colon
\cAH\simeq\cR\mathrm{ep}(S_{4}).
\end{gather}
For the list of equivalence classes of simple transitive
$2$-representations of $\cR\mathrm{ep}(S_{4})$, see \eqref{tables}.
This is the second smallest example in Weyl type with a
non-cell simple transitive $2$-representation.

\subsubsection*{Type $H_{3}$}\label{appendix.6}

This type needs \textbf{more} work:
\begin{gather*}
\renewcommand{\arraystretch}{1.25}
\scalebox{0.75}{$\begin{tabular}{|C||C|C|C|C|C|C|C|}
\hline
\text{cell} & 0 & 1 & 2 & 3{=}3' & 2' & 1' & 0' \\
\hline
\hline
\text{size} & 1 & 18 & 25 & 32 & 25 & 18 & 1 \\
\hline
\hline
\mathbf{a} & 0 & 1 & 2 & 3 & 5 & 6 & 15 \\
\hline
\hline
\cAH & (a) & (b) & (a) & \cellcolor{red!25}(c) & (a) & \cellcolor{red!25}(b) & (a) \\
\hline
\end{tabular}$}
\end{gather*}

\begin{enumerate}[label=(\alph*)]	

\item These cases are strongly regular two-sided cells.

\item In these cases, the cell is $2_{3,3}$, and the Grothendieck rings of $\cAH$
and $\cS\cO(3)_{3}$ coincide.

\item Cell $3$ is $2_{4,4}$, and the Grothendieck rings of $\cAH$
and $\cV\mathrm{ect}(\mathbb{Z}/2\mathbb{Z})$ coincide.
\end{enumerate}

By \cite[Subsection 2.5]{Os1}, the only two possibilities for (b) are $\cAH\simeq\cS\cO(3)_{3}$ or
$\cAH\simeq\mathcal{M}(2,5)$ (in the notation of Ostrik).
Similarly, by \cite[Subsection 2.4]{Os1}, the only two possibilities for (c) are
$\cAH\simeq\cV\mathrm{ect}(\mathbb{Z}/2\mathbb{Z})$ or
$\cAH\simeq\cV\mathrm{ect}^{\varsigma}(\mathbb{Z}/2\mathbb{Z})$.
However, only in the case of the cell $1$ do we know which option it is, namely $\cAH\simeq\cS\cO(3)_{3}$,
since this case is covered by \cite[Theorem 28]{KMMZ}. In this case the classification is as in \eqref{dihedral}.

\subsubsection*{Type $H_{4}$}\label{appendix.8}

This type needs \textbf{much more} work:
\begin{gather*}
\renewcommand{\arraystretch}{1.25}
\scalebox{0.75}{$\begin{tabular}{|C||C|C|C|C|C|C|C|C|C|C|C|C|C|}
\hline
\text{cell} & 0 & 1 & 2 & 3 & 4 & 5 & 6{=}6' & 5' & 4' & 3' & 2' & 1' & 0' \\
\hline
\hline
\text{size} & 1 & 32 & 162 & 512 & 625 & 1296 & 9144 & 1296 & 625 & 512 & 162 & 32 & 1 \\
\hline
\hline
\mathbf{a} & 0 & 1 & 2 & 3 & 4 & 5 & 6 & 15 & 16 & 18 & 22 & 31 & 60 \\
\hline
\hline
\cAH & (a) & (b) & \cellcolor{red!25}(b) & \cellcolor{red!25}(c) & (a) & (a) & \cellcolor{red!25}\eqref{eq:h4}
& (a) & (a) & \cellcolor{red!25}(c) & \cellcolor{red!25}(b) & \cellcolor{red!25}(b) & (a) \\
\hline
\end{tabular}$}
\end{gather*}
(a),(b),(c) are similar to (a),(b),(c) in type $H_{3}$, and the same remark as in (d) holds.
In the remaining case we have:
\begin{gather}\label{eq:h4}
\renewcommand{\arraystretch}{1.25}
\scalebox{0.75}{$\begin{tabular}{|C|C|C|}
\hline
\cellcolor{blue!25} 14_{8,8} & 13_{10,8} & 14_{6,8} \\
\hline
13_{8,10} & \cellcolor{blue!25} 18_{10,10} & 18_{6,10} \\
\hline
14_{8,6} & 18_{10,6} & \cellcolor{blue!25} 24_{6,6} \\
\hline
\end{tabular}$}
.
\end{gather}
We were not able to find $\cAH$ in the literature.
In fact, for none of the diagonal $\mathcal{H}$-cells do we know what $\cAH$ is;
we only know the multiplication tables of their Grothendieck rings with respect to the
asymptotic Kazhdan--Lusztig bases $\{a_{w}\mid w\in\mathcal{H}\}$, see also \cite{Al}.
For example, if $\mathcal{H}$ is in the
$14$-by-$14$ block, then the Grothendieck ring of $\cAH$ is not
commutative, $\cAH$ has Perron--Frobenius
dimension $120(9+4\sqrt{5})$ and a simple generating $1$-morphism of Perron--Frobenius
dimension $1+\sqrt{5}$ and fusion graph
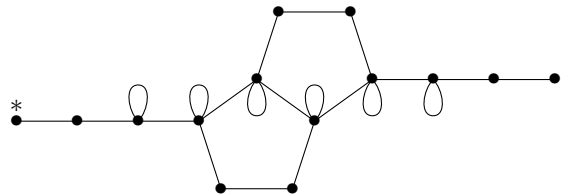
\begin{gather*}
\begin{tikzpicture}[anchorbase,scale=.8]
\draw[black] (0,1) to (.95,.31) to (.59,-.81) to (-.59,-.81) to (-.95,.31) to (0,1);
\draw[black] (0,1) to (.36,2.12) to (1.54,2.12) to (1.9,1) to (.95,.31);
\draw[black] (-.95,.31) to (-1.95,.31) to (-2.95,.31) to (-3.95,.31) node[above]{$*$};
\draw[black] (1.9,1) to (2.9,1) to (3.9,1) to (4.9,1);
\draw[black] (-.95,.31) to[out=135,in=180] (-.95,.91) to[out=0,in=45] (-.95,.31);
\draw[black] (-1.95,.31) to[out=135,in=180] (-1.95,.91) to[out=0,in=45] (-1.95,.31);
\draw[black] (.95,.31) to[out=135,in=180] (.95,.91) to[out=0,in=45] (.95,.31);
\draw[black] (0,1) to[out=225,in=180] (0,.4) to[out=0,in=315] (0,1);
\draw[black] (1.9,1) to[out=225,in=180] (1.9,.4) to[out=0,in=315] (1.9,1);
\draw[black] (2.9,1) to[out=225,in=180] (2.9,.4) to[out=0,in=315] (2.9,1);
\node at (0,1) {$\bullet$};
\node at (.95,.31) {$\bullet$};
\node at (-.95,.31) {$\bullet$};
\node at (.59,-.81) {$\bullet$};
\node at (-.59,-.81) {$\bullet$};
\node at (1.9,1) {$\bullet$};
\node at (1.54,2.12) {$\bullet$};
\node at (.36,2.12) {$\bullet$};
\node at (-1.95,.31) {$\bullet$};
\node at (-2.95,.31) {$\bullet$};
\node at (-3.95,.31) {$\bullet$};
\node at (2.9,1) {$\bullet$};
\node at (3.9,1) {$\bullet$};
\node at (4.9,1) {$\bullet$};
\end{tikzpicture}.
\end{gather*}

\begin{remark}
Note that the two-sided cell $6$ in type $H_{4}$, whose structure is detailed in \eqref{eq:h4},
contains a diagonal $\mathcal{H}$-cell
for which the fusion algebra $\mathsf{A}_{\mathcal{H}}$ is non-commutative. This implies, of course,
that the pivotal fusion category $\cAH$, whose precise structure
is unknown, is not braided. For completeness, we list all non-commutative $\mathsf{A}_{\mathcal{H}}$ in all finite Coxeter types below:
\begin{enumerate}[\textbullet]

\item In types $E_{6}$, the $\mathcal{H}$ of order 6 in the two-sided cell $8$.

\item In type $E_{7}$, the $\mathcal{H}$ of order 6 in the two-sided cells $11$ and $11^{\prime}$.

\item In type $E_{8}$,  the $\mathcal{H}$ of order 6 in the two-sided cells
$8,8^{\prime},9,9^{\prime}$, and the $\mathcal{H}$ of order $>8$
in the two-sided cell $23$.

\item In type $F_{4}$, the $\mathcal{H}$ of order 9 in the two-sided cell $5$.

\item In type $H_{4}$, all $\mathcal{H}$ in the two-sided cell $6$.

\end{enumerate}
\end{remark}


\vspace{2mm}

M.M.: Center for Mathematical Analysis, Geometry, and Dynamical Systems, Departamento de Matem{\'a}tica,
Instituto Superior T{\'e}cnico, 1049-001 Lisboa, PORTUGAL \& Departamento de Matem{\'a}tica, FCT,
Universidade do Algarve, Campus de Gambelas, 8005-139 Faro, PORTUGAL
\newline email: {\tt mmackaay\symbol{64}ualg.pt}

Vo.Ma.: Department of Mathematics, Uppsala University, Box. 480,
SE-75106, Uppsala, SWEDEN
\newline email: {\tt mazor\symbol{64}math.uu.se}

Va.Mi.: School of Mathematics, University of East Anglia,
Norwich NR4 7TJ, UK
\newline email: {\tt v.miemietz\symbol{64}uea.ac.uk}

D.T.: The University of Sydney,
School of Mathematics and Statistics F07,
Office Carslaw 827,
NSW 2006, AUSTRALIA, www.dtubbenhauer.com
\newline email: {\tt daniel.tubbenhauer\symbol{64}sydney.edu.au}

X.Z.: Beijing Advanced Innovation Center for Imaging Theory and Technology, Academy for Multidisciplinary Studies, Capital Normal University, Beijing 100048, CHINA
\newline email: {\tt xiaoting.zhang09\symbol{64}hotmail.com}

\end{document}